\documentclass[10pt, oneside]{amsart}
\usepackage[lmargin=1in,rmargin=1in, bmargin=1in, tmargin=1in]{geometry}
\usepackage{amsmath,amsthm,amssymb}
\usepackage{tikz-cd}
\usepackage{tikz}
\usepackage{todonotes}
\usepackage{mathrsfs}
\usepackage{extarrows}
\usepackage{graphicx}
\usepackage{mathtools}
\usepackage{adjustbox}
\usepackage{colonequals}
\usepackage{stmaryrd}

\usepackage[pagebackref, hidelinks]{hyperref}
\usepackage{IEEEtrantools}
\usepackage[utf8]{inputenc}
\usepackage[english]{babel}
\usepackage{comment}
\usepackage{csquotes}
\usepackage[shortlabels]{enumitem}
\usepackage{bbm}

\newcommand{\QQ}{\mathbb{Q}}
\newcommand{\defeq}{\colonequals}
\renewcommand{\tilde}{\widetilde}  
\DeclareMathOperator{\Gal}{Gal}

\numberwithin{equation}{subsection}

\newtheorem{theorem}[equation]{Theorem}
\newtheorem{proposition}[equation]{Proposition}
\newtheorem{lemma}[equation]{Lemma}
\newtheorem{corollary}[equation]{Corollary}
\newtheorem*{corollary*}{Corollary}

\newtheorem{assumption}[equation]{Assumption}

\newcounter{alphalabels}

\newtheorem{theoremx}[alphalabels]{Theorem}

\theoremstyle{definition}

\newtheorem{definition}[equation]{Definition}
\newtheorem{notation}[equation]{Notation}
\newtheorem{convention}[equation]{Convention}
\newtheorem{situation}[equation]{Situation}

\theoremstyle{remark}
\newtheorem{remark}[equation]{Remark}
\newtheorem{example}[equation]{Example}

\DeclareFontFamily{U}{wncy}{}
\DeclareFontShape{U}{wncy}{m}{n}{<->wncyr10}{}
\DeclareSymbolFont{mcy}{U}{wncy}{m}{n}
\DeclareMathSymbol{\sha}{\mathord}{mcy}{"58}

\newcommand{\ide}[1]{\mathfrak{#1}}
\newcommand{\mbb}[1]{\mathbb{#1}}

\newcommand{\opn}[1]{\operatorname{#1}}
\newcommand{\hatot}{\hat{\otimes}}
\newcommand{\mbf}[1]{\mathbf{#1}}
\newcommand{\tbyt}[4]{\left( \begin{array}{cc} #1 & #2 \\ #3 & #4 \end{array} \right)}

\newcommand{\Addresses}{{
  \bigskip
  \footnotesize

  (Graham) \textsc{Mathematical Institute, University of Oxford, Woodstock Road, Oxford OX2 6GG, United Kingdom}\par\nopagebreak
  \textit{E-mail address}: \texttt{andrew.graham@maths.ox.ac.uk}

  \medskip

  (Rockwood) \textsc{Department of Mathematics, King's College London, Strand, London WC2R 2LS, United Kingdom}\par\nopagebreak
  \textit{E-mail address}: \texttt{robert.rockwood@kcl.ac.uk}

}}

\newcommand{\Qpb}{\overline{\QQ}_p}

\title[Nearly higher Coleman theory and p-adic L-functions]{Nearly higher Coleman theory and p-adic L-functions for $\mathrm{GSp}(4) \times \mathrm{GL}(2)$ and $\mathrm{GSp}(4) \times \mathrm{GL}(2) \times \mathrm{GL}(2)$}
\author{Andrew Graham and Rob Rockwood}
\date{}

\begin{document}
\begin{abstract}
    We construct four-variable $p$-adic $L$-functions for the spin Galois representation of a Siegel modular form of genus 2 twisted by the Galois representation of a cuspidal modular form as the modular forms vary in Coleman families. The main ingredient is the construction of a space of nearly overconvergent modular forms in the coherent cohomology of the Siegel threefold, extending the spaces of overconvergent modular forms appearing in higher Coleman theory. In addition to this, we construct $p$-adic distributions interpolating the Gan--Gross--Prasad automorphic periods for $(\mathrm{GSpin}(5), \mathrm{GSpin}(4))$ which, conditional on the local and global Gan--Gross--Prasad conjectures for this pair of groups, provides a construction of ``square-root'' $p$-adic $L$-functions for $\mathrm{GSp}(4) \times \mathrm{GL}(2) \times \mathrm{GL}(2)$ as the automorphic forms vary in Coleman families.
\end{abstract}

\maketitle

\tableofcontents

\section{Introduction}

Over the past 30 years, there has been an extensive programme of study on Euler systems associated with triples of modular forms, including the Beilinson--Kato setting \cite{KatoES} (when two of the modular forms are Eisenstein series), the Beilinson--Flach setting \cite{BDRi, BDRii, BeilinsonFlach, KLZ2015, lz-coleman} (when one of the modular forms is an Eisenstein series), and the Gross--Kudla--Schoen setting \cite{darmon2014diagonal, darmon2017diagonal} (when all three modular forms are cuspidal). All three of these cases constitute a ``trilogy of Euler systems'' in the sense of \cite{trilogy}, and have led to several results towards the Bloch--Kato and Birch--Swinnerton-Dyer conjectures for modular forms, convolutions of modular forms, and elliptic curves twisted by Artin representations (as well as many other variants). Moreover, the theory of $p$-adic $L$-functions in this trilogy is incredibly rich, with different behaviour and constructions depending on many factors such as: signs of functional equations of $L$-functions, regions of interpolations, and slope behaviour of the modular forms at the prime $p$. This latter factor (that is, the availability of constructions of Euler systems and $p$-adic $L$-functions as the modular forms vary in Coleman families) is particularly important if one wants to understand the Bloch--Kato and Birch--Swinnerton-Dyer conjectures for every prime number $p$. 

More recently, there has been a substantial amount of work on generalising this trilogy of Euler systems to triples of automorphic representations of $\opn{GSp}_4$, $\opn{GL}_2$, and $\opn{GL}_2$ respectively. For example, if both of the $\opn{GL}_2$-representations are Eisenstein series then one obtains the Lemma--Flach Euler system in \cite{LSZ17}, and if one of the $\opn{GL}_2$-representations is cuspidal and the other is an Eisenstein series, then an Euler system has been constructed in \cite{hsu2020euler}. If one imposes a suitable ordinarity assumption on the automorphic representations, then these Euler systems can be used to prove cases of the Bloch--Kato conjecture \cite{LZBK20, LZBK21} and the Birch--Swinnerton-Dyer conjecture for abelian surfaces \cite{LZmodularabelian}, however the analogous results in the non-ordinary setting have not yet been established. 

One of the reasons why the non-ordinary setting is much harder for this trilogy (in contrast with the $\opn{GL}_2 \times \opn{GL}_2 \times \opn{GL}_2$ trilogy) is that one is often required to work with cohomology classes in higher degrees of coherent cohomology of Siegel threefolds (higher Coleman theory \cite{BoxerPilloni}) in order to construct the $p$-adic $L$-functions and establish the explicit reciprocity laws. Since families of \emph{real-analytic} Eisenstein series and Maass--Shimura differential operators naturally appear in these constructions, it is therefore necessary to have a good theory of nearly overconvergent automorphic forms in higher degrees of coherent cohomology. In addition to this, the methods for constructing the $p$-adic $L$-functions and Euler systems depend very much on the desired region of interpolation (partly because there is a smaller amount of symmetry for this trilogy); so far there is only one known Euler system construction in region (e) of \cite[Figure 2]{LZpadiclfunctionsdiagonalcycles}, which should be related to four different $p$-adic $L$-functions in regions (c), (d), (d'), (f) of the same figure (see below for the inequalities defining region (f)). 

The purpose of this article is to study the construction of $p$-adic $L$-functions in region (f) for the $\opn{GSp}_4 \times \opn{GL}_2 \times \opn{GL}_2$ trilogy in the non-ordinary setting. More precisely, in this article we construct $p$-adic $L$-functions for $\opn{GSp}_4 \times \opn{GL}_2$ and $\opn{GSp}_4 \times \opn{GL}_2 \times \opn{GL}_2$ in region (f) as the automorphic representations vary in Coleman families (with as much finite-slope variation as possible). The main novelty in the construction of these $p$-adic $L$-functions is a theory of nearly overconvergent automorphic forms for $\opn{GSp}_4$ in higher degrees of coherent cohomology, extending the spaces of overconvergent automorphic forms in higher Coleman theory. This construction is inspired by the theory of ``nearly sheaves'' in \cite{LPSZ} and recent work of Pilloni, Rodrigues, and the first named author \cite{DiffOps} on the $p$-adic interpolation of Maass--Shimura operators on nearly overconvergent modular forms. In fact, this latter work will also be a key ingredient in this article. In particular, we note that our construction of $p$-adic $L$-functions recovers \emph{all} other previous constructions of $p$-adic $L$-functions for $\opn{GSp}_4 \times \opn{GL}_2$ in region (f). 

In a sequel to this paper \cite{PartII-BK}, we will also show that the Euler system classes in \cite{LSZ17, hsu2020euler} in region (e) vary in Coleman families (with the maximal amount of variation), and prove an explicit reciprocity law relating the $p$-adic $L$-functions in this article to these Euler system classes. As a consequence, it will be possible to establish new cases of the Bloch--Kato conjecture for automorphic representations of $\opn{GSp}_4 \times \opn{GL}_2$ in the non-ordinary setting.

\begin{remark}
    To the best of the authors' knowledge, there are three other constructions of $p$-adic $L$-functions for $\opn{GSp}_4 \times \opn{GL}_2$ in regions different from (f). More precisely, there are constructions of three-variable $p$-adic $L$-functions in Coleman families for one half of region (d) in \cite{LRregiond}, and four-variable $p$-adic functions in Hida families for the whole region in \cite{LiuRegiondI, LiuRegiondII}. Furthermore, it is expected that the construction of a $p$-adic $L$-function in region (c) will appear in forthcoming work of Andreatta, Bertolini, Seveso, and Venerucci. In all of these cases, we are optimistic that (analogous versions of) the methods in the article can be employed to remove (some of) the restrictions and assumptions in these works (such as extending \cite{LRregiond} to a four-variable $p$-adic $L$-function, and extending \cite{LiuRegiondI, LiuRegiondII} to incorporate Coleman families).
\end{remark}

\subsection{Statement of main results}

Let us now describe the main results of this article. For simplicity of exposition, we suppose that all automorphic representations in this introduction are assumed to be unitary and we parameterise the weights of the automorphic representations in a slightly different way than from the rest of the article (see Remark \ref{IntroRemarkOnNormalisations}).

\subsubsection{$p$-adic $L$-functions for $\opn{GSp}_4 \times \opn{GL}_2$} \label{PAdicLsForCaseAIntroSec}

Let $\pi$ be a unitary discrete series cuspidal automorphic representation of $\opn{GSp}_4(\mbb{A})$ of weight $(k_1, k_2)$ (with $k_1 \geq k_2 \geq 3$) which is globally generic and non-CAP. We assume that $\pi$ is either non-endoscopic or of Yoshida-type. By the results of \cite{okazaki2019localarxiv}, $\pi$ is necessarily quasi-paramodular of level $(N, M)$ with $M^2 | N$, and the central character $\chi_{\pi}$ of $\pi$ is a Dirichlet character of conductor $M$. Let $\sigma$ denote the unitary automorphic representation of $\opn{GL}_2(\mbb{A})$ associated with a cuspidal newform of weight $c \geq 1$, level $\Gamma_1(N')$, and nebentypus $\chi_{\sigma}$. Let $r \colon {^L \opn{GSp}_4} \times {^L \opn{GL}_2} \to \opn{GL}_8(\mbb{C})$ denote the $8$-dimensional representation given by the tensor product of the $4$-dimensional spin representation of ${^L \opn{GSp}_4}$ and the standard representation of ${^L\opn{GL}_2}$. Fix an odd prime number $p \nmid NN'$. We are interested in interpolating special values of the $L$-function
\[
L(\pi \times \sigma, \chi, s) \defeq L((\pi \boxtimes \sigma) \otimes \chi, s; r)
\]
associated with the representation $r$, as $\chi$ varies through Dirichlet characters of $p$-th power conductor and $\pi$ and $\sigma$ vary in Coleman families. More precisely, we assume that $c \leq k_1 - k_2 + 1$ (so we are in region (f)) which implies that the critical values of the $L$-function $L(\pi \times \sigma, \chi, s)$ occur for $s = j - \tfrac{w}{2}$ with $j$ in the set:
\[
\opn{Crit}(\pi \times \sigma) \defeq [k_2+c-2, k_1-1] \cap \mbb{Z} 
\]
where $w \defeq k_1+k_2+c-4$.

To be able to $p$-adically interpolate these critical values, we need to impose some assumptions on the behaviour of the automorphic representations $\pi$ and $\sigma$ at the prime $p$. Recall that $p \nmid NN'$, so the local components $\pi_p$ and $\sigma_p$ are unramified. 
\begin{itemize}
    \item Let $P(X)$ denote the (degree $4$) Hecke polynomial at $p$ satisfying:
    \[
    P(p^{-s}) = L(\pi_p, s - \tfrac{w}{2})^{-1}
    \]
    where the right-hand side denotes the spin $L$-factor. This polynomial has coefficients in a number field, and we let $R \defeq \{ \alpha, \beta, \gamma, \delta \}$ denote the \emph{inverses} of the roots of $P(X)$ in an algebraic closure $\Qpb$ of $\mbb{Q}_p$, ordered so that $v_p(\alpha) \leq v_p(\beta) \leq v_p(\gamma) \leq v_p(\delta)$ and $\alpha \delta = \beta \gamma$ (where $v_p$ denotes the $p$-adic valuation normalised so that $v_p(p) = 1$). We assume that $\pi_p$ is $p$-regular (the roots of $P(X)$ are distinct) and that $\pi_p$ admits a ``small slope'' $p$-stabilisation, i.e., one has
    \[
    v_p(\alpha) < k_2-2, \quad \quad v_p(\alpha \beta) < k_1-1 .
    \]
    Under this assumption, we fix a small slope $p$-stabilisation $\tilde{\pi}$ of $\pi$, which is the data of an ordered pair $(\alpha_1, \alpha_2) \in (R \times R) - \{(\alpha, \delta), (\delta, \alpha), (\beta, \gamma), (\gamma, \beta) \}$ satisfying $v_p(\alpha_1) < k_2-2$ and $v_p(\alpha_1\alpha_2) < k_1 - 1$.
    \item We assume that $\sigma_p$ is $p$-regular, i.e., the roots of the Hecke polynomial at $p$
    \[
    Q(X) = 1 - a_p X + p^{c-1}\chi_{\sigma}(p)X^2
    \]
    are distinct. If $c=1$, then we assume that $\sigma$ does not admit real multiplication by a quadratic field in which $p$ splits. If $c \geq 2$, then we assume that the $p$-adic Galois representation associated with $\sigma$ does not split as a direct sum of one-dimensional representations when restricted to the decomposition group at $p$.\footnote{Both of these conditions imply that the weight map is \'{e}tale at the corresponding point on the Buzzard--Coleman--Mazur eigencurve.} In either case, we fix a ``noble'' $p$-stabilisation $\tilde{\sigma}$ of $\sigma$, which is the choice of an inverse root of $Q(X)$.
\end{itemize}

Let $\mathcal{W}$ denote the weight space parameterising continuous characters of $\mbb{Z}_p^{\times}$. Under the above assumptions, the weight map is \'{e}tale at the points of the eigenvarieties of $\opn{GSp}_4$ and $\opn{GL}_2$ corresponding to $\tilde{\pi}$ and $\tilde{\sigma}$ respectively. In particular, there exist Coleman families $\underline{\pi}$ and $\underline{\sigma}$ over open affinoid subspaces $U \subset \mathcal{W}^2$ and $V \subset \mathcal{W}$ respectively, passing through the points $\tilde{\pi}$ and $\tilde{\sigma}$. Moreover, there exists a Zariski dense subset $\Upsilon(\underline{\pi}) \subset U$ (resp. $\Upsilon(\underline{\sigma}) \subset V$) of classical weights\footnote{Points of the form $(k_1, k_2) \in \mbb{Z}^2$ (resp. $c \in \mbb{Z}$) satisfying $k_1 \geq k_2 \geq 3$ (resp. $c \geq 1$).} such that the specialisation $\underline{\pi}_{k_x}$ (resp. $\underline{\sigma}_{c_y}$) of $\underline{\pi}$ (resp. $\underline{\sigma}$) at the point $k_x = (k_{1, x}, k_{2, x}) \in \Upsilon(\underline{\pi})$ (resp. $c_y \in \Upsilon(\underline{\sigma})$) corresponds to a choice of small slope (resp. noble) $p$-stabilisation in a cuspidal automorphic representation of $\opn{GSp}_4(\mbb{A})$ (resp. $\opn{GL}_2(\mbb{A})$) satisfying the assumptions above. We refer the reader to \S \ref{ColemanFamiliesForBothSection} for more precise details. 

The first result we prove concerns the construction of four-variable $p$-adic $L$-functions associated with $\underline{\pi}$ and $\underline{\sigma}$. Let $\Upsilon(\underline{\pi}, \underline{\sigma}) \subset \Upsilon(\underline{\pi}) \times \Upsilon(\underline{\sigma})$ denote the subset of points $(k_x, c_y)$ satisfying $c_y \leq k_{1, x}-k_{2, x}+1$, which is Zariski dense in $U \times V$. Let $\mathscr{O}_{U \times V}$ denote the global sections of $U \times V$.

\begin{theoremx} \label{ThmAIntro}
    There exists a locally analytic distribution $\mathscr{L}_p \in \mathscr{D}^{\opn{la}}(\mbb{Z}_p^{\times}, \mathscr{O}_{U \times V})$ such that for all $z \defeq (k_x, c_y) \in \Upsilon(\underline{\pi}, \underline{\sigma})$, all integers $j \in \opn{Crit}(\underline{\pi}_{k_x}, \underline{\sigma}_{c_y})$, and all Dirichlet characters $\chi$ of conductor $p^{\beta}$ ($\beta \geq 0$), one has
    \[
    \mathscr{L}_p(k_x, c_y)(j+\chi) = \mathcal{Z}_S \cdot \mathcal{E}_p(\underline{\pi}_{k_x} \times \underline{\sigma}_{c_y}, \chi^{-1}, j-\tfrac{w}{2}) \cdot \frac{\Lambda^{\{p\}}(\underline{\pi}_{k_x} \times \underline{\sigma}_{c_y}, \chi^{-1}, j-\tfrac{w}{2})}{\Omega_x}
    \]
    where:
    \begin{itemize}
        \item $\Lambda^{\{p\}}(\cdots)$ denotes the completed $L$-function with the Euler factor at $p$ removed;
        \item $\Omega_x \in \mbb{C}^{\times}$ is a scalar independent of $j$, $\chi$, and $c_y$;
        \item $S$ denotes the set of primes dividing $NN'$, and $\mathcal{Z}_S = \mathcal{Z}_S(k_x, c_y; j+\chi)$ is a product of ratios of local zeta integrals and $L$-factors for primes in $S$ (see \S \ref{CaseALfunctionAndPeriodSSSec});
        \item $\mathcal{E}_p(\underline{\pi}_{k_x} \times \underline{\sigma}_{c_y}, \chi^{-1}, j-\tfrac{w}{2})$ denotes the modified $L$-factor at $p$ given by
        \[
        \prod_{a, b=1}^2 \frac{(1-p^{j-1}\alpha_{a, x}^{-1} \mu_{b,y}^{-1})}{(1-p^{-j}\alpha_{a, x} \mu_{b,y})}
        \]
        if $\beta = 0$, and given by
        \[
        G(\chi^{-1})^{-4} p^{4 \beta j} (\alpha_{1, x}\alpha_{2, x}\mu_{1, y} \mu_{2, y})^{-2\beta} 
        \]
        if $\beta \geq 1$. Here $(\alpha_{1, x}, \alpha_{2, x})$ denotes the $p$-stabilisation corresponding to $\underline{\pi}_{k_x}$, $(\mu_{1, y}, \mu_{2, y})$ are the inverses of the roots of the Hecke polynomial of $\underline{\sigma}_{c_y}$ at $p$, and $G(\chi^{-1})$ denotes the Gauss sum of $\chi^{-1}$. 
    \end{itemize}
    Moreover, the three-variable $p$-adic $L$-function on $U \times V$ given by $(k_x, c_y) \mapsto \mathscr{L}_p(k_x, c_y)(k_x-1)$ recovers the $p$-adic $L$-function in \cite{LZBK21}; and the one-variable $p$-adic $L$-function $\mathscr{L}_p(k, c) \in \mathscr{D}^{\opn{la}}(\mbb{Z}_p^{\times}, \mbb{C}_p)$ recovers the $p$-adic $L$-function in \cite{LPSZ} if $\pi_p$ is Klingen-ordinary. 
\end{theoremx}

\begin{remark} \label{IntroRemarkOnNormalisations}
    Our notation and normalisations in Theorem \ref{ThmAIntro} are slightly different from the main body of the article. More precisely, the parameters for the automorphic representations for $\opn{GSp}_4(\mbb{A})$ and $\opn{GL}_2(\mbb{A})$ will be $r=(r_1, r_2)$ and $t_2$ respectively, satisfying $(k_1, k_2) = (r_1+3, r_2+3)$ and $c = t_2+1$. Furthermore, we have shifted the variable in $\mathcal{W}$ in the interpolation formula by $r_2+t_2+2$; so in the main body of the article, we construct a $p$-adic $L$-function which satisfies the interpolation property
    \[
    \mathscr{L}_p(r_x, t_{2,y})(j + \chi) = (\star) \cdot \Lambda^{\{p\}}(\underline{\pi}_{k_x} \times \underline{\sigma}_{c_y}, \chi^{-1}, j + \tfrac{1-t_1}{2}), \quad \quad t_1 \defeq r_1-r_2-t_2,
    \]
    for $0 \leq j \leq t_1$. Finally, automorphic representation with lower-case letters will be normalised so that they are ``C-arithmetic'' (i.e., the corresponding Hecke eigensystem can be defined over a number field). We will denote the unitary normalisations of these automorphic representations with upper-case letters.
\end{remark}

\begin{remark}
    The $p$-adic $L$-function in Theorem \ref{ThmAIntro} depends on a choice of test data at primes in $S$, and it is known in many cases (and expected in general) that one can choose this test data such that $\mathcal{Z}_S(k, c; j) \neq 0$ for all $j \in \opn{Crit}(\pi, \sigma)$. For example, it is known if the central character of $\pi$ is a square (see Remark \ref{ZetaSgeneratedFractionalRem}). This property will suffice for the applications to the Bloch--Kato conjecture in \cite{PartII-BK}.

    Guaranteeing that $\mathcal{Z}_S(k_x, c_y; j) \neq 0$ as $(k_x, c_y)$ varies through $\Upsilon(\underline{\pi}, \underline{\sigma})$ is a much harder problem however. If we know that $\mathscr{L}_p(k,c) \neq 0$ and $v_p(\alpha_1 \alpha_2) < k_1 - c -2$, then after possibly shrinking $U \times V$ and using Theorem \ref{ThmBTemperedIntro} below, one can deduce that: for any $(k_x, c_y) \in \Upsilon(\underline{\pi}, \underline{\sigma})$, one has $\mathscr{L}_p(k_x, c_y)(j+\chi) \neq 0$ (and hence $\mathcal{Z}_S(k_x, c_y; j+\chi) \neq 0$) for all but finitely many $(j, \chi)$ with $j \in \opn{Crit}(\underline{\pi}_{k_x}, \underline{\sigma}_{c_y})$ and $\chi$ finite-order.
\end{remark}

As mentioned above (and explained in \S \ref{MethodOfProofIntroSec} below), the strategy for proving Theorem \ref{ThmAIntro} is to generalise the method in \cite{LZBK21} to incorporate families of nearly overconvergent automorphic forms in higher Coleman theory. Since this is manifestly a theory with $p$ inverted, we automatically lose control of the growth of $\mathscr{L}_p$ over the cyclotomic tower. On the other hand, the recent work of Boxer and Pilloni \cite{HHTBoxerPilloni} on \emph{higher Hida theory} for $\opn{GSp}_{2g}$ provides us with a natural lattice inside the spaces of overconvergent automorphic forms in higher Coleman theory. We exploit this property to show that the $p$-adic $L$-function in Theorem \ref{ThmAIntro} is tempered.

\begin{theoremx} \label{ThmBTemperedIntro}
    With notation as above, let $L_p(\pi \times \sigma, -) \in \mathscr{D}^{\opn{la}}(\mbb{Z}_p^{\times}, \mbb{C}_p)$ denote the specialisation of $\mathscr{L}_p$ at the point $((k_1, k_2), c)$. Let $h \in \mbb{Q}$ be such that $v_p(\alpha_1 \alpha_2) = h + k_2 - 2$. Then $L_p(\pi \times \sigma, -)$ is tempered with growth of order $\leq h+2$ in the sense of \cite[Definition 3.10]{barrera2021p}. In particular, if $h < k_1-k_2-c$ then $L_p(\pi \times \sigma, -)$ is uniquely determined by its interpolation property.    
\end{theoremx}

Unfortunately, this is not the most optimal result one expects: it should be the case that $L_p(\pi \times \sigma, -)$ is tempered with growth of order $\leq h$ so that, in particular, if $\pi_p$ is Klingen-ordinary (but not necessarily Siegel-ordinary), then $L_p(\pi \times \sigma, -)$ is a $p$-adic measure. The slightly sub-optimal bound of $h+2$ in Theorem \ref{ThmBTemperedIntro} arises from our lack of knowledge of the integrality properties of the nearly overconvergent family of Eisenstein series; we hope to improve this result in the near future.

\subsubsection{$p$-adic $L$-functions for $\opn{GSp}_4 \times \opn{GL}_2 \times \opn{GL}_2$}

We also obtain results towards $p$-adic $L$-functions for $\opn{GSp}_4 \times \opn{GL}_2 \times \opn{GL}_2$. More precisely, let $\pi$ be a cuspidal automorphic representation of weight $(k_1, k_2)$ and quasi-paramodular level $(N, M)$ satisfying the assumptions in \S \ref{PAdicLsForCaseAIntroSec} above (including the existence of a small slope $p$-stabilisation, which we fix). For $i=1, 2$, let $\sigma_i$ be cuspidal automorphic representations of $\opn{GL}_2(\mbb{A})$ of weights $c_i \geq 1$ and levels $\Gamma_1(N_i)$, satisfying the assumptions in \S \ref{PAdicLsForCaseAIntroSec}. We assume that $\chi_{\pi} \chi_{\sigma_1} \chi_{\sigma_2}=1$, and that $c_1+c_2 \leq k_1-k_2+2$ (so we are in region (f) of \cite[Figure 2]{LZpadiclfunctionsdiagonalcycles}).

We are interested in the ``triple-product'' $L$-function $L(\pi \times \sigma_1 \times \sigma_2, s)$ associated with the automorphic representation $\pi \boxtimes \sigma_1 \boxtimes \sigma_2$ and the $16$-dimensional representation of the dual group ${^L\opn{GSp}_4} \times {^L\opn{GL}_2} \times {^L\opn{GL}_2}$ obtained as the tensor product of the spin and standard representations. If we assume that the local signs for this $L$-function at all finite places are $+1$, then the (global) sign of the functional equation is $+1$ and one is led to study the central critical $L$-values $L(\pi \times \sigma_1 \times \sigma_2, 1/2)$ as $\pi, \sigma_1, \sigma_2$ vary in Coleman families.

Let $S$ denote the set of primes dividing $NN_1N_2$, and set $H = \opn{GL}_2 \times_{\mbb{G}_m} \opn{GL}_2 \hookrightarrow \opn{GSp}_4 =: G$. In this case, the Gan--Gross--Prasad conjecture for general spin groups \cite{emory2020global} predicts a relation between these central critical $L$-values and automorphic periods for $\pi \boxtimes \sigma_1 \boxtimes \sigma_2$. If we let $\varphi^{\opn{sph}} \in \pi$, and $\lambda_{i}^{\opn{sph}} \in \sigma_i$ denote the Whittaker new vectors, with $\varphi^{\opn{sph}}$ and  $\lambda_i^{\opn{sph}}$ specific vectors in the minimal $K_{\infty}$-type (depending on the weights $((k_1, k_2), c_1, c_2)$ -- see \S \ref{CaseBLfunctionAndPeriodSSec}), then the Gan--Gross--Prasad conjecture implies that $L(\pi \times \sigma_1 \times \sigma_2, 1/2) \neq 0$ if and only if there exists a pair $\gamma_S = (\gamma_{0, S}, \gamma_{1, S}) \in \opn{GSp}_4(\mbb{Q}_S) \times \opn{GL}_2(\mbb{Q}_S)$ such that
\[
\mathscr{P}^{\opn{sph}}(k, c_1, c_2; \gamma_S) \defeq \int_{H(\mbb{Q})Z_G(\mbb{A})\backslash H(\mbb{A})} \varphi^{\opn{sph}}(h \gamma_{0, S}) \left(\delta^{m}\lambda_{1}^{\opn{sph}}\right)(h_1 \gamma_{1, S}) \lambda_2^{\opn{sph}}(h_2) dh \; \neq \; 0 .
\]
Here $\delta$ denotes the Maass--Shimura differential operator, $(h_1, h_2) \in H(\mbb{A})$ is a general element, and $m = (k_1-k_2-c_1-c_2+2)/2$ (which is an integer by the condition on the central characters). Since the orthogonal Gan--Gross--Prasad conjecture is still open in general, we will therefore study the $p$-adic variation of the periods $\mathscr{P}^{\opn{sph}}(\cdots)$.

Let $\underline{\pi}$, $\underline{\sigma}_1$, $\underline{\sigma}_2$ be Coleman families over $U \subset \mathcal{W}^2$, $V_1 \subset \mathcal{W}$, $V_2 \subset \mathcal{W}$ passing through the fixed $p$-stabilisations $\tilde{\pi}$, $\tilde{\sigma}_1$, $\tilde{\sigma}_2$ respectively. Let $\Upsilon(\underline{\pi}) \subset U$, $\Upsilon(\underline{\sigma}_1) \subset V_1$, $\Upsilon(\underline{\sigma}_2) \subset V_2$ denote the Zariski dense subsets as in \S \ref{PAdicLsForCaseAIntroSec}, and set
\[
\Upsilon(\underline{\pi}, \underline{\sigma}_1, \underline{\sigma}_2) \defeq \{ (k_x, c_{y_1}, c_{y_2}) \in \Upsilon(\underline{\pi}) \times \Upsilon(\underline{\sigma}_1) \times \Upsilon(\underline{\sigma}_2) : c_{y_1} + c_{y_2} \leq k_{1, x} - k_{2, x} +2 \}
\]
which is Zariski dense in $U \times V_1 \times V_2$. Let $\mathscr{O}_{U \times V_1 \times V_2}$ denote the global sections of $U \times V_1 \times V_2$.

\begin{theoremx} \label{ThmCCaseBIntro}
    There exists a rigid-analytic function $\mathscr{L}_p \in \mathscr{O}_{U \times V_1 \times V_2}$ such that: for all $(k_x, c_{y_1}, c_{y_2}) \in \Upsilon(\underline{\pi}, \underline{\sigma}_1, \underline{\sigma}_2)$, one has 
    \[
    \mathscr{L}_p(k_x, c_{y_1}, c_{y_2}) = \mathcal{E}_p(\underline{\pi}_{k_x} \times \underline{\sigma}_{c_{y_1}} \times \underline{\sigma}_{c_{y_2}}, 1/2) \cdot \frac{\mathscr{P}^{\opn{sph}}(k_x, c_{y_1}, c_{y_2}; \gamma_S)}{\Omega_x}
    \]
    where:
    \begin{itemize}
        \item $\Omega_x \in \mbb{C}^{\times}$ is a scalar independent of $\underline{\sigma}_{c_{y_1}} \boxtimes \underline{\sigma}_{c_{y_2}}$ and $\gamma_S$;
        \item $\mathcal{E}_p(\underline{\pi}_{k_x} \times \underline{\sigma}_{c_{y_1}} \times \underline{\sigma}_{c_{y_2}}, 1/2)$ denotes the $p$-adic multiplier 
        \[
        \mathcal{E}_p(\underline{\pi}_{k_x} \times \underline{\sigma}_{c_{y_1}} \times \underline{\sigma}_{c_{y_2}}, 1/2) = \prod_{i, j, k}^2 \left( 1 - \frac{p^{-1/2}}{\alpha_{i,x}\mu_{j, y_1}\mu_{k, y_2}} \right)
        \]
        where $(\alpha_{1, x}, \alpha_{2, x})$ denotes the $p$-stabilisation corresponding to $\underline{\pi}_{k_x}$, and $(\mu_{1, y_a}, \mu_{2, y_a})$ are the inverses of the roots of the Hecke polynomial of $\underline{\sigma}_{c_{y_a}}$ at $p$ for $a=1,2$.
    \end{itemize}
    Moreover, if $\underline{\sigma}_1$ and $\underline{\sigma}_2$ are Hida families, then the two-variable slice $\mathscr{L}_p(k, -, -) \in \mathscr{O}_{V_1 \times V_2}$ coincides with the $p$-adic $L$-function in \cite[\S 5]{LZpadiclfunctionsdiagonalcycles} (after restricting the latter measure to $V_1 \times V_2 \subset \mathcal{W} \times \mathcal{W}$).
\end{theoremx}

\begin{remark}
    As in Remark \ref{IntroRemarkOnNormalisations}, the normalisations of the automorphic representations and weights in Theorem \ref{ThmCCaseBIntro} differ from the conventions in the main body of the article.
\end{remark}

\subsection{Method of proof} \label{MethodOfProofIntroSec}

Let us now sketch the main ingredients that go into the proofs of Theorem \ref{ThmAIntro}, Theorem \ref{ThmBTemperedIntro}, and Theorem \ref{ThmCCaseBIntro}. As in \cite{LPSZ}, the starting point for Theorem \ref{ThmAIntro} is an integral formula for the critical $L$-values due to Novodvorsky \cite{Novodvorsky}, which involves an automorphic period of the shape:
\begin{equation} \label{AutPerInIntroSec}
\int_{H(\mbb{Q})Z_G(\mbb{A})\backslash H(\mbb{A})} \varphi(h) \opn{Eis}(h_1) \lambda(h_2) dh
\end{equation}
where $\varphi$ and $\lambda$ are cuspidal automorphic forms for $\opn{GSp}_4(\mbb{A})$ and $\opn{GL}_2(\mbb{A})$ respectively, and $\opn{Eis}$ is a certain Eisenstein series. For this integral to be non-zero, we must assume that $\varphi$ is generic; and to obtain all critical $L$-values in general, it is necessary to work with real-analytic (but not necessarily holomorphic) Eisenstein series. To be able to interpret this automorphic period as a pairing in coherent cohomology, one therefore needs to not only view $\varphi$ as a $\opn{H}^2$ coherent cohomology class for the Siegel threefold, but also lift this cohomology class to a degree $2$ ``nearly holomorphic class'' (in order to pair with the real-analytic Eisenstein series). This is precisely the strategy in \cite{LPSZ} -- roughly, one can interpret the period in (\ref{AutPerInIntroSec}) as the value of a pairing 
\begin{equation} \label{TrilinearPerInIntro}
\opn{H}^2\left( X_G, \mathcal{N}^{\opn{hol}}_G(-D) \right) \otimes \opn{H}^0\left( X_{\opn{GL}_2}, \mathcal{N}^{\opn{hol}}_{\opn{GL}_2} \right) \otimes \opn{H}^0\left( X_{\opn{GL}_2}, \mathcal{M}^{\opn{hol}}_{\opn{GL}_2} \right) \to \mbb{C}
\end{equation}
where $\mathcal{N}^{\opn{hol}}_G(-D)$ denotes a certain sheaf of cuspidal nearly holomorphic forms on the Siegel threefold $X_G$, and $\mathcal{N}^{\opn{hol}}_{\opn{GL}_2}$ (resp. $\mathcal{M}^{\opn{hol}}_{\opn{GL}_2}$) is a certain sheaf of nearly holomorphic (resp. holomorphic) modular forms on the modular curve $X_{\opn{GL}_2}$. For this pairing to make sense, the weights of the (Siegel) modular forms must be compatible in an appropriate sense; this more or less amounts to the conditions defining the interpolation region (f). 

The strategy for constructing the $p$-adic $L$-function in Theorem \ref{ThmAIntro} is to then $p$-adically interpolate the pairing (and its inputs) in (\ref{TrilinearPerInIntro}). Under a \emph{Klingen-ordinarity} assumption on the Siegel modular form, this is accomplished in \cite{LPSZ} since one can work with higher Hida theory groups of $p$-adic modular forms. In the non-ordinarity setting however, one is forced to construct a $p$-adic version of (\ref{TrilinearPerInIntro}) involving spaces of \emph{nearly overconvergent} (Siegel) modular forms. This presents two complications:
\begin{itemize}
    \item Can one construct a nearly overconvergent family of Eisenstein series which $p$-adically interpolates the real-analytic Eisenstein series appearing in Novodvorsky's integral above?
    \item Can one define spaces of nearly overconvergent Siegel modular forms in $\opn{H}^2$ coherent cohomology of the Siegel threefold, extending the spaces of overconvergent forms appearing in Higher Coleman Theory \cite{BoxerPilloni} and the space of nearly holomorphic forms above?
\end{itemize}
If one can answer both of these questions in the positive, then it is possible (following a similar strategy as in \cite{LZBK21}) to construct the $p$-adic $L$-function in Theorem \ref{ThmAIntro}. 

For the first question, one can show using recent work of Pilloni, Rodrigues, and the first named author \cite{DiffOps} that the Eisenstein measure of Katz is nearly overconvergent. Here the space of nearly overconvergent modular forms $\mathscr{N}^{\dagger}_{\opn{GL}_2}$ is in the sense of \emph{op.cit.}, that is, overconvergent functions on the Katz Igusa tower inside the torsor $P_{\opn{GL}_2,\opn{dR}}$ parameterising trivialisations of the first de Rham homology of the universal elliptic curve respecting the Hodge filtration. To address the second question, we therefore need to construct a space of nearly overconvergent Siegel modular forms which ``pairs'' with $\mathscr{N}^{\dagger}_{\opn{GL}_2}$. The construction of this space is inspired by the theory of ``nearly sheaves'' in \cite{LPSZ} and takes place in \S \ref{padicRepTheorySubSec}, \S \ref{CohomologySubSec} -- one key ingredient in defining this space is the existence of reductions of structure of the de Rham torsor $P_{G, \opn{dR}}$ (see Proposition \ref{PGdrnReductionProposition}). In addition to this, the ideas involved in both of these aspects apply equally well to the triple product setting, leading to the construction of the $p$-adic $L$-function in Theorem \ref{ThmCCaseBIntro}.

This leaves the question of temperedness in Theorem \ref{ThmBTemperedIntro}. For this, we carefully study the integrality properties of the nearly overconvergent family of Eisenstein series, and the growth of the denominators of the cohomology classes associated with the Siegel modular form as one deepens the level at $p$. By the optimal slope bounds for the Hecke operators at $p$ in \cite{HHTBoxerPilloni}, it turns out that the latter denominators are precisely controlled by the $p$-adic valuation of the Klingen $U_p$-eigenvalue, which explains the presence of $h = v_p(\alpha_1 \alpha_2) - k_2+2$ in Theorem \ref{ThmBTemperedIntro}. Unfortunately our result is not the most optimal growth one expects -- this is related to our knowledge of the integrality of the nearly overconvergent family of Eisenstein series. If one can prove that this ``Eisenstein distribution'' is in fact a measure (which we believe should be the case), then one would be able to establish the optimal growth bounds for the $p$-adic $L$-function $L_p(\pi \times \sigma, -)$. We hope to revisit this question in future work.

\subsection{Acknowledgements}

The authors would like to thank David Loeffler and Sarah Zerbes for helpful discussions surrounding this article. Part of this work was completed at Max-Planck-Institut f\"{u}r Mathematik in April 2024, and the authors thank the institute for their hospitality. This research was (partly) funded by UK Research and Innovation grant MR/V021931/1. For the purpose of Open Access, the authors have applied a CC BY public copyright licence to any Author Accepted Manuscript (AAM) version arising from this submission.

\subsection{Notation}

We fix the following notations and conventions throughout the article:
\begin{itemize}
    \item We fix a prime number $p > 2$ throughout.
    \item If $\chi$ is a Dirichlet character, then we let $\widehat{\chi} \colon \mbb{A}^{\times} \to \mbb{C}^{\times}$ denote the associated algebraic Hecke character. In particular, for all but finitely many primes $\ell$, one has $\widehat{\chi}(\varpi_{\ell}) = \chi(\ell)$ for any choice of uniformiser $\varpi_{\ell} \in \mbb{Z}_{\ell}$. If the conductor of $\chi$ is a power of $p$, then we often let $c(\chi) \geq 0$ denote the integer such that $p^{c(\chi)}$ is the conductor of $\chi$. Note that, in this setting, one has $\widehat{\chi}|_{\mbb{Z}_p^{\times}} = \chi^{-1}$.
    \item If $A$ is a Banach $\mbb{Q}_p$-algebra and $f \colon \mbb{Z}_p \to A$ a continuous function, then we say $f$ is $r$-analytic if $z \mapsto f(a + p^r z)$ is given by a power series in $z$ for any $a \in \mbb{Z}_p$. For any compact open subset $U \subset \mbb{Z}_p$ and continuous function $f \colon U \to A$, we say $f$ is $r$-analytic if its extension-by-zero to $\mbb{Z}_p$ is $r$-analytic. For example, if $f \colon \mbb{Z}_p^{\times} \to A^{\times}$ is a continuous character, then $f$ is $r$-analytic (for $r \geq 1$) if $z \mapsto f(a + p^r z)$ is given by a power series in $z$ for any $a \in \mbb{Z}_p^{\times}$ (this convention occasionally differs from the literature). This definition extends to continuous functions on finite products of $\mbb{Z}_p^{\times}$ by requiring that it is $r$-analytic on each factor.
    \item We let $|\!| - |\!| \colon \mbb{A} \to \mbb{C}$ denote the adelic norm. 
    \item We say a compact open subgroup $K \subset G(\mbb{A}_f)$ is sufficiently small (or neat) if it satisfies the conditions in \cite[Definition B.6]{ACES}.
    \item If $\mathcal{X}$ is an adic space over $\opn{Spa}(\mbb{Q}_p, \mbb{Z}_p)$ and $U \subset \mathcal{X}$ is an open subspace, then we say that an open subspace $V \subset \mathcal{X}$ is a strict neighbourhood of $U$ in $\mathcal{X}$ if $V$ contains the closure (as topological spaces) of $U$ in $\mathcal{X}$. We will sometimes denote this by $U \Subset V$.
    \item If $X$ is a $p$-adic manifold, then we let $C_c^{\infty}(X)$ denote the space of Schwartz functions -- that is, locally constant and compactly supported functions $X \to \mbb{C}$. If $Y \subset X$ is a compact open subspace, then we let $\opn{ch}(Y) \in C_c^{\infty}(X)$ denote the characteristic function of $Y$.
\end{itemize}

\section{Groups and representations}

In this section we introduce the reductive groups underlying the various Shimura varieties we will consider in this article. In particular, we will discuss the various branching laws between these groups that will be needed for the construction of the $p$-adic $L$-functions.

\subsection{Algebraic groups} \label{AlgGroupsSubSec}

Throughout the whole article, we take $G \defeq \opn{GSp}_4$ to be the general symplectic group in four variables. Explicitly, we view
\[
G = \left\{ g \in \opn{GL}_4 : g^t J g = s(g) \cdot J \text{ for some } s(g) \in \mbb{G}_m \right\}, \quad J \defeq \left( \begin{smallmatrix} 0 & 0 & 0 & 1 \\ 0 & 0 & 1 & 0 \\ 0 & -1 & 0 & 0 \\ -1 & 0 & 0 & 0 \end{smallmatrix} \right),
\]
where $s \colon G \to \mbb{G}_m$ denotes the similitude character. We let $B_G \subset G$ denote the standard upper-triangular Borel subgroup, and $T \subset B_G$ the diagonal torus. In particular, we view elements of $T$ as diagonal matrices $\opn{diag}(z_1, z_2, sz_2^{-1}, sz_1^{-1})$ for $z_1, z_2, s \in \mbb{G}_m$. Via this presentation, we can identify algebraic characters $\kappa \in X^*(T)$ of the torus $T$ with tuples of integers $(r_1, r_2; c)$ satisfying
\[
\kappa(\opn{diag}(z_1, z_2, sz_2^{-1}, sz_1^{-1})) = z_1^{r_1} z_2^{r_2} s^c .
\]

\begin{remark}
    This presentation of $T$ and identification of $X^*(T)$ with tuples of integers is consistent with \cite{DiffOps}, however note that this convention differs from \cite{LSZ17, LPSZ, LZBK21}. We prefer to use the convention above since one is not required to fix a ``square-root'' later on when discussing $p$-adic characters of $T(\mbb{Z}_p)$. This downside to this choice is that the action of the Weyl group on $X^*(T)$ has a more complicated description; for the convenience of the reader, we explicitly describe this Weyl group action in Remark \ref{rem:ActionOfWeylGroup} below.
\end{remark}

Let $H \defeq \opn{GL}_2 \times_{\mbb{G}_m} \opn{GL}_2$ where the fibre product is over the determinant map. One has a natural embedding $\iota \colon H \hookrightarrow G$ given by
\[
\iota\left( \left( \begin{smallmatrix} a & b \\ c & d \end{smallmatrix} \right), \left( \begin{smallmatrix} a' & b' \\ c' & d' \end{smallmatrix} \right) \right) = \left( \begin{smallmatrix} a &  &  & b \\  & a' & b' &  \\  & c' & d' &  \\ c &  &  & d \end{smallmatrix} \right) .
\]
We let $B_H = B_G \cap H$ denote the standard upper-triangular Borel subgroup of $H$. Finally, we let $B_{\opn{GL}_2}$ denote the upper-triangular Borel subgroup of $\opn{GL}_2$, and let $T_{\opn{GL}_2}$ denote the diagonal torus (which we identify with matrices of the form $\opn{diag}(z, s z^{-1})$). Algebraic characters of $T_{\opn{GL}_2}$ will be denoted by tuples of integers $(r; c)$ corresponding to the character $\opn{diag}(z, s z^{-1}) \mapsto z^r s^c$.

\subsubsection{Parabolic subgroups and Weyl groups}

We let $P_G \subset G$ denote the upper-triangular Siegel parabolic subgroup, i.e., $P_G$ is the subgroup of matrices of the form $\left( \begin{smallmatrix} * & * & * & * \\ * & * & * & * \\ & & * & * \\ & & * & * \end{smallmatrix}  \right)$. We denote the standard block diagonal Levi subgroup of $P_G$ by $M_G$. For consistency, we also set $P_H = B_H$, $M_H = T$ and $P_{\opn{GL}_2} = B_{\opn{GL}_2}$, $M_{\opn{GL}_2} = T_{\opn{GL}_2}$. We let $B_{M_G}$ denote the upper-triangular Borel subgroup of $M_G$. We will write $\overline{(-)}$ to mean the opposite subgroup with respect to the fixed choices of maximal tori.

For any split reductive group $\mathscr{G}$, we let $W_{\mathscr{G}}$ denote its Weyl group. We let ${^MW_G} \subset W_G$ denote the set of minimal length representatives of the quotient $W_{M_G} \backslash W_G$ -- we can write ${^M W_G} = \{ w_0, w_1, w_2, w_3 \}$ where the length of $w_i$ is $i$. An explicit description of these elements is given in \cite[\S 2.1]{LZBK21}; in particular, we take 
\[
w_1 = \left( \begin{smallmatrix} 1 & & & \\ & & 1 & \\ & -1 & &  \\ & & & 1 \end{smallmatrix} \right) \in G .
\]
The action of $W_G$ on $X^*(T)$ is given by the formula: $(w \cdot \kappa)(-) = \kappa(w^{-1} \cdot - \cdot w)$ for $w \in W_G$ and $\kappa \in X^*(T)$. If we let $\rho_G = (2, 1; -3/2)$ denote the half-sum of the positive roots of $G$ (with respect to the choice $(B_G, T)$), then we will also use the notation $w \star \kappa \defeq w \cdot (\kappa + \rho_G) - \rho_G$ for the ``shifted action'' of the Weyl group on algebraic characters.

\begin{remark} \label{rem:ActionOfWeylGroup}
    Let $\kappa = (r_1, r_2; c) \in X^*(T)$. Then the action of ${^MW_G}$ is described as 
    \[
    w_1 \cdot (r_1, r_2; c) = (r_1, -r_2; c+r_2), \quad w_2 \cdot (r_1, r_2; c) = (r_2, -r_1; c+r_1), \quad w_3 \cdot (r_1, r_2; c) = (-r_2, -r_1; c+r_1+r_2)
    \]
    and $w_0$ is just the identity. Furthermore, if $s_1, s_2 \in W_G$ denote the simple reflections defined in \cite[\S 2.1]{LZBK21}, then we have $s_1 \cdot (r_1, r_2; c) = (r_1, -r_2; c+r_2)$ and $s_2 \cdot (r_1, r_2; c) = (r_2, r_1; c)$. Finally the longest Weyl element $w_G^{\opn{max}} \in W_G$ acts as $w_G^{\opn{max}} \cdot (r_1, r_2; c) = (-r_1, -r_2; c+r_1+r_2)$.
\end{remark}

\subsubsection{Open orbit representatives}

Throughout the article we set
\[
\gamma = \left( \begin{smallmatrix} 1 & & & \\ 1 & 1 & & \\ & & 1 & \\ & & -1 & 1 \end{smallmatrix} \right) \in M_G, \quad \quad \hat{\gamma} \defeq \gamma w_1 = \left( \begin{smallmatrix} 1 & & & \\ 1 & & 1 & \\ & -1 &  & \\ & 1 & & 1 \end{smallmatrix} \right) \in G . 
\]
Note that $\gamma$ (resp. $\hat{\gamma}$) is a representative for the Zariski open orbit of $M_H$ (resp. $H$) on the flag variety $M_G/B_{M_G}$ (resp. $G/B_G$). Such open orbits exist (and are unique) because the pairs of groups $(G, H)$ and $(M_G, M_H)$ are \emph{spherical}.

\subsection{Algebraic representation theory}

For one of the groups $\mathscr{G} = G, H, M_G, M_H$ and a $\mathscr{G}$-dominant weight $\lambda = (\lambda_1, \lambda_2; c)$, we will write $V_{\mathscr{G}}(\lambda) = V_{\mathscr{G}}(\lambda_1, \lambda_2; c)$ to be the algebraic representation of $\mathscr{G}$ of highest weight $\lambda$. In fact, we will fix an explicit presentation of this representation, namely we will take
\begin{align*} 
V_{G}(\lambda_1, \lambda_2; c) &\defeq \left\{ f \colon G \to \mbb{A}^1 : f(- \cdot b) = (w_1 w_G^{\opn{max}} \lambda)(b^{-1}) f(-) \text{ for all } b \in w_1 B_G w_1^{-1}  \right\} \\
V_{H}(\lambda_1, \lambda_2; c) &\defeq \left\{ f \colon H \to \mbb{A}^1 : f(- \cdot b) = (w_H^{\opn{max}} \lambda)(b^{-1}) f(-) \text{ for all } b \in B_H   \right\} \\
V_{M_G}(\lambda_1, \lambda_2; c) &\defeq \left\{ f \colon M_G \to \mbb{A}^1 : f(- \cdot b) = ( w_{M_G}^{\opn{max}} \lambda)(b^{-1}) f(-) \text{ for all } b \in B_{M_G}  \right\} \\
V_{M_H}(\lambda_1, \lambda_2; c) &\defeq \left\{ f \colon M_H \to \mbb{A}^1 : f(- \cdot b) = (w_{M_H}^{\opn{max}} \lambda)(b^{-1}) f(-) \text{ for all } b \in B_{M_H}  \right\}
\end{align*}
in the cases $\mathscr{G} = G, H, M_G, M_H$ respectively. Here $w_{\mathscr{G}}^{\opn{max}}$ denotes the longest Weyl element of $\mathscr{G}$. Note that $V_{M_H}(\lambda_1, \lambda_2; c)$ is simply the character $\lambda$ of $T$.

Similarly, for a $\opn{GL}_2$-dominant weight $\lambda = (r; c)$, we let
\[
V_{\opn{GL}_2}(\lambda) = V_{\opn{GL}_2}(r; c) \defeq \left\{ f \colon \opn{GL}_2 \to \mbb{A}^1 : f(- \cdot b) = (w_{\opn{GL}_2}^{\opn{max}} \lambda)(b^{-1}) f(-) \text{ for all } b \in B_{\opn{GL}_2}   \right\}
\]
denote the algebraic representation of $\opn{GL}_2$ of highest weight $\lambda$. Note that $V_{\opn{GL}_2}(r; c) \cong \opn{Sym}^r V_{\opn{GL}_2}(1;0) \otimes \opn{det}^c$, and $V_{\opn{GL}_2}(1;0)$ is isomorphic to the standard representation of $\opn{GL}_2$. Furthermore, one can easily check that we have a canonical isomorphism
\[
V_H(\lambda_1, \lambda_2; c) \cong V_{\opn{GL}_2}(\lambda_1; c_1) \boxtimes V_{\opn{GL}_2}(\lambda_2; c_2)
\]
for any choice of integers $c_1, c_2 \in \mbb{Z}$ satisfying $c_1 + c_2 = c$.

\subsubsection{Algebraic nearly representations} \label{algnearlyrepsSSec}

In this section, we will introduce certain algebraic representations which underlie the ``nearly sheaves'' introduced in \cite[\S 6.1]{LPSZ}. However, note that (once again) our conventions are slightly different from those in \emph{loc.cit.} as we work with lower-triangular parabolic subgroups.

\begin{notation}
    For any finite-dimensional representation $W$ of $\overline{P}_G$ (resp. $\overline{P}_H$), we let $\opn{Fil}_iW \subset W$ denote the subspace spanned by vectors $w \in W$ which satisfy 
\[
\left(\begin{smallmatrix} x & & & \\ & x & & \\ & & 1 & \\ & & & 1 \end{smallmatrix}\right) \cdot w = x^{j}w, \quad \quad x \in \mbb{G}_m
\]
for some $j \leq i$. This defines an ascending filtration stable under the action of $\overline{P}_G$ (resp. $\overline{P}_H$).
\end{notation}

To keep track of the various weights we consider, it will often be helpful to place ourselves in the following situation:

\begin{situation} \label{situation:Weights}
    In this situation, we will consider the following weights:
\begin{itemize}
    \item $\nu_G = (r_1, r_2; -(r_1+r_2))$ will be a $G$-dominant weight, i.e. $r_1 \geq r_2 \geq 0$.
    \item $\nu_H = (t_1, t_2; -r_1)$ will be a $H$-dominant weight (so $t_1 \geq 0$ and $t_2 \geq 0$) with the ``small'' condition that $t_1 + t_2 = r_1 - r_2$.
    \item $\kappa_G = -(w_{M_G}^{\opn{max}})^{-1} \cdot \left( w_1 \star (-w_G^{\opn{max}}\nu_G) \right) = (r_2 +2, -r_1; -r_2-1)$ will be a $M_G$-dominant weight.
    \item $\kappa_H = w_H^{\opn{max}}\nu_H + \rho_H = (1-t_1, 1-t_2; -r_2-1)$ will be a $M_H$-dominant weight, where $\rho_H = (1, 1; -1)$ denotes the half-sum of the positive roots of $H$.
\end{itemize}
We will also fix integers $\xi_1, \xi_2$ such that $\xi_1 + \xi_2 = 1-r_2$, and set $\zeta_{H_i} = (-1-t_i; \xi_i)$ for $i=1, 2$.
\end{situation}

We now introduce the nearly representations. 

\begin{definition}
    Suppose we are in Situation \ref{situation:Weights}. We let
    \begin{align*}
        \mbb{I}_{\overline{P}_G}(\kappa_G)^{\opn{nearly}} &\defeq \left( V_G(\nu_G) / \opn{Fil}_{-(r_1+1)} V_G(\nu_G) \right) \otimes V_{M_G}(-(w_{M_G}^{\opn{max}})^{-1} \cdot (w_1 \star \mbf{1})) \\
        \mbb{I}_{\overline{P}_H}(\kappa_H)^{\opn{nearly}} &\defeq V_H(\nu_H) \otimes V_{M_H}(\rho_H) \\
        \mbb{I}_{\overline{P}_{\opn{GL}_2}}(\zeta_{H_i})^{\opn{nearly}} &\defeq V_{\opn{GL}_2}(t_i; 0) \otimes V_{T_{\opn{GL}_2}}(-1; \xi_i) \quad \quad (i=1,2)
    \end{align*}
    where $\mbf{1}$ denotes the trivial character, which are algebraic representations of $\overline{P}_G$, $\overline{P}_H$ and $\overline{P}_{\opn{GL}_2}$ respectively. Explicitly, one has $-(w_{M_G}^{\opn{max}})^{-1} \cdot (w_1 \star \mbf{1}) = (2, 0;-1)$. We let 
    \begin{align*}
        \mathcal{F}_0 \left(\mbb{I}_{\overline{P}_G}(\kappa_G)^{\opn{nearly}}\right) &\defeq \left( \opn{Fil}_{-r_1}V_G(\nu_G) / \opn{Fil}_{-(r_1+1)} V_G(\nu_G) \right) \otimes V_{M_G}(-(w_{M_G}^{\opn{max}})^{-1} \cdot (w_1 \star \mbf{1})) \subset \mbb{I}_{\overline{P}_G}(\kappa_G)^{\opn{nearly}} \\
        \mathcal{F}_0 \left(\mbb{I}_{\overline{P}_H}(\kappa_H)^{\opn{nearly}}\right) &\defeq \opn{Fil}_{-r_1}V_H(\nu_H) \otimes V_{M_H}(\rho_H) \subset \mbb{I}_{\overline{P}_H}(\kappa_H)^{\opn{nearly}} \\
        \mathcal{F}_0 \left(\mbb{I}_{\overline{P}_{\opn{GL}_2}}(\zeta_{H_i})^{\opn{nearly}}\right) &\defeq \opn{Fil}_{0}V_{\opn{GL}_2}(t_i; 0) \otimes V_{T_{\opn{GL}_2}}(-1; \xi_i) \quad \quad (i=1,2)
    \end{align*}
    which are subrepresentations factoring through the projections $\overline{P}_G \to M_G$, $\overline{P}_H \to M_H$, and $\overline{P}_{\opn{GL}_2} \to T_{\opn{GL}_2}$ respectively. Note that $\mathcal{F}_0 \left(\mbb{I}_{\overline{P}_H}(\kappa_H)^{\opn{nearly}}\right)$ and $\mathcal{F}_0 \left(\mbb{I}_{\overline{P}_{\opn{GL}_2}}(\zeta_{H_i})^{\opn{nearly}}\right)$ are the lowest weight subspaces isomorphic to $V_{M_H}(\kappa_H)$ and $V_{T_{\opn{GL}_2}}(\zeta_{H_i})$ respectively. Furthermore, $\mathcal{F}_0 \left(\mbb{I}_{\overline{P}_G}(\kappa_G)^{\opn{nearly}}\right)$ contains $V_{M_G}(r_2, -r_1;-r_2) \otimes V_{M_G}(2, 0;-1)$ with multiplicity one, and the latter representation contains $V_{M_G}(\kappa_G)$ as a direct summand with multiplicity one. We therefore have a $M_G$-equivariant surjection $\mathcal{F}_0 \left(\mbb{I}_{\overline{P}_G}(\kappa_G)^{\opn{nearly}}\right) \twoheadrightarrow V_{M_G}(\kappa_G)$ factoring through $V_{M_G}(r_2, -r_1;-r_2) \otimes V_{M_G}(2, 0;-1)$, and any such surjection (satisfying this factorisation property) is unique up to scaling. 
\end{definition}

\begin{remark} \label{DecompositionOfNearlyIntoFactorsRem}
    One can easily verify that 
    \[
    \mbb{I}_{\overline{P}_H}(\kappa_H)^{\opn{nearly}} \otimes V_{M_H}(-2\rho_H) \cong \mbb{I}_{\overline{P}_{\opn{GL}_2}}(\zeta_{H_1})^{\opn{nearly}} \boxtimes \mbb{I}_{\overline{P}_{\opn{GL}_2}}(\zeta_{H_2})^{\opn{nearly}} .
    \]
\end{remark}

We have the following branching result.

\begin{proposition} \label{GHBranchingProp}
    Suppose that we are in Situation \ref{situation:Weights}. Then there is a unique (up to scaling) non-zero map
    \begin{equation} \label{HtoGBranchingEqn}
    V_H(t_1, t_2; -r_1) \hookrightarrow V_G(r_1, r_2; -(r_1+r_2))
    \end{equation}
    which is equivariant for the action of $H$ through the embedding $\gamma^{-1} H \gamma \subset G$.
\end{proposition}
\begin{proof}
    Recall that $t_1 + t_2 = r_1 - r_2$ and note that $| t_1 - t_2 | \leq r_1 - r_2$ is always satisfied because $t_1, t_2 \geq 0$. The result now follows from \cite[Proposition 4.3.1]{LSZ17}.
\end{proof}

We will fix a certain normalisation of the branching map (\ref{HtoGBranchingEqn}). We first have to make some choices. For any $H$-dominant weight $\lambda = (\lambda_1, \lambda_2; c)$, we let $w_{\lambda}^{\opn{hw}} \in V_H(\lambda_1, \lambda_2; c)$ denote the highest weight vector given by
\[
w_{\lambda}^{\opn{hw}}\left( \tbyt{x}{y}{w}{z} , \tbyt{x'}{y'}{w'}{z'} \right) = (-1)^{\lambda_1+\lambda_2}w^{\lambda_1} (w')^{\lambda_2} d^{-(c+\lambda_1+\lambda_2)}
\]
where $d = xz - wy = x'z' - w'y'$. It is the unique highest weight vector which is $1$ on the matrix $J$ (which is a representative for the longest Weyl element of $H$).

Similarly, for any $G$-dominant weight $\lambda = (\lambda_1, \lambda_2;c)$, we let $v_{\lambda}^{\opn{hw}} \in V_G(\lambda_1, \lambda_2; c)$ denote the unique highest weight vector such that $v_{\lambda}^{\opn{hw}}(J w_1^{-1}) = 1$ ($J$ is also a representative of the longest Weyl element of $G$). We fix a normalisation of the map (\ref{HtoGBranchingEqn}), which we will denote by $\opn{br} \colon V_H(t_1, t_2; -r_1) \hookrightarrow V_G(r_1, r_2; -(r_1+r_2))$, by sending $w_{(t_1, t_2; -r_1)}^{\opn{hw}}$ to the element
\[
\opn{br}(w_{(t_1, t_2; -r_1)}^{\opn{hw}}) = (-1)^{r_1-1}2^{-r_2} \gamma^{-1} \cdot \left[ ( v_{(1, 0;-1)}^{\opn{hw}} )^{t_1} \cdot ( X_{2, 1} \star v_{(1, 0;-1)}^{\opn{hw}} )^{t_2} \cdot ( Z \star v_{(1, 1; -2)}^{\opn{hw}} )^{r_2} \right]
\]
where the product is given by the Cartan product. Here
\[
X_{2, 1} = \left( \begin{smallmatrix} 0 & 0 & 0 & 0 \\ 1 & 0 & 0 & 0 \\ 0 & 0 & 0 & 0 \\ 0 & 0 & -1 & 0 \end{smallmatrix} \right) \text{ and } Z = \left( \begin{smallmatrix} 0 & 0 & 0 & 0 \\ 0 & 0 & 0 & 0 \\ 1 & 0 & 0 & 0 \\ 0 & 1 & 0 & 0 \end{smallmatrix} \right)
\]
are the elements of $\opn{Lie}(G)$ introduced in \cite[\S 4.3]{LSZ17}. One can easily check that the image of $\opn{br}$ intersects trivially with $\opn{Fil}_{-(r_1+1)} V_G(r_1, r_2; -(r_1+r_2))$, so we obtain an injective morphism:
\[
\opn{br} \colon V_H(t_1, t_2; -r_1) \hookrightarrow V_G(r_1, r_2; -(r_1+r_2))/\opn{Fil}_{-(r_1+1)} V_G(r_1, r_2; -(r_1+r_2)) .
\]
We also have a natural branching map
\begin{align*}
    \opn{br}' \colon V_{M_H}(1, 1; -1) &\hookrightarrow V_{M_G}(2, 0; -1) \\ \chi_{(1, 1;-1)} &\mapsto \left( x = \left( \begin{smallmatrix} x_{11} & x_{12} & 0 & 0 \\ x_{21} & x_{22} & 0 & 0 \\ 0 & 0 & x_{33} & x_{34} \\ 0 & 0 & x_{43} & x_{44} \end{smallmatrix} \right) \mapsto s(x)^{-1} x_{33} (x_{43} - x_{33}) \right)
\end{align*}
which is equivariant for the action of $M_H$ through the embedding $\gamma^{-1} M_H \gamma \subset M_G$, where $\chi_{(1, 1; -1)}$ denotes the character associated with the weight $(1, 1; -1)$ and $s \colon G \to \mbb{G}_m$ denotes the similitude character. Combining these two branching laws, we therefore obtain a morphism
\[
\opn{br}^{\opn{nearly}} = \opn{br} \otimes \opn{br}' \colon \mbb{I}_{\overline{P}_H}(\kappa_H)^{\opn{nearly}} \hookrightarrow \mbb{I}_{\overline{P}_G}(\kappa_G)^{\opn{nearly}}
\]
which is equivariant for the action of $\overline{P}_H$ through the embedding $\gamma^{-1} \overline{P}_H \gamma \subset \overline{P}_G$.

We have the following lemma.

\begin{lemma} \label{explicitBranchingUnipLemma}
    Suppose that we are in Situation \ref{situation:Weights}. Then one has
    \[
    w_{(t_1, t_2; -r_1)}^{\opn{hw}}\left( \left( \begin{smallmatrix} 1 & 0 \\ v & 1 \end{smallmatrix} \right), \left( \begin{smallmatrix} 1 & 0 \\ u & 1 \end{smallmatrix} \right) \right) = (-1)^{t_1+t_2} v^{t_1} u^{t_2}
    \]
    and
    \[
    \opn{br}(w_{(t_1, t_2; -r_1)}^{\opn{hw}})\left( \begin{smallmatrix} 1 & 0 & 0 & 0 \\ z & 1 & 0 & 0 \\ a & 0 & 1 & 0 \\ b & a & -z & 1 \end{smallmatrix} \right) = (-1)^{t_1+t_2-1} (b-a)^{t_1} a^{t_2} (1+z)^{r_2} .
    \]
\end{lemma}
\begin{proof}
    An explicit, but lengthy, calculation shows that
    \[
    \gamma \cdot \opn{br}(w_{(t_1, t_2; -r_1)}^{\opn{hw}})((x_{ij})) = (-1)^{t_1+t_2-1} 2^{-r_2} x_{41}^{t_1} x_{31}^{t_2} ( x_{21} x_{33} + x_{41}x_{13} - x_{11}x_{43} - x_{31}x_{23})^{r_2}
    \]
    for any matrix $(x_{ij}) \in G$. The result follows.
\end{proof}

We have a commutative diagram
\begin{equation} \label{brNearlyCommWithAlgDiagram}
\begin{tikzcd}
\mbb{I}_{\overline{P}_H}(\kappa_H)^{\opn{nearly}} \arrow[r, "\opn{br}^{\opn{nearly}}", hook]                                                                           & \mbb{I}_{\overline{P}_G}(\kappa_G)^{\opn{nearly}}                                                                 \\
\mathcal{F}_0 \left(\mbb{I}_{\overline{P}_H}(\kappa_H)^{\opn{nearly}}\right) \arrow[u, hook] \arrow[d, equals] \arrow[r, "\opn{br}^{\opn{nearly}}", hook] & \mathcal{F}_0 \left(\mbb{I}_{\overline{P}_G}(\kappa_G)^{\opn{nearly}}\right) \arrow[d, two heads] \arrow[u, hook] \\
{V_{M_H}(1-t_1, 1-t_2; -r_2-1)} \arrow[r, "\opn{br}_M" , hook]                                                                                                                    & {V_{M_G}(r_2+2, -r_1; -r_2-1)}                                                                               
\end{tikzcd}
\end{equation}
where the bottom horizontal map takes the character $\kappa_H$ to the unique function satisfying 
\[
\left( \begin{smallmatrix} 1 & & & \\ z & 1 & & \\ & & 1 & \\ & & -z & 1\end{smallmatrix} \right) \mapsto (1+z)^{r_1+1-t_1}
\]
and the bottom right surjection is normalised so that the diagram is commutative (and factors through $V_{M_G}(r_2, -r_1;-r_2) \otimes V_{M_G}(2, 0;-1)$).

\subsubsection{} \label{SplttingsOfMGReps}

For convenience later on, we fix the following notation. Let $s_1$, $s_2$ be $M_G$-equivariant splittings 
\[
\begin{tikzcd}
{V_{M_G}(r_2+2, -r_1; -r_2-1)} \arrow[r, "s_1", dotted, bend left] & \mathcal{F}_0 \left(\mbb{I}_{\overline{P}_G}(\kappa_G)^{\opn{nearly}}\right) \arrow[r, "f_2"] \arrow[l, "f_1"'] & \mbb{I}_{\overline{P}_G}(\kappa_G)^{\opn{nearly}} \arrow[l, "s_2"', dotted, bend right]
\end{tikzcd}
\]
where $f_1$ and $f_2$ are the morphisms in the right-hand column of (\ref{brNearlyCommWithAlgDiagram}). The splitting $s_2$ is automatically unique, and we take $s_1$ to be the splitting which factors through $V_{M_G}(r_2, -r_1;-r_2) \otimes V_{M_G}(2, 0;-1)$ (which is necessarily unique). We set $\jmath_{M_G} = f_2 \circ s_1$ and $\opn{spl}_{M_G} = f_1 \circ s_2$. Note that $\opn{spl}_{M_G} \circ \; \jmath_{M_G} = \opn{id}$.

\begin{notation} \label{NotationForSplittingsOfReps}
    We set $w_{\kappa_H} \defeq \gamma \cdot \opn{br}_M(\kappa_H)$, where $\kappa_H \in V_{M_H}(\kappa_H)$ denotes the canonical character. If $\mbb{D}_{\overline{P}_{\opn{GL}_2}}(\zeta_{H_i}^*)^{\opn{nearly}}$ denotes the linear dual of $\mbb{I}_{\overline{P}_{\opn{GL}_2}}(\zeta_{H_i})^{\opn{nearly}}$, then we let $v_{\zeta_{H_i}}^* \in \mbb{D}_{\overline{P}_{\opn{GL}_2}}(\zeta_{H_i}^*)^{\opn{nearly}}$ denote the image of $1$ under the unique $M_{\opn{GL}_2}$-equivariant morphism $V_{M_{\opn{GL}_2}}(\zeta_{H_i})^* \hookrightarrow \mbb{D}_{\overline{P}_{\opn{GL}_2}}(\zeta_{H_i}^*)^{\opn{nearly}}$ which splits the natural projection $\mbb{D}_{\overline{P}_{\opn{GL}_2}}(\zeta_{H_i}^*)^{\opn{nearly}} \to V_{M_{\opn{GL}_2}}(\zeta_{H_i})^*$.
\end{notation}

\subsection{\texorpdfstring{$p$}{p}-adic representation theory} \label{padicRepTheorySubSec}

In this section we introduce $p$-adic versions of the algebraic nearly representations defined in the previous section. For this, we now consider all of our groups as algebraic groups over $\mbb{Q}_p$ (or $\mbb{Z}_p$) but often omit this from the notation. Fix an integer $\beta \geq 1$ throughout.

Let $\overline{\mathcal{P}}_G$ denote the adic generic fibre of the formal completion of $\overline{P}_{G, \mbb{Z}_p}$ along the special fibre. We define $\overline{\mathcal{P}}_H$ and $\overline{\mathcal{P}}_{\opn{GL}_2}$ similarly. 

\begin{definition}
    Let $n \geq 1$ and let $\overline{\mathcal{P}}_{G, n} \subset \overline{\mathcal{P}}_G$ (resp. $\overline{\mathcal{P}}_{H, n} \subset \overline{\mathcal{P}}_H$, resp. $\overline{\mathcal{P}}_{\opn{GL}_2, n} \subset \overline{\mathcal{P}}_{\opn{GL}_2}$) denote the subgroup of elements which are congruent to the identity modulo $p^n$. Let $M_{G, \opn{Iw}}(p^{\beta})$ denote the depth $p^{\beta}$ upper triangular Iwahori subgroup in $M_G(\mbb{Z}_p)$, and let $T_{\diamondsuit}(p^{\beta}) \subset T(\mbb{Z}_p) = M_H(\mbb{Z}_p)$ denote the subgroup of elements
    \[
    \left( \begin{smallmatrix}%
   t_1 & & & \\ & t_2 & & \\ & & \nu t_2^{-1} & \\ & & & \nu t_1^{-1}    
\end{smallmatrix} \right) \in T(\mbb{Z}_p)
    \]
    with $t_1 \equiv t_2$ modulo $p^{\beta}$. Note that $T_{\diamondsuit}(p^{\beta}) = \gamma M_{G, \opn{Iw}}(p^{\beta}) \gamma^{-1} \cap M_H(\mbb{Q}_p)$. 

    Finally, we let $\overline{\mathcal{P}}_{G,n,\beta}^{\square} \defeq \overline{\mathcal{P}}_{G, n} M_{G, \opn{Iw}}(p^{\beta})$, $\overline{\mathcal{P}}^{\diamondsuit}_{H, n, \beta} \defeq \overline{\mathcal{P}}_{H, n} T_{\diamondsuit}(p^{\beta})$, and $\overline{\mathcal{P}}_{\opn{GL}_2, n}^{\square} \defeq \overline{\mathcal{P}}_{\opn{GL}_2, n} T_{\opn{GL}_2}(\mbb{Z}_p)$. Clearly we have $\gamma^{-1} \overline{\mathcal{P}}^{\diamondsuit}_{H, n, \beta} \gamma \subset \overline{\mathcal{P}}_{G,n,\beta}^{\square}$ and the image of $\overline{\mathcal{P}}^{\diamondsuit}_{H, n, \beta}$ under projection to either $\opn{GL}_2$-factor is equal to $\overline{\mathcal{P}}_{\opn{GL}_2, n}^{\square}$.
\end{definition}

We have the following useful lemma.

\begin{lemma} \label{PartialFactorisationLemma}
    Any element $g \in \overline{\mathcal{P}}_{G,n,\beta}^{\square}$ can be uniquely factored as
    \[
    g = \left( \begin{smallmatrix} 1 & & & \\ z & 1 & & \\ a & & 1 & \\ b & a & -z & 1 \end{smallmatrix} \right) \cdot x
    \]
    where:
    \begin{itemize}
        \item $x \in w_1 B_G^{\opn{an}} w_1^{-1} \cap \overline{\mathcal{P}}_{G,n,\beta}^{\square}$
        \item $z \in (p^{\beta} \mbb{Z}_p) + \mathcal{B}_n$, and $a, b \in \mathcal{B}_n$, where $\mathcal{B}_n$ denotes the rigid ball of radius $p^{-n}$.
    \end{itemize}
\end{lemma}
\begin{proof}
    Let $\mathcal{M}_{G, n, \beta}^{\square} = \mathcal{M}_{G, n} M_{G, \opn{Iw}}(p^{\beta})$. Then, using the Iwahori factorisation for this group, we can factor $g$ as
    \[
    \left( \begin{smallmatrix} 1 & & & \\  & 1 & & \\ a & & 1 & \\ b & a &  & 1 \end{smallmatrix} \right) \cdot \left( \begin{smallmatrix} 1 & & & \\  & 1 & & \\  & c & 1 & \\  &  &  & 1 \end{smallmatrix} \right) \cdot \left( \begin{smallmatrix} 1 & & & \\ z & 1 & & \\  & & 1 & \\  &  & -z & 1 \end{smallmatrix} \right) \cdot x'
    \]
    for some $c \in \mathcal{B}_n$ and $x' \in w_1 B_G^{\opn{an}} w_1^{-1} \cap \overline{\mathcal{P}}_{G,n,\beta}^{\square}$, where $a, b, z$ satisfy the same conditions as in the statement of the lemma. The desired factorisation now follows from the computation
    \[
    \left( \begin{smallmatrix} 1 & & & \\  & 1 & & \\  & c & 1 & \\  &  &  & 1 \end{smallmatrix} \right) \cdot \left( \begin{smallmatrix} 1 & & & \\ z & 1 & & \\  & & 1 & \\  &  & -z & 1 \end{smallmatrix} \right) = \left( \begin{smallmatrix} 1 & & & \\ z & 1 & & \\ cz & & 1 & \\  & cz & -z & 1 \end{smallmatrix} \right) \cdot \left( \begin{smallmatrix} 1 & & & \\  & 1 & & \\  & c & 1 & \\  &  &  & 1 \end{smallmatrix} \right) .
    \]
    The factorisation is unique because the group $w_1 B_G^{\opn{an}} w_1^{-1} \cap \overline{\mathcal{P}}_{G,n,\beta}^{\square}$ has trivial intersection with the group of matrices of the form $\left( \begin{smallmatrix} 1 & & & \\ z & 1 & & \\ a & & 1 & \\ b & a & -z & 1 \end{smallmatrix} \right)$ as in the statement of the lemma.
\end{proof}

\subsubsection{\texorpdfstring{$p$}{p}-adic nearly representations for \texorpdfstring{$G$}{G}}

Let $(A, A^+)$ be a complete Tate affinoid pair over $(\mbb{Q}_p, \mbb{Z}_p)$. Let $n \geq 1$ be an integer and let $\kappa_G \colon T(\mbb{Z}_p) \to A^{\times}$ be a $n$-analytic character; this can be viewed as a tuple $(\kappa_1, \kappa_2; \omega)$, where $\kappa_i \colon \mbb{Z}_p^{\times} \to A^{\times}$ and $\omega \colon \mbb{Z}_p^{\times} \to A^{\times}$ are $n$-analytic characters such that
\[
\kappa_G \left( \begin{smallmatrix} t_1 & & & \\ & t_2 & & \\ & & \nu t_2^{-1} & \\ & & & \nu t_1^{-1} \end{smallmatrix} \right) = \kappa_1(t_1) \kappa_2(t_2) \omega(\nu) .
\]

\begin{definition} \label{DefinitionOfIGpadicnearlyrep}
    Let $n \geq 1$. We let 
    \[
    \mbb{I}_{\overline{\mathcal{P}}_{G,n,\beta}^{\square}}(\kappa_G) \defeq \left\{ f \colon \overline{\mathcal{P}}_{G,n,\beta}^{\square} \to \mbb{A}^{1, \opn{an}}_A : f(- \cdot b) = (w_{M_G}^{\opn{max}} \kappa_G)(b^{-1}) f(-) \text{ for all } b \in \overline{\mathcal{P}}_{G,n,\beta}^{\square} \cap w_1 B_G^{\opn{an}} w_1^{-1} \right\} .
    \]
    This is a projective Banach $A$-module (Lemma \ref{PartialFactorisationLemma}) and comes equipped with an action of $\overline{\mathcal{P}}_{G,n,\beta}^{\square}$ via
    \[
    (p \cdot f)(-) = f(p^{-1} \cdot -), \quad \quad p \in \overline{\mathcal{P}}_{G,n,\beta}^{\square} .
    \]
    It also comes equipped with an action of $w_1 T^{G, +} w_1^{-1} \subset T(\mbb{Q}_p)$ (the submonoid of elements $t \in T(\mbb{Q}_p)$ which satisfy $v_p(\alpha(w_1^{-1} t w_1)) \geq 0$ for every positive root $\alpha$ of $G$) by the formula:
    \[
    (t \cdot f)(n x) = f((t^{-1}  n  t) x), \text{ for any } n = \left( \begin{smallmatrix} 1 & & & \\ z & 1 & & \\ a & & 1 & \\ b & a & -z & 1 \end{smallmatrix} \right), x \in w_1 B_G^{\opn{an}} w_1^{-1} \cap \overline{\mathcal{P}}_{G,n,\beta}^{\square} 
    \]
    where $a, b, z$ satisfy the conditions in the statement of Lemma \ref{PartialFactorisationLemma}.
\end{definition}

Note that, if $\kappa_G \in X^*(T)$ is as in Situation \ref{situation:Weights}, then we have a natural $\overline{\mathcal{P}}_{G,n,\beta}^{\square}$-equivariant map
\[
\mbb{I}_{\overline{P}_G}(\kappa_G)^{\opn{nearly}} \to \mbb{I}_{\overline{\mathcal{P}}_{G,n,\beta}^{\square}}(\kappa_G)
\]
given by restricting a function on $G$ to $\overline{\mathcal{P}}_{G,n,\beta}^{\square}$ and inflating a function on $M_G$ to $\overline{\mathcal{P}}_{G,n,\beta}^{\square}$. Such a map is well-defined because $-w_{M_G}^{\opn{max}}\kappa_G = w_1 \cdot (-w_G^{\opn{max}} \nu_G + \rho_G) - \rho_G$, where $\rho_G$ denotes the half-sum of the positive roots (note that $w_1 \cdot \rho_G - \rho_G = -w_{M_G}^{\opn{max}} \cdot (2, 0; -1)$).

We also introduce a version of this representation for the Levi subgroup.

\begin{definition}
    Let $n \geq 1$ and $\kappa_G \colon T(\mbb{Z}_p) \to A^{\times}$ a $n$-analytic character. Let $\mathcal{M}_{G, n, \beta}^{\square} = \mathcal{M}_{G, n} M_{G, \opn{Iw}}(p^{\beta})$. We define
    \[
    \mbb{I}_{\mathcal{M}_{G,n,\beta}^{\square}}(\kappa_G) \defeq \left\{ f \colon \mathcal{M}_{G,n,\beta}^{\square} \to \mbb{A}^{1, \opn{an}}_A : f(- \cdot b) = (w_{M_G}^{\opn{max}} \kappa_G)(b^{-1}) f(-) \text{ for all } b \in \mathcal{M}_{G,n,\beta}^{\square} \cap B_{M_G}^{\opn{an}} \right\}
    \]
    (note that $w_1 B_G^{\opn{an}} w_1^{-1} \cap M_G^{\opn{an}} = B_{M_G}^{\opn{an}}$). This comes equipped with actions of $\mathcal{M}_{G, n, \beta}^{\square}$ and $T^{M_G, +}$ by the same formulae above. We have an obvious inclusion 
    \begin{equation} \label{IMGpadicToIPGEqn}
    \mbb{I}_{\mathcal{M}_{G,n,\beta}^{\square}}(\kappa_G) \hookrightarrow \mbb{I}_{\overline{\mathcal{P}}_{G,n,\beta}^{\square}}(\kappa_G)
    \end{equation}
    induced from the map $\overline{\mathcal{P}}_{G,n,\beta}^{\square} \twoheadrightarrow \mathcal{M}_{G,n,\beta}^{\square}$.
\end{definition}

\subsubsection{\texorpdfstring{$p$}{p}-adic nearly representations for \texorpdfstring{$H$}{H}}

We now consider the analogous constructions for the group $H$. Let $\kappa_H \colon T(\mbb{Z}_p) \to A^{\times}$ be a $n$-analytic character. Then we define:
\[
\mbb{I}_{\overline{\mathcal{P}}_{H, n, \beta}^{\diamondsuit}}(\kappa_H) \defeq \left\{ f \colon \overline{\mathcal{P}}_{H, n, \beta}^{\diamondsuit} \to \mbb{A}^{1, \opn{an}}_A : f(- \cdot b) = (w_{M_H}^{\opn{max}} \kappa_H)(b^{-1}) f(-) \text{ for all } b \in \overline{\mathcal{P}}_{H, n, \beta}^{\diamondsuit} \cap B_H^{\opn{an}} \right\} 
\]
which is a projective Banach $A$-module and comes equipped with actions of $\overline{\mathcal{P}}_{H, n, \beta}^{\diamondsuit}$ and $T^{H, +}$. We also define $\mbb{I}_{\mathcal{M}_{H, n, \beta}^{\diamondsuit}}(\kappa_H)$ similarly, where $\mathcal{M}_{H, n, \beta}^{\diamondsuit} \defeq \mathcal{M}_{H, n} T_{\diamondsuit}(p^{\beta})$. Note that this latter representation is simply a line, and has a canonical basis given by the character $\kappa_H$. As above, we have a natural $\overline{\mathcal{P}}_{H, n, \beta}^{\diamondsuit}$-equivariant map $\mbb{I}_{\overline{P}_{H}}(\kappa_H)^{\opn{nearly}} \to \mbb{I}_{\overline{\mathcal{P}}_{H, n, \beta}^{\diamondsuit}}(\kappa_H)$ when $\kappa_H$ is as in Situation \ref{situation:Weights}, which is well-defined because $w^{\opn{max}}_{M_H} \kappa_H = \kappa_H = w_H^{\opn{max}} \nu_H + \rho_H = w_H^{\opn{max}} \nu_H + w_{M_H}^{\opn{max}}\rho_H$.

\subsubsection{\texorpdfstring{$p$}{p}-adic nearly representations for \texorpdfstring{$\opn{GL}_2$}{GL(2)}}

Let $\zeta \colon T_{\opn{GL}_2}(\mbb{Z}_p) \to A^{\times}$ be a $n$-analytic character. We define:
\[
\mbb{I}_{\overline{\mathcal{P}}_{\opn{GL}_2, n}^{\square}}(\zeta) \defeq \left\{ f \colon \overline{\mathcal{P}}_{\opn{GL}_2, n}^{\square} \to \mbb{A}^{1, \opn{an}}_A : f(- \cdot b) = (w_{M_{\opn{GL}_2}}^{\opn{max}} \zeta)(b^{-1}) f(-) \text{ for all } b \in \overline{\mathcal{P}}_{\opn{GL}_2, n}^{\square} \cap B_{\opn{GL}_2}^{\opn{an}} \right\} .
\]
If $\zeta = (-1 - t; \xi)$ is algebraic with $t \geq 0$, then we have a natural $\overline{\mathcal{P}}_{\opn{GL}_2, n}^{\square}$-equivariant map $\mbb{I}_{\overline{P}_{\opn{GL}_2}}(\zeta)^{\opn{nearly}} \to \mbb{I}_{\overline{\mathcal{P}}_{\opn{GL}_2, n}^{\square}}(\zeta)$.

\begin{remark} \label{DecompOfDDaggerIntoFactorsRem}
    We have the following analogue of Remark \ref{DecompositionOfNearlyIntoFactorsRem}, namely if $\kappa_H \colon T(\mbb{Z}_p) \to A^{\times}$ is a $n$-analytic character and we view $\kappa_H - 2\rho_H$ as a pair of $n$-analytic characters $(\zeta_{H_1}, \zeta_{H_2}) \colon T_{\opn{GL}_2}(\mbb{Z}_p) \times_{\mbb{Z}_p^{\times}} T_{\opn{GL}_2}(\mbb{Z}_p) \to A^{\times}$ via the identification $T(\mbb{Z}_p) = T_{\opn{GL}_2}(\mbb{Z}_p) \times_{\mbb{Z}_p^{\times}} T_{\opn{GL}_2}(\mbb{Z}_p)$, then 
    \[
    \mbb{I}_{\overline{\mathcal{P}}_{H, n, \beta}^{\diamondsuit}}(\kappa_H) \otimes V_{M_H}(-2\rho_H) \cong \mbb{I}_{\overline{\mathcal{P}}_{\opn{GL}_2, n}^{\square}}(\zeta_{H_1}) \; \widehat{\boxtimes} \; \mbb{I}_{\overline{\mathcal{P}}_{\opn{GL}_2, n}^{\square}}(\zeta_{H_2}) .
    \]
\end{remark}

\subsubsection{A \texorpdfstring{$p$}{p}-adic branching result}

We consider the following morphism.

\begin{proposition} \label{DefPropOfbrnan}
    Let $n \geq 1$ and let $\kappa_G \colon T(\mbb{Z}_p) \to A^{\times}$ be a $n$-analytic character. Let $\lambda \colon \mbb{Z}_p^{\times} \to A^{\times}$ be a $n$-analytic character and set $\kappa_H = w_{M_G}^{\opn{max}} \kappa_G + (\lambda, -\lambda; 0)$, i.e.,
    \[
    \kappa_H\left( \begin{smallmatrix} t_1 & & & \\ & t_2 & & \\ & & \nu t_1^{-1} & \\ & & & \nu t_2^{-1} \end{smallmatrix} \right) = \kappa_1(t_2) \kappa_2(t_1) \omega(\nu) \lambda(t_1 t_2^{-1}) .
    \]
    Consider the following $A$-linear map
    \[
    \opn{br}^{n\opn{-an}} \colon \mbb{I}_{\overline{\mathcal{P}}_{H, n, \beta}^{\diamondsuit}}(\kappa_H) \to \mbb{I}_{\overline{\mathcal{P}}_{G, n, \beta}^{\square}}(\kappa_G)
    \]
    uniquely determined by the property
    \[
    \opn{br}^{n\opn{-an}}(G)\left( \begin{smallmatrix} 1 & & & \\ z & 1 & & \\ a & & 1 & \\ b & a & -z & 1 \end{smallmatrix} \right) = \lambda(1+z) G \left( \left( \begin{smallmatrix} 1 & \\ b-a & 1\end{smallmatrix} \right), \left( \begin{smallmatrix} 1 & \\ \frac{a}{1+z} & 1 \end{smallmatrix} \right) \right)
    \]
    (which is well-defined because $1+z \in 1 + (p^{\beta} \mbb{Z}_p)+\mathcal{B}_n \subset \mbb{Z}_p^{\times}(1 + \mathcal{B}_n)$). Then $\opn{br}^{n\opn{-an}}$ is equivariant for the action of $\overline{\mathcal{P}}_{H, n, \beta}^{\diamondsuit}$ through the embedding $\gamma^{-1} \overline{\mathcal{P}}_{H, n, \beta}^{\diamondsuit} \gamma \subset \overline{\mathcal{P}}_{G, n, \beta}^{\square}$.
\end{proposition}
\begin{proof}
    We first compute the actions of $\overline{\mathcal{P}}_{H, n, \beta}^{\diamondsuit}$ and $\gamma^{-1} \overline{\mathcal{P}}_{H, n, \beta}^{\diamondsuit} \gamma$ on $\mbb{I}_{\overline{\mathcal{P}}_{H, n, \beta}^{\diamondsuit}}(\kappa_H)$ and $\mbb{I}_{\overline{\mathcal{P}}_{G, n, \beta}^{\square}}(\kappa_G)$ respectively. More precisely, let 
    \[
    h = \left( \begin{smallmatrix} t_1 & & & \\ & t_2 & & \\ & & \nu t_2^{-1} & \\ & & & \nu t_1^{-1} \end{smallmatrix} \right) \in \overline{\mathcal{P}}_{H, n, \beta}^{\diamondsuit} .
    \]
    Then one can compute that:
    \begin{align*}
        \gamma^{-1} h^{-1} \gamma \cdot \left( \begin{smallmatrix} 1 & & & \\ z & 1 & & \\ a & & 1 & \\ b & a & -z & 1 \end{smallmatrix} \right) &= \left( \begin{smallmatrix} t_1^{-1} & & & \\ t_2^{-1} - t_1^{-1} & t_2^{-1} & & \\ & & \nu^{-1} t_2 & \\ & & \nu^{-1}(t_2-t_1) & \nu^{-1} t_1 \end{smallmatrix} \right) \cdot \left( \begin{smallmatrix} 1 & & & \\ z & 1 & & \\ a & & 1 & \\ b & a & -z & 1 \end{smallmatrix} \right) \\
         &= \left( \begin{smallmatrix} t_1^{-1} & & & \\ t_2^{-1}(1+z) - t_1^{-1} & t_2^{-1} & & \\ \nu^{-1}t_2 a & & \nu^{-1} t_2 & \\ \nu^{-1}t_1 b + \nu^{-1}(t_2 - t_1) a & \nu^{-1}t_1 a & \nu^{-1}(t_2-t_1) - z\nu^{-1} t_1 & \nu^{-1} t_1 \end{smallmatrix} \right) \\
         &= \left( \begin{smallmatrix} 1 & & & \\ (t_1t_2^{-1})(1+z) - 1 & 1 & & \\ \nu^{-1}t_1 t_2 a & & 1 & \\ \nu^{-1}t_1t_2 a + \nu^{-1} t_1^2 (b-a) & \nu^{-1}t_1 t_2 a & 1- (t_1t_2^{-1})(1+z) & 1 \end{smallmatrix} \right) \cdot \left( \begin{smallmatrix} t_1^{-1} & & & \\ & t_2^{-1} & & \\ & & \nu^{-1} t_2 & \\ & & & \nu^{-1} t_1 \end{smallmatrix} \right) .
    \end{align*}
    Therefore, we see that for any $G \in \mbb{I}_{\overline{\mathcal{P}}_{H, n, \beta}^{\diamondsuit}}(\kappa_H)$ and $F \in \mbb{I}_{\overline{\mathcal{P}}_{G, n, \beta}^{\square}}(\kappa_G)$:
    \begin{align*}
        (h \cdot G)\left( \left( \begin{smallmatrix} 1 & \\ Y & 1 \end{smallmatrix} \right), \left( \begin{smallmatrix} 1 & \\ X & 1 \end{smallmatrix} \right) \right) &= \kappa_H(h) G\left( \left( \begin{smallmatrix} 1 & \\ \nu^{-1}t_1^2 Y & 1 \end{smallmatrix} \right), \left( \begin{smallmatrix} 1 & \\ \nu^{-1}t_2^2 X & 1 \end{smallmatrix} \right) \right) \\
        ((\gamma^{-1}h \gamma) \cdot F)\left( \begin{smallmatrix} 1 & & & \\ z & 1 & & \\ a & & 1 & \\ b & a & -z & 1 \end{smallmatrix} \right) &= (w_{M_G}^{\opn{max}} \kappa_G)(h) F\left( \begin{smallmatrix} 1 & & & \\ (t_1t_2^{-1})(1+z) - 1 & 1 & & \\ \nu^{-1}t_1 t_2 a & & 1 & \\ \nu^{-1}t_1 t_2 a + \nu^{-1} t_1^2 (b-a) & \nu^{-1}t_1 t_2 a & 1-(t_1t_2^{-1})(1+z) & 1 \end{smallmatrix} \right) .
    \end{align*}
    Similarly, suppose that 
    \[
    h = \left( \begin{smallmatrix} 1 & & & \\ & 1 & & \\ & u & 1 & & \\ v & & & 1 \end{smallmatrix} \right) \in \overline{\mathcal{P}}_{H, n, \beta}^{\diamondsuit} .
    \]
    Then one can compute that
    \[
    \gamma^{-1} h^{-1} \gamma \cdot \left( \begin{smallmatrix} 1 & & & \\ z & 1 & & \\ a & & 1 & \\ b & a & -z & 1 \end{smallmatrix} \right) = \left( \begin{smallmatrix} 1 & & & \\ z & 1 & & \\ a - u(1+z) & & 1 & \\ b - v - u(1+z) & a - u(1+z) & -z & 1 \end{smallmatrix} \right) \cdot \left( \begin{smallmatrix} 1 & & & \\  & 1 & & \\  & -u & 1 & \\  &  &  & 1 \end{smallmatrix} \right)
    \]
    so for any $G \in \mbb{I}_{\overline{\mathcal{P}}_{H, n, \beta}^{\diamondsuit}}(\kappa_H)$ and $F \in \mbb{I}_{\overline{\mathcal{P}}_{G, n, \beta}^{\square}}(\kappa_G)$:
    \begin{align*}
        (h \cdot G)\left( \left( \begin{smallmatrix} 1 & \\ Y & 1 \end{smallmatrix} \right), \left( \begin{smallmatrix} 1 & \\ X & 1 \end{smallmatrix} \right) \right) &= G\left( \left( \begin{smallmatrix} 1 & \\ Y-v & 1 \end{smallmatrix} \right), \left( \begin{smallmatrix} 1 & \\ X - u & 1 \end{smallmatrix} \right) \right) \\
        ((\gamma^{-1}h \gamma) \cdot F)\left( \begin{smallmatrix} 1 & & & \\ z & 1 & & \\ a & & 1 & \\ b & a & -z & 1 \end{smallmatrix} \right) &= F\left( \begin{smallmatrix} 1 & & & \\ z & 1 & & \\ a - u(1+z) & & 1 & \\ b - v - u(1+z) & a - u(1+z) & -z & 1 \end{smallmatrix} \right) .
    \end{align*}
    One can easily check this implies the required equivariance of the morphism $\opn{br}^{n\opn{-an}}$.
\end{proof}

We have the following important compatibility between the branching maps $\opn{br}^{\opn{nearly}}$ and $\opn{br}^{n\opn{-an}}$.

\begin{proposition} \label{BrNearlyBrNanCompatProposition}
    Suppose that $\kappa_G$ and $\kappa_H$ are as in Situation \ref{situation:Weights}. Then we have a commutative diagram
    \[
\begin{tikzcd}
{\mbb{I}_{\overline{\mathcal{P}}_{H, n, \beta}^{\diamondsuit}}(\kappa_H)} \arrow[r, "\opn{br}^{n\opn{-an}}"] & {\mbb{I}_{\overline{\mathcal{P}}_{G, n, \beta}^{\square}}(\kappa_G)} \\
\mbb{I}_{\overline{P}_H}(\kappa_H)^{\opn{nearly}} \arrow[u] \arrow[r, "\opn{br}^{\opn{nearly}}"]             & \mbb{I}_{\overline{P}_G}(\kappa_G)^{\opn{nearly}} \arrow[u]         
\end{tikzcd}
    \]
\end{proposition}
\begin{proof}
    Since all the maps are $\overline{\mathcal{P}}_{H, n, \beta}^{\diamondsuit}$-equivariant and $\mbb{I}_{\overline{P}_H}(\kappa_H)^{\opn{nearly}}$ is generated by the action of $\mathcal{U}(\overline{\mathfrak{p}}_H)$ on $\xi \defeq w^{\opn{hw}}_{(t_1, t_2; -r_1)} \chi_{(1, 1; -1)}$, it suffices to check that $\xi$ is sent to the same element when one traverses the diagram in either direction. Going round the diagram in a clockwise direction, we see that $\xi$ is sent to the function
    \[
    \left( \begin{smallmatrix} 1 & & & \\ z & 1 & & \\ a & & 1 & \\ b & a & -z & 1 \end{smallmatrix} \right) \mapsto (-1)^{t_1+t_2} (1+z)^{r_2+1+t_2} (b-a)^{t_1} \left( \frac{a}{1+z} \right)^{t_2} .
    \]
    On the other hand, the image of $\xi$ going round the diagram in an anticlockwise direction is given by:
    \[
    \left( \begin{smallmatrix} 1 & & & \\ z & 1 & & \\ a & & 1 & \\ b & a & -z & 1 \end{smallmatrix} \right) \mapsto (-1)^{t_1+t_2-1} (b-a)^{t_1} a^{t_2} (1+z)^{r_2} (-1-z)
    \]
    Both of these computations use Lemma \ref{explicitBranchingUnipLemma}. These two expressions are evidently equal.
\end{proof}

\begin{remark}
    The morphism $\opn{br}^{n\opn{-an}}$ clearly restricts to a morphism
    \[
    \opn{br}^{n\opn{-an}} \colon \mbb{I}_{\mathcal{M}_{H, n, \beta}^{\diamondsuit}}(\kappa_H) \to \mbb{I}_{\mathcal{M}_{G, n, \beta}^{\square}}(\kappa_G)
    \]
    which satisfies 
    \[
    \opn{br}^{n\opn{-an}}(\kappa_H)\left( \begin{smallmatrix} 1 & & & \\ z & 1 & & \\  & & 1 & \\  &  & -z & 1 \end{smallmatrix} \right) = \lambda(1+z) .
    \]
    This coincides with the ``kraken'' $\mathcal{H}^{\lambda}$ introduced in \cite[Definition 8.3.1]{LZBK21}.
\end{remark}

\section{Shimura varieties and algebraic periods}

We now introduce the relevant Shimura varieties for $G$, $H$ and $\opn{GL}_2$ and construct an algebraic trilinear period which is closely related to the automorphic periods needed for the $p$-adic $L$-functions.

\subsection{The Shimura varieties}

For $\mathscr{G} \in \{ G, H, \opn{GL}_2 \}$ and a compact open subgroup $K \subset \mathscr{G}(\mbb{A}_f)$ which is sufficiently small, we let $Y_{\mathscr{G}}(K)$ denote the Shimura variety over $\mbb{Q}$ associated with $\mathscr{G}$ of level $K$. We let $X_{\mathscr{G}}(K)$ denote a smooth toroidal compactification of $Y_{\mathscr{G}}(K)$, and we denote the boundary divisor by $D_{\mathscr{G}} = X_{\mathscr{G}}(K) - Y_{\mathscr{G}}(K)$. This toroidal compactification depends on a choice of rational polyhedral cone decomposition (which we will always omit from the notation); following \cite[\S 2.3.3]{LPSZ} we assume that these choices are ``good'' in the sense of \cite[\S 6.1.5]{Pilloni}. When varying the level $K$, these choices of cone decompositions will change, however we can always choose good refinements so that the $\mathscr{G}$-equivariance properties of $Y_{\mathscr{G}}(-)$ extend to $X_{\mathscr{G}}(-)$, as well as the functoriality in $\mathscr{G}$. For $\mathscr{G} = H, \opn{GL}_2$, we may choose the cone decomposition so that the corresponding toroidal compactification agrees with the minimal compactification.

\subsubsection{} \label{LevelSubgroupsSSSec}

We will consider the following level subgroups for these Shimura varieties. Fix a sufficiently small compact open subgroup $K^p \subset G(\mbb{A}_f^p)$.

\begin{definition}
    Let $\beta \geq 1$ be an integer. We let:
    \begin{enumerate}
        \item $K^G_{\opn{Iw}}(p^{\beta}) \subset G(\mbb{Z}_p)$ denote the subgroup of elements which lie in $B_G$ modulo $p^{\beta}$;
        \item $K^H_{\diamondsuit}(p^{\beta}) \subset H(\mbb{Z}_p)$ denote the subgroup of elements $(h_1, h_2) \in H(\mbb{Z}_p)$ which satisfy:
        \[
        (h_1, h_2) \equiv \left( \left( \begin{smallmatrix} x & y \\ 0 & z \end{smallmatrix} \right), \left( \begin{smallmatrix} x & -y \\ 0 & z \end{smallmatrix} \right) \right) \quad \text{ modulo } p^{\beta}
        \]
        for some $x,z \in \mbb{Z}_p^{\times}$ and $y \in \mbb{Z}_p$ (which satisfies $K^H_{\diamondsuit}(p^{\beta}) = \hat{\gamma} K^G_{\opn{Iw}}(p^{\beta}) \hat{\gamma}^{-1} \cap H(\mbb{Q}_p)$);
        \item $K^{\opn{GL}_2}_{\opn{Iw}}(p^{\beta}) \subset \opn{GL}_2(\mbb{Z}_p)$ denote the subgroup of elements which lie in $B_{\opn{GL}_2}$ modulo $p^{\beta}$.
    \end{enumerate}
    Note that if $\opn{pr}_i \colon H \to \opn{GL}_2$ is projection to one of the $\opn{GL}_2$-factors, then $\opn{pr}_i(K^H_{\diamondsuit}(p^{\beta})) = K^{\opn{GL}_2}_{\opn{Iw}}(p^{\beta})$. We set
    \[
    K^G_{\beta} \defeq K^p K^G_{\opn{Iw}}(p^{\beta}), \quad K^H_{\beta} \defeq (K^p \cap H(\mbb{A}_f^p)) K^H_{\diamondsuit}(p^{\beta}), \quad K^{H_i}_{\beta} = \opn{pr}_i (K^p \cap H(\mbb{A}_f^p)) K^{\opn{GL}_2}_{\opn{Iw}}(p^{\beta})
    \]
    for $i=1, 2$ (where $H = H_1 \times_{\mbb{G}_m} H_2$).
\end{definition}

\subsubsection{} Let $\mathscr{G} \in \{ G, H, \opn{GL}_2 \}$. Then $Y_{\mathscr{G}}(K)$ represents a certain moduli problem of principally polarised abelian varieties with extra structure. More precisely, $Y_{G}(K)$ (resp. $Y_{\opn{GL}_2}(K)$) parameterises principally polarised abelian surfaces (resp. elliptic curves) with $K$-level structure, and $Y_{H}(K)$ parameterises pairs of elliptic curves with $K$-level structure. Over $Y_{\mathscr{G}}(K)$, we have a natural $\overline{P}_{\mathscr{G}}$-torsor
\[
P_{\mathscr{G}, \opn{dR}} \to Y_{\mathscr{G}}(K)
\]
which parameterises trivialisations of the first relative de Rham homology of the universal object over $Y_{\mathscr{G}}(K)$ respecting the Hodge filtration. These torsors carry a $\mathscr{G}$-equivariant structure as the level $K$ varies. These torsors (and their equivariant structure) extend to the toroidal compactifications $X_{\mathscr{G}}(K)$ by considering trivialisations of the canonical extension of the first relative de Rham homology of the universal object -- we will continue to denote this torsor by $P_{\mathscr{G}, \opn{dR}} \to X_{\mathscr{G}}(K)$.

\begin{notation}
    For any (not necessarily finite-dimensional) algebraic representation $V$ of $\overline{P}_{\mathscr{G}}$, we let 
    \[
    [V] = \left(\pi_* \mathcal{O}_{P_{\mathscr{G}, \opn{dR}}} \otimes V \right)^{\overline{P}_{\mathscr{G}}}
    \]
    denote the corresponding quasi-coherent sheaf on $X_{\mathscr{G}}(K)$. Here $\pi \colon P_{\mathscr{G}, \opn{dR}} \to X_{\mathscr{G}}(K)$ denotes the structural map. If $V$ is additionally a $(\opn{Lie}\mathscr{G}, \overline{P}_{\mathscr{G}} )$-module, then $[V]$ comes equipped with an integrable connection with logarithmic poles along the boundary divisor $D_{\mathscr{G}}$.
\end{notation}

\subsubsection{} \label{MapsBetweenShVarsSSec}

Since $K^H_{\beta} = \hat{\gamma} K^G_{\beta} \hat{\gamma}^{-1} \cap H(\mbb{A}_f)$, we obtain a finite unramified morphism
\[
\hat{\iota} \colon X_H(K^H_{\beta}) \to X_G(K^G_{\beta})
\]
induced from right-translation by $\hat{\gamma}$ (and the morphism of Shimura data between $H$ and $G$). For any $\beta \geq 1$, the morphism $\hat{\iota}$ is the pullback of the morphism for $\beta = 1$ with respect to the various forgetful maps; more precisely, one has a Cartesian diagram:
\[
\begin{tikzcd}
X_{H}(K^H_{\beta+1}) \arrow[d] \arrow[r, "\hat{\iota}"] & X_{G}(K^G_{\beta+1}) \arrow[d] \\
X_{H}(K^H_{\beta}) \arrow[r, "\hat{\iota}"]             & X_{G}(K^G_{\beta})            
\end{tikzcd}
\]
for any $\beta \geq 1$ (since both vertical maps have degree $p^4 = [K^H_{\diamondsuit}(p^{\beta}) : K^H_{\diamondsuit}(p^{\beta+1})] = [K^G_{\opn{Iw}}(p^{\beta}) : K^G_{\opn{Iw}}(p^{\beta+1})]$). One can easily check that one has a reduction of structure $\hat{\iota}^{-1}P_{G, \opn{dR}} = P_{H, \opn{dR}} \times^{[\overline{P}_H, \gamma]} \overline{P}_{G}$ where the superscript denotes the pushout along the map $\overline{P}_{H} \hookrightarrow \overline{P}_{G}$, $h \mapsto \gamma^{-1} h \gamma$ induced from conjugation by $\gamma^{-1}$ (recall that $\gamma \in M_G$).

Similarly, since $K^{H_i}_{\beta} = \opn{pr}_i(K^H_{\beta})$, we obtain proper morphisms
\[
\opn{pr}_i \colon X_{H}(K^H_{\beta}) \twoheadrightarrow X_{\opn{GL}_2}(K^{H_i}_{\beta})
\]
for $i=1,2$. One has a reduction of structure $\opn{pr}_i^{-1} P_{\opn{GL}_2, \opn{dR}} = P_{\opn{dR}, H} \times^{\overline{P}_H} \overline{P}_{\opn{GL}_2}$ where the pushout is along the projection to the $i$-th factor.

\subsection{Coherent cohomology}

We now introduce some coherent cohomology complexes.

\begin{definition} \label{DefOfNearlyAlgCohComplexes}
    Suppose that we are in Situation \ref{situation:Weights}. Let $\mbb{D}_{\overline{P}_G}(\kappa_G^*)^{\opn{nearly}}$ and $\mbb{D}_{\overline{P}_H}(\kappa_H^*)^{\opn{nearly}}$ denote the linear duals of $\mbb{I}_{\overline{P}_G}(\kappa_G)^{\opn{nearly}}$ and $\mbb{I}_{\overline{P}_H}(\kappa_H)^{\opn{nearly}}$ respectively. We set:
    \begin{align*}
        R\Gamma( K^G_{\beta}, \kappa_G^*; \opn{cusp})^{\opn{nearly}} &\defeq R\Gamma(X_{G}(K^G_{\beta}), [\mbb{D}_{\overline{P}_G}(\kappa_G^*)^{\opn{nearly}}](-D_G) ) \\
        R\Gamma( K^H_{\beta}, \kappa_H^*; \opn{cusp})^{\opn{nearly}} &\defeq R\Gamma(X_{H}(K^H_{\beta}), [\mbb{D}_{\overline{P}_H}(\kappa_H^*)^{\opn{nearly}}](-D_H) ) \\
        R\Gamma( K^H_{\beta}, \kappa_H - 2\rho_H)^{\opn{nearly}} &\defeq R\Gamma(X_{H}(K^H_{\beta}), [\mbb{I}_{\overline{P}_H}(\kappa_H)^{\opn{nearly}} \otimes V_{M_H}(-2\rho_H)] ) \\
        R\Gamma( K^{H_i}_{\beta}, \zeta_{H_i})^{\opn{nearly}} &\defeq R\Gamma(X_{\opn{GL}_2}(K^{H_i}_{\beta}), [\mbb{I}_{\overline{P}_{\opn{GL}_2}}(\zeta_{H_i})^{\opn{nearly}}] ) .
    \end{align*}
    We also use the notation $R\Gamma( K^G_{\beta}, \kappa_G^*; \opn{cusp})$, $R\Gamma( K^H_{\beta}, \kappa_H^*; \opn{cusp})$, $ R\Gamma( K^H_{\beta}, \kappa_H - 2\rho_H)$, $R\Gamma( K^{H_i}_{\beta}, \zeta_{H_i})$ for the cohomology complexes defined in exactly the same way but replacing $\mbb{D}_{\overline{P}_G}(\kappa_G^*)^{\opn{nearly}}$, $\mbb{D}_{\overline{P}_H}(\kappa_H^*)^{\opn{nearly}}$, $\mbb{I}_{\overline{P}_H}(\kappa_H)^{\opn{nearly}}$, $\mbb{I}_{\overline{P}_{\opn{GL}_2}}(\zeta_{H_i})^{\opn{nearly}}$ with $V_{M_G}(\kappa_G)^*$, $V_{M_H}(\kappa_H)^*$, $V_{M_H}(\kappa_H)$, $V_{M_{\opn{GL}_2}}(\zeta_{H_i})$ respectively. 
\end{definition}

The cohomology groups of the complexes in Definition \ref{DefOfNearlyAlgCohComplexes} should be viewed as certain spaces of ``nearly holomorphic'' automorphic forms for the corresponding reductive groups with fixed weight and bounded degree/order (with a cuspidality condition in the first two cases). For example, set 
\[
\mathscr{N}_{H_i, \beta}^{\opn{nhol}} \defeq \opn{H}^0\left( X_{\opn{GL}_2}(K^{H_i}_{\beta}), \pi_*\mathcal{O}_{P_{\opn{GL}_2, \opn{dR}}} \right)
\]
which one should view of as the space of all nearly holomorphic modular forms of level $K^{H_i}_{\beta}$ without a fixed weight. This is naturally a $(\mathfrak{gl}_2, \overline{P}_{\opn{GL}_2})$-module with an ascending filtration 
\[
\opn{Fil}_i\mathscr{N}_{H_i, \beta}^{\opn{nhol}} \defeq \left\{ F \in \mathscr{N}_{H_i, \beta}^{\opn{nhol}} : \left( \begin{smallmatrix} 0 & 0 \\ 1 & 0 \end{smallmatrix} \right)^{i+1} \cdot F = 0 \right\}, \quad \quad i \geq 0  
\]
with $\opn{Fil}_0 \mathscr{N}_{H_i, \beta}^{\opn{nhol}}$ equal to the space of holomorphic modular forms of level $K^{H_i}_{\beta}$. The action of $\left( \begin{smallmatrix} 0 & 1 \\ 0 & 0 \end{smallmatrix} \right)$ is an algebraic interpretation of the Maass--Shimura differential operator and satisfies Griffiths transversality with respect to the above filtration. We then have identifications:
\[
\opn{H}^0(K^{H_i}_{\beta}, \zeta_{H_i})^{\opn{nearly}} = \opn{Hom}_{T_{\opn{GL}_2}}\left( -\zeta_{H_i}, \opn{Fil}_{t_i}\mathscr{N}^{\opn{nhol}}_{H_i, \beta} \right), \quad \quad \opn{H}^0(K^{H_i}_{\beta}, \zeta_{H_i}) = \opn{Hom}_{T_{\opn{GL}_2}}\left( -\zeta_{H_i}, \opn{Fil}_{0}\mathscr{N}^{\opn{nhol}}_{H_i, \beta} \right) .
\]
As explained in \cite[Remark 6.3.2]{DiffOps}, $\opn{H}^0(K^{H_i}_{\beta}, \zeta_{H_i})^{\opn{nearly}}$ (resp. $\opn{H}^0(K^{H_i}_{\beta}, \zeta_{H_i})$) is identified with the usual space of nearly holomorphic forms of weight $1+t_i$, level $K^{H_i}_{\beta}$ and degree $\leq t_i$ (resp. degree $0$) after restricting weights to $\opn{SL}_2$ and identifying the universal elliptic curve over $Y_{\opn{GL}_2}$ with its dual via the principal polarisation.

\subsubsection{} Suppose we are in Situation \ref{situation:Weights}. Then the reduction of structure for the groups $G$, $H$ in \S \ref{MapsBetweenShVarsSSec} and the map $\mbb{D}_{\overline{P}_G}(\kappa_G^*)^{\opn{nearly}} \to \mbb{D}_{\overline{P}_H}(\kappa_H^*)^{\opn{nearly}}$ induced from the dual of $\opn{br}^{\opn{nearly}}$ induces a pullback map 
\[
\hat{\iota}^* \colon R\Gamma( K^G_{\beta}, \kappa_G^*; \opn{cusp})^{\opn{nearly}} \to R\Gamma( K^H_{\beta}, \kappa_H^*; \opn{cusp})^{\opn{nearly}} .
\]
Similarly, the reduction of structure for the groups $H$, $H_i = \opn{GL}_2$ in \S \ref{MapsBetweenShVarsSSec} and the decomposition in Remark \ref{DecompositionOfNearlyIntoFactorsRem} induces an exterior cup product map 
\begin{align*}
    \opn{H}^0( K^{H_1}_{\beta}, \zeta_{H_1})^{\opn{nearly}} \otimes \opn{H}^0( K^{H_2}_{\beta}, \zeta_{H_2})^{\opn{nearly}} &\to \opn{H}^0( K^H_{\beta}, \kappa_H - 2\rho_H)^{\opn{nearly}} \\
    (\omega_1, \omega_2) &\mapsto \omega_1 \sqcup \omega_2 \defeq  \opn{pr}_1^*(\omega_1) \smile \opn{pr}_2^*(\omega_2) .
\end{align*}
Finally, note that for any $i=0, 1, 2$, we have Serre duality pairings
\[
\langle \cdot , \cdot \rangle \colon \opn{H}^i\left( K^H_{\beta}, \kappa_H^*; \opn{cusp} \right)^{\opn{nearly}} \times \opn{H}^{2-i}\left( K^H_{\beta}, \kappa_H - 2\rho_H \right)^{\opn{nearly}} \to \mbb{Q} .
\]
induced from the natural duality pairing $\mbb{D}_{\overline{P}_H}(\kappa_H^*)^{\opn{nearly}} \otimes \mbb{I}_{\overline{P}_H}(\kappa_H)^{\opn{nearly}} \to \mbb{Q}$ and the identification $[V_{M_H}(-2\rho_H)] \cong \Omega^2_{X_H(K^H_{\beta})}(\opn{log}D_H)$. We define the algebraic trilinear period as follows:

\begin{definition}
    Suppose we are in Situation \ref{situation:Weights} and let $L$ be a characteristic zero field. Then we let $\mathcal{P}^{\opn{alg}}_L$ denote the $L$-linear morphism
    \begin{align*}
        \mathcal{P}^{\opn{alg}}_L \colon \opn{H}^2( K^G_{\beta}, \kappa_G^*; \opn{cusp})_L^{\opn{nearly}} \otimes \opn{H}^0( K^{H_1}_{\beta}, \zeta_{H_1})_L^{\opn{nearly}} \otimes \opn{H}^0( K^{H_2}_{\beta}, \zeta_{H_2})_L^{\opn{nearly}} &\to L \\
        (\eta, \omega_1, \omega_2 ) &\mapsto \langle \hat{\iota}^* \eta, \omega_1 \sqcup \omega_2 \rangle .
    \end{align*}
\end{definition}

\subsection{The archimedean period} \label{TheArchimideanPeriodSSec}

In this subsection, we describe the relation between $\mathcal{P}^{\opn{alg}}_{\mbb{C}}$ and certain automorphic periods for automorphic forms of $G$ and $H$. We first describe the relation between the cohomology complexes in Definition \ref{DefOfNearlyAlgCohComplexes} and automorphic forms for the corresponding group. In this subsection, gothic letters will denote the corresponding complexified Lie algebras. 

Let $\mathscr{A}(G)$ denote the space of automorphic forms for $G$ and let $\mathscr{A}_0(G) \subset \mathscr{A}(G)$ denote the subspace of cuspforms. We use similar notation for $H$ and $\opn{GL}_2$. Let $K_{G, \infty} = \mbb{R}^{\times} \cdot U_2(\mbb{R}) \subset G(\mbb{R})$ (resp. $K_{H, \infty} = Z_H(\mbb{R}) \cdot (U_1(\mbb{R}) \times U_1(\mbb{R})) \subset H(\mbb{R})$, resp. $K_{\opn{GL}_2, \infty} = \mbb{R}^{\times} \cdot U_1(\mbb{R}) \subset \opn{GL}_2(\mbb{R})$) denote the standard maximal compact-mod-centre subgroups. For $\mathscr{G} \in \{ G, H, \opn{GL}_2 \}$, we set $K_{\mathscr{G}, \infty}^+ = K_{\mathscr{G}, \infty} \cap \mathscr{G}(\mbb{R})_+$, where $\mathscr{G}(\mbb{R})_+ \subset \mathscr{G}(\mbb{R})$ denotes the connected component of the identity in the $\mbb{R}$-analytic topology. Let $\mu_{\mathscr{G}} \colon \mbb{G}_m \to \mathscr{G}$ denote the standard choices of Hodge cocharacters given by
\[
\mu_G(z) = \left( \begin{smallmatrix} z & & & \\ & z & & \\ & & 1 & \\ & & & 1 \end{smallmatrix} \right), \quad \mu_H(z) = \left( \left( \begin{smallmatrix} z & \\ & 1 \end{smallmatrix} \right), \left( \begin{smallmatrix} z & \\ & 1 \end{smallmatrix} \right) \right), \quad \mu_{\opn{GL}_2}(z) = \left( \begin{smallmatrix} z & \\ & 1 \end{smallmatrix} \right) .
\]
We let $K_{\mathscr{G}, \infty}^+$ act on any algebraic representation of $\overline{P}_{\mathscr{G}}(\mbb{C})$ through the map $K_{\mathscr{G}, \infty}^+ \hookrightarrow \overline{P}_{\mathscr{G}}(\mbb{C})$ given by conjugation by an element in $\mathscr{G}(\mbb{C})$ which transforms the Hodge cocharacter associated with $K_{\mathscr{G}, \infty}^+$ into the standard Hodge cocharacter $\mu_{\mathscr{G}}$.

\begin{proposition}[Harris, Su] \label{HarrisSuProposition}
    There are identifications
    \begin{align}
        \opn{H}^2( K^G_{\beta}, \kappa_G^*; \opn{cusp})_{\mbb{C}}^{\opn{nearly}} &= \opn{H}^2_{(\overline{\mathfrak{p}}_G, K_{G,\infty}^+)}\left( \mathscr{A}_0(G)^{K^G_{\beta}} \otimes \mbb{D}_{\overline{P}_G}(\kappa_G^*)^{\opn{nearly}} \right) \label{LieAlgGEqn} \\
        \opn{H}^0( K^{H_i}_{\beta}, \zeta_{H_i})_{\mbb{C}}^{\opn{nearly}} &= \opn{H}^0_{(\overline{\mathfrak{p}}_{\opn{GL}_2}, K_{\opn{GL}_2, \infty}^+)}\left( \mathscr{A}(\opn{GL}_2)^{K^{H_i}_{\beta}} \otimes \mbb{I}_{\overline{P}_{\opn{GL}_2}}(\zeta_{H_i})^{\opn{nearly}} \right) \label{LieAlgHiEqn}
    \end{align}
    where the right-hand side denotes the relative Lie algebra cohomology.
\end{proposition}
\begin{proof}
    This is the main result of \cite{Su19}.
\end{proof}

The Lie algebra cohomology group in (\ref{LieAlgGEqn}) can be calculated via the Chevalley--Eilenberg complex whose $j$-th term is
\[
\opn{Hom}_{K_{G, \infty}^+}\left( \wedge^j \overline{\mathfrak{u}}_G \otimes \mbb{I}_{\overline{P}_G}(\kappa_G)^{\opn{nearly}}, \mathscr{A}_0(G)^{K^G_{\beta}} \right) .
\]
We have a similar description for the Lie algebra cohomology group in (\ref{LieAlgHiEqn}). In particular, if $\eta \in \opn{H}^2( K^G_{\beta}, \kappa_G^*; \opn{cusp})_{\mbb{C}}^{\opn{nearly}}$, then we have an associated $K_{G, \infty}^+$-equivariant morphism
\[
\phi_{\eta} \colon \wedge^2 \overline{\ide{u}}_G \otimes \mbb{I}_{\overline{P}_G}(\kappa_G)^{\opn{nearly}} \to \mathscr{A}_0(G)^{K^G_{\beta}} .
\]
Similarly, if $\omega_i \in \opn{H}^0( K^{H_i}_{\beta}, \zeta_{H_i})_{\mbb{C}}^{\opn{nearly}}$, then we have an associated $K_{\opn{GL}_2, \infty}^+$-equivariant morphism:
\[
F_{\omega_i} \colon \mbb{D}_{\overline{P}_{\opn{GL}_2}}(\zeta_{H_i}^*)^{\opn{nearly}} \to \mathscr{A}(\opn{GL}_2)^{K^{H_i}_{\beta}} .
\]
Let $[H] = \mbb{R}_{>0}^{\times} H(\mbb{Q}) \backslash H(\mbb{A})$, where $\mbb{R}_{>0}^{\times} \hookrightarrow Z_H(\mbb{R})$ is embedded diagonally, and let $dh$ denote the Tamagawa measure on $[H]$. Let $\alpha_i \in \overline{\ide{u}}_{H_i}$ denote a choice of basis element defined over $\mbb{Q}$, and set $\alpha_{1, 2} = \alpha_1 \otimes \alpha_2 \in \wedge^2 \overline{\ide{u}}_H = \overline{\ide{u}}_{H_1} \otimes \overline{\ide{u}}_{H_2}$. The natural map $\wedge^2\overline{\ide{u}}_H \hookrightarrow \wedge^2 \overline{\ide{u}}_G$ sends $\alpha_{1, 2}$ to $\alpha_1 \wedge \alpha_2$.

\begin{proposition} \label{PAlgAutPeriodProposition}
    With notation as above and as in \S \ref{SplttingsOfMGReps}: 
    \begin{itemize}
    \item set $\phi \defeq \hat{\gamma} \cdot \phi_{\eta}((\alpha_1 \wedge \alpha_2) \otimes \jmath_{M_G}(w_{\kappa_H}))$, where $\hat{\gamma} \in G(\mbb{Q}_p)$ acts through the natural map $G(\mbb{Q}_p) \hookrightarrow G(\mbb{A})$;
    \item for $i=1, 2$, let $F_i \defeq F_{\omega_i}(\upsilon_{\zeta_{H_i}}^*)$.
    \end{itemize}
    Suppose that the morphism $\phi_{\eta}$ factors through the morphism $\opn{id} \otimes \opn{spl}_{M_G} \colon \wedge^2 \overline{\ide{u}}_G \otimes \mbb{I}_{\overline{P}_G}(\kappa_G)^{\opn{nearly}} \to \wedge^2 \overline{\ide{u}}_G \otimes V_{M_G}(\kappa_G)$. Then, up to a non-zero rational constant, we have
    \begin{equation} \label{PalgEqualsAutPeriodEqn}
    \mathcal{P}^{\opn{alg}}_{\mbb{C}}(\eta, \omega_1, \omega_2) = (2 \pi i )^{-2} \opn{Vol}(\mathcal{K}_{\beta}; dh)^{-1} \int_{[H]} \phi(h) F_1(h_1) F_2(h_2) dh 
    \end{equation}
    where $\mathcal{K}_{\beta}$ denotes the image of $K^+_{H, \infty}K^H_{\beta}$ in $[H]$.
\end{proposition}
\begin{proof}
    The pullback $\hat{\iota}^*(\eta)$ is represented by the $K_{H, \infty}^+$-equivariant homomorphism:
    \begin{align*}
        \hat{\iota}^*\phi_{\eta} \colon \wedge^2 \overline{\ide{u}}_H \otimes \mbb{I}_{\overline{P}_{H}}(\kappa_H)^{\opn{nearly}} &\to \mathscr{A}_0(H)^{K^H_{\beta}} \\
        \alpha_{1, 2} \otimes x &\mapsto \left.\left[ \hat{\gamma} \cdot \phi_{\eta}((\alpha_1 \wedge \alpha_2) \otimes \gamma \cdot \opn{br}^{\opn{nearly}}(x) ) \right] \right|_{H(\mbb{A})} .
    \end{align*}
    Similarly, the exterior cup product $\omega_1 \sqcup \omega_2$ is represented by the $K_{H, \infty}^+$-equivariant homomorphism:
    \begin{align*}
        F_{\omega_1 \sqcup \omega_2} \colon \mbb{D}_{\overline{P}_H}(\kappa_H^*)^{\opn{nearly}} \otimes V_{M_H}(-2\rho_H)^* = \mbb{D}_{\overline{P}_{\opn{GL}_2}}(\zeta_{H_1}^*)^{\opn{nearly}} \boxtimes \mbb{D}_{\overline{P}_{\opn{GL}_2}}(\zeta_{H_2}^*)^{\opn{nearly}} \to \mathscr{A}(H)^{K^H_{\beta}} \\
        y \boxtimes z \mapsto \left( (h_1, h_2) \mapsto F_{\omega_1}(y)(h_1)F_{\omega_2}(z)(h_2) \right) .
    \end{align*}
    Consider the $M_{H}$-equivariant morphism
    \[
    \lambda \colon \mbb{I}_{\overline{P}_{H}}(\kappa_H)^{\opn{nearly}} \xrightarrow{\gamma \cdot \opn{br}^{\opn{nearly}}} \mbb{I}_{\overline{P}_{G}}(\kappa_G)^{\opn{nearly}} \xrightarrow{\opn{spl}_{M_G}} V_{M_G}(\kappa_G) .
    \]
    This necessarily factors as
    \[
    \mbb{I}_{\overline{P}_{H}}(\kappa_H)^{\opn{nearly}} \xrightarrow{s'} V_{M_H}(\kappa_H) \xrightarrow{\opn{br}_M} V_{M_G}(\kappa_G)
    \]
    and $s'$ is the unique splitting of the natural map $V_{M_H}(\kappa_H) \hookrightarrow \mbb{I}_{\overline{P}_{H}}(\kappa_H)^{\opn{nearly}}$. Fix a basis of $\mbb{I}_{\overline{P}_{H}}(\kappa_H)^{\opn{nearly}}$ given by $\{ e'=\kappa_H \} \cup \{ e_j : j \in J \}$, where $\kappa_H \in V_{M_H}(\kappa_H)$ is the canonical character, and $\{ e_j \}$ is a basis of the kernel of $s'$. Let $\{ (e')^* \} \cup \{ e_j^* : j \in J \}$ denote the dual basis. Then the cup product $\hat{\iota}^*(\eta) \cup (\omega_1 \sqcup \omega_2 )$ is represented by the $K_{H, \infty}^+$-equivariant morphism
    \[
    \Xi \colon \wedge^2 \overline{\ide{u}}_H \otimes V_{M_H}(-2\rho_H)^* \to \mathscr{A}_0(H)^{K^H_{\beta}}
    \]
    given by 
    \[
    \Xi( \alpha_{1, 2} \otimes x) = (\hat{\iota}^*\phi_{\eta})(\alpha_{1, 2} \otimes e') F_{\omega_1 \sqcup \omega_2}((e')^* \otimes x) + \sum_{j \in J} (\hat{\iota}^*\phi_{\eta})(\alpha_{1, 2} \otimes e_j) F_{\omega_1 \sqcup \omega_2}(e_j^* \otimes x) .
    \]
    By the assumption that $\phi_{\eta}$ factors through $\opn{id} \otimes \opn{spl}_{M_G}$, we see that for $x$ equal to the vector dual to the canonical basis of $V_{M_H}(-2\rho_H)$
    \[
    \Xi(\alpha_{1, 2} \otimes x) = (\hat{\iota}^*\phi_{\eta})(\alpha_{1, 2} \otimes e') F_{\omega_1 \sqcup \omega_2}((e')^* \otimes x) = \left( (h_1, h_2) \mapsto \phi(h) F_1(h_1) F_2(h_2) \right)
    \]
    where we have used the fact that $(e')^* \otimes x = \upsilon_{\zeta_{H_1}}^* \boxtimes \upsilon_{\zeta_{H_2}}^*$. We can identify $(\wedge^2 \overline{\ide{u}}_H)^* \otimes V_{M_H}(-2\rho_H)$ with the top exterior power of the sheaf of (real-analytic) differentials on $X_H(K^H_{\beta})$, and the cup product $\hat{\iota}^*(\eta) \cup (\omega_1 \sqcup \omega_2 )$ is therefore identified (up to non-zero rational scalar) with a $K_{H, \infty}^+$-invariant volume form 
    \[
    (h_1, h_2) \mapsto \phi(h) F_1(h_1) F_2(h_2) \cdot \alpha_{1, 2}^* \otimes x^* .
    \]
    The result now follows from \cite[Proposition 3.8]{HarrisPartial}.
\end{proof}

\begin{remark}
    The non-zero rational constant in Proposition \ref{PAlgAutPeriodProposition} is independent of the input data $(\eta, \omega_1, \omega_2)$ as well as the level subgroup $K^H_{\beta}$. Therefore, we can (and do) choose $\alpha_i$ such that the equality (\ref{PalgEqualsAutPeriodEqn}) actually holds (i.e., not just up to rational scalar) for any choice of input data and level subgroup.
\end{remark}

\section{Adic Shimura varieties}

In this section, we construct a $p$-adic version of the trilinear pairing $\mathcal{P}^{\opn{alg}}$ by studying the coherent cohomology of the adic Shimura varieties associated with the groups $G$, $H$, $\opn{GL}_2$. In this section, we work over $\mbb{Q}_p$ but will often omit this from the notation.

\subsection{Geometry}

For any $\mathscr{G} \in \{ G, H, \opn{GL}_2 \}$ and sufficiently small compact open subgroup $K \subset \mathscr{G}(\mbb{A}_f)$, we let $\mathcal{Y}_{\mathscr{G}}(K)$ (resp. $\mathcal{X}_{\mathscr{G}}(K)$) denote the adic space over $\mbb{Q}_p$ associated with $Y_{\mathscr{G}}(K)_{\mbb{Q}_p}$ (resp. $X_{\mathscr{G}}(K)_{\mbb{Q}_p}$) via rigid analytification. We will define certain subspaces of these adic Shimura varieties via the Hodge--Tate period map.

\subsubsection{} 

Suppose that $K=K^pK_p \subset \mathscr{G}(\mbb{A}_f)$ is a sufficiently small compact open subgroup. As explained in \cite[\S 4.4.10]{BoxerPilloni}, one has a (truncated) Hodge--Tate period map:
\[
\pi_{\opn{HT}, \mathscr{G}, K_p} \defeq \pi_{\opn{HT}, K_p} \colon \mathcal{X}_{\mathscr{G}}(K) \to \mathtt{FL}^{\mathscr{G}} / K_p
\]
where $\mathtt{FL}^{\mathscr{G}} = P_{\mathscr{G}}^{\opn{an}} \backslash {\mathscr{G}}^{\opn{an}}$ denotes the adic space over $\mbb{Q}_p$ associated with the (partial) flag variety $\opn{FL}_{\mathscr{G}} \defeq P_{\mathscr{G}} \backslash \mathscr{G}$ (here the rational polyhedral cone decomposition is chosen to be independent of $K_p$ so that $\varprojlim_{K_p' \subset K_p} \mathcal{X}_{\mathscr{G}}(K^pK_p')$ exists as a perfectoid space). The morphism $\pi_{\opn{HT}, \mathscr{G}, K_p}$ is functorial in the level $K_p$ and Shimura datum.

\subsubsection{} \label{FlagVarietyStrataSSec}

Let ${^M W_{\mathscr{G}}}$ denote the set of minimal length representatives for $W_{M_{\mathscr{G}}} \backslash W_{\mathscr{G}}$. Then one has a Bruhat decomposition:
\[
\opn{FL}_{\mathscr{G}, \mbb{F}_p} = \bigcup_{w \in {^M W_{\mathscr{G}}}} P_{\mathscr{G}} \backslash \left( P_{\mathscr{G}} \cdot w \cdot B_{\mathscr{G}} \right) = \bigcup_{w \in {^M W_{\mathscr{G}}}} C_w^{\mathscr{G}} .
\]
Let $X^{\mathscr{G}}_w \defeq \cup_{w' \leq w} C_{w'}^{\mathscr{G}}$ and $Y^{\mathscr{G}}_w \defeq \cup_{w' \geq w} C_{w'}^{\mathscr{G}}$ where $\leq$ denotes the Bruhat ordering on ${^M W_{\mathscr{G}}}$. For any locally closed subset $Z \subset \opn{FL}_{\mathscr{G}, \mbb{F}_p}$, we let ${]Z[} \subset \mathtt{FL}^{\mathscr{G}}$ denote the tube of $Z$ in the adic flag variety, i.e., ${]Z[}$ is defined to be the interior of the preimage of $Z$ under the specialisation map $\opn{sp} \colon \mathtt{FL}^{\mathscr{G}} \to \opn{FL}_{\mathscr{G}, \mbb{F}_p}$. For any $m, n \in \mbb{Q} \cup \{ -\infty \}$, we let ${]C^{\mathscr{G}}_w[_{m, n}}$, $]C^{\mathscr{G}}_w[_{\overline{m}, n}$, etc. denote the (partially compactified) tubes of varying radii introduced in \cite[\S 3.3.6]{BoxerPilloni}.

\begin{definition} \label{DefOfOrdinaryStrata}
    Let $\beta \geq 1$ be an integer.
    \begin{enumerate}
        \item For $K_p = K^G_{\opn{Iw}}(p^{\beta})$, let $\mathcal{X}_{G, w_1, \beta} \subset \mathcal{X}_G(K^G_{\beta})$ denote the quasi-compact open subspace given by the interior of 
        \[
        \bigcap_{k \geq 1} \pi_{\opn{HT}, K_p}^{-1}\left( ]C^G_{w_1}[_{k, k} K_p \right) .
        \]
        We also let $\mathcal{Z}_{G, <1, \beta} \subset \mathcal{X}_G(K^G_{\beta})$ denote the closed subset given by the closure of $\pi_{\opn{HT}, K_p}^{-1}\left( P_{G}^{\opn{an}} \mathcal{K} K_p \right)$ where $\mathcal{K} \subset G^{\opn{an}}$ denotes the analytic upper-triangular Klingen parabolic subgroup, i.e., for $S=\opn{Spa}(R, R^+)$ over $\opn{Spa}(\mbb{Q}_p, \mbb{Z}_p)$, the group $\mathcal{K}(S)$ is equal to all elements in $G(R^+)$ of the form $\left( \begin{smallmatrix} * & * & * & * \\ & * & * & * \\ & * & * & * \\ & & & * \end{smallmatrix} \right)$. Note that $\mathcal{X}_{G, w_1, \beta} \subset \mathcal{Z}_{G, <1, \beta}$.
        \item For $K_p = K^H_{\diamondsuit}(p^{\beta})$ (resp. $K_p = K^{\opn{GL}_2}_{\opn{Iw}}(p^{\beta})$), let $\mathcal{X}_{H, \opn{id}, \beta} \subset \mathcal{X}_H(K^H_{\beta})$ (resp. $\mathcal{X}_{H_i, \opn{id}, \beta} \subset \mathcal{X}_{\opn{GL}_2}(K^{H_i}_{\beta})$) denote the quasi-compact open subspace given by the interior of 
        \[
        \bigcap_{k \geq 1} \pi_{\opn{HT}, K_p}^{-1}\left( ]C^H_{\opn{id}}[_{k, k} K_p \right) \quad \quad \left( \text{ resp. } \bigcap_{k \geq 1} \pi_{\opn{HT}, K_p}^{-1}\left( ]C^{\opn{GL}_2}_{\opn{id}}[_{k, k} K_p \right) \right) .
        \]
        Let $\mathcal{Z}_{H, \opn{id}, \beta} \subset \mathcal{X}_{H}(K^H_{\beta})$ (resp. $\mathcal{Z}_{H_i, \opn{id}, \beta} \subset \mathcal{X}_{\opn{GL}_2}(K^{H_i}_{\beta})$) denote the closure of $\mathcal{X}_{H, \opn{id}, \beta}$ (resp. $\mathcal{X}_{H_i, \opn{id}, \beta}$).
    \end{enumerate}
\end{definition}

\begin{remark}
    The subspaces in Definition \ref{DefOfOrdinaryStrata} have more tractable descriptions away from the boundary and over the good reduction locus. More precisely (over the good reduction locus and away from the boundary) the subspace $\mathcal{X}_{G, w_1, \beta}$ parameterises abelian surfaces $A$ (with a polarisation and prime-to-$p$ level structure) with good ordinary reduction and a filtration
    \[
    0 \subset H_1 \subset H_2 = H_2^{\perp} \subset H_1^{\perp} \subset A[p]
    \]
    by cyclic finite flat subgroup schemes (with $\opn{rk}H_i = p^{i\beta}$) such that $H_1$ (resp. $H_2/H_1$) extends to a multiplicative (resp. \'{e}tale) group scheme over the integral model. The subspace $\mathcal{Z}_{G, <1, \beta}$ is then the closure of the locus parameterising abelian surfaces $A$ (with polarisation and prime-to-$p$ level structure) and $0 \subset H_1 \subset H_2$ as above, such that $H_1$ extends to a multiplicative group scheme. 

    Similarly, $\mathcal{X}_{H_i, \opn{id}, \beta}$ is the canonical ordinary locus inside the modular curve, i.e., over the good reduction locus and away from the boundary it parameterises elliptic curves $E$ with a cyclic rank $p^{\beta}$ subgroup scheme $C_i \subset E[p^{\beta}]$ such that $C_i$ extends to a multiplicative group scheme over the integral model of $E$. The subspace $\mathcal{X}_{H, \opn{id}, \beta}$ is the finite \'{e}tale cover of $\mathcal{X}_{H_1, \opn{id}, \beta} \times \mathcal{X}_{H_2, \opn{id}, \beta}$ parameterising isomorphisms $C_1 \cong C_2$.
\end{remark}

Via an explicit computation on the level of flag varieties (c.f., the proof of \cite[Proposition 5.2.1]{LZBK21}), one has $\hat{\iota}(\mathcal{X}_{H, \opn{id}, \beta}) \subset \mathcal{X}_{G, w_1, \beta}$ and $\hat{\iota}^{-1}(\mathcal{Z}_{G, <1, \beta}) = \mathcal{Z}_{H, \opn{id}, \beta}$. Similarly one clearly has $\opn{pr}_i(\mathcal{X}_{H, \opn{id}, \beta}) \subset \mathcal{X}_{H_i, \opn{id}, \beta}$ for any $i=1, 2$.

\subsection{Torsors}

The goal of this section is to obtain certain reductions of structure of $P_{\mathscr{G},\opn{dR}}$ which will allow one to construct Banach sheaves locally modelled on the representations in \S \ref{padicRepTheorySubSec}. To do this, we first have to introduce certain partially compactified Igusa varieties over the components of the ordinary locus introduced in the previous section.

\subsubsection{} 

We now review the construction of the partially compactified Igusa varieties in \cite[\S 3.4]{HHTBoxerPilloni} for the group $G$. We fix a good $K^pG(\mbb{Z}_p)$-admissible cone decomposition underlying the toroidal compactifications we consider in this section. Note that we can choose the cone decomposition in such a way that it is good for any finite index subgroup of $K^p G(\mbb{Z}_p)$.\footnote{See \cite[\S 2.5]{CSnoncompact} for example.} Let $\mathfrak{X}_{G}(K^pG(\mbb{Z}_p))$ denote the formal completion of the toroidal compactification $X_{G}(K^pG(\mbb{Z}_p))/\mbb{Z}_p$ (as constructed in \cite{LanArithmetic}) along the special fibre, and let $\mathfrak{X}_{G}(K^pG(\mbb{Z}_p))^{\opn{ord}} \subset \mathfrak{X}_{G}(K^pG(\mbb{Z}_p))$ denote the ordinary locus. Let $\mathfrak{Y}_G(K^pG(\mbb{Z}_p)) \subset \mathfrak{X}_G(K^pG(\mbb{Z}_p))$ denote the formal completion of the open Shimura variety, and $\mathfrak{Y}_G(K^pG(\mbb{Z}_p))^{\opn{ord}} \subset \mathfrak{Y}_G(K^pG(\mbb{Z}_p))$ the ordinary locus.

Consider the ``standard ordinary'' $p$-divisible group $\mbb{X}_{\opn{ord}} \defeq \mu_{p^{\infty}}^{\oplus 2} \oplus \left( \mbb{Q}_p/ \mbb{Z}_p \right)^{\oplus 2}$ over $\opn{Spf}\mbb{Z}_p$, which we equip with the symplectic pairing $\langle \cdot, \cdot \rangle \colon \mbb{X}_{\opn{ord}} \times \mbb{X}_{\opn{ord}} \to \mu_{p^{\infty}}$ defined by the matrix $J$ (see \S \ref{AlgGroupsSubSec}). We let $J_{\opn{ord}}^+$ denote the fpqc group scheme over $\opn{Spf}\mbb{Z}_p$ which on $\mbb{Z}_p$-algebras $R$ for which $p$ is nilpotent is described as:
\[
J_{\opn{ord}}^+(R) \defeq \opn{Aut}_R\left( (\mbb{X}_{\opn{ord}, R}, \langle \cdot, \cdot \rangle) \right)
\]
i.e., automorphisms of $\mbb{X}_{\opn{ord}, R}$ preserving the symplectic pairing \emph{up to similitude} in $\mbb{Z}_p^{\times}$.\footnote{There is a minor clash of notation between the matrix $J$ and the group $J_{\opn{ord}}^+$, however we will not mention the former again in this section.} This can be described as the semi-direct product $J_{\opn{ord}}^+ = U(J_{\opn{ord}}^+) \rtimes M(J_{\opn{ord}}^+)$, where $M(J_{\opn{ord}}^+)$ is the pro-\'{e}tale group scheme given by $\underline{M_G(\mbb{Z}_p)}$ and the unipotent part $U(J_{\opn{ord}}^+)$ is isomorphic to the abelian group scheme of $2 \times 2$ matrices $X = \left( \begin{smallmatrix} x_{11} & x_{12} \\ x_{21} & x_{22} \end{smallmatrix} \right)$ with coefficients in $T_p\mu_{p^{\infty}}$ satisfying $x_{11} = x_{22}$.

\begin{definition}
    Let $\mathfrak{IG}_G \to \mathfrak{Y}_G(K^pG(\mbb{Z}_p))^{\opn{ord}}$ denote the fpqc $J_{\opn{ord}}^+$-torsor parameterising trivialisations $\mbb{X}_{\opn{ord}} \xrightarrow{\sim} A[p^{\infty}]$ respecting polarisations up to similitude, where $A$ denotes the universal ordinary abelian surface over $\mathfrak{Y}_G(K^pG(\mbb{Z}_p))^{\opn{ord}}$. 
    
    For any finite index fpqc subgroup scheme $J_p \subset J_{\opn{ord}}^+$, we set $U(J_p) \defeq J_p \cap U(J_{\opn{ord}}^+)$ and $M(J_p) = J_p/U(J_p) \hookrightarrow M(J_{\opn{ord}}^+)$, and we let $\mathfrak{IG}_G(K^p U(J_p))$ and $\mathfrak{IG}_G(K^p J_p)$ denote the categorical quotients of $\mathfrak{IG}_G$ by $U(J_p)$ and $J_p$ respectively. Note that the natural map
    \[
    \mathfrak{IG}_G(K^p U(J_p)) \to \mathfrak{IG}_G(K^p J_p)
    \]
    is a pro-\'{e}tale $M(J_p)$-torsor (see \cite[Proposition 3.4.9(5)]{HHTBoxerPilloni}). We let $\mathcal{IG}_G(K^p U(J_p))$ and $\mathcal{IG}_G(K^p J_p)$ denote the adic generic fibres of $\mathfrak{IG}_G(K^p U(J_p))$ and $\mathfrak{IG}_G(K^p J_p)$ respecively.
\end{definition}

Note that $\mathfrak{IG}_G(K^p U(J_{\opn{ord}}^+)) \to \mathfrak{IG}_G(K^pJ_{\opn{ord}}^+) = \mathfrak{Y}_G(K^pG(\mbb{Z}_p))^{\opn{ord}}$ is the usual Mantovan Igusa variety parameterising trivialisations of the connected and \'{e}tale parts of the $p$-divisible group $A[p^{\infty}]$ respecting the polarisation up to similitude. This moduli problem is equivalent to trivialising the connected part $A[p^{\infty}]^{\circ}$ and making a choice of similitude factor in $\mbb{Z}_p^{\times}$; since $A[p^{\infty}]^{\circ}$ extends to a $p$-divisible group over $\mathfrak{X}_G(K^pG(\mbb{Z}_p))^{\opn{ord}}$, we therefore see that $\mathfrak{IG}_G(K^p U(J_{\opn{ord}}^+))$ extends to a pro-\'{e}tale $M_G(\mbb{Z}_p)$-torsor $\mathfrak{IG}_G(K^pU(J_p))^{\opn{tor}} \to \mathfrak{X}_G(K^pG(\mbb{Z}_p))^{\opn{ord}}$. The Mantovan Igusa variety carries a natural lift of Frobenius $\varphi \colon \mathfrak{IG}_G(K^p U(J_{\opn{ord}}^+)) \to \mathfrak{IG}_G(K^p U(J_{\opn{ord}}^+))$ by quotienting by the canonical subgroup, and this extends to $\varphi \colon \mathfrak{IG}_G(K^p U(J_{\opn{ord}}^+))^{\opn{tor}} \to \mathfrak{IG}_G(K^p U(J_{\opn{ord}}^+))^{\opn{tor}}$.

We now define $\mathfrak{IG}_G^{\opn{tor}} \defeq \varprojlim_{\varphi} \mathfrak{IG}_G(K^pU(J_p))^{\opn{tor}}$ where the inverse limit has transition maps given by $\varphi$. This exists as a formal scheme because $\varphi$ is affine, and one has
\[
\mathfrak{IG}_G = \mathfrak{IG}_G^{\opn{tor}} \times_{\mathfrak{X}_G(K^pG(\mbb{Z}_p))^{\opn{ord}}} \mathfrak{Y}_G(K^pG(\mbb{Z}_p))^{\opn{ord}} .
\]
By \cite[Proposition 3.4.16]{HHTBoxerPilloni}, the action of $J_{\opn{ord}}^+$ extends to $\mathfrak{IG}_G^{\opn{tor}}$ and for any finite index fpqc subgroup scheme $J_p \subset J_{\opn{ord}}^+$, the categorical quotients $\mathfrak{IG}_G(K^pU(J_p))^{\opn{tor}} = \mathfrak{IG}_G^{\opn{tor}}/U(J_p)$ and $\mathfrak{IG}_G(K^pJ_p)^{\opn{tor}} = \mathfrak{IG}_G^{\opn{tor}}/J_p$ exist. Moreover, the map $\mathfrak{IG}_G(K^pU(J_p))^{\opn{tor}} \to \mathfrak{IG}_G(K^pJ_p)^{\opn{tor}}$ is a pro-\'{e}tale $M(J_p)$-torsor. As above, we let $\mathcal{IG}_G(K^pU(J_p))^{\opn{tor}}$ and $\mathcal{IG}_G(K^pJ_p)^{\opn{tor}}$ denote the adic generic fibres of $\mathfrak{IG}_G(K^pU(J_p))^{\opn{tor}}$ and $\mathfrak{IG}_G(K^pJ_p)^{\opn{tor}}$ respectively.

\subsubsection{} Let $\mbb{Q}_p^{\opn{cycl}}$ denote the $p$-adic completion of $\mbb{Q}_p(\mu_{p^{\infty}})$, and let $\mbb{Z}_p^{\opn{cycl}} \subset \mbb{Q}_p^{\opn{cycl}}$ denote its ring of integers. Let $J_{\opn{ord}, \mbb{Q}_p}^+ \defeq J_{\opn{ord}}^+ \times_{\opn{Spa}(\mbb{Z}_p, \mbb{Z}_p)} \opn{Spa}(\mbb{Q}_p, \mbb{Z}_p)$, and let $J_{\opn{ord}, \mbb{Q}_p^{\opn{cycl}}}^+ = J_{\opn{ord}, \mbb{Q}_p}^+ \times_{\opn{Spa}(\mbb{Q}_p, \mbb{Z}_p)} \opn{Spa}(\mbb{Q}_p^{\opn{cycl}}, \mbb{Z}_p^{\opn{cycl}})$. We use similar notation for the other groups considered in the previous subsection. Fix an identification $T_p \mu_{p^{\infty}} = \mbb{Z}_p(1) \cong \mbb{Z}_p$ over $\mbb{Q}_p^{\opn{cycl}}$. Depending on this identification, we have an isomorphism $J^+_{\opn{ord}, \mbb{Q}_p^{\opn{cycl}}} \cong P_G(\mbb{Z}_p)$.

Let $K_p \subset G(\mbb{Z}_p)$ be a compact open subgroup and let $J_p \defeq w_1 K_p w_1^{-1} \cap P_G(\mbb{Z}_p)$ (which is a compact open subgroup of $P_G(\mbb{Z}_p)$). By abuse of notation, we also let $J_p$ denote the schematic closure of $J_p$ inside $J^+_{\opn{ord}, \mbb{Z}_p^{\opn{cycl}}} = J_{\opn{ord}}^+ \times \opn{Spf}\mbb{Z}_p^{\opn{cycl}}$ (which is a finite-index fpqc subgroup scheme -- see \cite[Lemma 3.4.8]{HHTBoxerPilloni}). One has a natural open immersion 
\begin{equation} \label{GeneralOpenImmersionOfIG}
\mathcal{IG}_G(K^p J_p)_{\mbb{Q}_p^{\opn{cycl}}} \hookrightarrow \mathcal{Y}_G(K^pK_p)_{\mbb{Q}_p^{\opn{cycl}}}
\end{equation}
induced from right-translation by $w_1$, and this extends to an open immersion $\mathcal{IG}_G(K^p J_p)^{\opn{tor}}_{\mbb{Q}_p^{\opn{cycl}}} \hookrightarrow \mathcal{X}_G(K^pK_p)_{\mbb{Q}_p^{\opn{cycl}}}$. Here $\mathcal{X}_G(K^pK_p)$ denotes the adic space associated with $X_G(K^pK_p)_{\mbb{Q}_p}$, and $\mathcal{Y}_G(K^pK_p) \subset \mathcal{X}_G(K^pK_p)$ denotes the good reduction locus away from the boundary.

If $J_p$ splits as a semidirect product $J_p = U(J_p) \rtimes M(J_p)$, then $J_p$ is Galois stable and descends to a compact open subgroup of $J_{\opn{ord}}^+$ (see \cite[Remark 3.4.7]{HHTBoxerPilloni}). In this setting, the quotient $\mathcal{IG}_G(K^p J_p)$ is therefore defined over $\mbb{Q}_p$, however it may still be the case that (\ref{GeneralOpenImmersionOfIG}) is only defined over $\mbb{Q}_p(\mu_{p^m})$ for some integer $m \geq 0$. If we assume that $T(\mbb{Z}_p) \subset K_p$ though, then both (\ref{GeneralOpenImmersionOfIG}) and its extension to the (partial) toroidal compactifications descend to open immersions over $\mbb{Q}_p$. 

In practice, we will only consider $K_p$ such that $T(\mbb{Z}_p) \subset K_p$ and $J_p = w_1 K_p w_1^{-1} \cap P_G(\mbb{Z}_p)$ splits as a semidirect product. For example, for any $\beta \geq 1$ and $t \in T^{G, \pm}$, the compact open subgroup $K_p = t K^G_{\opn{Iw}}(p^{\beta}) t^{-1} \cap K^G_{\opn{Iw}}(p^{\beta})$ satisfies both of these properties. 

\begin{example}
    Let $\beta \geq 1$ be an integer. Then $\mathcal{IG}_G(K^G_{\beta})^{\opn{tor}}$ is identified with $\mathcal{X}_{G, w_1, \beta}$.
\end{example}

\subsubsection{}

We now suppose that $K_p \subset G(\mbb{Z}_p)$ is a compact open subgroup such that $T(\mbb{Z}_p) \subset K_p$ and $J_p \defeq w_1 K_p w_1^{-1} \cap P_G(\mbb{Z}_p)$ splits as a semidirect product (as in the previous section). Recall that $\mathcal{IG}_G(K^pU(J_p))^{\opn{tor}} \to \mathcal{IG}_G(K^p J_p)^{\opn{tor}}$ is a pro-\'{e}tale $M(J_p)$-torsor. Over $\mathfrak{IG}_G(K^pU(J_p))$ one has a universal trivialisation $\mu_{p^{\infty}}^{\oplus 2} \xrightarrow{\sim} A[p^{\infty}]^{\circ}$ of the connected part of the $p$-divisible group associated with the universal ordinary abelian surface. This induces a trivialisation of $\omega_A$ which extends to a trivialisation of the canonical extension of $\omega_A$ over $\mathcal{IG}_G(K^pU(J_p))^{\opn{tor}}$. Combining this with the choice of symplectic similitude, we therefore obtain a trivialisation of the graded pieces of the Hodge filtration on $\mathcal{H}$, where $\mathcal{H}$ denotes the pullback to $\mathcal{IG}_G(K^pU(J_p))^{\opn{tor}}$ of the canonical extension of $\mathcal{H}_1^{\opn{dR}}(A/\mathcal{Y}_G(K^pK_p))$. Since the unit root splitting extends to the canonical extension $\mathcal{H}$, we therefore obtain a trivialisation of $\mathcal{H}$ over $\mathcal{IG}_G(K^pU(J_p))^{\opn{tor}}$ respecting the Hodge filtration and polarisation up to similitude. In other words, we have a commutative diagram:
\[
\begin{tikzcd}
\mathcal{IG}_G(K^pU(J_p))^{\opn{tor}} \arrow[d] \arrow[r] & {P^{\opn{an}}_{G, \opn{dR}}} \arrow[d] \\
\mathcal{IG}_G(K^p J_p)^{\opn{tor}} \arrow[r]      & \mathcal{X}_G(K^pK_p)      
\end{tikzcd}
\]
where the bottom horizontal arrow denotes the extension of (\ref{GeneralOpenImmersionOfIG}) to partial toroidal compactifications. The top horizontal map is equivariant for the action of $M(J_p)$ (via the embedding $M(J_p) \hookrightarrow M_G^{\opn{an}} \hookrightarrow \overline{P}_G^{\opn{an}}$) and hence defines a reduction of structure over $\mathcal{IG}_G(K^p J_p)^{\opn{tor}}$.

\begin{proposition} \label{PGdrnReductionProposition}
    Let $n \geq 1$ be an integer. Then there exists a quasi-compact open strict neighbourhood $U$ of $\mathcal{IG}_G(K^p J_p)^{\opn{tor}}$ inside $\mathcal{X}_G(K^pK_p)$ and an \'{e}tale $\overline{\mathcal{P}}_{G, n} M(J_p)$-torsor $\mathcal{P}_{G, \opn{dR}, n} \to U$ such that:
    \begin{itemize}
        \item $\mathcal{P}_{G, \opn{dR}, n}$ is a reduction of structure of $P_{G, \opn{dR}}^{\opn{an}}$ over $U$;
        \item $\mathcal{IG}_G(K^pU(J_p))^{\opn{tor}}$ is a reduction of structure of $\mathcal{P}_{G, \opn{dR}, n}$ over $\mathcal{IG}_G(K^p J_p)^{\opn{tor}}$.
    \end{itemize}
    Moreover, if we view this torsor as an open subspace of $P_{G, \opn{dR}}^{\opn{an}}$, then this torsor (with these properties) is unique up to possibly shrinking $U$.
\end{proposition}
\begin{proof}
    The proof of this proposition is very similar to that of \cite[Proposition 6.2.1]{DiffOps}. We first claim that there exists a sheaf $\mathcal{L}$ of $\mathcal{O}^+_{\mathcal{X}_G(K^pK_p)}$-modules, which is locally free of rank one, and a section $h \in \opn{H}^0(\mathcal{X}_G(K^pK_p), \mathcal{L}/p\mathcal{L})$ such that $\mathcal{IG}_G(K^p J_p)^{\opn{tor}}$ is the locus of points $| \cdot |_x$ satisfying $|\tilde{h}|_x=1$ for any local lift $\tilde{h}$ of $h$. Indeed, $\mathcal{IG}_G(K^p J_p)^{\opn{tor}}$ is a connected component of $\mathcal{X}_G(K^pK_p)^{\opn{ord}} \defeq \mathcal{X}_G(K^pK_p) \times_{\mathcal{X}_G(K^pG(\mbb{Z}_p))} \mathcal{X}_G(K^pG(\mbb{Z}_p))^{\opn{ord}}$, hence there exists a section $h' \in \opn{H}^0( \mathcal{X}_G(K^pK_p)^{\opn{ord}}, \mathcal{O}^+_{\mathcal{X}_G(K^pK_p)^{\opn{ord}}})$ such that $\mathcal{IG}_G(K^p J_p)^{\opn{tor}}$ is the locus in $\mathcal{X}_G(K^pK_p)^{\opn{ord}}$ where $|h'|=1$. Let $\opn{Ha} \in \opn{H}^0(\mathcal{X}_G(K^pK_p), \opn{det}(\omega_A^+)^{\otimes p-1})$ be a lift of the Hasse invariant. There then exists a sufficiently large integer $N \geq 1$ and section $h'' \in \opn{H}^0(\mathcal{X}_G(K^pK_p), \opn{det}(\omega_A^+)^{\otimes N(p-1)}/p)$ such that $h''$ is congruent to $\opn{Ha}^N h'$ modulo $p$ over $\mathcal{X}_G(K^pK_p)^{\opn{ord}}$. We may then take $\mathcal{L} = \opn{det}(\omega_A^+)^{\otimes (N+1)(p-1)}$ and $h = \opn{Ha} h''$. The upshot of this is that we have a cofinal system $\{ \mathcal{X}_r \}_{r \geq 1}$ of strict quasi-compact open neighbourhoods of $\mathcal{IG}_G(K^p J_p)^{\opn{tor}}$ in $\mathcal{X}_G(K^pK_p)$ with $\mathcal{X}_r$ equal to the locus of points $| \cdot |_x$ satisfying $|\tilde{h}|_x^{p^{r+1}} \geq |p|$. 
    
    Let $\mathscr{G}_k = \overline{\mathcal{P}}_{G, k}M(J_p)/\overline{\mathcal{P}}_{G, k}$ which is also equal to the image of $M(J_p)$ in $M_G(\mbb{Z}/p^k \mbb{Z})$. Let $\mathcal{I}_{k, \infty}$ denote the quotient of $\mathcal{IG}_G(K^pU(J_p))^{\opn{tor}}$ by the kernel of $M(J_p) \to \mathscr{G}_k$. Then $\mathcal{I}_{k, \infty} \to \mathcal{X}_{\infty} \defeq \mathcal{IG}_G(K^p J_p)^{\opn{tor}}$ is a finite \'{e}tale $\mathscr{G}_k$-torsor. We also let $\mathfrak{I}_{k, \infty} \to \mathfrak{X}_{\infty} \defeq \mathfrak{IG}_G(K^p J_p)^{\opn{tor}}$ denote the finite \'{e}tale $\mathscr{G}_k$-torsor given by the quotient of $\mathfrak{IG}_G(K^pU(J_p))^{\opn{tor}}$ by the kernel of $M(J_p) \to \mathscr{G}_k$. Let $C_k \subset A[p^k]$ denote the canonical subgroup of order $p^{2k}$. This extends to the boundary and there exists an integer $r$ such that $C_k$ extends to $\mathcal{X}_r$. Let $\mathcal{F}_{k,r} \to \mathcal{X}_r$ denote the finite \'{e}tale $M_G(\mbb{Z}/p^k\mbb{Z})$-torsor parameterising trivialisations of $C_r$ and a choice of similitude. We have a natural $\mathscr{G}_k$-equivariant map $\mathcal{I}_{k, \infty} \to \mathcal{F}_{k, r}$ which provides a reduction of structure over $\mathcal{X}_{\infty}$, hence by the arguments in \cite[Proposition 2.2.1]{KisinLai}, we see that there exists a sufficiently large integer $r$ and a finite \'{e}tale $\mathscr{G}_k$-torsor $\mathcal{I}_{k, r} \to \mathcal{X}_r$ fitting into the following Cartesian diagram
    \[
\begin{tikzcd}
{\mathcal{I}_{k, \infty}} \arrow[d] \arrow[r] & {\mathcal{I}_{k, r}} \arrow[d] \\
\mathcal{X}_{\infty} \arrow[r]                & \mathcal{X}_r                 
\end{tikzcd}
    \]
    where the top map is $\mathscr{G}_k$-equivariant.

    The rest of the proof now follows \cite[Proposition 6.2.1]{DiffOps} \emph{mutatis mutandis} by first defining a $\overline{\mathcal{P}}_{G, n}$-torsor over $\mathcal{I}_{k, \infty}$ and showing that this descends to an \'{e}tale $\overline{\mathcal{P}}_{G, n}M(J_p)$-torsor over $\mathcal{X}_{\infty}$, then using the fact that $\mathcal{X}_{\infty}$ is the locus where $|\tilde{h}| = 1$ and the above Cartesian diagram to overconverge this torsor.
\end{proof}

\begin{remark}
    We note that there is some ambiguity when referring to ``the'' reduction of structure $\mathcal{P}_{G, \opn{dR}, n}$, since this torsor depends on $U$ (and may not be unique over $U$). It is however unique up to shrinking $U$, and since the majority of the objects considered later on are up to shrinking $U$, we can safely pretend that $\mathcal{P}_{G, \opn{dR}, n}$ is unique. 
\end{remark}

\subsubsection{} 

Exactly the same constructions in the previous sections carry over to the setting of Shimura varieties for $H$ or $\opn{GL}_2$ (replacing the element $w_1$ with the identity element). In particular, we will use similar notation for the corresponding objects for the groups $H$, $\opn{GL}_2$. We note that, for $\beta \geq 1$, it is not true that $T(\mbb{Z}_p) \subset K^H_{\diamondsuit}(p^{\beta})$ -- instead one only has $T_{\diamondsuit}(p^{\beta}) \subset K^H_{\diamondsuit}(p^{\beta})$, however this still implies that $\mathcal{IG}_H(K^pJ_p)^{\opn{tor}} \hookrightarrow \mathcal{X}_H(K^p K_p)$ is defined over $\mbb{Q}_p$ for any $K_p \subset H(\mbb{Z}_p)$ of the form $t K^H_{\diamondsuit}(p^{\beta}) t^{-1} \cap K^H_{\diamondsuit}(p^{\beta})$ with $t \in T^{H, \pm}$. In particular, we have an analogue of Proposition \ref{PGdrnReductionProposition} -- we denote the corresponding reductions of structure by $\mathcal{P}_{\mathscr{G}, \opn{dR}, n}$ for $\mathscr{G} \in \{H, \opn{GL}_2 \}$. Note that when $\mathscr{G} = \opn{GL}_2$ and $K_p = H(\mbb{Z}_p)$, this is nothing but \cite[Proposition 6.2.1]{DiffOps}.

\subsubsection{Maps between torsors} \label{MapsBetweenTorsorsSubSubSec}

Let $\beta \geq 1$ be an integer. For ease of notation set $\mathcal{IG}_{G, w_1, \beta}^{\opn{tor}} = \mathcal{IG}_G(K^p J_p)^{\opn{tor}}$ (resp. $\mathcal{IG}_{H, \opn{id}, \beta}^{\opn{tor}} = \mathcal{IG}_H((K^p \cap H(\mbb{A}_f^p)) J_p)^{\opn{tor}}$) for $J_p = w_1 K^G_{\opn{Iw}}(p^{\beta}) w_1^{-1} \cap P_G(\mbb{Z}_p)$ (resp. $J_p = K^H_{\diamondsuit}(p^{\beta}) \cap P_H(\mbb{Z}_p)$). Then we have a commutative cube
\[
\begin{tikzcd}
                                                                              &                                                      & {P_{H, \opn{dR}}^{\opn{an}}} \arrow[d, dotted] \arrow[r, "\gamma"] & {P_{G, \opn{dR}}^{\opn{an}}} \arrow[d] \\
{\mathcal{IG}_{H, \opn{id}, \beta}^{\opn{tor}}} \arrow[d] \arrow[r, "\gamma"] \arrow[rru] & {\mathcal{IG}_{G, w_1, \beta}^{\opn{tor}}} \arrow[d] \arrow[rru] & \mathcal{X}_{H}(K^H_{\beta}) \arrow[r, "\hat{\iota}"]              & \mathcal{X}_{G}(K^G_{\beta})           \\
{\mathcal{X}_{H, \opn{id}, \beta}} \arrow[rru, dotted] \arrow[r]              & {\mathcal{X}_{G, w_1, \beta}} \arrow[rru]            &                                                                    &                                       
\end{tikzcd}
\]
where the horizontal maps in the top square are induced from right-translation by $\gamma$ (recall that $\gamma \in M_G(\mbb{Z}_p)$). The bottom square is Cartesian, and since $\hat{\iota}$ is a finite morphism between qcqs adic spaces, one can show that $\{ \hat{\iota}^{-1}(U) \}$ is a cofinal collection of quasi-compact open subspaces containing the closure of $\mathcal{X}_{H, \opn{id}, \beta}$ as $U$ runs over quasi-compact open subspaces of $\mathcal{X}_G(K^G_{\beta})$ containing the closure of $\mathcal{X}_{G, w_1, \beta}$.\footnote{Indeed, it suffices to check $\bigcap_U \hat{\iota}^{-1}(U)$ and $\mathcal{X}_{H, \opn{id}, \beta}$ have the same rank one points, but this follows from the fact that $\left( \cap_U U \right)^{\opn{rk} 1} = \mathcal{X}_{G, w_1, \beta}^{\opn{rk} 1}$.}

Let $\mathcal{P}_{G, \opn{dR}, n} \to U$ be the reduction of structure of $P_{G, \opn{dR}}^{\opn{an}}$ as constructed in Proposition \ref{PGdrnReductionProposition}. Then we can consider the pullback $\mathcal{P}_{G, \opn{dR},n} \times_{P^{\opn{an}}_{G, \opn{dR}}, \gamma} P_{H, \opn{dR}}^{\opn{an}} \to \hat{\iota}^{-1}U$, which is an \'{e}tale torsor under the group $\gamma \overline{\mathcal{P}}_{G, n, \beta}^{\square} \gamma^{-1} \cap \overline{P}_H^{\opn{an}} = \overline{\mathcal{P}}_{H, n, \beta}^{\diamondsuit}$. Therefore, by the uniqueness part of (the analogue of) Proposition \ref{PGdrnReductionProposition} for $H$, we can find a sufficiently small $U$ and a commutative diagram:
\[
\begin{tikzcd}
{\mathcal{P}_{H, \opn{dR},n}} \arrow[d] \arrow[r, "\gamma"] & {\mathcal{P}_{G, \opn{dR},n}} \arrow[d] \\
\hat{\iota}^{-1}U \arrow[r, "\hat{\iota}"]                                  & U                                      
\end{tikzcd}
\]
which provides a reduction of structure $\mathcal{P}_{H, \opn{dR},n} \times^{[\overline{\mathcal{P}}_{H, n, \beta}^{\diamondsuit}, \gamma]} \overline{\mathcal{P}}_{G, n, \beta}^{\square} \cong \mathcal{P}_{G,\opn{dR}, n} \times_U \hat{\iota}^{-1}(U) $.

We also have morphisms between the reductions of structure of $P_{\bullet, \opn{dR}}^{\opn{an}}$ for $\bullet \in \{ H, \opn{GL}_2 \}$. More precisely, one can show that for any quasi-compact open subspace $U \subset \mathcal{X}_{\opn{GL}_2}(K^{H_i}_{\beta})$ containing the closure of $\mathcal{X}_{H_i, \opn{id}, \beta}$ such that one has a reduction of structure $\mathcal{P}_{\opn{GL}_2, \opn{dR}, n} \to U$, there exists a sufficiently small quasi-compact open subspace $V \subset \mathcal{X}_{H}(K^H_{\beta})$ containing the closure of $\mathcal{X}_{H, \opn{id}, \beta}$ and a reduction of structure $\mathcal{P}_{H, \opn{dR}, n} \to V$ fitting into a commutative diagram:
\[
\begin{tikzcd}
{\mathcal{P}_{H, \opn{dR},n}} \arrow[d] \arrow[r] & {\mathcal{P}_{\opn{GL}_2, \opn{dR},n}} \arrow[d] \\
V \arrow[r, "\opn{pr}_i"]                                   & U                                               
\end{tikzcd}
\]
Here the top arrow is compatible with the natural map $P_{H, \opn{dR}}^{\opn{an}} \to P_{\opn{GL}_2, \opn{dR}}^{\opn{an}}$ induced from the projection $\opn{pr}_i \colon H \twoheadrightarrow \opn{GL}_2$ and provides a reduction of structure.

\subsection{p-adic nearly sheaves}

We now consider the sheaves associated with the torsors in Proposition \ref{PGdrnReductionProposition}.

\begin{notation}
    Let $(A, A^+)$ be a complete Tate affinoid pair over $(\mbb{Q}_p, \mbb{Z}_p)$. For any projective Banach $\overline{\mathcal{P}}_{G, n, \beta}^{\square}$-representation $V$ over $(A, A^+)$, we let 
    \[
    [V] \defeq \left( \pi_*\mathcal{O}_{\mathcal{P}_{G, \opn{dR}, n}} \hatot V \right)^{\overline{\mathcal{P}}_{G, n, \beta}^{\square}}
    \]
    denote the associated sheaf on $U \times \opn{Spa}(A, A^+)$, where $\pi \colon \mathcal{P}_{G, \opn{dR}, n} \to U$ denotes the reduction of structure in Proposition \ref{PGdrnReductionProposition}. We use similar notation for $H$, $\opn{GL}_2$ and projective Banach representations of $\overline{\mathcal{P}}_{H, n, \beta}^{\diamondsuit}$, $\overline{\mathcal{P}}_{\opn{GL}_2, n, \beta}^{\square}$. Note that $[V]$ is a locally projective Banach sheaf by a similar argument as in \cite[Proposition 6.3.3]{BoxerPilloni}.
\end{notation}

\begin{remark} \label{OnlyUniqueAfterShrinkingRem}
    The sheaves $[V]$ are only independent of choices up to shrinking $U$ because the same is true for the torsors $\mathcal{P}_{G, \opn{dR}, n}$. Since we will only interested in the cohomology of these sheaves modulo shrinking $U$, this ambiguity will not affect the constructions in the following sections.
\end{remark}

If $V$ (resp. $W$) is a projective Banach $\overline{\mathcal{P}}_{G, n, \beta}^{\square}$-representation (resp. $\overline{\mathcal{P}}_{H, n, \beta}^{\diamondsuit}$-representation) and we have a continuous morphism $f \colon V \to W$ which is equivariant for the action of $\overline{\mathcal{P}}_{H, n, \beta}^{\diamondsuit}$ through the inclusion $\gamma^{-1}\overline{\mathcal{P}}_{H, n, \beta}^{\diamondsuit} \gamma \subset \overline{\mathcal{P}}_{G, n, \beta}^{\square}$, then the discussion in \S \ref{MapsBetweenTorsorsSubSubSec} implies that we have a natural morphism $[V] \to \hat{\iota}_*[W]$ for $U$ sufficiently small induced from $f$ (and $[W]$ is defined over $\hat{\iota}^{-1}U$). Similarly, if $W_i$ is a projective Banach $\overline{\mathcal{P}}_{\opn{GL}_2, n, \beta}^{\square}$-representation and $g \colon W_i \to W$ is equivariant through the map $\opn{pr}_i \colon \overline{\mathcal{P}}_{H, n, \beta}^{\diamondsuit} \to \overline{\mathcal{P}}_{\opn{GL}_2, n, \beta}^{\square}$, then we obtain a morphism $[W_i] \to (\opn{pr}_i)_*[W]$ induced from $g$, where $\opn{pr}_i \colon V \to U$ is as in \S \ref{MapsBetweenTorsorsSubSubSec}.

We will consider the above sheaves for a small modification of the specific representations appearing in \S \ref{padicRepTheorySubSec}. More precisely:

\begin{notation}
    Let $\mathbb{I}_{\overline{\mathcal{P}}_{G, n, \beta}^{\square, \circ}}(\kappa_G)$ be defined in exactly the same way as in Definition \ref{DefinitionOfIGpadicnearlyrep}, but replacing $\overline{\mathcal{P}}_{G, n, \beta}^{\square}$ with the subgroup $\overline{\mathcal{P}}_{G, n, \beta}^{\square, \circ} = \overline{\mathcal{P}}^{\circ}_{G, n}M_{G, \opn{Iw}}(p^{\beta}) \subset \overline{\mathcal{P}}_{G, n, \beta}^{\square}$ where $\overline{\mathcal{P}}^{\circ}_{G, n} \subset \overline{\mathcal{P}}_{G, n}$ denotes the subgroup of elements which are congruent to the identity modulo $p^{n +\varepsilon}$ for some $\varepsilon > 0$. We let $\mathbb{D}_{\overline{\mathcal{P}}_{G, n+1, \beta}^{\square}}(\kappa^*_G)$ denote the continuous dual of $\mathbb{I}_{\overline{\mathcal{P}}_{G, n, \beta}^{\square, \circ}}(\kappa_G)$, which we view as a projective Banach $\overline{\mathcal{P}}_{G, n+1, \beta}^{\square}$-representation.

    Similarly, we let $\mathbb{D}_{\overline{\mathcal{P}}_{H, n+1, \beta}^{\diamondsuit}}(\kappa_H^*)$ denote the continuous dual of $\mathbb{I}_{\overline{\mathcal{P}}_{H, n, \beta}^{\diamondsuit, \circ}}(\kappa_H)$, which is a projective Banach $\overline{\mathcal{P}}_{H, n+1, \beta}^{\diamondsuit}$-module, where $\overline{\mathcal{P}}_{H, n, \beta}^{\diamondsuit, \circ} = \overline{\mathcal{P}}_{H, n}^{\circ} T_{\diamondsuit}(p^{\beta})$. We will also use similar notation for the versions $\mathbb{D}_{\mathcal{M}_{G, n+1, \beta}^{\square}}(\kappa^*_G)$ and $\mathbb{D}_{\mathcal{M}_{H, n+1, \beta}^{\diamondsuit}}(\kappa_H^*)$ defined using the Levi subgroups.
\end{notation}

\begin{remark}
    If the action of $\overline{\mathcal{P}}_{G, n, \beta}^{\square}$ factors through $\overline{\mathcal{P}}_{G, n, \beta}^{\square} \to \mathcal{M}_{G, n, \beta}^{\square}$, then $[V]$ agrees with the sheaf constructed in \cite[\S 6.3]{BoxerPilloni} after possibly shrinking $U$. In the setting of \emph{loc.cit.}, one can be more precise about the locus over which $[V]$ is defined (however we will not need this).
\end{remark}

\subsubsection{Cohomological correspondences}

Fix an integer $\beta \geq 1$ and let $t \in T^{G, -}$. Set $K_p = K^G_{\opn{Iw}}(p^{\beta})$ and $J_p = w_1 K_p w_1^{-1} \cap P_G(\mbb{Z}_p)$, and set $K_p' = tK_pt^{-1} \cap K_p$, $K_p'' = t^{-1}K_p't$, $J_p' = (w_1 t w_1^{-1}) J_p (w_1 t w_1^{-1})^{-1} \cap J_p = w_1 K_p' w_1^{-1} \cap P_G(\mbb{Z}_p)$, and $J_p'' = (w_1 t w_1^{-1})^{-1} J_p' (w_1 t w_1^{-1}) = w_1 K_p'' w_1^{-1} \cap P_G(\mbb{Z}_p)$. Consider the following commutative diagram:
\[
\begin{tikzcd}
\mathcal{IG}_G(K^pJ_p)^{\opn{tor}} \arrow[d, hook] & \mathcal{IG}_G(K^pJ_p')^{\opn{tor}} \arrow[d, hook] \arrow[l, "q_1"'] \arrow[r, "q_2"] & \mathcal{IG}_G(K^pJ_p)^{\opn{tor}} \arrow[d, hook] \\
\mathcal{X}_G(K^p K_p)                             & \mathcal{X}_G(K^p K_p') \arrow[l, "p_1"'] \arrow[r, "p_2"]                            & \mathcal{X}_G(K^p K_p)                            
\end{tikzcd}
\]
where $p_1$, $q_1$ are the natural forgetful maps, $p_2$ is induced from right-translation by $t$, $q_2$ is induced from right-translation by $w_1 t w_1^{-1}$, and the vertical arrows are induced from right-translation by $w_1$. Note that $p_1$, $p_2$, $q_1$, $q_2$ are all finite flat. The squares in this diagram are often not Cartesian, however we have the following property:

\begin{lemma} \label{p1p2inverseimageLemma}
    With notation as above, one has 
    \[
    p_1^{-1}\left( \mathcal{IG}_G(K^pJ_p)^{\opn{tor}} \right) \cap p_2^{-1}\left( \mathcal{IG}_G(K^pJ_p)^{\opn{tor}} \right) =  \mathcal{IG}_G(K^pJ_p')^{\opn{tor}} ,
    \] 
    where the intersection takes place in $\mathcal{X}_G(K^pK_p')$.
\end{lemma}
\begin{proof}
    Clearly $\mathcal{IG}_G(K^pJ_p')^{\opn{tor}}$ is contained in this intersection, so it suffices to prove the reverse inclusion. Furthermore, it suffices to check this on rank one points. Consider the truncated Hodge--Tate period map 
    \[
    \pi_{\opn{HT}, K_p'} \colon \mathcal{X}_{G}(K^pK'_p) \to \mathtt{FL}^G/K_p' .
    \]
    Then (on rank one points):
    \begin{itemize}
        \item $p_1^{-1}\left( \mathcal{IG}_G(K^pJ_p)^{\opn{tor}} \right)$ is equal to the intersection over all $k \geq 1$ of $\pi_{\opn{HT}, K_p'}^{-1}\left( {]C^G_{w_1}[_{k, k}} K_p \right)$;
        \item $p_2^{-1}\left( \mathcal{IG}_G(K^pJ_p)^{\opn{tor}} \right)$ is equal to the intersection over all $k \geq 1$ of $\pi_{\opn{HT}, K_p'}^{-1}\left( {]C^G_{w_1}[_{k, k}} K_p t^{-1} \right)$ .
    \end{itemize}
    It therefore suffices to show that, for any $k \geq 1$, there exists $l \gg k$ such that
    \[
    {]C^G_{w_1}[_{l, l}} K_p \cap {]C^G_{w_1}[_{l, l}} K_p t^{-1} \subset {]C^G_{w_1}[_{k, k}} K_p' .
    \]
    Let $t = \opn{diag}(t_1, t_2, \nu t_2^{-1}, \nu t_1^{-1})$ and let $l \gg \beta$ and $l \gg \opn{max}(t) \defeq -\opn{min}\{ v_p(\alpha(t)) : \alpha \in \Phi_G^+ \}$. By the Iwahori factorisation for $K_p$, we see that an element $X \in {]C^G_{w_1}[_{l, l}} K_p$ is uniquely represented by 
    \[
    w_1 \cdot \left( \begin{smallmatrix} 1 & & & \\ x_1 & 1 & y & \\  & & 1 & \\ x_2 &  & -x_1 & 1 \end{smallmatrix} \right)
    \]
    with $x_1, x_2 \in p^{\beta}\mbb{Z}_p + \mathcal{B}_l^{\circ}$, $y \in \mathcal{B}_l$. Similarly, ${]C^G_{w_1}[_{l, l}} K_p t^{-1}$ is uniquely represented by 
    \[
    w_1 \cdot \left( \begin{smallmatrix} 1 & & & \\ t_2 t_1^{-1} x'_1 & 1 & \nu^{-1} t_2^2 y' & \\  & & 1 & \\ \nu t_1^{-2} x'_2 &  & -t_2t_1^{-1}x'_1 & 1 \end{smallmatrix} \right) .
    \]
    with $x_1', x_2' \in p^{\beta}\mbb{Z}_p + \mathcal{B}_l^{\circ}$, $y' \in \mathcal{B}_l$. Both of these are unique representations of the elements in the translate of the big Bruhat cell by $w_1$, hence if $X \in {]C^G_{w_1}[_{l, l}} K_p \cap {]C^G_{w_1}[_{l, l}} K_p t^{-1}$, we must have $x_1 = t_2t_1^{-1}x_1'$, $x_2 = \nu t_1^{-2} x_2'$ and $y=\nu^{-1}t_2^2 y'$, hence $X \in {]C^G_{w_1}[_{k, k}} K_p'$ as required. 
\end{proof}

\begin{remark}
    At no point in the proof of Lemma \ref{p1p2inverseimageLemma} did we use the fact that $t \in T^{G, -}$, and the proof also applies when $t \in T^{G, +}$.
\end{remark}

As a consequence of Lemma \ref{p1p2inverseimageLemma}, we see that $\{ p_1^{-1}(U) \cap p_2^{-1}(U') \}$ forms a cofinal system of quasi-compact open neighbourhoods of the closure of $\mathcal{IG}_G(K^pJ_p')^{\opn{tor}}$ as $U$, $U'$ run through quasi-compact open neighbourhoods of the closure of $\mathcal{IG}_G(K^pJ_p)^{\opn{tor}}$. Let $d \geq \opn{max}(t)$. We can therefore find sufficiently small $U$, $U'$ such that we have a diagram
\[
\begin{tikzcd}
{\mathcal{P}_{G, \opn{dR}, n+d, K_p'}} \arrow[d] \arrow[rr, "\phi_t"] &                               & {t^{-1}\mathcal{P}_{G, \opn{dR}, n, K_p''}} \arrow[d]  \\
{p_1^{-1}\mathcal{P}_{G, \opn{dR}, n+d, K_p}} \arrow[rd]              &                               & {p_2^{-1}\mathcal{P}_{G, \opn{dR}, n, K_p}} \arrow[ld] \\
                                                                      & p_1^{-1}(U) \cap p_2^{-1}(U') &                                                       
\end{tikzcd}
\]
where the vertical arrows are reductions of structure (here we have temporarily included the level subgroup in the notation and $t \colon \mathcal{X}_G(K^pK_p') \xrightarrow{\sim} \mathcal{X}_G(K^pK_p'')$ denotes the map induced from right-translation by $t$). The top horizontal arrow is induced from the morphism $\phi_t \colon p_1^{-1}P_{G, \opn{dR}}^{\opn{an}} \xrightarrow{\sim} p_2^{-1}P_{G, \opn{dR}}^{\opn{an}}$ given by the composition of the equivariant structure and right-translation by $w_1 t w_1^{-1}$ (through the torsor structure). This morphism fits into the commutative diagram
\[
\begin{tikzcd}
\mathcal{IG}_G(K^pU(J_p'))^{\opn{tor}} \arrow[d, "w_1tw_1^{-1}"'] \arrow[r] & {\mathcal{P}_{G, \opn{dR}, n+d, K_p'}} \arrow[r] \arrow[d, "\phi_t"] & {p_1^{-1}P_{G, \opn{dR}}^{\opn{an}}} \arrow[d, "\phi_t"] \\
t^{-1}\mathcal{IG}_G(K^pU(J_p''))^{\opn{tor}} \arrow[r]                     & {t^{-1}\mathcal{P}_{G, \opn{dR}, n, K_p''}} \arrow[r]                & {p_2^{-1}P_{G, \opn{dR}}^{\opn{an}}}                    
\end{tikzcd}
\]

\begin{definition} \label{PsitCohCorrespDef}
    Let $V$ be projective Banach $\overline{\mathcal{P}}_{G, n, \beta}^{\square}$-representation over $(A, A^+)$ which is additionally equipped with an action of $w_1 T^{G, -} w_1^{-1}$. Then we let $(w_1^{-1} \star \mbf{1})(t)\psi_t \colon p_2^*[V] \to p_1^*[V]$ denote the morphism (defined over $p_1^{-1}(U) \cap p_2^{-1}(U')$ for sufficiently small $U$, $U'$ as above) given by 
    \[
    p_2^*[V] = \left( (\pi'')_* \mathcal{O}_{t^{-1}\mathcal{P}_{G, \opn{dR}, n, K_p''}} \hatot V \right)^{\overline{\mathcal{P}}_{G, n}M(J_p'')} \xrightarrow{\phi_t^* \otimes (w_1 t w_1^{-1}) \cdot -} \left( (\pi')_* \mathcal{O}_{\mathcal{P}_{G, \opn{dR}, n+d, K_p'}} \hatot V \right)^{\overline{\mathcal{P}}_{G, n+d}M(J_p')} = p_1^*[V]
    \]
    where $\pi' \colon \mathcal{P}_{G, \opn{dR}, n+d, K_p'} \to p_1^{-1}(U) \cap p_2^{-1}(U')$ and $\pi'' \colon t^{-1}\mathcal{P}_{G, \opn{dR}, n, K_p''} \to p_1^{-1}(U) \cap p_2^{-1}(U')$ denote the structural maps.
\end{definition}

\begin{remark}
    The optimal normalisation of the trace map associated with $p_1$ in neighbourhoods of $\mathcal{IG}_G(K^pJ_p)^{\opn{tor}}$ is $(w_1^{-1}\star \mbf{1})(t)^{-1}\opn{Tr}_{p_1}$, which explains the appearance of the factor in the above definition.
\end{remark}

\subsection{Cohomology} \label{CohomologySubSec}

We now discuss various overconvergent cohomology groups. 

\subsubsection{Classical weight overconvergent cohomology}

We first introduce overconvergent cohomology complexes associated with the algebraic ``nearly'' representations appearing in Definition \ref{DefOfNearlyAlgCohComplexes}. So the notation is not overloaded, we will henceforth use the notation ``$\ddagger$'' to mean ``nearly overconvergent''.

\begin{definition} \label{DDaggerComplexClassicalDef}
    Suppose that we are in Situation \ref{situation:Weights}. We set:
    \begin{align*}
        R\Gamma( K^G_{\beta}, \kappa_G^*; \opn{cusp})^{(\ddagger, -)} &\defeq \varinjlim_U R\Gamma_{\mathcal{Z}_{G, <1, \beta} \cap U}(U, [\mbb{D}_{\overline{P}_G}(\kappa_G^*)^{\opn{nearly}}](-D_G) ) \\
        R\Gamma( K^H_{\beta}, \kappa_H^*; \opn{cusp})^{(\ddagger, -)} &\defeq \varinjlim_U R\Gamma_{\mathcal{Z}_{H, \opn{id}, \beta}}(U, [\mbb{D}_{\overline{P}_H}(\kappa_H^*)^{\opn{nearly}}](-D_H) ) \\
        R\Gamma( K^H_{\beta}, \kappa_H - 2\rho_H)^{(\ddagger, +)} &\defeq \varinjlim_U R\Gamma(U, [\mbb{I}_{\overline{P}_H}(\kappa_H)^{\opn{nearly}} \otimes V_{M_H}(-2\rho_H)] ) \\
        R\Gamma( K^{H_i}_{\beta}, \zeta_{H_i})^{(\ddagger, +)} &\defeq \varinjlim_U R\Gamma(U, [\mbb{I}_{\overline{P}_{\opn{GL}_2}}(\zeta_{H_i})^{\opn{nearly}}] ) .
    \end{align*}
    where the colimit is over all strict neighbourhoods $U$ of $\mathcal{X}_{G, w_1, \beta}$ (in the first case), $\mathcal{X}_{H, \opn{id}, \beta}$ (in the second and third cases), $\mathcal{X}_{H_i, \opn{id},\beta}$ (in the fourth case). We define the complexes $R\Gamma( K^G_{\beta}, \kappa_G^*; \opn{cusp})^{(\dagger, -)}$, $R\Gamma( K^H_{\beta}, \kappa_H^*; \opn{cusp})^{(\dagger, -)}$, $R\Gamma( K^H_{\beta}, \kappa_H - 2\rho_H)^{(\dagger, +)}$, $R\Gamma( K^{H_i}_{\beta}, \zeta_{H_i})^{(\dagger, +)}$ in exactly the same way replacing the nearly representations with the corresponding algebraic representations of the Levi subgroup.
\end{definition}

The complexes in Definition \ref{DDaggerComplexClassicalDef} carry an action of Hecke operators away from $p$ as well as actions of $T^{G, -}$, $T^{H, -}$, $T^{H, +}$, $T^{H_i, +}$ respectively through Hecke operators at $p$ (the $\pm$-superscript denotes the relevant submonoid). For example, for the first complex the Hecke action at $p$ is constructed using the (unnormalised) cohomological correspondence\footnote{Or more simply the algebraic version over $X_G(K^pK_p')$.} $(w_1^{-1} \star \mbf{1})(t)\psi_t$ in Definition \ref{PsitCohCorrespDef} in combination with the following lemma (using the general process described in \cite[\S 4]{UFJII} for example).

\begin{lemma} \label{OpenClosedSupportLemma}
    Set $K_p = K^G_{\opn{Iw}}(p^{\beta})$. Let $\mathcal{Z}_3^{-} \defeq \pi_{\opn{HT}, K_p}^{-1}\left( \overline{]X^G_{\opn{id}}[} \right)$ and $\mathcal{Z}_2^{-} \defeq \pi_{\opn{HT}, K_p}^{-1}\left( \overline{]X^G_{w_1}[} \right)$. Consider the correspondence associated with $t \in T^{G, -}$
    \begin{equation} \label{EqnCorrespFort}
\begin{tikzcd}
                       & \mathcal{X}_G(K^p K_p') \arrow[ld, "p_1"'] \arrow[rd, "p_2"] &                        \\
\mathcal{X}_G(K^p K_p) &                                                              & \mathcal{X}_G(K^p K_p)
\end{tikzcd}
    \end{equation}
    where $K_p' = tK_pt^{-1}\cap K_p$. Set $T(-) = p_2p_1^{-1}(-)$ and $T^t(-) = p_1 p_2^{-1}(-)$. Then one has:
    \begin{enumerate}
        \item $\bigcap_{m \geq 1} (T^t)^m(\mathcal{Z}^{-}_2) = \mathcal{Z}_{G, <1, \beta}$;
        \item $\bigcap_{n, m \geq 1} T^n(\mathcal{X}_{G}(K^pK_p) - \mathcal{Z}_3^{-}) \cap (T^t)^m(\mathcal{Z}^{-}_2) = \overline{\mathcal{X}}_{G, w_1, \beta}$;
        \item both $(\mathcal{X}_{G}(K^pK_p), \mathcal{Z}_2^{-})$ and $(\mathcal{X}_{G}(K^pK_p) - \mathcal{Z}_3^{-}, \mathcal{Z}_2^{-})$ are open/closed support conditions for the correspondence (\ref{EqnCorrespFort}), in the sense of \cite[\S 6.1.1]{HHTBoxerPilloni}.
    \end{enumerate}
\end{lemma}
\begin{proof}
    This follows from the same strategy as in \cite[Proposition 6.1.11]{HHTBoxerPilloni} (although note that \emph{loc.cit.} is for the ``$+$''-version of higher Coleman theory). More precisely, parts (2) and (3) follow immediately from the analogous ``$-$''-version of the proof in \emph{loc.cit.}. For part (1), the proof in \emph{loc.cit.} shows that $\bigcap_{m \geq 1} \overline{]X^G_{w_1}[} K_p t^{-m} K_p$ is equal to the closure of $\mathcal{X}_{w_1}^G K_p$, where $\mathcal{X}_{w_1}^G \subset \mathtt{FL}^G$ denotes the analytic closed subset (the analytic version of $X_{w_1}^G \subset \mathrm{FL}_{G, \mbb{F}_p}$). But one can easily observe that $\mathcal{X}^G_{w_1} = P_G^{\opn{an}} \backslash P_G^{\opn{an}} \mathcal{K}$, where $\mathcal{K}$ denotes the analytic Klingen parabolic as in Definition \ref{DefOfOrdinaryStrata}.
\end{proof}

\subsubsection{\texorpdfstring{$p$}{p}-adic weight overconvergent cohomology}

We now define the analogous cohomology complexes as in Definition \ref{DDaggerComplexClassicalDef} for the $p$-adic representations appearing in \S \ref{padicRepTheorySubSec}. 

\begin{definition} \label{DDaggerComplexPADICDef}
    Let $(A, A^+)$ be a complete Tate affinoid pair over $(\mbb{Q}_p, \mbb{Z}_p)$. Let $n \geq 1$ be an integer and let $\kappa_G, \kappa_H \colon T(\mbb{Z}_p) \to A^{\times}$ and $\zeta_{H_i} \colon T_{\opn{GL}_2}(\mbb{Z}_p) \to A^{\times}$ be $n$-analytic characters. We set:
    \begin{align*}
        R\Gamma( K^G_{\beta}, \kappa_G^*; \opn{cusp})^{(\ddagger, n\opn{-an}, -)} &\defeq \varinjlim_U R\Gamma_{\mathcal{Z}_{G, <1, \beta} \cap U}(U, [\mathbb{D}_{\overline{\mathcal{P}}_{G, n+1, \beta}^{\square}}(\kappa^*_G)](-D_G) ) \\
        R\Gamma( K^H_{\beta}, \kappa_H^*; \opn{cusp})^{(\ddagger, n\opn{-an}, -)} &\defeq \varinjlim_U R\Gamma_{\mathcal{Z}_{H, \opn{id}, \beta}}(U, [\mbb{D}_{\overline{\mathcal{P}}_{H, n+1, \beta}^{\diamondsuit}}(\kappa_H^*)](-D_H) ) \\
        R\Gamma( K^H_{\beta}, \kappa_H - 2\rho_H)^{(\ddagger, n\opn{-an}, +)} &\defeq \varinjlim_U R\Gamma(U, [\mbb{I}_{\overline{\mathcal{P}}_{H, n, \beta}^{\diamondsuit}}(\kappa_H) \otimes V_{M_H}(-2\rho_H)] ) \\
        R\Gamma( K^{H_i}_{\beta}, \zeta_{H_i})^{(\ddagger, n\opn{-an}, +)} &\defeq \varinjlim_U R\Gamma(U, [\mbb{I}_{\overline{\mathcal{P}}^{\square}_{\opn{GL}_2, n}}(\zeta_{H_i})] ) .
    \end{align*}
    where the colimit is over all strict neighbourhoods $U$ of $\mathcal{X}_{G, w_1, \beta}$ (in the first case), $\mathcal{X}_{H, \opn{id}, \beta}$ (in the second and third cases), $\mathcal{X}_{H_i, \opn{id},\beta}$ (in the fourth case) over which the sheaves are defined. We define the complexes $R\Gamma( K^G_{\beta}, \kappa_G^*; \opn{cusp})^{(\dagger, n\opn{-an}, -)}$, $R\Gamma( K^H_{\beta}, \kappa_H^*; \opn{cusp})^{(\dagger, n\opn{-an}, -)}$, $R\Gamma( K^H_{\beta}, \kappa_H - 2\rho_H)^{(\dagger, n\opn{-an}, +)}$, $R\Gamma( K^{H_i}_{\beta}, \zeta_{H_i})^{(\dagger, n\opn{-an}, +)}$ in exactly the same way replacing the nearly $p$-adic representations with the versions for the Levi subgroups (i.e., $\mathbb{D}_{\mathcal{M}_{G, n+1, \beta}^{\square}}(\kappa^*_G)$, $\mbb{D}_{\mathcal{M}_{H, n+1, \beta}^{\diamondsuit}}(\kappa_H^*)$, etc.).
\end{definition}

The cohomology complexes in Definition \ref{DDaggerComplexPADICDef} carry an action of the Hecke algebra away from $p$, and also actions of $T^{G, -}$, $T^{H, -}$, $T^{H, +}$, $T^{H_i, +}$ respectively via the (analogous versions of the) normalised cohomological correspondences $\psi_t$ in Definition \ref{PsitCohCorrespDef}.

\begin{convention} \label{ConventionOnUnnormalisedAction}
    Suppose that we are in Situation \ref{situation:Weights}. From now on, we will often equip the complex $R\Gamma( K^G_{\beta}, \kappa_G^*; \opn{cusp})^{(\ddagger, n\opn{-an}, -)}$ with the ``unnormalised'' action of $T^{G, -}$ given by $(w_1^{-1} \star \mbf{1})(t)\kappa_G^*(w_1tw_1^{-1}) t \cdot -$, where $\cdot$ denotes the action above. With this action, the natural map
    \[
    R\Gamma( K^G_{\beta}, \kappa_G^*; \opn{cusp})^{(\ddagger, n\opn{-an}, -)} \to R\Gamma( K^G_{\beta}, \kappa_G^*; \opn{cusp})^{(\ddagger, -)}
    \]
    is $T^{G, -}$-equivariant. We also use similar conventions for the latter three complexes in Definition \ref{DDaggerComplexPADICDef} (as well as the overconvergent versions) so that the natural maps between these complexes and the versions in Definition \ref{DDaggerComplexClassicalDef} are equivariant for the Hecke operators at $p$.
\end{convention}

\subsubsection{Results from higher Coleman theory} \label{ResultsFromHCTSSsec}

In this section, we describe certain classicality results from higher Coleman theory which are relevant for the construction of the $p$-adic $L$-functions. We begin with the following lemma.

\begin{lemma} \label{BigSetOfWeightsLemma}
    Let $\kappa_G$ be as in Situation \ref{situation:Weights}. Then the weights of $\mbb{D}_{\overline{P}_G}(\kappa_G^*)^{\opn{nearly}}$ are
    \begin{equation} \label{SetOfWeightNEarly}
    \bigcup_{(t_1, t_2) \in \Sigma} \bigcup_{\substack{\delta_1, \delta_2 \in \{0,1, 2\}\\ \delta_1+\delta_2 =2}} \left\{ (2i-t_1-\delta_1, 2j-t_2-\delta_2; \xi) : \begin{array}{c} i=0, \dots, t_1 \\ j=0, \dots, t_2 \\ i+j \leq \frac{t_1+t_2}{2} + \frac{r_1-r_2}{2} \\ 2\xi = r_1+r_2 + t_1 + t_2 -2(i+j-1) \end{array} \right\} 
    \end{equation}
    where $\Sigma$ is the following set:
    \[
    \Sigma = \left\{ (t_1, t_2) \in \mbb{Z}_{\geq 0} \times \mbb{Z}_{\geq 0} : \begin{array}{c} t_1 + t_2 \equiv r_1 + r_2 \text{ modulo } 2 \\ r_1 - r_2 \leq t_1+t_2 \leq r_1+r_2 \\ |t_1 - t_2| \leq r_1 - r_2 \end{array} \right\} .
    \]
    Moreover, for any $\kappa$ in the set of weights in (\ref{SetOfWeightNEarly}), one has
    \[
    v_p(\kappa_G^*(t)) \leq v_p(\kappa(t))
    \]
    for any $t \in T^{G, -}$.
\end{lemma}
\begin{proof}
    The first part follows from \cite[Proposition 4.3.1]{LSZ17}. For the last claim, recall that $\kappa_G^* = (r_1, -(r_2+2); r_2+1)$. It suffices to check the inequality for the following matrices
    \[
    t_{\opn{S}} = \opn{diag}(1, 1, p, p), \quad t_{\opn{Kl}} = \opn{diag}(1, p, p, p^2)
    \]
    since these matrices and the centre generate $T^{G, -}$. One can easily compute:
    \begin{align*}
        v_p(\kappa_G^*(t_{\opn{S}})) = r_2 + 1, \quad & \quad v_p(\kappa(t_{\opn{S}})) = \frac{r_1+r_2}{2} + \frac{t_1+t_2}{2} - (i+j) + 1, \\ 
        v_p(\kappa_G^*(t_{\opn{Kl}})) = r_2, \quad & \quad v_p(\kappa(t_{\opn{Kl}})) = r_1+r_2 + t_1 + \delta_1 - 2i . 
    \end{align*}
    The inequality $v_p(\kappa_G^*(t_{\opn{S}})) \leq v_p(\kappa(t_{\opn{S}}))$ follows from the inequality involving $i+j$ in the definition (\ref{SetOfWeightNEarly}). For the inequality $v_p(\kappa_G^*(t_{\opn{Kl}})) \leq v_p(\kappa(t_{\opn{Kl}}))$, we note that the conditions in $\Sigma$ imply that $t_1 \leq r_1$, hence
    \[
    r_1 + r_2 + t_1 +\delta_1 - 2i \geq r_1 + r_2 - t_1 \geq r_2
    \]
    as required.
\end{proof}

Let $\kappa_G$ be as in Situation \ref{situation:Weights} and recall the definition of ${^MW_G}$ from \S \ref{FlagVarietyStrataSSec}. Let $\lambda \colon T^{G, -} \to \Qpb^{\times}$ be a monoid homomorphism. We say that $\lambda$ satisfies $\mathrm{ss}_{w_1}^M(\kappa_G^*)$ if: for any $w \in {^MW_G}$ with $w \neq w_1$, there exists an element $y \in T^{G, -}$ such that $v_p(\lambda(y)) < v_p((w^{-1} \star \kappa_G^*)(y))$. We now state the first classicality result that we need. 

\begin{proposition} \label{NearlyClassicalityProp}
    Let $\beta \geq 1$ be an integer and $\kappa_G$ as in Situation \ref{situation:Weights}. Consider the following morphisms
    \[
    \opn{H}^2\left( K^G_{\beta}, \kappa_G^*; \opn{cusp} \right)^{(\ddagger, -)} \xleftarrow{\opn{res}} \opn{H}^2_{\mathcal{Z}_{G, <1, \beta}}(\mathcal{X}_{G}(K^G_{\beta}), [\mbb{D}_{\overline{P}_G}(\kappa_G^*)^{\opn{nearly}}](-D_G)) \xrightarrow{\opn{cores}} \opn{H}^2\left( K^G_{\beta}, \kappa_G^*; \opn{cusp} \right)^{\opn{nearly}} .
    \]
    where the first map is induced from restriction and the second map is induced from corestriction (and rigid GAGA). Then:
    \begin{enumerate}
        \item The above morphisms are $T^{G, -}$-equivariant, and for any $t \in T^{G, --}$ and $h \in \mbb{Q}$, all three cohomology groups admit slope $\leq h$ decompositions with respect to the action of $t$.
        \item The right-hand map is surjective on finite-slope parts.
        \item The left-hand map is an isomorphism on small slope $\mathrm{ss}_{w_1}^M(\kappa_G^*)$ parts.
    \end{enumerate}
\end{proposition}
\begin{proof}
    The fact that the morphisms are $T^{G, -}$-equivariant follows from the fact that they are built are from the same cohomological correspondence. Moreover, Lemma \ref{OpenClosedSupportLemma} and \cite[\S 6.1.1]{HHTBoxerPilloni} imply that for $t \in T^{G, --}$:
    \begin{itemize}
        \item The square of the operator corresponding to $t$ acts compactly on $R\Gamma\left( K^G_{\beta}, \kappa_G^*; \opn{cusp}\right)^{(\ddagger, -)}$ and hence this complex admits slope $\leq h$ decompositions with respect to the action of $t$ for any $h \in \mbb{Q}$. In particular, the finite slope part $R\Gamma\left( K^G_{\beta}, \kappa_G^*; \opn{cusp}\right)^{(\ddagger, -, \opn{fs})}$ is quasi-isomorphic to the finite slope part of 
        \[
        R\Gamma_{\mathcal{Z}_2^{-}-\mathcal{Z}_3^{-}}\left( \mathcal{X}_G(K^G_{\beta}) - \mathcal{Z}_3^{-}, [\mbb{D}_{\overline{P}_G}(\kappa_G^*)^{\opn{nearly}}](-D_G) \right) .
        \]
        \item The square of the operator corresponding to $t$ acts compactly on 
        \begin{equation} \label{IntermediateCohComplexEqn}
        R\Gamma_{\mathcal{Z}_{G, <1, \beta}}(\mathcal{X}_{G}(K^G_{\beta}), [\mbb{D}_{\overline{P}_G}(\kappa_G^*)^{\opn{nearly}}](-D_G))
        \end{equation}
        and hence this complex admits slope $\leq h$ decompositions with respect to the action of $t$ for any $h \in \mbb{Q}$. In particular, the finite slope part of (\ref{IntermediateCohComplexEqn}) is quasi-isomorphic to the finite slope part of 
        \[
        R\Gamma_{\mathcal{Z}_2^{-}}\left( \mathcal{X}_G(K^G_{\beta}), [\mbb{D}_{\overline{P}_G}(\kappa_G^*)^{\opn{nearly}}](-D_G) \right) .
        \]
    \end{itemize}
    Recall that ${^MW_G} = \{ \opn{id}, w_1, w_2, w_3 \}$ comprises of four elements, where the length of $w_i$ is $i$. For brevity, set $\mathscr{F} \defeq [\mbb{D}_{\overline{P}_G}(\kappa_G^*)^{\opn{nearly}}](-D_G)$. Consider the following filtration
    \[
    \varnothing = \mathcal{Z}^{-}_{4} \subset \mathcal{Z}^{-}_3 \subset \mathcal{Z}^{-}_2 \subset \mathcal{Z}^{-}_{1} = \pi_{\opn{HT}, K_p}^{-1}\left( \overline{]X^G_{w_2}[} \right) \subset \mathcal{Z}^{-}_0 = \mathcal{X}_G(K^G_{\beta})
    \]
    (c.f., \cite[\S 5.5.2]{BoxerPilloni}) which gives rise to the spectral sequence
    \[
    E_1^{p, q} = \opn{H}^{p+q}_{\mathcal{Z}_{p}^{-} - \mathcal{Z}_{p+1}^{-}}(\mathcal{X}_G(K^G_{\beta}) - \mathcal{Z}_{p+1}^{-}, \mathscr{F})^{\opn{fs}} \Rightarrow \opn{H}^{p+q}(\mathcal{X}_G(K^G_{\beta}), \mathscr{F})^{\opn{fs}} 
    \]
    where the finite-slope part is with respect to the action of $T^{G, -}$. By \cite[Corollary 6.2.10]{HHTBoxerPilloni}, the conjectural lower bounds on slopes in \cite[Conjecture 5.9.2]{BoxerPilloni} hold. This implies that: for any eigensystem $\lambda \colon T^{G, -} \to \Qpb^{\times}$ occurring in the cohomology $\mathrm{R}\Gamma_{\mathcal{Z}_{3}^{-}}(\mathcal{X}_G(K^G_{\beta}), \mathscr{F})^{\opn{fs}}$, one has
    \[
    v_p(\lambda(t)) \geq v_p(\kappa_G^*(t)) .
    \]
    for any $t \in T^{G, -}$. Indeed, it suffices to verify this for the sheaves associated with the graded pieces of $\mbb{D}_{\overline{P}_G}(\kappa_G^*)^{\opn{nearly}}$ (which are $M_G$-representations); but by Lemma \ref{BigSetOfWeightsLemma}, we know that the highest weight $\kappa$ of any such graded piece satisfies $v_p(\kappa(t)) \geq v_p(\kappa_G^*(t))$ for any $t \in T^{G, -}$. We therefore see that $E_1^{3, q}$ vanishes on small slope $\mathrm{ss}_{w_1}^M(\kappa_G^*)$ parts.

    We have an exact triangle:
    \[
    R\Gamma_{\mathcal{Z}_3^{-}}\left( \mathcal{X}_G(K^G_{\beta}), \mathscr{F} \right) \to R\Gamma_{\mathcal{Z}_2^{-}}\left( \mathcal{X}_G(K^G_{\beta}), \mathscr{F} \right) \xrightarrow{\opn{res}} R\Gamma_{\mathcal{Z}_2^{-}-\mathcal{Z}_3^{-}}\left( \mathcal{X}_G(K^G_{\beta}) - \mathcal{Z}_3^{-}, \mathscr{F} \right) \xrightarrow{+1}
    \]
    hence the above vanishing implies that the restriction map is an isomorphism on small slope $\mathrm{ss}_{w_1}^M(\kappa_G^*)$ parts of cohomology. This proves part (3).

    Finally we note that, since $\mathscr{F}$ is coherent and vanishing along the boundary divisor, one has 
    \[
    \opn{H}^i\left( \mathcal{X}_G(K^G_{\beta}) - \mathcal{Z}_2^{-}, \mathscr{F} \right)^{\opn{fs}} = 0
    \]
    for $i \geq 2$. Indeed, it suffices to show that $E_1^{0,q}$ and $E_1^{1, q}$ vanish on finite slope parts for $q \geq 2$, but this follows from \cite[Theorem 5.6.1]{BoxerPilloni}. This implies that the map
    \[
    \opn{H}^2_{\mathcal{Z}_2^{-}}\left( \mathcal{X}_G(K^G_{\beta}), \mathscr{F} \right) \xrightarrow{\opn{cores}} \opn{H}^2\left( \mathcal{X}_G(K^G_{\beta}), \mathscr{F} \right)
    \]
    is surjective on finite-slope parts and concludes the proof of part (2).
\end{proof}

We will also need a classicality result relating $p$-adic weight overconvergent cohomology with classical weight overconvergent cohomology. Let $\kappa_G$ be as in Situation \ref{situation:Weights}. We say that $\lambda \colon T^{G, -} \to \Qpb^{\times}$ satisfies $\opn{ss}_{M, w_1}(\kappa_G^*)$ if: for any $w \in W_{M_G}$ with $w \neq \opn{id}$, there exists $y \in T^{G, -}$ such that $v_p(\lambda(y)) < v_p(((w_1^{-1}w) \star \kappa_G^*)(y))$.

\begin{proposition}
    Let $\kappa_G$ be as in Situation \ref{situation:Weights} and $n \geq 1$ any integer. Then the natural map
    \[
    R\Gamma( K^G_{\beta}, \kappa_G^*; \opn{cusp})^{(\dagger, n\opn{-an}, -)} \to R\Gamma( K^G_{\beta}, \kappa_G^*; \opn{cusp})^{(\dagger, -)}
    \]
    is a quasi-isomorphism on small slope $\opn{ss}_{M, w_1}(\kappa_G^*)$ parts (for the unnormalised action as in Convention \ref{ConventionOnUnnormalisedAction}).
\end{proposition}
\begin{proof}
    This follows from the same proof as in \cite[Corollary 6.8.4]{BoxerPilloni} in combination with \cite[Corollary 6.2.16]{HHTBoxerPilloni}, which allows us to replace the ``strictly small slope'' condition with the small slope $\opn{ss}_{M, w_1}(\kappa_G^*)$ condition.
\end{proof}

Finally, we will need to know when we can lift overconvergent cohomology classes to nearly overconvergent classes.

\begin{proposition}
    Let $(A, A^+)$ be a complete Tate affinoid pair over $(\mbb{Q}_p, \mbb{Z}_p)$ and $n \geq 1$ an integer. Let $\kappa_G \colon T(\mbb{Z}_p) \to A^{\times}$ be an $n$-analytic character. Set $S = \opn{Spa}(A, A^+)$.
    \begin{enumerate}
        \item For any $h \in \mbb{Q}$, $t \in T^{G, --}$, and $x \in S(\mbb{C}_p)$, there exists a sufficiently small open quasi-compact affinoid neighbourhood $V = \opn{Spa}(B, B^+)$ of $x$ in $S$ such that both $R\Gamma( K^G_{\beta}, \kappa_{G,V}^*; \opn{cusp})^{(\dagger, n\opn{-an}, -)}$ and $R\Gamma( K^G_{\beta}, \kappa_{G,V}^*; \opn{cusp})^{(\ddagger, n\opn{-an}, -)}$ admit slope $\leq h$ decompositions with respect to the action of $t$, where $\kappa_{G, V}$ denotes the composition $T(\mbb{Z}_p) \xrightarrow{\kappa_G} A^{\times} \to B^{\times}$.
        \item Let $h \in \mbb{Q}$ and suppose that $R\Gamma( K^G_{\beta}, \kappa_G^*; \opn{cusp})^{(\dagger, n\opn{-an}, -)}$ and $R\Gamma( K^G_{\beta}, \kappa_G^*; \opn{cusp})^{(\ddagger, n\opn{-an}, -)}$ admit slope $\leq h$ decompositions with respect to the action of a fixed element $t \in T^{G, --}$. Then the natural map 
        \begin{equation} \label{NearlyH2toOCH2surjectiveEqn}
        \opn{H}^2( K^G_{\beta}, \kappa_G^*; \opn{cusp})^{(\ddagger, n\opn{-an}, -)} \to \opn{H}^2( K^G_{\beta}, \kappa_G^*; \opn{cusp})^{(\dagger, n\opn{-an}, -)}
        \end{equation}
        induced from the map of representations (\ref{IMGpadicToIPGEqn}) is surjective on slope $\leq h$ parts.
    \end{enumerate} 
\end{proposition}
\begin{proof}
    Part (1) follows from the fact that the action of $w_1 t w_1^{-1}$ on $\mbb{D}_{\overline{\mathcal{P}}^{\square}_{G, n+1, \beta}}(\kappa_G^*)$ and $\mbb{D}_{\mathcal{M}^{\square}_{G, n+1, \beta}}(\kappa_G^*)$ is compact (and a similar argument as in \cite[\S 6.4.2]{BoxerPilloni}). In particular, one can easily check that there exists an integer $m \geq 0$ such that
    \[
    R\Gamma( K^G_{\beta}, \kappa_{G,V}^*; \opn{cusp})^{(\ddagger, n\opn{-an}, -), \leq h} \xrightarrow{\sim} \left(\varinjlim_U R\Gamma_{\mathcal{Z}_{G, <1, \beta} \cap U}(U, [\mathcal{F}_m\mathbb{D}_{\overline{\mathcal{P}}_{G, n+1, \beta}^{\square}}(\kappa^*_{G,V})](-D_G) )\right)^{\leq h}
    \]
    where $\mathcal{F}_m\mathbb{D}_{\overline{\mathcal{P}}_{G, n+1, \beta}^{\square}}(\kappa^*_{G,V}) = \mathbb{D}_{\overline{\mathcal{P}}_{G, n+1, \beta}^{\square}}(\kappa^*_{G,V}) / \opn{Fil}^{m+1}\mathbb{D}_{\overline{\mathcal{P}}_{G, n+1, \beta}^{\square}}(\kappa^*_{G,V})$ and $\opn{Fil}^m\mathbb{D}_{\overline{\mathcal{P}}_{G, n+1, \beta}^{\square}}(\kappa^*_{G,V})$ denotes the sub-representation of linear functionals $\mathbb{I}_{\overline{\mathcal{P}}_{G, n, \beta}^{\square, \circ}}(\kappa_{G,V}) \to B$ which vanish on the subspace of functions which are polynomial in the coordinates $a,b$ from Lemma \ref{PartialFactorisationLemma} of total degree $\leq m-1$ (we set $\opn{Fil}^0\mathbb{D}_{\overline{\mathcal{P}}_{G, n+1, \beta}^{\square}}(\kappa^*_{G,V}) = \mathbb{D}_{\overline{\mathcal{P}}_{G, n+1, \beta}^{\square}}(\kappa^*_{G,V})$). In particular, we have $\mathcal{F}_0\mathbb{D}_{\overline{\mathcal{P}}_{G, n+1, \beta}^{\square}}(\kappa^*_{G,V}) = \mathbb{D}_{\mathcal{M}_{G, n+1, \beta}^{\square}}(\kappa^*_{G,V})$.

    By an inductive argument, it therefore suffices to show that
    \[
    \left(\varinjlim_U \opn{H}^3_{\mathcal{Z}_{G, <1, \beta} \cap U}(U, [\opn{Gr}^j\mathbb{D}_{\overline{\mathcal{P}}_{G, n+1, \beta}^{\square}}(\kappa^*_{G,V})](-D_G) )\right)^{\leq h} = 0
    \]
    for any $0 \leq j \leq m$, where $\opn{Gr}^j\mathbb{D}_{\overline{\mathcal{P}}_{G, n+1, \beta}^{\square}}(\kappa^*_{G,V}) = \opn{Fil}^j\mathbb{D}_{\overline{\mathcal{P}}_{G, n+1, \beta}^{\square}}(\kappa^*_{G,V})/\opn{Fil}^{j+1}\mathbb{D}_{\overline{\mathcal{P}}_{G, n+1, \beta}^{\square}}(\kappa^*_{G,V}) =: M_j$. Note that $M_j$ is a projective Banach $\mathcal{M}^{\square}_{G, n+1, \beta}$-representation over $B$. It suffices to show that 
    \[
    R\Gamma(\pi_{\opn{HT}, K_p}^{-1}(\mathcal{U}), [M_j](-D_G)) 
    \]
    is concentrated in degree $0$ for any open affinoid $\mathcal{U} \subset {]C^G_{w_1}[_{k, k}} K_p$ (where $k \gg 1$ is a sufficiently large integer so that $\mathcal{P}_{G, \opn{dR}, n}$ is defined over $\pi_{\opn{HT}, K_p}^{-1}\left( {]C^G_{w_1}[_{k, k}} K_p \right)$). Indeed granted this, we can therefore follow the same strategy as in \cite[Theorem 5.6.1]{BoxerPilloni} noting that $]C^G_{w_1}[_{k, k} \backslash ]C^G_{w_1}[_{\bar{l}, k}$ ($l > k$) is covered by two acyclic open subspaces. But this claim about being concentrated in degree $0$ follows from exactly the same proof as in \cite[Lemma 6.6.2]{BoxerPilloni}.
\end{proof}

\subsection{\texorpdfstring{$p$}{p}-adic trilinear periods}

We now define $p$-adic versions of $\mathcal{P}^{\opn{alg}}$. Fix an integer $\beta \geq 1$ throughout.

\subsubsection{Classical weight nearly overconvergent}

Suppose that we are in Situation \ref{situation:Weights}. In what follows, we work with adic Shimura varieties (and their cohomology groups) over a fixed finite extension $L/\mbb{Q}_p$, but omit this from the notation. Recall that $\hat{\iota}^{-1}(\mathcal{Z}_{G, <1, \beta}) = \mathcal{Z}_{H, \opn{id}, \beta}$ and we have a natural map $[\mbb{D}_{\overline{P}_G}(\kappa_G^*)^{\opn{nearly}}](-D_G) \to \hat{\iota}_* [\mbb{D}_{\overline{P}_H}(\kappa_H^*)^{\opn{nearly}}](-D_H)$ induced from the branching law $\opn{br}^{\opn{nearly}}$ from \S \ref{algnearlyrepsSSec}. This implies that we have a commutative diagram:
\begin{equation} \label{CoresResCohDiagram}
\begin{tikzcd}
{\opn{H}^2\left(K^G_{\beta}, \kappa_G^*; \opn{cusp} \right)^{\opn{nearly}}} \arrow[r, "\hat{\iota}^*"]                                                                                                          & {\opn{H}^2\left(K^H_{\beta}, \kappa_H^*; \opn{cusp} \right)^{\opn{nearly}}}                           \\
{\opn{H}^2_{\mathcal{Z}_{G, <1, \beta}}(\mathcal{X}_{G}(K^G_{\beta}), [\mbb{D}_{\overline{P}_G}(\kappa_G^*)^{\opn{nearly}}](-D_G))} \arrow[u, "\opn{cores}"] \arrow[d, "\opn{res}"'] \arrow[r, "\hat{\iota}^*"] & {\opn{H}^2\left(K^H_{\beta}, \kappa_H^*; \opn{cusp} \right)^{(\ddagger, -)}} \arrow[u, "\opn{cores}"] \\
{\opn{H}^2\left( K^G_{\beta}, \kappa_G^*; \opn{cusp} \right)^{(\ddagger, -)}} \arrow[ru, "\hat{\iota}^*"']                                                                                                      &                                                                                                      
\end{tikzcd}
\end{equation}
Furthermore, we have an exterior cup product pairing:
\begin{align}
    \opn{H}^0\left( K^{H_1}_{\beta}, \zeta_{H_1} \right)^{(\ddagger, +)} \otimes_L \opn{H}^0\left( K^{H_2}_{\beta}, \zeta_{H_2} \right)^{(\ddagger, +)} &\to \opn{H}^0\left( K^{H}_{\beta}, \kappa_H - 2\rho_H \right)^{(\ddagger, +)} \label{DDaggerClassicalExtCupEqn} \\
    (\omega_1, \omega_2 ) &\mapsto \omega_1 \sqcup \omega_2 \defeq \opn{pr}_1^*(\omega_1) \smile \opn{pr}_2^*(\omega_2) . \nonumber
\end{align}
Let $\mathscr{N}^{\dagger}_{H_i, \beta}$ denote the space of nearly overconvergent modular forms of level $K^{H_i}_{\beta}$ as constructed in \cite[Definition 2.3.1]{DiffOps} (for $g=\opn{id}$), i.e., one has $\mathscr{N}^{\dagger}_{H_i, \beta} = \varinjlim_U \opn{H}^0\left( U, \mathcal{O}_{P_{H_i, \opn{dR}}^{\opn{an}}} \right)$ where $U$ runs over all quasi-compact open subspaces of $P_{H_i, \opn{dR}}^{\opn{an}}$ containing the closure of $\mathcal{IG}_{H_i, \opn{id}, \beta}^{\opn{tor}}$. Then we have identifications:
\[
\opn{H}^0\left( K_{\beta}^{H_i}, \zeta_{H_i} \right)^{(\ddagger, +)} = \opn{Hom}_{T_{\opn{GL}_2}(\mbb{Z}_p)}\left( -\zeta_{H_i}, \opn{Fil}_{t_i} \mathscr{N}^{\dagger}_{H_i, \beta} \right)
\]
where $\opn{Fil}_{\bullet}$ denotes the filtration in \cite[Theorem 2.3.2]{DiffOps}.

\begin{definition} \label{PddaggerLClassicalWtDef}
    With notation as above, let $\mathcal{P}^{\ddagger}_L$ denote the trilinear map given by
    \begin{align*}
        \mathcal{P}_L^{\ddagger} \colon \opn{H}^2\left( K^G_{\beta}, \kappa_G^*; \opn{cusp} \right)^{(\ddagger, -)} \otimes \opn{H}^0\left( K_{\beta}^{H_1}, \zeta_{H_1} \right)^{(\ddagger, +)} \otimes \opn{H}^0\left( K_{\beta}^{H_2}, \zeta_{H_2} \right)^{(\ddagger, +)} &\to L \\
        (\eta, \omega_1, \omega_2) &\mapsto \langle \hat{\iota}^*\eta, \omega_1 \sqcup \omega_2 \rangle 
    \end{align*}
    where $\langle \cdot, \cdot \rangle$ denotes the natural pairing between $\opn{H}^2\left(K^H_{\beta}, \kappa_H^*; \opn{cusp} \right)^{(\ddagger, -)}$ and $\opn{H}^0\left( K^{H}_{\beta}, \kappa_H - 2\rho_H \right)^{(\ddagger, +)}$. 
\end{definition}

\begin{remark}
    One has an overconvergent version of Definition \ref{PddaggerLClassicalWtDef}: 
    \[
    \mathcal{P}_L^{\dagger} \colon \opn{H}^2\left( K^G_{\beta}, \kappa_G^*; \opn{cusp} \right)^{(\dagger, -)} \otimes \opn{H}^0\left( K_{\beta}^{H_1}, \zeta_{H_1} \right)^{(\dagger, +)} \otimes \opn{H}^0\left( K_{\beta}^{H_2}, \zeta_{H_2} \right)^{(\dagger, +)} \to L
    \]
    defined in exactly the same way (using the branching law from the bottom horizontal arrow of (\ref{brNearlyCommWithAlgDiagram})). This is one of the trilinear pairings used in \cite{LZBK21}. If 
    \[
    \omega_i \in \opn{H}^0\left( K_{\beta}^{H_i}, \zeta_{H_i} \right)^{(\dagger, +)} \subset \opn{H}^0\left( K_{\beta}^{H_i}, \zeta_{H_i} \right)^{(\ddagger, +)}
    \]
    and $\eta^{\ddagger} \in \opn{H}^2\left( K^G_{\beta}, \kappa_G^*; \opn{cusp} \right)^{(\ddagger, -)}$, then one has the compatibility
    \[
    \mathcal{P}^{\ddagger}_L(\eta^{\ddagger}, \omega_1, \omega_2) = \mathcal{P}^{\dagger}_L(\eta^{\dagger}, \omega_1, \omega_2)
    \]
    where $\eta^{\dagger}$ denotes the image of $\eta^{\ddagger}$ under the natural map $\opn{H}^2\left( K^G_{\beta}, \kappa_G^*; \opn{cusp} \right)^{(\ddagger, -)} \to \opn{H}^2\left( K^G_{\beta}, \kappa_G^*; \opn{cusp} \right)^{(\dagger, -)}$.
\end{remark}

We note the following compatibility of the nearly overconvergent trilinear pairing with the algebraic version.

\begin{lemma} \label{DDaggerPerEqualsAlgPerLem}
    Let $\eta \in \opn{H}^2\left( K^G_{\beta}, \kappa_G^*; \opn{cusp} \right)^{(\ddagger, -)}$ be a class in the small slope $\opn{ss}^M_{w_1}(\kappa_G^*)$ part, and let $\eta^{\opn{cl}} \in \opn{H}^2\left( K^G_{\beta}, \kappa_G^*; \opn{cusp} \right)^{\opn{nearly}}$ denote the corresponding small slope $\opn{ss}^M_{w_1}(\kappa_G^*)$ class via the morphisms in Proposition \ref{NearlyClassicalityProp}. Let $\omega_i \in \opn{H}^0( K^{H_i}_{\beta}, \zeta_{H_i})^{\opn{nearly}}$ and let $\opn{res}(\omega_i) \in \opn{H}^0( K^{H_i}_{\beta}, \zeta_{H_i})^{(\ddagger, +)}$ denote its image under the restriction map. Then
    \[
    \mathcal{P}^{\ddagger}_L(\eta, \opn{res}(\omega_1), \opn{res}(\omega_2)) = \mathcal{P}^{\opn{alg}}_L(\eta^{\opn{cl}}, \omega_1, \omega_2) .
    \]
\end{lemma}
\begin{proof}
    This follows immediately from the commutative diagram (\ref{CoresResCohDiagram}).
\end{proof}

\subsubsection{\texorpdfstring{$p$}{p}-adic weight nearly overconvergent} \label{padicWtNOCsssec}

Let $(A, A^+)$ be a complete Tate affinoid pair over $(\mbb{Q}_p, \mbb{Z}_p)$. Let $n \geq 1$ and let $\kappa_G \colon T(\mbb{Z}_p) \to A^{\times}$ be a $n$-analytic character. Let $\lambda \colon \mbb{Z}_p^{\times} \to A^{\times}$ be a $n$-analytic character and set $\kappa_H = w_{M_G}^{\opn{max}} \kappa_G + (\lambda, -\lambda; 0)$. We can lift $\kappa_H - 2\rho_H$ to a pair of characters $(\zeta_{H_1}, \zeta_{H_2}) \colon T_{\opn{GL}_2}(\mbb{Z}_p) \times T_{\opn{GL}_2}(\mbb{Z}_p) \to A^{\times}$ via the identification $T(\mbb{Z}_p) = T_{\opn{GL}_2}(\mbb{Z}_p) \times_{\mbb{Z}_p^{\times}} T_{\opn{GL}_2}(\mbb{Z}_p)$.

Note that we have
\[
\opn{H}^0\left( K_{\beta}^{H_i}, \zeta_{H_i} \right)^{(\ddagger, n\opn{-an}, +)} \subset \opn{Hom}_{T_{\opn{GL}_2}(\mbb{Z}_p)}\left( - \zeta_{H_i}, \mathscr{N}^{\dagger}_{H_i, \beta} \right) =: \mathscr{N}^{\dagger}_{H_i, \beta, \zeta_{H_i}}
\]
and $\varinjlim_{m \geq n} \opn{H}^0\left( K_{\beta}^{H_i}, \zeta_{H_i} \right)^{(\ddagger, m\opn{-an}, +)} = \mathscr{N}^{\dagger}_{H_i, \beta, \zeta_{H_i}}$. By Remark \ref{DecompOfDDaggerIntoFactorsRem} and the discussion following Remark \ref{OnlyUniqueAfterShrinkingRem}, we have an ``$n$-analytic'' version of the exterior cup product pairing in (\ref{DDaggerClassicalExtCupEqn}). We also have a natural pullback map 
\[
\hat{\iota}^* \colon \opn{H}^2\left( K^G_{\beta}, \kappa_G^*; \opn{cusp} \right)^{(\ddagger, n\opn{-an}, -)} \to \opn{H}^2\left( K^H_{\beta}, \kappa_H^*; \opn{cusp} \right)^{(\ddagger, n\opn{-an}, -)}
\]
again via the branching law $\opn{br}^{n\opn{-an}}$ in Proposition \ref{DefPropOfbrnan} and the discussion following Remark \ref{OnlyUniqueAfterShrinkingRem}.

\begin{definition}
    With notation as above, let $\mathcal{P}_{\lambda}^{\ddagger, n\opn{-an}}$ denote the following $A$-trilinear pairing
    \begin{align*}
        \mathcal{P}_{\lambda}^{\ddagger, n\opn{-an}} \colon \opn{H}^2\left( K^G_{\beta}, \kappa_G^*; \opn{cusp} \right)^{(\ddagger, n\opn{-an}, -)} \otimes \opn{H}^0\left( K_{\beta}^{H_1}, \zeta_{H_1} \right)^{(\ddagger, n\opn{-an}, +)} \otimes \opn{H}^0\left( K_{\beta}^{H_2}, \zeta_{H_2} \right)^{(\ddagger, n\opn{-an}, +)} \to A 
    \end{align*}
    given by $\mathcal{P}_{\lambda}^{\ddagger, n\opn{-an}}(\eta, \omega_1, \omega_2) \defeq \langle \hat{\iota}^*\eta, \omega_1 \sqcup \omega_2 \rangle$. We let $\mathcal{P}_{\lambda}^{\dagger, n\opn{-an}}$ denote the overconvergent version, defined in exactly the same way but replacing the cohomology groups with their ``$\dagger$-versions''.
\end{definition}

\begin{remark}
    As above, the pairings $\mathcal{P}_{\lambda}^{\ddagger, n\opn{-an}}$ and $\mathcal{P}_{\lambda}^{\dagger, n\opn{-an}}$ are compatible with each other by viewing $\omega_1, \omega_2$ as nearly overconvergent classes, and considering the image of $\eta$ under the map from nearly overconvergent cohomology to overconvergent cohomology. The pairing $\mathcal{P}_{\lambda}^{\dagger, n\opn{-an}}$ appears in the construction of the $p$-adic $L$-function in \cite[Definition 10.6.2]{LZBK21}.
\end{remark}

\begin{remark}
    Since the $p$-adic branching laws in Remark \ref{DecompOfDDaggerIntoFactorsRem} and Proposition \ref{DefPropOfbrnan} are compatible with changing $n$, the pairings $\mathcal{P}_{\lambda}^{\ddagger, n\opn{-an}}$ induce an $A$-trilinear pairing
    \[
    \mathcal{P}_{\lambda}^{\ddagger, \opn{la}} \colon \opn{H}^2\left( K^G_{\beta}, \kappa_G^*; \opn{cusp} \right)^{(\ddagger, \opn{la}, -)} \otimes \mathscr{N}^{\dagger}_{H_1, \beta, \zeta_{H_1}} \otimes \mathscr{N}^{\dagger}_{H_2, \beta, \zeta_{H_2}} \to A
    \]
    where $\opn{H}^2\left( K^G_{\beta}, \kappa_G^*; \opn{cusp} \right)^{(\ddagger, \opn{la}, -)} = \varprojlim_{m \geq n} \opn{H}^2\left( K^G_{\beta}, \kappa_G^*; \opn{cusp} \right)^{(\ddagger, m\opn{-an}, -)}$. A similar statement is true for $\mathcal{P}^{\dagger, \opn{la}}_{\lambda}$, and we note that $\opn{H}^0\left( K^{H_i}_{\beta}, \zeta_{H_i} \right)^{(\dagger, n\opn{-an}, +)}$ is independent of $n$ and equal to the space $\mathscr{M}^{\dagger}_{H_i, \beta, \zeta_{H_i}} = \opn{Fil}_0 \mathscr{N}^{\dagger}_{H_i, \beta, \zeta_{H_i}}$ of weight $\zeta_{H_i}$, level $K^{H_i}_{\beta}$ overconvergent modular forms.
\end{remark}

\begin{remark}
    It is an easy check to see that $\mathcal{P}^{\ddagger, n\opn{-an}}_{\lambda}$ (and hence $\mathcal{P}^{\ddagger, \opn{la}}_{\lambda}$) is functorial in the pair $(A, A^+)$.
\end{remark}

We will also need the following compatibility with the classical weight nearly overconvergent pairing.

\begin{lemma} \label{DDaggerNanPeriodEqualsDDaggerPerLem}
    Suppose that we are in Situation \ref{situation:Weights}, and set $(A, A^+) = (L, \mathcal{O}_L)$ and $\lambda = r_1 + 1 - t_1$. Let $\eta^{\ddagger, n\opn{-an}} \in \opn{H}^2\left( K^G_{\beta}, \kappa_G^*; \opn{cusp} \right)^{(\ddagger, n\opn{-an}, -)}$ and let $\eta^{\ddagger} \in \opn{H}^2\left( K^G_{\beta}, \kappa_G^*; \opn{cusp} \right)^{(\ddagger, -)}$ denote its image under the natural map induced from $\mbb{D}_{\overline{\mathcal{P}}_{G, n+1, \beta}^{\square}}(\kappa_G^*) \to \mbb{D}_{\overline{P}_G}(\kappa_G^*)^{\opn{nearly}}$. For $i=1, 2$, let $\omega_i \in \opn{H}^0\left( K^{H_i}_{\beta}, \zeta_{H_i} \right)^{(\ddagger, +)} \subset \opn{H}^0\left( K^{H_i}_{\beta}, \zeta_{H_i} \right)^{(\ddagger, n\opn{-an}, +)}$. Then one has
    \[
    \mathcal{P}_{\lambda}^{\ddagger, n\opn{-an}}(\eta^{\ddagger, n\opn{-an}}, \omega_1, \omega_2) = \mathcal{P}_L^{\ddagger}(\eta^{\ddagger}, \omega_1, \omega_2) .
    \]
\end{lemma}
\begin{proof}
    This follows immediately from the compatibility of the branching laws $\opn{br}^{\opn{nearly}}$ and $\opn{br}^{n\opn{-an}}$ in Proposition \ref{BrNearlyBrNanCompatProposition}.
\end{proof}

\subsubsection{Trace compatibility}

The final compatibility we will need is when $\beta \geq 1$ varies. We first note that there are natural (unnormalised) trace maps
\begin{equation} \label{TrGbetaEqn}
\opn{Tr}_{G,\beta} \colon R\Gamma(K^G_{\beta+1}, \kappa_G^*; \opn{cusp})^{(\ddagger, n\opn{-an}, -)} \to R\Gamma(K^G_{\beta}, \kappa_G^*; \opn{cusp})^{(\ddagger, n\opn{-an}, -)}
\end{equation}
via the formalism in \cite[Lemma 2.1.2]{BoxerPilloni} and using the natural map $\mbb{D}_{\overline{\mathcal{P}}^{\square}_{G, n+1, \beta+1}}(\kappa_G^*) \to \mbb{D}_{\overline{\mathcal{P}}^{\square}_{G, n+1, \beta}}(\kappa_G^*)$. This latter map of representations is compatible with the branching law $\opn{br}^{n\opn{-an}}$ and the analogous change-of-$\beta$ map for $\mbb{D}_{\overline{\mathcal{P}}^{\diamondsuit}_{H, n+1, \beta}}(\kappa_H^*)$. Combining this with fact that both of the squares
\[
\begin{tikzcd}
{\mathcal{X}_{H, \opn{id}, \beta+1}} \arrow[d] \arrow[r] & {\mathcal{X}_{G, w_1, \beta+1}} \arrow[d] &  & {\mathcal{Z}_{H, \opn{id}, \beta+1}} \arrow[d] \arrow[r] & {\mathcal{Z}_{G, <1, \beta+1}} \arrow[d] \\
{\mathcal{X}_{H, \opn{id}, \beta}} \arrow[r]             & {\mathcal{X}_{G, w_1, \beta}}             &  & {\mathcal{Z}_{H, \opn{id}, \beta}} \arrow[r]             & {\mathcal{Z}_{G, <1, \beta}}            
\end{tikzcd}
\]
are Cartesian, implies that $\hat{\iota}^* \circ \opn{Tr}_{G, \beta} = \opn{Tr}_{H, \beta} \circ \; \hat{\iota}^*$, where 
\[
\opn{Tr}_{H, \beta} \colon R\Gamma(K^H_{\beta+1}, \kappa_H^*; \opn{cusp})^{(\ddagger, n\opn{-an}, -)} \to R\Gamma(K^H_{\beta}, \kappa_H^*; \opn{cusp})^{(\ddagger, n\opn{-an}, -)}
\]
is the unnormalised trace map for $H$. We have the following compatibility:

\begin{lemma} \label{PlambdaIsIndepOfBetaTraceLemma}
    Suppose we are in the setting of \S \ref{padicWtNOCsssec}. Let $\eta \in \opn{H}^2\left( K^G_{\beta+1}, \kappa_G^*; \opn{cusp} \right)^{(\ddagger, n\opn{-an}, -)}$ and $\omega_i \in \opn{H}^0\left( K_{\beta}^{H_i}, \zeta_{H_i} \right)^{(\ddagger, n\opn{-an}, +)}$ for $i=1, 2$. Then
    \[
    \mathcal{P}^{\ddagger, n\opn{-an}}_{\lambda, \beta}( \opn{Tr}_{G, \beta}(\eta), \omega_1, \omega_2) = \mathcal{P}^{\ddagger, n\opn{-an}}_{\lambda, \beta+1}(\eta, \omega_1, \omega_2)
    \]
    where we have added the subscript to denote the level we are working at and, for the right-hand side, we naturally view $\omega_i \in \opn{H}^0\left( K_{\beta+1}^{H_i}, \zeta_{H_i} \right)^{(\ddagger, n\opn{-an}, +)}$ via pullback.\footnote{In fact, $\opn{H}^0\left( K_{\beta}^{H_i}, \zeta_{H_i} \right)^{(\ddagger, n\opn{-an}, +)}$ is independent of $\beta$.}
\end{lemma}
\begin{proof}
    This follows from the adjunction between $\opn{Tr}_{H, \beta}$ and restriction (from level $K^H_{\beta}$ to level $K^H_{\beta+1}$).
\end{proof}

\begin{remark}
We also note that if the target of (\ref{TrGbetaEqn}) admits a slope $\leq h$ decomposition with respect to some fixed $t \in T^{G, --}$, then so does the source of (\ref{TrGbetaEqn}) (since the action of $t$ factors through $\opn{Tr}_{G, \beta}$; c.f., \cite[Lemma 4.2.14]{BoxerPilloni}). Moreover, in this case, the trace map $\opn{Tr}_{G, \beta}$ is then an isomorphism on slope $\leq h$ parts. We will use this to lift small slope classes along the trace maps $\opn{Tr}_{G, \beta}$ for any $\beta \geq 1$.
\end{remark}

\section{Automorphic representations}

We now introduce the automorphic representations of $\opn{GSp}_4$ and $\opn{GL}_2$ that we will consider in this article. 

\subsection{Automorphic representations of \texorpdfstring{$\opn{GSp}_4$}{GSp(4)}} \label{AutoRepsForGsp4Section}

Following \cite[\S 10.1]{LSZ17}, for the rest of the article we let $\pi$ denote a cuspidal automorphic representation of $\opn{GSp}_4(\mbb{A})$ which is:
\begin{itemize}
    \item non-CAP, and either non-endoscopic or of Yoshida-type;
    \item discrete series at $\infty$;
    \item globally generic.
\end{itemize}
Let $\pi = \bigotimes'_{v} \pi_v$ denote the restricted tensor product decomposition of $\pi$ over all places of $\mbb{Q}$. Then we can (and do) assume that $\pi$ is cohomological of weight $\nu_G = (r_1, r_2; -(r_1+r_2))$ with $r_1 \geq r_2 \geq 0$, i.e., the $(\mathfrak{g}, K^+_{G,\infty})$-cohomology of $\pi_{\infty} \otimes V_G(\nu_G)^*$ is non-trivial. Here $K_{G,\infty}^+ = \mbb{R}_{>0}^{\times} \cdot U_2(\mbb{R})$ denotes the standard maximal compact-mod-centre subgroup of $\opn{GSp}_4(\mbb{R})_+$. Note that the finite part $\pi_f$ is definable over a number field (this is the ``arithmetic normalisation'' appearing in \cite[Note 6.1.2]{LZBK20}). 

\begin{assumption} \label{PipUnramifiedAssumption}
    We assume that $\pi_p$ is unramified. In particular, by the globally generic assumption, $\pi_p$ is a generic unramified principal series representation of $G(\mbb{Q}_p)$.
\end{assumption}

Since $\pi$ is globally generic, it is quasi-paramodular of level $(N, M)$ in the sense of \cite{okazaki2019localarxiv}, where $M^2 \mid N$, $M$ is the conductor of the finite part of the central character of $\pi$, and $p \nmid N$ by Assumption \ref{PipUnramifiedAssumption}. If we write the central character of $\pi$ as $\widehat{\chi} \cdot |\!|s(-)|\!|^{-(r_1+r_2)/2}$ with $\chi$ finite-order, then for brevity, we say that $\pi$ has weight $\nu_G$, level $(N, M)$, and character $\chi$ (where $M$ is the conductor of $\chi$). 

\subsubsection{} Let $\ell$ be a prime for which $\pi_{\ell}$ is unramified. By the globally generic assumption, we can write $\pi_{\ell} = I_B^G(\theta_{\ell})$ as a normalised induction from $B(\mbb{Q}_{\ell})$ to $G(\mbb{Q}_{\ell})$ of a smooth unramified character $\theta_{\ell} \colon T(\mbb{Q}_{\ell}) \to \mbb{C}^{\times}$. Fix a Haar measure on $G(\mbb{Q}_\ell)$ such that $K_{\ell} \defeq G(\mbb{Z}_{\ell})$ has volume $1$, and let $\mathcal{H}_{\ell, R} \defeq \mathcal{H}(K_{\ell} \backslash G(\mbb{Q}_{\ell}) / K_{\ell}, R)$ denote the spherical Hecke algebra over a $\mbb{Q}$-algebra $R$ with convolution product with respect to the fixed Haar measure above. 

\begin{definition}
    We let $\Theta_{\pi, \ell} \colon \mathcal{H}_{\ell, \mbb{C}} \to \mbb{C}$ denote the $\mbb{C}$-algebra homomorphism through which $\mathcal{H}_{\ell, \mbb{C}}$ acts on $\pi_{\ell}^{G(\mbb{Z}_\ell)}$. For a finite set of places $S$ containing $\infty$, $p$ and all primes where $\pi$ ramifies, we let
    \[
    \Theta^S_{\pi} \colon \mathcal{H}^S_{\mbb{C}} \defeq \bigotimes_{\ell \not\in S} \mathcal{H}_{\ell, \mbb{C}} \to \mbb{C}
    \]
    denote the $\mbb{C}$-algebra homomorphism given by $\Theta^S_{\pi} = \otimes_{\ell \notin S} \Theta_{\pi, \ell}$. This homomorphism is definable over a number field $E$, which will assume from now on and suppress from the notation.
\end{definition}

\subsubsection{} 

We now introduce certain Hecke operators at $p$. Let $\beta \geq 1$ be an integer and let $K_{\opn{Kl}}(p^{\beta}) \subset G(\mbb{Z}_p)$ (resp. $K_{\opn{Si}}(p^{\beta}) \subset G(\mbb{Z}_p)$) denote the compact open subgroup of elements which are of the form $\left( \begin{smallmatrix} * & * & * & * \\ & * & * & * \\ & * & * & * \\ & & & * \end{smallmatrix} \right)$ (resp. $\left( \begin{smallmatrix} * & * & * & * \\ * & * & * & * \\ &  & * & * \\ & & * & * \end{smallmatrix} \right)$) modulo $p^{\beta}$. Let $K_{\opn{B}}(p^{\beta}) = K_{\opn{Si}}(p^{\beta}) \cap K_{\opn{Kl}}(p^{\beta})$ denote the upper-triangular Iwahori subgroup of depth $p^{\beta}$. We will often use the notation $K_{p, \opn{Si}} = K_{\opn{Si}}(p)$, $K_{p, \opn{Kl}} = K_{\opn{Kl}}(p)$, $K_{p, \opn{B}} = K_{\opn{B}}(p)$. For $\bullet \in \{ \opn{Si}, \opn{Kl}, \opn{B} \}$ we let $U_{p, \bullet} = U_{\bullet}(p)$ and $U'_{\bullet}(p^{\beta})$ denote the following Hecke operators:
\begin{align*}
U_{p, \opn{Si}} \defeq \upsilon_{\opn{Si}, 1}^{-1} \opn{ch}\left( K_{p, \opn{Si}} \left( \begin{smallmatrix} p & & & \\ & p & & \\ & & 1 & \\ & & & 1 \end{smallmatrix} \right) K_{p, \opn{Si}} \right), &\quad U'_{\opn{Si}}(p^{\beta}) \defeq \upsilon_{\opn{Si}, \beta}^{-1} \opn{ch}\left( K_{\opn{Si}}(p^{\beta}) \left( \begin{smallmatrix} 1 & & & \\ & 1 & & \\ & & p^{\beta} & \\ & & & p^{\beta} \end{smallmatrix} \right) K_{\opn{Si}}(p^{\beta}) \right) \\
U_{p, \opn{Kl}} \defeq \upsilon_{\opn{Kl}, 1}^{-1} \opn{ch}\left( K_{p, \opn{Kl}} \left( \begin{smallmatrix} p^2 & & & \\ & p & & \\ & & p & \\ & & & 1 \end{smallmatrix} \right) K_{p, \opn{Kl}} \right), &\quad U'_{\opn{Kl}}(p^{\beta}) \defeq \upsilon_{\opn{Kl}, \beta}^{-1} \opn{ch}\left( K_{\opn{Kl}}(p^{\beta}) \left( \begin{smallmatrix} 1 & & & \\ & p^{\beta} & & \\ & & p^{\beta} & \\ & & & p^{2\beta} \end{smallmatrix} \right) K_{\opn{Kl}}(p^{\beta}) \right) \\
U_{p, \opn{B}} \defeq \upsilon_{\opn{B}, 1}^{-1} \opn{ch}\left( K_{p, \opn{B}} \left( \begin{smallmatrix} p^3 & & & \\ & p^2 & & \\ & & p & \\ & & & 1 \end{smallmatrix} \right) K_{p, \opn{B}} \right), &\quad U'_{\opn{B}}(p^{\beta}) \defeq \upsilon_{\opn{B}, \beta}^{-1} \opn{ch}\left( K_{\opn{B}}(p^{\beta}) \left( \begin{smallmatrix} 1 & & & \\ & p^{\beta} & & \\ & & p^{2\beta} & \\ & & & p^{3\beta} \end{smallmatrix} \right) K_{\opn{B}}(p^{\beta}) \right)
\end{align*}
where for brevity we have written $\upsilon_{\bullet, \beta} \defeq \opn{Vol}(K_{\bullet}(p^{\beta}))$. We will often simply write $U_{p, \bullet}' = U'_{\bullet}(p^{\beta})$ when $\beta$ is clear from the context. Furthermore, we will also write $U_{p, \opn{Si}}, U_{p, \opn{Kl}}$ for the versions of the Hecke operators at Iwahori level (given by the same description replacing $\upsilon_{\bullet, 1}, K_{p, \bullet}$ with $\upsilon_{\opn{B}, 1}, K_{p, \opn{B}}$), and similarly for $U'_{p, \opn{Si}}, U'_{p, \opn{Kl}}$ at (depth $p^{\beta}$) Iwahori level.

For a $G$-dominant weight $\nu_G=(r_1, r_2; -(r_1+r_2))$, these operators naturally act on the singular cohomology groups $\opn{H}^{i}_c\left( Y_{G}(K^pK_{\opn{B}}(p^{\beta}))(\mbb{C}), \mathscr{V}_{\mbb{Q}} \right)$, where $\mathscr{V}_{\mbb{Q}}$ denotes the $\mbb{Q}$-local system associated with the algebraic representation $V_G(\nu_G)^*$. We let
\[
U_{p, \opn{Si}}^{\circ} \defeq U_{p, \opn{Si}}, \quad U_{p, \opn{Kl}}^{\circ} \defeq p^{-r_2} U_{p, \opn{Kl}}, \quad U_{p, B}^{\circ} \defeq p^{-r_2} U_{p, B} 
\]
which are the optimal normalisations for the action of the Hecke operators on $\opn{H}^{i}_c\left( Y_{G}(K^pK_{\opn{B}}(p^{\beta}))(\mbb{C}), \mathscr{V}_{\mbb{Q}} \right)$ (see \cite[Proposition 5.10.11]{BoxerPilloni}).

More generally, recall that $T^{G,+} \subset T(\mbb{Q}_p)$ is the submonoid of elements $t \in T(\mbb{Q}_p)$ which satisfy $v_p(\alpha(t)) \geq 0$ for all positive roots $\alpha \in \Phi_G^+$. For any $\mbb{Q}$-algebra $R$, we will consider the following (commutative) subalgebra
\begin{align*}
    R[T^{G,+}] &\hookrightarrow \mathcal{H}\left( K_{p, \opn{B}} \backslash G(\mbb{Q}_p) / K_{p, \opn{B}}, R \right) \\
    t &\mapsto \opn{Vol}(K_{p, \opn{B}})^{-1} \opn{ch}( K_{p, \opn{B}} \cdot t \cdot K_{p, \opn{B}} ) .
\end{align*}
The following lemma describes the possible Hecke eigensystems at $p$ that can occur in $\pi_p^{K_{p, \opn{B}}}$.

\begin{lemma}
    Let $\pi$ be as above and write $\pi_p = I_B^G (\theta_p)$ as a normalised induction of a smooth unramified character $\theta_p \colon T(\mbb{Q}_p) \to \mbb{C}^{\times}$. Then the semisimplification of $\pi_p^{K_{p, \opn{B}}}$ as a $\mbb{C}[T^{G,+}]$-module can be described as
    \begin{equation} \label{SSdecompositionPip}
    \bigoplus_{w \in W_G} \delta_B^{-1/2} (w \cdot \theta_p)
    \end{equation}
    where $\delta_B^{-1/2} (w \cdot \theta_p) \colon T^{G,+} \to \mbb{C}^{\times}$ denotes the character given by $\left[ \delta_B^{-1/2} (w \cdot \theta_p) \right](t) = \delta_B(t)^{-1/2} \theta_p(w^{-1}t w)$ and $\delta_B$ denotes the standard modulus character:
    \[
    \delta_B(t_1, t_2, s t_2^{-1}, st_1^{-1} ) = | t_1 |^4 \cdot |t_2|^2 \cdot |s|^{-3}.
    \]
\end{lemma}
\begin{proof}
    This follows from \cite[Theorem 3.3.3, Proposition 6.4.1]{Casselman}.
\end{proof}

We will impose the following assumption.

\begin{assumption} \label{pregularAssumption}
    We assume that $\pi_p$ is $p$-regular, i.e., the eigencharacters appearing in (\ref{SSdecompositionPip}) are pairwise distinct (so, in particular, $\pi_p^{K_{p, \opn{B}}}$ is a semisimple $\mbb{C}[T^{G,+}]$-module). 
\end{assumption}

\subsubsection{}

For convenience later on, we make the following definition.

\begin{definition}
    By a \emph{good $p$-stabilised} automorphic representation $\tilde{\pi} = (\pi, \Theta_{\pi, p})$ of $\opn{GSp}_4(\mbb{A})$, we mean the data of:
    \begin{itemize}
        \item a cuspidal automorphic representation $\pi$ of $\opn{GSp}_4(\mbb{A})$ satisfying the assumptions at the start of \S \ref{AutoRepsForGsp4Section}, as well as Assumption \ref{PipUnramifiedAssumption} and Assumption \ref{pregularAssumption};
        \item a character $\Theta_{\pi, p} \colon T^{G, +} \to \mbb{C}^{\times}$ which appears in the semisimple $\mbb{C}[T^{G, +}]$-module $\pi_p^{K_{p, \mathrm{B}}}$.
    \end{itemize}
    By a \emph{good $p$-stabilisation} $\phi \in \tilde{\pi}$, we mean a decomposible (non-zero) automorphic form
    \[
    \phi = \phi_{\infty} \otimes \phi^p \otimes \phi_p \in \pi_{\infty} \otimes \left( \otimes_{\ell \neq p} \pi_{\ell} \right)^{K(N, M)^p} \otimes \pi_p^{K_{p, \opn{B}}}
    \]
    where $\phi_p$ is an eigenvector for the action of $T^{G, +}$ with eigencharacter $\Theta_{\pi, p}$, and $\phi_{\infty}$ lies in the (unique) $K_{G, \infty}^+$-type of highest weight $\kappa_G + (-1, -3; 2)$. Here $K(N, M)^p \subset \opn{GSp}_4(\mbb{A}_f^p)$ denotes the quasi-paramodular subgroup of level $(N, M)$ (with $N$ equal to the conductor of $\pi$ and $M$ equal to the conductor of the finite-part of the central character of $\pi$).

    It will also be useful to introduce a ``normalised'' version of these eigensystems. More precisely, given a good $p$-stabilised automorphic representation $\tilde{\pi}$, we let $\Theta^+_{\pi, p} \colon T^{G, +} \to \mbb{C}^{\times}$ denote the character satisfying 
    \[
    \Theta^+_{\pi, p}(t) \defeq \nu_G(t)\Theta_{\pi, p}(t), \quad \quad t \in T^{G, +},
    \]
    where $\nu_G$ is the weight of $\pi$.
\end{definition}

We now introduce the relevant small slope conditions on good $p$-stabilised automorphic representations.

\begin{definition}
    Let $\tilde{\pi} = (\pi, \Theta_{\pi, p})$ be a good $p$-stabilised automorphic representation of $\opn{GSp}_4(\mbb{A})$. Via our fixed isomorphism $\mbb{C} \cong \Qpb$, we view $\Theta_{\pi, p}$ as a character of $T^{G,+}$ valued in $\Qpb^{\times}$. Let $v_p$ denote the $p$-adic valuation on $\Qpb^{\times}$, normalised so that $v_p(p) =1$. Then we say $\Theta_{\pi, p}$ (or $\tilde{\pi}$) is:
    \begin{itemize}
        \item \emph{Klingen small slope} if 
        \[
        v_p(\Theta_{\pi, p}(U_{p, \opn{Kl}}^{\circ})) < 1 + r_1 - r_2 ;
        \]
        \item \emph{Siegel small slope} if
        \[
        v_p(\Theta_{\pi, p}(U_{p, \opn{Si}}^{\circ})) < 1 + r_2 ;
        \]
        \item \emph{Borel small slope} (or, simply, \emph{small slope}) if $\Theta_{\pi, p}$ is both Klingen and Siegel small slope.
    \end{itemize}
    We say $\Theta_{\pi, p}$ (or $\tilde{\pi}$) is Klingen (resp. Siegel, resp. Borel) ordinary if $v_p(\Theta_{\pi, p}(U_{p, \opn{Kl}}^{\circ})) = 0$ (resp. $v_p(\Theta_{\pi, p}(U_{p, \opn{Si}}^{\circ})) = 0$, resp. $v_p(\Theta_{\pi, p}(U_{p, \opn{B}}^{\circ})) = 0$). 
\end{definition}

\begin{remark}
    The above definition of (Borel) small slope is equivalent to the ``$(+, \opn{ss}(\nu_G))$'' small slope condition in \cite[Definition 5.11.2]{BoxerPilloni} and is the usual requirement that appears in the theory of $p$-adic families of automorphic forms via Betti cohomology. Furthermore, if $\Theta_{\pi, p}$ is small slope, then $w_G^{\opn{max}} \cdot \Theta_{\pi, p}$ also satisfies the $\opn{ss}^M_{w_1}(\kappa_G^*)$ and $\opn{ss}_{M, w_1}(\kappa_G^*)$ conditions from \S \ref{ResultsFromHCTSSsec} (with $\kappa_G = (r_2+2, -r_1; -r_2-1)$); indeed, this follows from \cite[Proposition 5.11.10]{BoxerPilloni}.
\end{remark}

\subsubsection{Dual $p$-stabilisations} \label{DualPStabilisationsSSec}

It will be useful to consider $p$-stabilisations for the transpose Hecke operators at $p$. Let $\phi \in \tilde{\pi}$ denote a good $p$-stabilisation in a good $p$-stabilised cuspidal automorphic representation of $G(\mbb{A})$ as above. 

\begin{definition}
    Let $\beta \geq 1$. We let $\phi'_{\beta} \in \pi_{\infty} \otimes \pi_f^{K(N, M)^pK^G_{\opn{Iw}}(p^{\beta})}$ denote the automorphic form given by
    \[
    \phi'_{\beta} \defeq \Theta_{\pi, p}(t_{p, \opn{B}})^{-\beta} [K^G_{\opn{Iw}}(p^{\beta}) s_{p, \opn{B}}^{\beta} w_G^{\opn{max}} K^G_{\opn{Iw}}(p)] \cdot \phi, \quad t_{p, \opn{B}} = \left( \begin{smallmatrix} p^3 & & & \\ & p^2 & & \\ & & p & \\ & & & 1 \end{smallmatrix} \right), \; s_{p, \opn{B}} = w_{G}^{\opn{max}} t_{p, \opn{B}} (w_{G}^{\opn{max}})^{-1}.
    \]
    This is an eigenvector for the action of $\bigotimes_{\ell \nmid Np} \mathcal{H}_{\ell, \mbb{C}}$ with eigencharacter $\Theta_{\pi}^{Np} = \otimes_{\ell \nmid Np} \Theta_{\pi, \ell}$, and an eigenvector for the action of $T^{G, -}$ with eigencharacter $w_G^{\opn{max}} \cdot \Theta_{\pi, p}$. Furthermore, if $\opn{Tr}_{G, \beta}$ denotes the (unnormalised) trace map from level $K(N, M)^p K^G_{\opn{Iw}}(p^{\beta+1})$ to $K(N, M)^p K^G_{\opn{Iw}}(p^{\beta})$, then $\opn{Tr}_{G, \beta}(\phi'_{\beta+1}) = \phi'_{\beta}$.

    Conversely, given any trace compatible system of eigenvectors $\phi'_{\beta}$ as above, one can recover the good $p$-stabilisation $\phi$ by reversing this process. Hence $\phi$ and $\phi'_{\beta}$ encode the same information.
\end{definition}

\subsection{Automorphic representations of \texorpdfstring{$\opn{GL}_2$}{GL(2)}} \label{AutoRepsForGL2SSec}

We now introduce the cuspidal automorphic representations of $\opn{GL}_2(\mbb{A})$ that we will consider. 

\subsubsection{} 

Let $N \geq 1$ and $t \geq 0$ be integers. Let $f$ be a normalised cuspidal new eigenform of weight $1+t$, level $\Gamma_1(N)$, and nebentypus $\chi$. We impose the following assumption:

\begin{assumption} \label{PRegAssumptionsOnModularForm}
    We assume that $p \nmid N$. Furthermore, we assume that:
    \begin{enumerate}
        \item ($p$-regularity) The roots of the Hecke polynomial $X^2 - a_p(f)X + p^t \chi(p)$ are distinct.
        \item If $t \geq 1$ and $\rho_f \colon \Gal(\overline{\mbb{Q}}/\mbb{Q}) \to \opn{GL}_2(\Qpb)$ denotes the two-dimensional $p$-adic Galois representation associated with $f$, then $\rho_f|_{\Gal(\Qpb/\mathbb{Q}_p)}$ does not split as a direct sum of two characters (note that this is automatic if $v_p(a_p(f)) \neq 0$).
        \item If $t = 0$, then $f$ does not have real multiplication by a quadratic field in which $p$ splits (in the sense of \cite[Notation 4.7.1]{lz-coleman}).
    \end{enumerate}
\end{assumption}

Let $\sigma$ denote the cuspidal automorphic representation of $\opn{GL}_2(\mbb{A})$ associated with $f$ with central character $|\!| \opn{det}(-) |\!|^{(1+t)/2} \widehat{\chi}$, where $\chi$ is a Dirichlet character. For brevity, we say that $\sigma$ has weight $1+t$, level $N$, and character $\chi$.

\begin{definition}
    By a \emph{good $p$-stabilised} automorphic representation $\tilde{\sigma} = (\sigma, \alpha)$ of $\opn{GL}_2(\mbb{A})$, we mean the data of:
    \begin{itemize}
        \item a cuspidal automorphic representation $\sigma$ of weight $1+t$, level $N$, and character $\chi$ as above, satisfying Assumption \ref{PRegAssumptionsOnModularForm};
        \item a choice of root $\alpha$ of the Hecke polynomial in Assumption \ref{PRegAssumptionsOnModularForm}(1).
    \end{itemize}
    By the \emph{good $p$-stabilisation} $\phi \in \tilde{\sigma}$, we mean the automorphic form $\phi \in \sigma$ associated with the (normalised) $p$-stabilisation $f_{\alpha}(-) \defeq f(-) - p^t \chi(p) \alpha^{-1} f(p \cdot -)$; explicitly, the automorphic form $\phi \colon \opn{GL}_2(\mbb{Q}) \backslash \opn{GL}_2(\mbb{A}) \to \mbb{C}$ is described as:
    \[
    \phi(\gamma g_{\infty} k) = (\opn{det}g_{\infty})^{1+t} (c i + d)^{-(1+t)} f_{\alpha}\left( \frac{a i + b}{c i + d} \right), \quad \quad \gamma \in \opn{GL}_2(\mbb{Q}), \; g_{\infty} = \left(\begin{smallmatrix} a & b \\ c & d \end{smallmatrix} \right) \in \opn{GL}_2(\mbb{R})_+, \; k \in K_1(N),
    \]
    where $K_1(N) \subset \opn{GL}_2(\widehat{\mbb{Z}})$ denotes the adelic version of the arithmetic subgroup $\Gamma_1(N)$. 
\end{definition}

\subsection{Set-up and tame test data}

As in \cite[\S 10.2]{LZBK21}, we fix the following data in the two situations which will correspond to the $p$-adic $L$-functions for $\opn{GSp}_4 \times \opn{GL}_2$ and $\opn{GSp}_4 \times \opn{GL}_2 \times \opn{GL}_2$ respectively.

\subsubsection{Case A} \label{CaseASubSubsection}

Let $(N_0, M_0)$ be positive integers prime to $p$ with $M_0^2 \mid N_0$, and let $\chi_0$ be a Dirichlet character of conductor $M_0$. Let $(N_2, M_2)$ be positive integers prime to $p$ with $M_2 \mid N_2$, and let $\chi_2$ be a Dirichlet character of conductor $M_2$. Let $L/\mbb{Q}_p$ denote a finite extension containing $\mu_{N_0}$ and $\mu_{N_2}$, and such that $\chi_0, \chi_2$ are valued in $L$.

Let $S$ denote the set of primes dividing $N_0N_2$. As in \cite[\S 10.2]{LZBK21}, we consider the following tame test data $\gamma_S = (\gamma_{0, S}, \Phi_S)$, where: $\gamma_{0, S} \in G(\mbb{Q}_S)$ and $\Phi_S \in C^{\infty}_c(\mbb{Q}_S^{\oplus 2}, L)$ is a Schwartz function lying in the $(\widehat{\chi}_0 \widehat{\chi}_2)^{-1}$-eigenspace for the diagonal action of $\mbb{Z}_S^{\times}$.

Let $T$ denote an auxiliary finite set of primes which is disjoint from $S \cup \{ p \}$. Let $K_S \subset G(\mbb{Q}_S)$ denote the quasi-paramodular subgroup of level $(N_0, M_0)$. We consider compact open subgroups $\widehat{K}_S \subset G(\mbb{Q}_S)$, $K_T \subset G(\mbb{Z}_T)$ satisfying the following conditions:
\begin{itemize}
    \item $K_T$ is a normal subgroup of $G(\mbb{Z}_T)$;
    \item $\widehat{K}_S \subset \gamma_{0, S} K_S \gamma_{0, S}^{-1}$;
    \item the projection of $\widehat{K}_S \cap H$ to the first factor acts trivially on $\Phi_S$;
    \item the projection of $\widehat{K}_S \cap H$ to the second factor is contained in $\left\{ \left( \begin{smallmatrix} * & * \\ 0 & 1 \end{smallmatrix} \right) \text{ mod } N_2 \right\}$;

    \item $\mathcal{K} = K_T \prod_{\ell \not\in T} G(\mbb{Z}_{\ell}) \subset G(\mbb{A}_f)$, $\mathcal{K} \cap H$ and the projections of $\mathcal{K} \cap H$ to each factor are neat compact open subgroups.
\end{itemize}
One can easily see that such compact open subgroups exist (for example, one can take $T$ to have least three elements and $K_T$ is the subgroup of elements which are congruent to the identity modulo the primes in $T$). With the choice of these compact open subgroups, we set
\begin{align*} 
K^p = K_S \cdot K_T \cdot \prod_{\ell \notin S \cup T \cup \{p\}} G(\mbb{Z}_{\ell}) \quad \subset \quad G(\mbb{A}_f^p) \\
\widehat{K}^p = \widehat{K}_S \cdot K_T \cdot \prod_{\ell \notin S \cup T \cup \{p\}} G(\mbb{Z}_{\ell}) \quad \subset \quad G(\mbb{A}_f^p) .
\end{align*}
As in \S \ref{LevelSubgroupsSSSec}, we define $K^G_{\beta}$, $\widehat{K}^G_{\beta}$ etc. with respect to these choices of $K^p$ and $\widehat{K}^p$.

\subsubsection{Case B} \label{CaseBSubSubSection}

In this case, we let $(N_0, M_0, \chi_0)$ and $(N_2, M_2, \chi_2)$ as above. We let $(N_1, M_2)$ be positive integers prime to $p$ with $M_1 \mid N_1$, and take $\chi_1$ to be a Dirichlet character of conductor $M_1$. Let $L/\mbb{Q}_p$ be a finite extension containing $\mu_{N_i}$ ($i=0, 1, 2$) and such that $\chi_0, \chi_1, \chi_2$ are valued in $L$. We let $S$ denote the set of primes dividing $N_0N_1N_2$ and let $T$ denote an auxiliary set of primes disjoint from $S \cup \{p\}$.

In this case, we assume that $\chi_0 \chi_1 \chi_2 = 1$ and we fix elements $\gamma_S = (\gamma_{0, S}, \gamma_{1, S}) \in G(\mbb{Q}_S) \times \opn{GL}_2(\mbb{Q}_S)$. As above, we let $\widehat{K}_S \subset G(\mbb{Q}_S)$ and $K_T \subset G(\mbb{Z}_T)$ be compact open subgroups satisfying all the bullet points in \S \ref{CaseASubSubsection}, except the third one. Instead, we replace the third bullet point with the condition:
\begin{itemize}
    \item the projection of $\widehat{K}_S \cap H$ to the first factor is contained in $\gamma_{1, S} \left\{ \left( \begin{smallmatrix} * & * \\ 0 & 1 \end{smallmatrix} \right) \text{ mod } N_1 \right\} \gamma_{1, S}^{-1}$.
\end{itemize}
We define $K^p$ and $\widehat{K}^p$ in exactly the same way as in \S \ref{CaseASubSubsection}.

\subsection{\texorpdfstring{$L$}{L}-functions}

We recall the $L$-functions we are interested in studying.

\subsubsection{} \label{CaseALfunctionAndPeriodSSSec}

Suppose that we are in Case A (\S \ref{CaseASubSubsection}) and Situation \ref{situation:Weights}. Let $\tilde{\pi} = (\pi, \Theta_{\pi, p})$ be a good $p$-stabilised automorphic representation of $G(\mbb{A})$ of weight $\nu_G$, level $(N_0, M_0)$ and character $\chi_0$. Let $\tilde{\sigma} = (\sigma, \alpha)$ be a good $p$-stabilised automorphic representation of $\opn{GL}_2(\mbb{A})$ of weight $1+t_2$, level $N_2$, and character $\chi_2$. Let $\Pi = \pi \otimes |\!|s(-)|\!|^{(r_1+r_2)/2}$ and $\Sigma = \sigma \otimes |\!| \opn{det}(-) |\!|^{-(1+t_2)/2}$ denote the unitary normalisations of $\pi$ and $\sigma$ respectively. 

\begin{definition}
    Let $\chi$ be a Dirichlet character of $p$-th power conductor. 
    \begin{enumerate}
        \item Let $L(\Pi \times \Sigma, \chi, s)$ denote the $L$-function associated with the automorphic representation $\left(\Pi \otimes \widehat{\chi}\right) \boxtimes \Sigma$ of $G(\mbb{A}) \times \opn{GL}_2(\mbb{A})$ with respect to the $8$-dimensional representation of $\widehat{G}(\mbb{C}) \times \opn{GL}_2(\mbb{C})$ given by the tensor product of the $4$-dimensional spin representation of $\widehat{G}(\mbb{C})$ and the $2$-dimensional standard representation of $\opn{GL}_2(\mbb{C})$. In particular, if $\ell \not\in S \cup \{p\}$ and $(\theta_1, \theta_2, \theta_3, \theta_4)$ (resp. $(\gamma_1, \gamma_2)$) denotes the Satake parameters\footnote{In particular, the eigenvalues of $U_{\ell, \opn{Si}}$ on $\Pi_{\ell}^{K_{\ell, \opn{B}}}$ are given by $\{\ell^{3/2} \theta_1, \ell^{3/2} \theta_2, \ell^{3/2} \theta_3, \ell^{3/2} \theta_4 \}$.} of $\Pi_{\ell}$ (resp. $\Sigma_{\ell}$) via the identification $\widehat{G} \cong \opn{GSp}_4$, then the local $L$-factor at $\ell$ is described as:
        \[
        L_{\ell}(\Pi \times \Sigma, \chi, s) = L(\Pi_{\ell} \times \Sigma_{\ell}, \chi_{\ell}, s) = \prod_{i=1}^4 \prod_{j=1}^2 \left( 1 -  \chi(\ell) \theta_i \gamma_j \ell^{-s} \right)^{-1} .
        \]
        \item Let $\Lambda(\Pi \times \Sigma, \chi, s) = L_{\infty}(\Pi \times \Sigma, \chi, s) L(\Pi \times \Sigma, \chi, s)$ denote the completed $L$-function, and set $w = r_1+r_2+t_2+3$. Explicitly, the archimedean $L$-factor is given by
        \[
        L_{\infty}(\Pi \times \Sigma, \chi, s) = \Gamma_{\mbb{C}}(s+w/2-r_2-t_2-1)\Gamma_{\mbb{C}}(s+w/2 - r_2-1)\Gamma_{\mbb{C}}(s+w/2-t_2)\Gamma_{\mbb{C}}(s+w/2)
        \]
        where $\Gamma_{\mbb{C}}(s) = 2 (2\pi)^{-s} \Gamma(s)$ (and $\Gamma(-)$ is the usual Gamma function). Furthermore, the completed $L$-function satisfies the following functional equation:
        \[
        \Lambda(\Pi \times \Sigma, (\chi \chi_0 \chi_2)^{-1}, 1 - s) = \varepsilon(\Pi \times \Sigma, \chi, s) \Lambda(\Pi \times \Sigma, \chi, s)
        \]
        where $\varepsilon(\Pi \times \Sigma, \chi, s)$ is entire and nowhere vanishing.
        \item The critical values of the $L$-function $L(\Pi \times \Sigma, \chi, s)$ are of the form $s = j-w/2$, where $r_2+t_2+2 \leq j \leq r_1 +2$.
        \item We define the modified $L$-factor at $p$ to be:
        \[
        \mathcal{E}_p(\Pi \times \Sigma, \chi, s) \defeq \left\{ \begin{array}{cc} \prod_{i, j=1}^2 \frac{1 - p^{s-1}\theta_i^{-1}\gamma_j^{-1}}{1 - p^{-s}\theta_i\gamma_j} & \text{ if } \chi = 1 \\ G(\chi)^{-4}p^{4 c s} \left( \theta_1 \theta_2 \gamma_1 \gamma_2 \right)^{-2c}   & \text{ if } \opn{cond}(\chi) = p^c \text{ with } c \geq 1   \end{array} \right.
        \]
        where $\{ \theta_i : i=1, \dots, 4 \}$ are ordered so that $\Theta_{\Pi,p}(U_{p, \opn{Kl}}) = p^2 \theta_1 \theta_2$.
    \end{enumerate}
    For any set $S'$ of places of $\mbb{Q}$, we let $\Lambda^{S'}(\Pi \times \Sigma, \chi, s)$ denote the $L$-function with Euler factors at places in $S'$ removed.
\end{definition}

Since $\Pi$ is globally generic, we can consider its associated Whittaker model (with respect to the standard additive character of $\mbb{A}/\mbb{Q}$) as in \cite[\S 9.1]{LPSZ}. We can (and do) assume that this Whittaker model decomposes as a product of local Whittaker models satisfying the conditions in \emph{loc.cit.}. We make similar assumptions on the local Whittaker models for $\Sigma$. Let $v$ be a place of $\mbb{Q}$ and let $\varphi_v \in \Pi_v$, $\lambda_v \in \Sigma_v$, $\Phi_v \in C^{\infty}_{c}(\mbb{Q}_v^{\oplus 2})$, and $\chi_v \colon \mbb{Q}_v^{\times} \to \mbb{C}^{\times}$ a smooth unitary character. Then, as in \cite[Definition 8.3]{LPSZ}, we can consider the local zeta integral:
\[
Z(\varphi_v, \lambda_v, \Phi_v; \chi_v, s) \defeq \int_{Z_G(\mbb{Q}_v) N_H(\mbb{Q}_v) \backslash H(\mbb{Q}_v)} W_{\varphi_v}(h) f^{\Phi_v}(h_1; \chi_v^2\chi_{0,v}\chi_{2,v}, s) W_{\lambda_v}(h_2)\chi_v(\opn{det}h_2) dh
\]
where $W_{\varphi_v}$ (resp. $W_{\lambda_v}$) denotes the image of $\varphi_v$ (resp. $\lambda_v$) under the local Whittaker model of $\Pi_v$ (resp. $\Sigma_v$), and $f^{\Phi_v}$ is the Siegel section defined in \cite[\S 8.1]{LPSZ}. If $v$ is non-archimedean, then the Haar measure $dh$ is normalised so that the volume of $H(\mbb{Z}_v)$ is $1$.

We consider the following data:
\begin{itemize}
    \item For any $\beta \geq 1$, let 
    \[
    \eta_{\pi, \beta} \in \opn{H}^2\left( K^G_{\beta}, \kappa_G^*; \opn{cusp} \right)_{\mbb{C}}^{G(\mbb{Z}_T)/K_T} 
    \]
    be a generator of the one-dimensional eigenspace under the action of the Hecke algebra $\mathcal{H}^{S_0}_{\mbb{C}} \otimes \mbb{C}[T^{G,-}]$ with eigencharacter $\Theta_{\pi}^{S_0} \otimes \left( w_G^{\opn{max}} \cdot \Theta_{\pi, p} \right)$, where $S_0$ denotes the set of places dividing $\infty N_0 p$. This eigenspace is indeed one-dimensional by the arguments in \cite[Proposition 2.7.2]{LZBK21}. More precisely, by the work of Su \cite{Su19}, if we localise this cohomology group at the kernel $I$ of the above eigencharacter, then $\pi$ is the only automorphic representation that appears in this localised cohomology group (if $\pi'$ appears in this localised cohomology group, then it must be globally generic and locally isomorphic to $\pi$ away from $S_0$ -- this implies $\pi' \cong \pi$ \cite{Soudry87}). We see that
    \[
    \opn{dim}_{\mbb{C}} \opn{H}^2\left( K^G_{\beta}, \kappa_G^*; \opn{cusp} \right)_{\mbb{C}, I}^{G(\mbb{Z}_T)/K_T} = m(\pi) \opn{dim}_{\mbb{C}} \left( \left( \otimes_{\ell \neq p} \pi_{\ell} \right)^{K(N_0, M_0)^p} \otimes \pi_p^{K^G_{\opn{Iw}}(p^{\beta})}[w_G^{\opn{max}} \cdot \Theta_{\pi, p}] \right)
    \]
    where $[\cdots]$ denotes the generalised eigenspace and $m(\pi)$ is the multiplicity of $\pi$ in the discrete spectrum. Since $m(\pi) = 1$ (see \cite{JiangSoudry07}), $\opn{dim}_{\mbb{C}} \pi_p^{K^G_{\opn{Iw}}(p^{\beta})}[w_G^{\opn{max}} \cdot \Theta_{\pi, p}] = \opn{dim}_{\mbb{C}} \pi_p^{K_{p, \opn{B}}}[\Theta_{\pi, p}] = 1$ (see \S \ref{DualPStabilisationsSSec}), and $K(N_0, M_0)$ is the quasi-paramodular group of level $(N_0, M_0)$, we see that this localised cohomology group is one-dimensional.

    We can (and do) assume that $\eta_{\pi, \beta}$ is defined over the field of definition of $\pi$ (which is a number field). Furthermore, if we view $\eta_{\pi, \beta} \in \opn{Hom}_{K_{G, \infty}^+}(\wedge^2 \overline{\ide{u}}_G \otimes V_{M_G}(\kappa_G), \mathscr{A}_0(G)^{K^G_{\beta}})$ via \cite{Su19}, then we note that $\eta_{\pi, \beta}(v) = \phi'_{\beta}$ for some good $p$-stabilisation $\phi$ in $\tilde{\pi}$ for any $v \in \wedge^2 \overline{\ide{u}}_G \otimes V_{M_G}(\kappa_G)$ (if non-zero, which is equivalent to the image of $v$ under the map $\wedge^2 \overline{\ide{u}}_G \otimes V_{M_G}(\kappa_G) \to V_{M_G}(\kappa_G + (-1,-3;2))$ being non-zero).
    \item Suppose that $\xi_2 = 0$ (with notation as in Situation \ref{situation:Weights}). Let 
    \[
    \omega_{\sigma, \beta} \in \opn{H}^0\left( \widehat{K}^{H_2}_{\beta}, \zeta_{H_2} \right)_{\mbb{C}}
    \]
    denote (the restiction to level $\widehat{K}^{H_2}_{\beta}$ of) the cohomology class corresponding to the good $p$-stabilisation in $\tilde{\sigma}$. This cohomology class is defined over the field obtained by adjoining $\mu_{N_2}$ to the field of definition of $\sigma$.\footnote{One can normalise $\omega_{\sigma, \beta}$ further by the Gauss sum of $\chi_2^{-1}$ so that the class is defined over the field of definition of $\sigma$, however we will not consider this class in this article. See \cite[\S 10.5]{LZBK21} for more details.}
    \item Let $\chi$ be a Dirichlet character of $p$-th power conductor, and let $0 \leq j \leq r_1-r_2-t_2$. Let $\Phi^{(p)}$ denote the Schwartz function on $(\mbb{A}_f^p)^{\oplus 2}$ given by $\Phi_S \otimes \bigotimes_{\ell \not\in S} \opn{ch}(\mbb{Z}_\ell^{\oplus 2})$. For $\beta \geq \opn{max}(2\opn{max}(1, c(\chi)), \opn{max}(1, c(\chi))+1)$, we let
    \[
    \omega_{\opn{Eis}, j, \beta} \defeq \mathcal{E}^{\Phi^{(p)}}_{\xi_1}(d - j - \chi, j+\chi) \in \opn{H}^0\left( \widehat{K}^{H_1}_{\beta}, \zeta_{H_1} \right)_{\mbb{C}}^{\opn{nearly}}, \quad d =r_1 - r_2 - t_2,
    \]
    denote the Eisenstein class as in \S \ref{ClassicalSpecOfEisSeriesSSec}. This is defined over a number field.
\end{itemize}

For brevity, set $\mathscr{F} = \mathcal{F}_0\left(\mbb{I}_{\overline{P}_G}(\kappa_G)^{\opn{nearly}} \right)^*$. Consider the diagram:
\[
\opn{H}^2\left( K^G_{\beta}, \kappa_G^*; \opn{cusp} \right)^{\opn{nearly}} \xrightarrow{f_2^{\vee}} \opn{H}^2\left( X_{G}(K^G_{\beta}), [\mathscr{F}](-D_G) \right) \xleftarrow{f_1^{\vee}} \opn{H}^2\left( K^G_{\beta}, \kappa_G^*; \opn{cusp} \right) .
\]
where $f_1$ and $f_2$ are the morphisms in \S \ref{SplttingsOfMGReps}. As explained in \cite[Proposition 6.3]{LPSZ}, these morphisms become isomorphisms after localising at the kernel of $\Theta_{\pi}^{S\cup T \cup \{ p \}}$ (recall that $S$ is the set of primes dividing $N_0N_2$). We let $\tilde{\eta}_{\pi, \beta} \in \opn{H}^2\left( K^G_{\beta}, \kappa_G^*; \opn{cusp} \right)^{\opn{nearly}}$ denote the unique lift of $\eta_{\pi, \beta}$ lying in the (generalised) eigenspace for $\Theta_{\pi}^{S\cup T \cup \{ p \}}$. By our assumptions, one has a natural map $X_G(\widehat{K}^G_{\beta}) \to X_G(K^G_{\beta})$ induced from right-translation by $\gamma_{0, S}$, and we consider the class
\[
\gamma_{0, S} \cdot \tilde{\eta}_{\pi, \beta} \in \opn{H}^2\left( \widehat{K}^G_{\beta}, \kappa_G^*; \opn{cusp} \right)^{\opn{nearly}}
\]
obtained as the pullback of $\tilde{\eta}_{\pi, \beta}$ under this morphism. Finally, we consider the following product of local zeta integrals and $L$-factors:
\[
\mathcal{Z}_S(\Pi \times \Sigma, \gamma_S; \chi_S, s) \defeq \prod_{\ell \in S} \frac{Z(\gamma_{0, \ell} \cdot \varphi_{\ell}, \lambda_{\ell}, \Phi_{\ell}; \chi_{\ell}, s)}{L(\Pi_{\ell} \times \Sigma_\ell, \chi_{\ell}, s)}
\]
for (non-zero) $\varphi_S = \otimes_{\ell \in S} \varphi_{\ell} \in \bigotimes_{\ell \in S} \Pi_{\ell}$, $\lambda_S = \otimes_{\ell \in S} \lambda_{\ell} \in \bigotimes_{\ell \in S} \Sigma_{\ell}$, and $\chi_S \colon \mbb{Q}_S^{\times} \to \mbb{C}^{\times}$ a smooth unitary character. 

\begin{remark} \label{ZetaSgeneratedFractionalRem}
For any $\ell \in S$, one can show that the local zeta integrals $Z(\varphi_{\ell}, \lambda_{\ell}, \Phi_{\ell}; \chi_{\ell}, s)$ as $(\varphi_{\ell}, \lambda_{\ell}, \Phi_{\ell})$ vary form a fractional ideal of $\mbb{C}[\ell^{\pm s}]$ containing the constant functions (see \cite[Theorem 8.9]{LPSZ}). It is known in many cases that this fractional ideal is generated by $L(\Pi_{\ell} \times \Sigma_{\ell}, \chi_{\ell}, s)$ (see \cite{LoefflerLocalZeta}), although the general case is still open. For example, suppose that $\chi_S$ is unramified. Then in any of the following (non-exhaustive list of) cases:
\begin{itemize}
    \item the central character of $\Pi$ is a square;
    \item $\Sigma$ is not supercuspidal at any prime in $S$;
    \item the integers $N_0$ and $N_2$ are coprime;
\end{itemize}
one can always find a finite linear combination of $\mathcal{Z}_S(\Pi \times \Sigma, \gamma_S; \chi_S, s)$ (for different choices of $\gamma_S$) such that this linear combination is equal to the constant function $1$ in the variable $s \in \mbb{C}$.
\end{remark}

The following proposition describes the archimedean period $\mathcal{P}_{\mbb{C}}^{\opn{alg}}$ in terms of the $L$-function associated with $\Pi \times \Sigma$. 

\begin{proposition} \label{PropPAlgCEqualsCompletedLfunction}
    Let $\chi$ be a Dirichlet character of $p$-th power conductor, and let $0 \leq j \leq r_1-r_2-t_2$. Let $\beta \geq \opn{max}(2\opn{max}(1, c(\chi)), \opn{max}(1, c(\chi))+1)$. Then there exists a scalar $\Omega_{\pi} \in \mbb{C}^{\times}$, independent of $\chi$, $j$, $\beta$, and $\gamma_S$, such that
    \[
    \mathcal{P}_{\mbb{C}}^{\opn{alg}}(\gamma_{0, S} \cdot \tilde{\eta}_{\pi, \beta}, \omega_{\opn{Eis}, j, \beta}, \omega_{\sigma, \beta} ) = \mathcal{Z}_{\gamma_S,j, \chi^{-1}} \cdot \mathcal{E}_p(\Pi \times \Sigma, \chi^{-1}, j+(1-d)/2) \cdot \frac{\Lambda^{\{p\}}(\Pi \times \Sigma, \chi^{-1}, j+(1-d)/2)}{\Omega_{\pi}} 
    \]
    where $\mathcal{Z}_{\gamma_S,j, \chi^{-1}} = \mathcal{Z}_S(\Pi \times \Sigma, \gamma_S; \widehat{\chi}_S^{-1}, j+(1-d)/2)$ with $\varphi = \phi \otimes |\!|s(-)|\!|^{(r_1+r_2)/2}$ and $\lambda = \psi \otimes |\!|\opn{det}(-)|\!|^{-(1+t_2)/2}$ for:
    \begin{itemize}
        \item $\phi$ equal to the unique good $p$-stabilisation in $\tilde{\pi}$ such that $\eta_{\pi, \beta}((\alpha_1 \wedge \alpha_2) \otimes \jmath_{M_G}(w_{\kappa_H})) = \phi'_{\beta}$ with notation as in Proposition \ref{PAlgAutPeriodProposition};
        \item $\psi$ equal to the good $p$-stabilisation in $\tilde{\sigma}$.
    \end{itemize}
\end{proposition}
\begin{proof}
    Throughout the course of this proof, we write $\doteq$ to mean equal up to a non-zero element of $\mbb{C}^{\times}$ independent of $\chi$, $j$, $\beta$, and $\gamma_S$. By \cite[Lemma 6.5]{LPSZ}, the hypotheses of Proposition \ref{PAlgAutPeriodProposition} are satisfied. Therefore, we see that
    \begin{align}
    \mathcal{P}_{\mbb{C}}^{\opn{alg}}&(\gamma_{0, S} \cdot \tilde{\eta}_{\pi, \beta}, \omega_{\opn{Eis}, j, \beta}, \omega_{\sigma, \beta} ) \nonumber \\ &= (2 \pi i)^{-2} \opn{Vol}(\mathcal{K}_{\beta}; dh)^{-1} \int_{[H]} (\gamma_{0, S} \hat{\gamma} \cdot \phi'_{\beta})(h) |\!| \opn{det}h_1 |\!|^{\tfrac{1+d}{2}-\xi_1} \widehat{\chi}(\opn{det}h_1)^{-1} E^{\Phi}(h_1; \widehat{\chi}, j - \tfrac{d-1}{2}) \psi(h_2) dh \nonumber \\
    &\doteq p^{4 \beta} \int_{Z_G(\mbb{A}) H(\mbb{Q}) \backslash H(\mbb{A})} (\gamma_{0, S} \hat{\gamma} \cdot \varphi_{\beta, \chi}')(h) E^{\Phi}(h_1; \widehat{\chi}, j - \tfrac{d-1}{2}) \lambda(h_2) dh \; =: \; A \nonumber
    \end{align}
    where $\varphi'_{\beta, \chi} \defeq \phi'_{\beta} \otimes (|\!|-|\!|^{(r_1+r_2)/2} \widehat{\chi}^{-1} \circ s) \in \Pi \otimes \widehat{\chi}^{-1}$, the Schwartz function is $\Phi = \Phi^{(d+1)}_{\infty}\Phi^{(p)} \Phi_{p, \chi^{-1}, \chi}$ (with notation as in \cite[Definition 7.5, \S 8.5.1]{LPSZ}), and we have used the fact that the argument is invariant under the action of $Z_G(\mbb{A})$. Applying \cite[\S 8.5, Proposition 9.3(ii)]{LPSZ}, we see that
    \[
    A \doteq \Lambda^{\{p \}}(\Pi \times \Sigma, \chi^{-1}, j+(1-d)/2) \cdot \mathcal{Z}_S(\Pi \times \Sigma, \gamma_S; \widehat{\chi}_S^{-1}, j+(1-d)/2) \cdot \left( p^{4 \beta} Z(\hat{\gamma} \cdot \varphi'_{\beta,p}, \lambda_p, \Phi_{p, \chi^{-1}, \chi}; \widehat{\chi}^{-1}_p, j+(1-d)/2) \right) .
    \]
    The proposition now follows from Corollary \ref{FinalZetapFormulaCorollary} (note that $\widehat{\chi}^{-1}_p|_{\mbb{Z}_p^{\times}} = \chi$).
\end{proof}

\subsubsection{}  \label{CaseBLfunctionAndPeriodSSec}

Suppose that we are in Case B (\S \ref{CaseBSubSubSection}) and Situation \ref{situation:Weights}. Let $d_2 = t_2$ and let $0 \leq d_1 \leq t_1$ be any choice of integer. Let $\tilde{\pi} = (\pi, \Theta_{\pi, p})$ be a good $p$-stabilised automorphic representation of $G(\mbb{A})$ of weight $\nu_G$, level $(N_0, M_0)$ and character $\chi_0$. For $i=1, 2$, let $\tilde{\sigma}_i = (\sigma_i, \alpha_i)$ be good $p$-stabilised automorphic representations of $\opn{GL}_2(\mbb{A})$ of weight $1+d_i$, level $N_i$, and character $\chi_i$. Let $\Pi = \pi \otimes |\!|s(-)|\!|^{(r_1+r_2)/2}$ and $\Sigma_i = \sigma_i \otimes |\!| \opn{det}(-) |\!|^{-(1+d_i)/2}$ denote the unitary normalisations of $\pi$ and $\sigma_i$ respectively. 

\begin{remark}
    Note that the assumption $\chi_0 \cdot \chi_1 \cdot \chi_2 = 1$ in \S \ref{CaseBSubSubSection} implies that $d_1+d_2 \equiv r_1-r_2$ modulo $2$ (and hence $d_1 \equiv t_1$ modulo $2$).
\end{remark}

In this setting, we are interested in the ``triple-product'' $L$-function $L(\Pi \times \Sigma_1 \times \Sigma_2, s)$ associated with the automorphic representation $\Pi \boxtimes \Sigma_1 \boxtimes \Sigma_2$ of $\opn{GSp}_4(\mbb{A}) \times \opn{GL}_2(\mbb{A}) \times \opn{GL}_2(\mbb{A})$ with respect to the $16$-dimensional representation of the dual group given by the tensor product of the $4$-dimensional spin representation and the $2$-dimensional standard representations for each $\opn{GL}_2$-factor. The central critical value of this $L$-function is at $s=1/2$. Under certain hypotheses on these automorphic representations, the Gan--Gross--Prasad conjecture for general spin groups \cite{emory2020global} predicts a relation between the central critical $L$-value $L(\Pi \times \Sigma_1 \times \Sigma_2, 1/2)$ and the automorphic period
\[
\mathscr{P}(\varphi, \lambda_1, \lambda_2) \defeq \int_{H(\mbb{Q})Z_G(\mbb{A}) \backslash H(\mbb{A})} \varphi(h) \lambda_1(h_1) \lambda_2(h_2) dh
\]
for $\varphi \in \Pi$ and $\lambda_i \in \Sigma_i$. This conjecture is known when $\Pi$ is of Yoshida or twisted Yoshida type, but is open in general. Nevertheless, it is natural to study the $p$-adic interpolation of the automorphic periods $\mathscr{P}$ in place of the central critical $L$-values.

In fact, we are interested in these automorphic periods for specific choices of input data. More precisely, let $\varphi^{\opn{sph}} = \otimes_{v} \varphi^{\opn{sph}}_v \in \Pi = \bigotimes'_v \Pi_v$, with:
\begin{itemize}
    \item $\varphi_{\infty}^{\opn{sph}} \in \Pi_{\infty}$ is the (unique up to scalar) vector obtained as the image of $(\alpha_1 \wedge \alpha_2) \odot \jmath_{M_G}(w_{\kappa_H})$ under the (unique up to scalar) $K_{G, \infty}^{\circ}$-equivariant homomorphism $V_{M_G}(\kappa_G - 2\rho_{G, \opn{nc}}) \hookrightarrow \Pi_{\infty}$. Here $K_{G, \infty}^{\circ} \subset K_{G, \infty}$ is the maximal compact subgroup, $\odot$ denotes the Cartan product, and $\rho_{G, \opn{nc}}$ is the half-sum of positive non-compact roots.
    \item For $\ell$ finite, $\varphi_{\ell}^{\opn{sph}} \in \Pi_{\ell}$ is the Whittaker new vector. In particular $\varphi_{p}^{\opn{sph}}$ is the spherical vector.
\end{itemize}
We also let $\lambda_i^{\opn{sph}} \in \Sigma_i$ denote the automorphic form associated with the holomorphic new eigenform associated with $\Sigma_i$ (so $\lambda_i^{\opn{sph}}$ lies in the minimal $K_{\opn{GL}_2, \infty}^+$-type at $\infty$ and is the Whittaker new vector at finite primes. Let $\delta$ denote the operator on (the $K_{\opn{GL}_2, \infty}^+$-finite vectors in) $\Sigma_{1, \infty}$ given by the action of $\left( \begin{smallmatrix} 0 & 1 \\ 0 & 0 \end{smallmatrix} \right) \in \mathfrak{gl}_{2, \mbb{C}}$ (the Maass--Shimura differential operator).

We have the following result:

\begin{lemma}
    Suppose that the local signs of $L(\Pi \times \Sigma_1 \times \Sigma_2, s)$ at all finite places are equal to $+1$, and suppose that the local and global Gan--Gross--Prasad conjectures for general spin groups hold (as described in \cite[\S 2]{LZpadiclfunctionsdiagonalcycles}). Then $L(\Pi \times \Sigma_1 \times \Sigma_2, 1/2) \neq 0$ if and only if there exists a choice of $\gamma_S = (\gamma_{0, S}, \gamma_{1, S}) \in G(\mbb{Q}_S) \times \opn{GL}_2(\mbb{Q}_S)$ such that $\mathscr{P}(\gamma_{0,S} \cdot \varphi^{\opn{sph}}, \gamma_{1, S} \cdot \delta^{(t_1-d_1)/2}\lambda_1^{\opn{sph}}, \lambda_2^{\opn{sph}}) \neq 0$. 
\end{lemma}
\begin{proof}
    Note that the conditions in Situation \ref{situation:Weights} (which place us in region (f) of \cite[Figure 2]{LZpadiclfunctionsdiagonalcycles}) imply that the local sign at $\infty$ is $+1$. Fix bases $\ide{z}_v \in \opn{Hom}_{H(\mbb{Q}_v)}(\Pi_v \boxtimes \Sigma_{1, v} \boxtimes \Sigma_{2, v}, \mbb{C})$ (which exist by the local Gan--Gross--Prasad conjecture). The ``if direction'' of the lemma follows immediately from the global Gan--Gross--Prasad conjecture, so from now on we will assume that $L(\Pi \times \Sigma_1 \times \Sigma_2, 1/2) \neq 0$. In particular, this implies that the automorphic period $\mathscr{P}(\cdots)$ is not identically zero on $\Pi \boxtimes \Sigma_1 \boxtimes \Sigma_2$. 
    
    This implies that there exist decomposable $\psi \in \Pi$, $\mu_1 \in \Sigma_1$, $\mu_2 \in \Sigma_2$ such that $\mathscr{P}(\psi, \mu_1, \mu_2) \neq 0$. In particular, we must have $\ide{z}_v(\psi_v, \mu_{1, v}, \mu_{2, v}) \neq 0$. By local multiplicity one, we have
    \[
    \mathscr{P}(\psi', \mu_{1}', \mu'_2) = \prod_{v}\left( \frac{\ide{z}_v(\psi'_v, \mu_{1, v}', \mu_{2, v}')}{\ide{z}_v(\psi_v, \mu_{1, v}, \mu_{2, v})} \right) \cdot \mathscr{P}(\psi, \mu_{1}, \mu_2)
    \]
    for any decomposable $\psi' \in \Pi$, $\mu_{i}' \in \Sigma_i$ (note that almost all terms in the product are $1$). We note that:
    \begin{itemize}
        \item $\ide{z}_{\infty}(\varphi^{\opn{sph}}_{\infty}, \delta^{(t_1-d_1)/2}\lambda_{1, \infty}^{\opn{sph}}, \lambda_{2, \infty}^{\opn{sph}}) \neq 0$ by the discussion in \cite[\S 8.5]{LPSZ};
        \item $\xi_{\ell}(\varphi^{\opn{sph}}_{\ell}, \lambda_{1, \ell}^{\opn{sph}}, \lambda_{2, \ell}^{\opn{sph}}) \neq 0$ for $\ell \not\in S \cup\{ \infty \}$ by \cite[Theorem 2.3]{LoefflerUnramifiedZeta};
        \item For any $\ell \in S$, there exists $\gamma_{\ell} = (\gamma_{0, \ell}, \gamma_{1, \ell}) \in G(\mbb{Q}_{\ell}) \times \opn{GL}_2(\mbb{Q}_{\ell})$ such that $\ide{z}_{\ell}(\gamma_{0,\ell} \cdot \varphi_{\ell}^{\opn{sph}}, \gamma_{1,\ell} \cdot \lambda_{1, \ell}^{\opn{sph}}, \lambda_{2, \ell}^{\opn{sph}}) \neq 0$. Indeed, this follows from the $H(\mbb{Q}_\ell)$-invariance of the linear functional.
    \end{itemize}
    Putting this all together proves the lemma.
\end{proof}

We now discuss the relation with the coherent cohomology pairing $\mathcal{P}_{\mbb{C}}^{\opn{alg}}$. Let $\eta_{\pi, \beta}$, $\tilde{\eta}_{\pi, \beta}$, and $\gamma_{0,S} \cdot \tilde{\eta}_{\pi, \beta}$ be defined in exactly the same way as in \S \ref{CaseALfunctionAndPeriodSSSec}. For $i=1, 2$, let 
\[
\omega_{\sigma_i, \beta} \in \opn{H}^0\left( \widehat{K}^{H_i}_{\beta}, (-1-d_i;0) \right)_{\mbb{C}}
\]
be as in \S \ref{CaseALfunctionAndPeriodSSSec}. For $\beta \geq 2$, let $\omega_{\sigma_i, \beta}^{[p]}$ denote the cohomology class associated with the $p$-depletion of the good $p$-stabilisation in $\tilde{\sigma}_i$ (which is a classical modular form of level $\Gamma_1(N_1) \cap \Gamma_0(p^{\beta})$ for any $\beta \geq 2$). We set $\xi_2 = 0$ (so $\zeta_{H_2} = (-1-d_2;0) = (-1-t_2; 0)$), and for $\beta \geq 2$ let 
\[
\omega^{[p],\opn{nearly}}_{\sigma_1, \beta} \defeq \opn{det}^{r_2-1+(t_1-d_1)/2}\nabla^{(t_1-d_1)/2}\omega_{\sigma_1, \beta}^{[p]} \in \opn{H}^0\left( \widehat{K}^{H_1}_{\beta}, \zeta_{H_1} \right)_{\mbb{C}}^{\opn{nearly}} 
\]
where the class $\opn{det}$ is defined in Definition \ref{DefOfDetTrivClass}. Here $\nabla$ denotes the action of the Gauss--Manin connection on nearly holomorphic forms as in \cite{DiffOps}. Note that $\gamma_{1, S} \cdot \omega_{\sigma_1, \beta}^{[p], \opn{nearly}}$ is also a cohomology class of level $\widehat{K}^{H_i}_{\beta}$ (by the definition of tame test data in \S \ref{CaseBSubSubSection}).

\begin{proposition} \label{TripleProdPalgWithSpherPeriodProp}
    With notation as above, let $\beta \geq 2$. Then there exists a scalar $\Omega_{\pi} \in \mbb{C}^{\times}$, independent of $\beta$, $\gamma_S$, and $\sigma_1 \boxtimes \sigma_2$, such that
    \[
    \mathcal{P}^{\opn{alg}}_{\mbb{C}}(\gamma_{0,S} \cdot \tilde{\eta}_{\pi, \beta}, \gamma_{1, S} \cdot \omega^{[p],\opn{nearly}}_{\sigma_1, \beta}, \omega^{[p]}_{\sigma_2, \beta}) = \mathcal{E}_p(\Pi \times \Sigma_1 \times \Sigma_2, 1/2) \frac{\mathscr{P}(\gamma_{0,S} \cdot \varphi^{\opn{sph}}, \gamma_{1, S} \cdot \delta^{(t_1-d_2)/2}\lambda_1^{\opn{sph}}, \lambda_2^{\opn{sph}})}{\Omega_{\pi}}
    \]
    where $\mathcal{E}_p(\cdots)$ is the $p$-adic multiplier:
    \[
    \mathcal{E}_p(\Pi \times \Sigma_1 \times \Sigma_2, s) = \prod_{i, j, k=1}^2 \left( 1 - \frac{p^{-s}}{\theta_i \gamma_j^{(1)}\gamma_k^{(2)}} \right)
    \]
    where $\{ \theta_1, \dots, \theta_4 \}$ are the Satake parameters of $\Pi_p$ ordered so that $\Theta_{\Pi, p}(U_{p, \opn{Kl}}) = p^2 \theta_1 \theta_2$, and $\{ \gamma_1^{(a)}, \gamma_2^{(a)} \}$ denote the Satake parameters of $\Sigma_{a,p}$ for $a=1, 2$.
\end{proposition}
\begin{proof}
    Write $\doteq$ to mean equal up to a non-zero element of $\mbb{C}^{\times}$ independent of $\beta$, $\gamma_S$ and $\sigma_1 \boxtimes \sigma_2$. Then, by Proposition \ref{PalgEqualsAutPeriodEqn} (and \cite[Lemma 6.5]{LPSZ}), we see that
    \[
    \mathcal{P}^{\opn{alg}}_{\mbb{C}}(\gamma_{0,S} \cdot \tilde{\eta}_{\pi, \beta}, \gamma_{1, S} \cdot \omega^{[p],\opn{nearly}}_{\sigma_1, \beta}, \omega^{[p]}_{\sigma_2, \beta}) \doteq p^{4\beta} \mathscr{P}(\gamma_{0,S}\hat{\gamma} \cdot \varphi_{\beta}', \gamma_{1, S} \cdot \delta^{(t_1-d_1)/2}\lambda_{1}^{[p]}, \lambda_{2}^{[p]}) 
    \]
    where $\lambda_i^{[p]}$ denotes the $p$-depletion of $\lambda_i^{\opn{sph}}$ (or equivalently the $p$-depletion of the Iwahori level automorphic form in $\Sigma_i$ corresponding to the good $p$-stabilisation in $\tilde{\sigma}_i$), and $\varphi'_{\beta} = \phi_{\beta}' \otimes |\!| s(-) |\!|^{(r_1+r_2)/2} \in \Pi$ with $\phi_{\beta}'$ as in Proposition \ref{PropPAlgCEqualsCompletedLfunction}. Note that, by \cite[Theorem 2.3]{LoefflerUnramifiedZeta}, there exists a unique $\ide{z}_p \in \opn{Hom}_{H(\mbb{Q}_p)}(\Pi_p \boxtimes \Sigma_{1, p} \boxtimes \Sigma_{2, p}, \mbb{C})$ which satisfies $\ide{z}_p(\varphi^{\opn{sph}}_{p}, \lambda_{1, p}^{\opn{sph}}, \lambda_{2, p}) = 1$. It therefore suffices to show that:
    \[
    \ide{z}_p(\hat{\gamma} \cdot \varphi_{\beta,p}', \lambda_{1, p}^{[p]}, \lambda_{2, p}^{[p]}) \doteq p^{-4\beta} \mathcal{E}_p(\Pi \times \Sigma_1 \times \Sigma_2, 1/2).
    \]
    But this follows from the same strategy as in Appendix \ref{LocalZetaSSecAppendix} using the fact that $\ide{z}_p(\gamma \cdot \varphi_{\beta,p}', \lambda_{1, p}^{[p]}, \lambda_{2, p}^{[p]}) = 0$ (\cite[Proposition 5.8]{LoefflerUnramifiedZeta} and the fact that the Whittaker function associated with $\varphi_{\beta, p}'$ vanishes at the identity for $\beta \geq 2$), and the invariance of $\lambda_{2, p}^{[p]}$ under $B_{\opn{GL}_2}(\mbb{Z}_p)$. The Klingen level statement is simply \cite[Proposition 5.8]{LoefflerUnramifiedZeta} (note the differences in the normalisations of the Satake parameters and the trace compatible system of eigenvectors).
\end{proof}

\section{Families of automorphic forms and \texorpdfstring{$p$}{p}-adic \texorpdfstring{$L$}{L}-functions}

In this section, we construct the $p$-adic $L$-functions which appear in Theorem \ref{ThmAIntro} and Theorem \ref{ThmCCaseBIntro}. 

\subsection{Families of Eisenstein series}

We briefly recall families of Eisenstein series following \cite[\S 7]{LPSZ} and \cite[\S 10]{LZBK21}. As we are keeping track of the action of the centre for automorphic vector bundles of $\opn{GL}_2$, we introduce the following class:

\begin{definition} \label{DefOfDetTrivClass}
    Set $K = K^p K^{\opn{GL}_2}_{\opn{Iw}}(p^{\beta})$. Let $M_{\opn{GL}_2, \opn{dR}} \defeq P_{\opn{GL}_2, \opn{dR}} \times^{\overline{P}_{\opn{GL}_2}} T_{\opn{GL}_2}$ denote the pushout along the projection to the torus: for any $S \to X_{\opn{GL}_2}(K)$, $M_{\opn{GL}_2, \opn{dR}}(S)$ is the set of all trivialisations $\psi_1 \colon \mathcal{O}_S \xrightarrow{\sim} \opn{Lie}(E)$, $\psi_2 \colon \mathcal{O}_S \to \omega_{E^D}$, where $E$ denotes the generalised elliptic curve corresponding to $S \to X_{\opn{GL}_2}(K)$. We let 
    \[
    \opn{det} \colon M_{\opn{GL}_2, \opn{dR}} \to \mbb{A}^1
    \]
    denote the global section sending $(\psi_1, \psi_2)$ to $\Psi(1) \in \mathcal{O}_S(S)$, where $\Psi \colon \mathcal{O}_S \xrightarrow{\sim} \mathcal{O}_S$ is the isomorphism given by the composition:
    \[
        \mathcal{O}_S \xrightarrow{\psi_1 \otimes \psi_2} \opn{Lie}(E) \otimes \omega_{E^D} \xrightarrow{\lambda \otimes 1} \opn{Lie}(E^D) \otimes \omega_{E^D} \xrightarrow{\sim} \mathcal{O}_S
    \]
    where $\lambda \colon \opn{Lie}(E) \xrightarrow{\sim} \opn{Lie}(E^D)$ is induced from the polarisation, and the last map is the natural one. Then $\opn{det} \in \opn{Fil}_0\mathscr{N}_{\opn{GL}_2, \beta, (0;-1)}^{\opn{nhol}}$ is a holomorphic modular form of weight $(0; -1)$ and level $K$, which is invertible as a section of $M_{\opn{GL}_2, \opn{dR}}$. The restriction of $\opn{det}$ along the natural map $\mathcal{IG}_{\opn{GL}_2}(K^pU(J_p))^{\opn{tor}} \to M_{\opn{GL}_2, \opn{dR}}$ is valued in $\mbb{Z}_p^{\times}$, hence for any locally analytic character $\chi \colon \mbb{Z}_p^{\times} \to A^{\times}$, we can define $\opn{det}^{\chi} \defeq \chi \circ \opn{det}$, which is a $p$-adic modular form of weight $(0; -\chi)$. One can easily show that $\opn{det}^{\chi}$ is overconvergent, i.e., $\opn{det}^{\chi} \in \mathscr{M}^{\dagger}_{\opn{GL}_2, \beta, (0; -\chi)}$.
\end{definition}

We now define the families of Eisenstein series.

\begin{definition}
    Let $L/\mbb{Q}_p$ be a finite extension and $(A, A^+)$ a complete Tate affinoid pair over $(L, \mathcal{O}_L)$. Let $\Phi^{(p)}$ be a Schwartz function on $(\mbb{A}_f^p)^2$ with values in a number field contained in $L$ (via our fixed isomorphism $\mbb{C} \cong \Qpb$), and let $\chi^{(p)}$ be a Dirichlet character of prime-to-$p$ conductor such that one has 
    \[
    \left( \begin{smallmatrix} a & 0 \\ 0 & a \end{smallmatrix} \right) \cdot \Phi^{(p)} = \widehat{\chi}^{(p)}(a)^{-1}\Phi^{(p)}
    \]
    for all $a \in (\widehat{\mbb{Z}}^{(p)})^{\times}$. Let $\kappa_1, \kappa_2, \xi \colon \mbb{Z}_p^{\times} \to A^{\times}$ be locally analytic characters. We let 
    \[
    \mathcal{E}_{\xi}^{\Phi^{(p)}}(\kappa_1, \kappa_2) = \mathcal{E}_{\xi}^{\Phi^{(p)}}(\kappa_1, \kappa_2; \chi^{(p)}) \defeq \opn{det}^{\kappa_1 + \kappa_2 - \xi} \cdot \mathcal{E}^{\Phi^{(p)}}(\kappa_1, \kappa_2; \chi^{(p)})
    \]
    where $\mathcal{E}^{\Phi^{(p)}}(\kappa_1, \kappa_2; \chi^{(p)})$ denotes the $p$-adic family of Eisenstein series as in \cite[Theorem 7.6]{LPSZ}.
\end{definition}

Note that by \cite[Lemma 2.4.1]{DiffOps}, $\mathscr{N}^{\dagger}_{\opn{GL}_2, \beta}$ is independent of $\beta$, so we will henceforth drop it from the notation.

\begin{lemma} \label{TheEisSeriesIsNearlyOCLemma}
    Let $\zeta = (-\kappa_1-\kappa_2-1; \xi)$ and suppose that $K^p$ fixes $\Phi^{(p)}$. Then one has $\mathcal{E}_{\xi}^{\Phi^{(p)}}(\kappa_1, \kappa_2) \in \mathscr{N}^{\dagger}_{\opn{GL}_2, \zeta}$, i.e., the family of Eisenstein series is a weight $\zeta$ nearly overconvergent modular form over $A$.
\end{lemma}
\begin{proof}
    The main result of \cite{DiffOps} is that $\mathscr{N}^{\dagger}_{\opn{GL}_2} \hatot A$ comes equipped with an $A$-algebra action $\star$ of locally analytic functions $\opn{C}^{\opn{la}}(\mbb{Z}_p, A)$ such that: the structural map $\mbb{Z}_p \to A$ acts through the Maass--Shimura differential operator $\delta$ (which on $q$-expansions is given by $q \frac{d}{dq}$), and the indicator function $1_{\mbb{Z}_p^{\times}}$ of $\mbb{Z}_p^{\times}$ acts as $p$-depletion. Let $\phi \colon \mbb{Z}_p \to A$ be the locally analytic function which is supported on $\mbb{Z}_p^{\times}$ and satisfies $\phi(x) = \kappa_1(x)$ for all $x \in \mbb{Z}_p^{\times}$.

    It is already known (\cite[Proposition 10.1.2]{LZBK21}) that $\mathcal{E}_{\xi-\kappa_1}^{\Phi^{(p)}}(0, \kappa_2 - \kappa_1)$ is overconvergent of weight $(-1-\kappa_2 + \kappa_1; \xi-\kappa_1)$, and one can easily check that the weight $(-\kappa_1-\kappa_2-1; \xi)$ nearly overconvergent form $\phi \star \mathcal{E}_{\xi-\kappa_1}^{\Phi^{(p)}}(0, \kappa_2 - \kappa_1)$ has the same $q$-expansion as $\mathcal{E}_{\xi}^{\Phi^{(p)}}(\kappa_1, \kappa_2)$. This proves the claim.
\end{proof}

\subsubsection{Classical specialisations} \label{ClassicalSpecOfEisSeriesSSec}
    Let $A = L$ and let $\rho$ be a Dirichlet character of conductor $p^{\beta'}$ for some $\beta' \geq 0$. Suppose that $\kappa_1 = a - \rho$, $\kappa_2 = b + \rho$, $\xi = c$ are locally algebraic characters of $\mbb{G}_m$, with $a, b \geq 0$ and $c$ integers. Let $\beta \geq \opn{max}(2\opn{max}(1, \beta'), \opn{max}(1, \beta')+1)$ and let $K^p \subset \opn{GL}_2(\mbb{A}_f^p)$ be a neat compact open subgroup which fixes $\Phi^{(p)}$. Then one has
    \[
    \mathcal{E}^{\Phi^{(p)}}_{\xi}(\kappa_1, \kappa_2) \in \opn{Fil}_{a} \mathscr{N}^{\opn{nhol}}_{\opn{GL}_2, \beta, \zeta}, \quad \quad \zeta = (-a-b-1; c),
    \]
    where $\mathscr{N}^{\opn{nhol}}_{\opn{GL}_2, \beta} = \opn{H}^0\left( \mathcal{X}_{\opn{GL}_2}(K^p K^{\opn{GL}_2}_{\opn{Iw}}(p^{\beta}))_L, \pi_* \mathcal{O}_{P^{\opn{an}}_{\opn{GL}_2, \opn{dR}}} \right)$ is the space of nearly holomorphic forms of level $K \defeq K^p K^{\opn{GL}_2}_{\opn{Iw}}(p^{\beta})$. Indeed, this follows from \cite[Proposition 7.3]{LPSZ} and the fact that $\Phi_{p, \rho^{-1}, \rho}$ (Definition 7.5 in \emph{op.cit.}) is an eigenvector under the action of $K_{\opn{Iw}}^{\opn{GL}_2}(p^{\beta})$ with eigencharacter $\widehat{\rho}(\opn{det}(-))$.

    Moreover, via the identification $\mbb{C} \cong \Qpb$, rigid GAGA, and the description of the coherent cohomology of modular curves in \S \ref{TheArchimideanPeriodSSec}, one can view
    \begin{equation} \label{NholEisLieAlgEqn}
    \opn{Fil}_a\mathscr{N}^{\opn{nhol}}_{\opn{GL}_2, \beta, \zeta} \subset \opn{Hom}_{(\overline{\mathfrak{p}}_{\opn{GL}_2}, K^+_{\opn{GL}_2,\infty})}\left( \mbb{D}_{\overline{P}_{\opn{GL}_2}}(\zeta^*)^{\opn{nearly}}, \mathscr{A}(\opn{GL}_2)^K \right) .
    \end{equation}
    Let $v_{\zeta}^* \in \mbb{D}_{\overline{P}_{\opn{GL}_2}}(\zeta^*)^{\opn{nearly}}$ denote the element in Notation \ref{NotationForSplittingsOfReps}. Then, if $\mathcal{E}^{\Phi^{(p)}}_{\xi}(\kappa_1, \kappa_2)$ corresponds to the morphism $\mathcal{E} \colon \mbb{D}_{\overline{P}_{\opn{GL}_2}}(\zeta^*)^{\opn{nearly}} \to \mathscr{A}(\opn{GL}_2)^K$ via the inclusion in (\ref{NholEisLieAlgEqn}), the automorphic form $\mathcal{E}(v_{\zeta}^*)$ is given by
    \[
    \mathcal{E}(v_{\zeta}^*)(g) = |\!|\opn{det}g|\!|^{\tfrac{1+a+b}{2}-c} \widehat{\rho}(\opn{det}g)^{-1} E^{\Phi}(g; \widehat{\chi}, \tfrac{b-a+1}{2}), \quad \quad g \in \opn{GL}_2(\mbb{A}),
    \]
    where $\chi \defeq \chi^{(p)}\rho^{-2}$ and $E^{\Phi}(-; \widehat{\chi}, s)$ is the Eisenstein series in \cite[\S 9.2]{LPSZ} with respect to the Schwartz function $\Phi = \Phi^{(a+b+1)}_{\infty}\Phi^{(p)} \Phi_{p, \rho^{-1}, \rho}$ (with notation as in \cite[Definition 7.5, \S 8.5.1]{LPSZ}).

\begin{remark}
    As explained in \cite[Theorem 7.6]{LPSZ}, the family of Eisenstein series $\mathcal{E}_{\xi}^{\Phi^{(p)}}(\kappa_1, \kappa_2)$ is zero on the component of $\opn{Spec}A$ where $\kappa_1(-1)\kappa_2(-1) \neq -\chi^{(p)}(-1)$.
\end{remark}

\subsection{Coleman families} \label{ColemanFamiliesForBothSection}

In this section, we introduce the Coleman families for $G$ and $\opn{GL}_2$ that will provide the input into the $p$-adic $L$-functions. Our definition differs slightly from the literature since we do not want to fix a ``base-point'' for the family, however we explain in Examples \ref{ExampleOfCFForGL2} and \ref{ExampleOfCFForGSP4} that any good $p$-stabilisation has a Coleman family passing through it (in our sense). 

\subsubsection{Coleman families for \texorpdfstring{$\opn{GL}_2$}{GL(2)}} \label{CFFORGl2SSSEC}

Let $\mathcal{W}(\mbb{Z}_p^{\times})$ denote the adic space over $\mbb{Q}_p$ whose $S$-points parameterise continuous characters $\mbb{Z}_p^{\times} \to \mathcal{O}(S)^{\times}$. Let $N \geq 1$ be an integer prime to $p$, and $L/\mbb{Q}_p$ a finite extension containing $\mu_N$. Let $\mathscr{E}_{\opn{GL}_2} = \mathscr{E}_{\opn{GL}_2}(N)$ denote the Coleman--Mazur--Buzzard cuspidal eigencurve over $\opn{Spa}(L, \mathcal{O}_L)$ of tame level $\Gamma_1(N)$. We normalise the weight map $w \colon \mathscr{E}_{\opn{GL}_2} \to \mathcal{W}(\mbb{Z}_p^{\times})$ such that the weight of a point $x \in \mathscr{E}_{\opn{GL}_2}$ is given by $w(x) + 1$.

We consider the following special locus of points in this eigencurve.

\begin{notation}
    Let $\chi$ be a Dirichlet character of conductor dividing $N$. We let $\mathscr{E}_{\opn{GL}_2}^{\opn{good}, \chi} \subset \mathscr{E}_{\opn{GL}_2}(\mbb{C}_p)$ denote the subset of points corresponding to good $p$-stabilised automorphic representations $\tilde{\sigma}$ of $\opn{GL}_2(\mbb{A})$ of weight $1+t$, level $N$, and character $\chi$, for some $t \geq 0$. We note that the weight map $w$ is \'{e}tale at any point in $\mathscr{E}_{\opn{GL}_2}^{\opn{good}, \chi}$ (see \cite{bellpadic, BDetalewt1}).
\end{notation}

We now introduce the definition of Coleman families for $\opn{GL}_2$.

\begin{definition}
    Let $U = \opn{Spa}(\mathscr{O}_U, \mathscr{O}_U^+)$ be a reduced affinoid adic space over $\opn{Spa}(L, \mathcal{O}_L)$. A Coleman family $\underline{\sigma}$ over $U$, of tame level $N$ and character $\chi$, is a morphism $f \colon U \to \mathscr{E}_{\opn{GL}_2}$ over $\opn{Spa}(L, \mathcal{O}_L)$ such that $\Upsilon(\underline{\sigma}) \defeq f^{-1}(\mathscr{E}_{\opn{GL}_2}^{\opn{good}, \chi})$ is Zariski dense in $U$.
\end{definition}

\begin{example} \label{ExampleOfCFForGL2}
    Suppose that $\tilde{\sigma}_t$ is a good $p$-stabilised automorphic representation of $\opn{GL}_2(\mbb{A})$ of weight $1+t$, level $N$, and character $\chi$, for some $t \geq 0$. After possibly enlarging $L$, let $x_t \in \mathscr{E}_{\opn{GL}_2}(L)$ denote the corresponding point on the eigencurve. Then there exists a sufficiently small open neighbourhood $U \subset \mathcal{W}(\mbb{Z}_p^{\times})$ of $w(x_t) = t$ and a Coleman family $\underline{\sigma}$ over $U$ (of tame level $N$ and character $\chi$) passing through $x_t$, i.e., one has $f(t) = x_t$, where $f \colon U \to \mathscr{E}_{\opn{GL}_2}$ denotes the Coleman family. Indeed, this follows from the fact that $w$ is \'{e}tale at the point $x_t$.
\end{example}

To any Coleman family $\underline{\sigma}$ as above, one can associate a family of cohomology classes. More precisely, let $T$ be a set of rational primes which do not divide $Np$, and let $K_T \subset \opn{GL}_2(\mbb{Z}_T)$ be a normal compact open subgroup. For $\ell^{n_{\ell}} \mid\mid N$ with $n_{\ell} \geq 1$, we let $K_{1}(\ell^{n_{\ell}}) \subset \opn{GL}_2(\mbb{Z}_{\ell})$ denote the subgroup of matrices $\left( \begin{smallmatrix} a & b \\ c & d \end{smallmatrix} \right)$ satisfying $c \equiv 0$ and $d \equiv 1$ modulo $\ell^{n_{\ell}}\mbb{Z}_{\ell}$. We set
\[
K^p \defeq K_T \cdot \prod_{\substack{v \nmid Np \\ v \notin T}} \opn{GL}_2(\mbb{Z}_v) \cdot \prod_{\ell^{n_{\ell}} \mid\mid N} K_{1}(\ell^{n_{\ell}}) \quad \subset \quad \opn{GL}_2(\mbb{A}_f^p)
\]
and assume that $K_T$ is chosen in such a way that $K^p$ is a neat compact open subgroup. Set $K^{\opn{GL}_2}_{\beta} = K^p K^{\opn{GL}_2}_{\opn{Iw}}(p^{\beta})$. Then, associated with $f \colon U \to \mathscr{E}_{\opn{GL}_2}$, there exists a cohomology class
\[
\omega_{\underline{\sigma}} \in \opn{H}^0\left( K^{\opn{GL}_2}_{\beta}, \zeta_U \right)^{(\dagger, n\opn{-an}, +),\opn{GL}_2(\mbb{Z}_T)/K_T} \quad \quad (n \gg 0)
\]
which is an eigenvector for the action of (normalised) Hecke operators at $p$, the Hecke operators $U_{\ell}$ for $\ell | N$, and the Hecke operators $T_{\ell}$ for $\ell \nmid Np$ with eigencharacter determined by the morphism $f$. Here $\zeta_U = (-1-\kappa_U; 0)$ where $\kappa_U$ denotes the pullback of the universal character under the composition $w \circ f \colon U \to \mathcal{W}(\mbb{Z}_p^{\times})$. This cohomology class satisfies the following specialisation property: for any $x \in \Upsilon(\underline{\sigma})$, the specialisation of $\omega_{\underline{\sigma}}$ at $x$ is the restriction of a classical cohomology class 
\[
\omega_{\underline{\sigma}_x} \in \opn{H}^0\left( K^{\opn{GL}_2}_{\beta}, \zeta_x \right)^{\opn{GL}_2(\mbb{Z}_T)/K_T} = \opn{H}^0\left( \mathcal{X}_{\opn{GL}_2}(K^{\opn{GL}_2}_{\beta}), [V_{T_{\opn{GL}_2}}(\zeta_x)] \right)^{\opn{GL}_2(\mbb{Z}_T)/K_T}, \quad \zeta_x = (-1-w(x); 0),
\]
which we can view as a morphism 
\[
\omega_{\underline{\sigma}_x} \in \opn{Hom}_{(\overline{\ide{p}}_{\opn{GL}_2}, K^+_{\opn{GL}_2, \infty})}\left( V_{T_{\opn{GL}_2}}(\zeta_x)^*, \mathscr{A}(\opn{GL}_2)^{K^{\opn{GL}_2}_{\beta}} \right) 
\]
via rigid GAGA, the identification in Proposition \ref{HarrisSuProposition}, and the identification $\mbb{C} \cong \Qpb$. Moreover, the automorphic form $\omega_{\underline{\sigma}_x}(1)$ is equal to the good $p$-stabilisation in the good $p$-stabilised automorphic representation $(\underline{\sigma}_x, \alpha_x)$ corresponding to the point $x$.

\begin{remark}
    In the definition of a Coleman family, one often assumes that $f \colon U \to \mathscr{E}_{\opn{GL}_2}$ is an open immersion; for maximal flexibility, we allow more general morphisms. In particular, we allow the degenerate case of $f \colon \opn{Spa}(\mbb{C}_p, \mathcal{O}_{\mbb{C}_p}) \to \mathscr{E}_{\opn{GL}_2}$ corresponding to a point in $\mathscr{E}_{\opn{GL}_2}^{\opn{good}, \chi}$. 
\end{remark}

\subsubsection{Coleman families for \texorpdfstring{$\opn{GSp}_4$}{GSp(4)}} \label{CFsForGSP4SSSEC}

Let $\mathcal{W}((\mbb{Z}_p^{\times})^{\oplus 2})$ denote the adic space over $\mbb{Q}_p$ whose $S$-points parameterise continuous characters $(\mbb{Z}_p^{\times})^{\oplus 2} \to \mathcal{O}(S)^{\times}$. Let $(N, M)$ be a pair of positive integers prime to $p$ such that $M^2 | N$, and let $L/\mbb{Q}_p$ be a finite extension. Let $T$ be a finite set of primes not dividing $Np$, and let $K_T \subset G(\mbb{Z}_T)$ be a normal compact open subgroup. We set
\[
K^p \defeq K_T \cdot \prod_{\substack{\ell \nmid Np \\ \ell \not\in T}} G(\mbb{Z}_{\ell}) \cdot \prod_{\ell | N} K(N, M)_{\ell} \quad \subset \quad G(\mbb{A}_f^p)
\]
where $K(N, M)_{\ell} \subset G(\mbb{Q}_{\ell})$ denotes the $\ell$-component of the quasi-paramodular subgroup of level $(N, M)$. We assume that $K^p$ is neat. Let $S$ denote the set of primes dividing $N$.

We let $\mathscr{E}_{G} = \mathscr{E}_G(N, M)$ denote the cuspidal eigenvariety\footnote{Obtained from considering only the ``$-$'' versions of the cohomology groups of overconvergent forms.} constructed in \cite[\S 6.9]{BoxerPilloni} of tame quasi-paramodular level $K(N, M)$ -- in particular, there is a weight map $w \colon \mathscr{E}_G \to \mathcal{W}((\mbb{Z}_p^{\times})^{\oplus 2})$ and morphisms $f \colon U \to \mathscr{E}_G$ out of affinoid adic spaces $U = \opn{Spa}(\mathscr{O}_U, \mathscr{O}_U^+)$ correspond to finite-slope Hecke eigensystems $(\lambda^S, \lambda_p) \colon \mathcal{H}_{\mathscr{O}_U}^S \otimes \mathscr{O}_U[T^{G,-}] \to \mathscr{O}_U$ appearing in the cohomology 
\begin{equation} \label{BigGradedCohEqn}
\bigoplus_{\substack{w \in {^MW_G} \\ k \in \mbb{Z}}} \opn{H}^k_{w, \opn{an}}\left( K^p, \nu_U, \opn{cusp} \right)^{-, \leq h, G(\mbb{Z}_T)/K_T} 
\end{equation}
for some $h \in \mbb{Q}$ for which there exist slope $\leq h$ decompositions with respect to $s_{p, \opn{B}} \in T^{G, --}$ (with notation as in \emph{loc.cit.}). Here $\nu_U = (r_{1, U}, r_{2, U}; -(r_{1, U}+r_{2, U}))$ where $(r_{1, U}, r_{2, U})$ denotes the pullback of the universal character under $w \circ f \colon U \to \mathcal{W}((\mbb{Z}_p^{\times})^{\oplus 2})$. We note that
\[
\opn{H}^i_{w_1, \opn{an}}\left( K^p, \nu_U, \opn{cusp} \right)^{-, \leq h} \cong \opn{H}^i(K^G_{\beta}, \kappa_U^*; \opn{cusp})^{(\dagger, n\opn{-an},-), \leq h}, \quad \quad (i \in \mbb{Z}),
\]
for any $\beta \geq 1$ and sufficiently large $n \geq 1$, where $\kappa_U = (r_{2,U}+2, -r_{1,U}; -r_{2,U}-1)$.

\begin{notation}
    Let $\chi$ be a Dirichlet character of conductor $M$. We let $\mathscr{E}_G^{\opn{good}, \chi} \subset \mathscr{E}_G(\mbb{C}_p)$ denote the subset of points corresponding to small slope good $p$-stabilised automorphic representations $\tilde{\pi} = (\pi, \Theta_{\pi, p})$ of $G(\mbb{A})$ of weight $\nu = (r_1, r_2; -(r_1+r_2))$, level $(N, M)$ and character $\chi$, for some $r_1 \geq r_2 \geq 0$. We note that the Hecke eigensystem corresponding such a point in $\mathscr{E}_{G}^{\opn{good}, \chi}$ is given by $(\lambda^S, \lambda_p)$ satisfying $\lambda^S = \Theta^S_{\pi}$ and $\lambda_p|_{T^{G, -}} = w_G^{\opn{max}} \cdot \Theta_{\pi, p}^+$. Let $I \subset \mathcal{H}^S_{\mbb{C}_p} \otimes \mbb{C}_p[T^{G, -}]$ denote the kernel of $\lambda^S \otimes \lambda_p$.. By the classicality and vanishing results of higher Coleman theory (\cite[Corollary 6.8.4, Theorem 5.12.3]{BoxerPilloni} supplemented with \cite[Proposition 2.7.2]{LZBK21} and \cite[Theorem 1.4.10]{HHTBoxerPilloni}), we see that
    \begin{equation} \label{LocalisedHCTcohgroupEqn}
    \opn{H}^i_{w, \opn{an}}\left( K^p, \nu, \opn{cusp} \right)^{-, \opn{fs}, G(\mbb{Z}_T)/K_T}_{I} = 0 \quad \text{ if } \quad i \neq 3-l(w) . 
    \end{equation}
    Furthermore, if $\pi$ is non-endoscopic and $i = 3-l(w)$, then the localised cohomology group in (\ref{LocalisedHCTcohgroupEqn}) is one-dimensional over $\mbb{C}_p$ for all $w \in {^MW_G}$. If $\pi$ is of Yoshida-type and $i=3-l(w)$, then the localised cohomology group in (\ref{LocalisedHCTcohgroupEqn}) is one-dimensional over $\mbb{C}_p$ for $l(w) = 1, 2$, and vanishes when $l(w) = 0,3$. In both cases, a Tor-spectral sequence argument shows that the weight map $w$ is \'{e}tale at points in $\mathscr{E}_G^{\opn{good}, \chi}$ (see \cite[\S 3.4]{LZBK21}).
\end{notation}

We now introduce Coleman families for $G$.

\begin{definition} \label{ColemanFamilyFORGSp4}
    Let $U = \opn{Spa}(\mathscr{O}_U, \mathscr{O}_U^+)$ be a reduced finite-type affinoid adic space over $\opn{Spa}(L, \mathcal{O}_L)$. Let $|\!| \cdot |\!|$ denote the Banach norm on $\mathscr{O}_U$ satisfying $|\!|p|\!| = p^{-1}$. A Coleman family $\underline{\pi}$ over $U$, of tame level $(N, M)$ and character $\chi$, is a morphism $f \colon U \to \mathscr{E}_G$ over $\opn{Spa}(L, \mathcal{O}_L)$ such that:
    \begin{enumerate}
        \item $\Upsilon(\underline{\pi}) \defeq f^{-1}(\mathscr{E}_G^{\opn{good}, \chi})$ is Zariski dense in $U$.
        \item The cohomology groups $\opn{H}^k_{w, \opn{an}}\left( K^p, \nu_U, \opn{cusp} \right)^{-}$ admit slope $\leq h$ decompositions with respect to the action of $s_{p, \opn{B}}$ for some $h$ satisfying $|\!| \lambda_p(s_{p, \opn{B}}) |\!| \geq p^{-h}$. If $I_U$ denotes the kernel of $\lambda^S \otimes \lambda_p$, then 
        \[
        \opn{H}^k_{w, \opn{an}}\left( K^p, \nu_U, \opn{cusp} \right)^{-, \leq h, G(\mbb{Z}_T)/K_T}_{I_U} = 0
        \]
        if $k \neq 3- l(w)$, and either: $S^k(\underline{\pi}) \defeq \opn{H}^k_{w_{3-k}, \opn{an}}\left( K^p, \nu_U, \opn{cusp} \right)^{-, \leq h, G(\mbb{Z}_T)/K_T}_{I_U}$ is free of rank one over $\mathscr{O}_U$ for all $k=0, 1, 2, 3$ (general-type); or $S^0(\underline{\pi}) = S^3(\underline{\pi}) = 0$ and $S^k(\underline{\pi})$ is free of rank one over $\mathscr{O}_U$ for $k=1, 2$ (Yoshida-type).
        \item For $K^G_{\beta} = K^p K^G_{\opn{Iw}}(p^{\beta})$ with $K^p$ as above, $\kappa_G = (r_{2,U}+2, -r_{1,U}; -r_{2,U}-1)$, and $n \geq 1$ sufficiently large, the cohomology groups in (\ref{NearlyH2toOCH2surjectiveEqn}) admit slope $\leq h$ decompositions with respect to $s_{p, \opn{B}}$ and the map is surjective on slope $\leq h$ parts.
    \end{enumerate}
\end{definition}

\begin{example} \label{ExampleOfCFForGSP4}
    Suppose that $\tilde{\pi}$ is a good small slope $p$-stabilised automorphic representation of $G(\mbb{A})$ of weight $\nu = (r_1, r_2; -(r_1+r_2))$, tame level $(N, M)$ and character $\chi$. After possibly enlarging $L$, let $x \in \mathscr{E}_G^{\opn{good}, \chi}$ denote the corresponding point of the eigenvariety. Then since the weight map is \'{e}tale at this point, there exists a Coleman family $\underline{\pi}$ over a sufficiently small open neighbourhood $U$ of $w(x)$ in $\mathcal{W}((\mbb{Z}_p^{\times})^{\oplus 2})$ passing through $\tilde{\pi}$ (one can always shrink $U$ so that properties (2) and (3) of Definition \ref{ColemanFamilyFORGSp4} are satisfied). The degenerate case of the point $\opn{Spa}(\mbb{C}_p, \mathcal{O}_{\mbb{C}_p}) \to \mathscr{E}_G$ corresponding to $x$ is also included in our definition of a Coleman family. 
\end{example}

To any Coleman family $\underline{\pi}$ as above, we can choose $\eta_{\underline{\pi}, \beta}$ to be a $\mathscr{O}_U$-basis of 
\[
\opn{H}^2\left( K^G_{\beta}, \kappa_G^*; \opn{cusp} \right)^{(\dagger, n\opn{-an}, -), \leq h, G(\mbb{Z}_T)/K_T}_{I_U} \cong S^2(\underline{\pi})
\]
which is independent of the choice of $n$, and $\opn{Tr}_{G, \beta}(\eta_{\underline{\pi}, \beta+1}) = \eta_{\underline{\pi}, \beta}$ (where $\opn{Tr}_{G, \beta}$ denotes the unnormalised trace map). Furthermore, if $x \in \Upsilon(\underline{\pi})$ and $(\underline{\pi}_x, \Theta_{\underline{\pi}_x, p})$ denotes the corresponding small slope good $p$-stabilised automorphic representation, then the specialisation of $\eta_{\underline{\pi}, \beta}$ at $x$ is (up to a non-zero scalar multiple) equal to the class $\eta_{\underline{\pi}_x, \beta}$ from \S \ref{CaseALfunctionAndPeriodSSSec} via the classicality results of higher Coleman theory.

\subsection{Interpolation regions}

We now introduce the interpolation regions for the $p$-adic $L$-functions.

\subsubsection{} \label{InterpolationRegionCaseASSSec}

Suppose we are in the setting of Case A (\S \ref{CaseASubSubsection}). Let $\underline{\pi}$ be a Coleman family for $G$ over $U_0$ of tame level $(N_0, M_0)$ and character $\chi_0$. Let $\underline{\sigma}$ be a Coleman family for $\opn{GL}_2$ over $U_2$ of tame level $N_2$ and character $\chi_2$. We let 
\[
\Upsilon^{\opn{crit}} = \left\{ ( x, j + \chi, y) \in \Upsilon(\underline{\pi}) \times \mathcal{W}(\mbb{Z}_p^{\times}) \times \Upsilon(\underline{\sigma}) : \begin{array}{c}  0 \leq t_2 \leq r_1 - r_2, \; 0 \leq j \leq r_1 - r_2 - t_2 \\ w(x) = (r_1, r_2), \; w(y) = t_2 \\ \chi \text{ finite-order } \end{array} \right\} .
\]
This region is precisely the range of specialisations where the $L$-function has critical values, i.e., if $\underline{\pi}_x$ and $\underline{\sigma}_y$ denote the specialisations of the Coleman families and $\underline{\Pi}_x$ and $\underline{\Sigma}_y$ their unitary normalisations, then $L(\underline{\Pi}_x \times \underline{\Sigma}_y, \chi^{-1}, j+\tfrac{1-r_1+r_2+t_2}{2})$ is a critical $L$-value for $(x, j+\chi, y) \in \Upsilon^{\opn{crit}}$. Note that if $\underline{\pi}$ and $\underline{\sigma}$ are Coleman families as in Examples \ref{ColemanFamilyFORGSp4} and \ref{ExampleOfCFForGL2} respectively (over open subsets of weight space), then this critical interpolation region $\Upsilon^{\opn{crit}}$ is Zariski dense.

\begin{remark}
    The inequalities in this definition describe ``region (f)'' in \cite[Figure 2]{LZpadiclfunctionsdiagonalcycles}. 
\end{remark}

\subsubsection{} \label{UpsilonCritForCaseBSSSec}

We also consider the interpolation region (f) in Case B (\S \ref{CaseBSubSubSection}). More precisely, let $\underline{\pi}$ be a Coleman family for $G$ over $U_0$ of tame level $(N_0, M_0)$ and character $\chi_0$. For $i=1, 2$, let $\underline{\sigma}_i$ be Coleman families for $\opn{GL}_2$ over $U_i$ of tame level $N_i$ and character $\chi_i$. We consider the region:
\[
\Upsilon^{\opn{crit}} = \left\{ ( x, y_1, y_2) \in \Upsilon(\underline{\pi}) \times \Upsilon(\underline{\sigma}_1) \times \Upsilon(\underline{\sigma}_2) : \begin{array}{c}  0 \leq d_2 \leq r_1 - r_2, \; 0 \leq d_1 \leq r_1 - r_2 - d_2 \\ w(x) = (r_1, r_2), \; w(y_1) = d_1, \; w(y_2) = d_2 \end{array} \right\} .
\]
In this region, we will consider the automorphic period
\[
\mathscr{P}^{\opn{sph}}(x, y_1, y_2; \gamma_S) \defeq \mathscr{P}( \gamma_{0, S} \cdot \underline{\varphi}^{\opn{sph}}_x, \gamma_{1, S} \cdot \delta^{(r_1-r_2-d_1-d_2)/2}\underline{\lambda}_{1, y_1}^{\opn{sph}}, \underline{\lambda}_{2, y_2}^{\opn{sph}} ), \quad \quad (x, y_1, y_2) \in \Upsilon^{\opn{crit}}
\]
which is a substitute for the central critical $L$-value $L(\underline{\Pi}_x \times \underline{\Sigma}_{1, y_1} \times \underline{\Sigma}_{2, y_2}, 1/2)$. Here the notation is as in \S \ref{CaseBLfunctionAndPeriodSSec} (i.e., $\underline{\varphi}_x^{\opn{sph}}$ is the new vector in $\underline{\Pi}_x$ etc.).

\subsection{\texorpdfstring{$p$}{p}-adic \texorpdfstring{$L$}{L}-functions for \texorpdfstring{$\opn{GSp}_4 \times \opn{GL}_2$}{GSp(4) x GL(2)}}

We now construct the $p$-adic $L$-function in Coleman families. We place ourselves in Case A (\S \ref{CaseASubSubsection}) and let $\underline{\pi}$ and $\underline{\sigma}$ be Coleman families for $G$ and $\opn{GL}_2$ respectively, as in \S \ref{InterpolationRegionCaseASSSec}.

\begin{theorem} \label{FourVarCaseATheorem}
    There exists a locally analytic distribution $\mathscr{L}_p \in \mathcal{O}(U_0 \times \mathcal{W}(\mbb{Z}_p^{\times}) \times U_2)$ such that: for all $( x, j + \chi, y) \in \Upsilon^{\opn{crit}}$ and $\beta \geq \opn{max}(2\opn{max}(1, c(\chi)), \opn{max}(1, c(\chi))+1)$, there exists a non-zero scalar $c_{x} \in \Qpb^{\times}$ (independent of $j$, $\chi$, $\beta$, and $\gamma_S$) such that 
    \[
    \mathscr{L}_p(x, j+\chi, y) = c_{x} \cdot \mathcal{P}_{\mbb{C}}^{\opn{alg}}(\gamma_{0, S} \cdot \tilde{\eta}_{\underline{\pi}_x, \beta}, \omega_{\opn{Eis}, j, \beta}, \omega_{\underline{\sigma}_y, \beta} )
    \]
    where the inputs in the right-hand side are as in \S \ref{CaseALfunctionAndPeriodSSSec}, and we compare both sides via the identification $\mbb{C} \cong \Qpb$.
\end{theorem}
\begin{proof}
    Let $\eta_{\underline{\pi}, \beta}$ and $\omega_{\underline{\sigma}}$ be the cohomology classes from \S \ref{CFsForGSP4SSSEC} and \S \ref{CFFORGl2SSSEC}. By property (3) in Definition \ref{ColemanFamilyFORGSp4} and the fact that $\opn{Tr}_{G, \beta}$ is an isomorphism of finite-slope parts, we can lift $\eta_{\underline{\pi}, \beta}$ to classes
    \[
    \tilde{\eta}_{\underline{\pi}, \beta} \in \opn{H}^2\left( K^G_{\beta}, \kappa_G^*; \opn{cusp} \right)^{(\ddagger, n\opn{-an}, -), \leq h, G(\mbb{Z}_T)/K_T}_{I_U}
    \]
    which satisfy $\opn{Tr}_{G, \beta}(\tilde{\eta}_{\underline{\pi}, \beta+1}) = \tilde{\eta}_{\underline{\pi}, \beta}$. Let $x \in \Upsilon{\underline{\pi}}$. We have a commutative diagram:

    \vspace{0.5cm}
    
    \adjustbox{scale=0.8, center}{%
\begin{tikzcd}
{\opn{H}^2\left( K^G_{\beta}, \kappa_G^*; \opn{cusp} \right)^{(\ddagger, n\opn{-an}, -), \leq h, G(\mbb{Z}_T)/K_T}_{I_U}} \arrow[r, "\opn{sp}_x"] \arrow[d] & {\opn{H}^2\left( K^G_{\beta}, \kappa_x^*; \opn{cusp} \right)^{(\ddagger, -), \leq h, G(\mbb{Z}_T)/K_T}_{I_U}} \arrow[d] \arrow[r] & {\opn{H}^2\left( K^G_{\beta}, \kappa_x^*; \opn{cusp} \right)^{\opn{nearly}, \leq h, G(\mbb{Z}_T)/K_T}_{I_U}} \arrow[d, "f"] \\
{\opn{H}^2\left( K^G_{\beta}, \kappa_G^*; \opn{cusp} \right)^{(\dagger, n\opn{-an}, -), \leq h, G(\mbb{Z}_T)/K_T}_{I_U}} \arrow[r, "\opn{sp}_x"]            & {\opn{H}^2\left( K^G_{\beta}, \kappa_x^*; \opn{cusp} \right)^{(\dagger, -), \leq h, G(\mbb{Z}_T)/K_T}_{I_U}} \arrow[r]            & {\opn{H}^2\left( K^G_{\beta}, \kappa_x^*; \opn{cusp} \right)^{\leq h, G(\mbb{Z}_T)/K_T}_{I_U}}                             
\end{tikzcd}
    }

    \vspace{0.5cm}

    where $\opn{sp}_x$ denotes specialisation at $x$ and projecting to the overconvergent cohomology of the dual of the (nearly) algebraic representation of highest weight $\kappa_x = (r_2+2, -r_1; -r_2-1)$ with $w(x) = (r_1, r_2)$, the left and centre vertical maps are the natural ones, the right-hand horizontal maps denote the classicality morphisms from higher Coleman theory and Proposition \ref{NearlyClassicalityProp}, and $f = (f_1^{\vee})^{-1} \circ f_2^{\vee}$ with notation as in \S \ref{CaseALfunctionAndPeriodSSSec}. This implies that the image of $\tilde{\eta}_{\underline{\pi}, \beta}$ in the right-hand top corner of the above diagram is (up to a scalar $c_x \in \Qpb^{\times}$ independent of $\beta$) equal to the unique lift of $\eta_{\underline{\pi}, \beta}$ under $f$.
    
    Let $U_1^{(0)} \subset U_1^{(1)} \subset \cdots \subset \mathcal{W}(\mbb{Z}_p^{\times})$ be an increasing cover by quasi-compact open affinoid subspaces $U_1^{(i)} = \opn{Spa}(\mathscr{O}_{U_1^{(i)}}, \mathscr{O}_{U_1^{(i)}}^+)$ such that the universal character $\kappa_i$ is $i$-analytic. Set $\lambda = r_{2, U_0} + t_{U_2} + 1$ and $\xi = 1-r_{2, U_0}$, where $t_{U_2}$ denotes the universal character over $U_2$. Let $n \gg i$. We define
    \[
    \Xi_i \defeq \mathcal{P}^{\ddagger, n\opn{-an}}_{\lambda, \beta}\left( \gamma_{0, S} \cdot \tilde{\eta}_{\underline{\pi}, \beta}, \mathcal{E}_{\xi}^{\Phi^{(p)}}(r_{1, U_0} - r_{2, U_0} - t_{U_2} - \kappa_i, \kappa_i), \omega_{\underline{\sigma}} \right) \quad \in \quad \mathscr{O}_{U_0} \hatot \mathscr{O}_{U^{(i)}_1} \hatot \mathscr{O}_{U_2}
    \]
    which is independent of $n$ and $\beta$ (see Lemma \ref{PlambdaIsIndepOfBetaTraceLemma}). Here $\gamma_{0, S} \cdot \tilde{\eta}_{\underline{\pi}, \beta}$ denotes the image of the class $\tilde{\eta}_{\underline{\pi}, \beta}$ under the pullback map $\opn{H}^2(K^G_{\beta}, \kappa_G^*; \opn{cusp})^{(\ddagger, n\opn{-an}, -)} \to \opn{H}^2(\widehat{K}^G_{\beta}, \kappa_G^*; \opn{cusp})^{(\ddagger, n\opn{-an}, -)}$. Furthermore, these classes are compatible with changing $i$, so we can define 
    \[
    \mathscr{L}_p \defeq (\Xi_i)_{i \geq 0} \in \varprojlim_{i \geq 0} \mathscr{O}_{U_0} \hatot \mathscr{O}_{U^{(i)}_1} \hatot \mathscr{O}_{U_2} = \mathcal{O}(U_0 \times \mathcal{W}(\mbb{Z}_p^{\times}) \times U_2) .
    \]
    The interpolation property now follows from Lemma \ref{DDaggerNanPeriodEqualsDDaggerPerLem} and Lemma \ref{DDaggerPerEqualsAlgPerLem}.
\end{proof}

Combining this theorem with Proposition \ref{PropPAlgCEqualsCompletedLfunction} completes the proof of Theorem \ref{ThmAIntro} in the introduction.

\subsection{\texorpdfstring{$p$}{p}-adic \texorpdfstring{$L$}{L}-functions for \texorpdfstring{$\opn{GSp}_4 \times \opn{GL}_2 \times \opn{GL}_2$}{GSp(4) x GL(2) x GL(2)}}

Suppose we are in the setting of Case B (\S \ref{CaseBSubSubSection}) and let $\underline{\pi}$ and $\underline{\sigma}_i$ be as in \S \ref{UpsilonCritForCaseBSSSec}. Recall that we are assuming that $p > 2$.

\begin{theorem} \label{FourVarCaseBTheorem}
    There exists a rigid-analytic function $\mathscr{L}_p \in \mathcal{O}(U_0 \times U_1 \times U_2)$ such that: for all $(x, y_1, y_2) \in \Upsilon^{\opn{crit}}$ and $\beta \geq 2$, there exists a non-zero scalar $c_x \in \Qpb^{\times}$ (independent of $\beta$, $\gamma_S$, $\underline{\sigma}_1$, and $\underline{\sigma}_2$) such that
    \[
    \mathscr{L}_p(x, y_1, y_2) = c_x \cdot \mathcal{P}_{\mbb{C}}^{\opn{alg}}\left(\gamma_{0, S} \cdot \tilde{\eta}_{\underline{\pi}_x, \beta}, \gamma_{1, S} \cdot \omega^{[p], \opn{nearly}}_{\underline{\sigma}_{1, y_1}, \beta}, \omega^{[p]}_{\underline{\sigma}_{2, y_2}, \beta} \right)
    \]
    where the inputs in the right-hand side are as in \S \ref{CaseBLfunctionAndPeriodSSec}, and we compare both sides via the identification $\mbb{C} \cong \Qpb$.
\end{theorem}
\begin{proof}
    This follows from exactly the same strategy as in Theorem \ref{FourVarCaseATheorem} replacing the family of Eisenstein series with the family of nearly overconvergent modular forms
    \[
    \omega^{[p],\opn{nearly}}_{\underline{\sigma}_1, \beta} \defeq \opn{det}^{(r_{1, U_0}+r_{2,U_0}-d_{U_1} - d_{U_2}-2)/2}\nabla^{(r_{1, U_0}-r_{2,U_0}-d_{U_1} - d_{U_2})/2}\omega_{\underline{\sigma}_1, \beta}^{[p]} \; \in \; \mathscr{N}^{\dagger}_{\opn{GL}_2} \hatot \mathscr{O}_{U_0 \times U_1 \times U_2}
    \]
    and replacing $\omega_{\underline{\sigma}_2, \beta}$ with its $p$-depleted version $\omega_{\underline{\sigma}_2, \beta}^{[p]}$. Here $d_{U_i}$ is the universal character of $U_i$ (determined by its map to one-dimensional weight space). More precisely, we have
    \[
    \omega_{\underline{\sigma}_2, \beta}^{[p]} = 1_{\mbb{Z}_p^{\times}} \star \omega_{\underline{\sigma}_2, \beta}, \quad \nabla^{(r_{1, U_0}-r_{2,U_0}-d_{U_1} - d_{U_2})/2}\omega_{\underline{\sigma}_1, \beta}^{[p]} = (\rho 1_{\mbb{Z}_p^{\times}}) \star \omega_{\underline{\sigma}_1, \beta}
    \]
    where $\rho \colon \mbb{Z}_p^{\times} \to \mathscr{O}_{U_0 \times U_1 \times U_2}^{\times}$ is a square-root of $r_{1, U_0}-r_{2,U_0}-d_{U_1} - d_{U_2}$ (which exists because, by assumption, there exists a Zariski dense subset of points where this quantity specialises to an even integer), $\star$ denotes the action of $C^{\opn{la}}(\mbb{Z}_p, \mathscr{O}_{U_0 \times U_1 \times U_2})$ on nearly overconvergent modular forms in \cite{DiffOps}, and $1_{\mbb{Z}_p^{\times}} \in C^{\opn{la}}(\mbb{Z}_p, \mathscr{O}_{U_0 \times U_1 \times U_2})$ denotes the indicator function of $\mbb{Z}_p^{\times}$. By abuse of notation, $\rho 1_{\mbb{Z}_p^{\times}}$ denotes the locally analytic function on $\mbb{Z}_p$ obtained by extending $\rho$ by zero.
\end{proof}

Combining this theorem with Proposition \ref{TripleProdPalgWithSpherPeriodProp}, we therefore obtain a locally analytic distribution interpolating $\mathscr{P}^{\opn{sph}}(x, y_1, y_2; \gamma_S)$ for $(x, y_1, y_2) \in \Upsilon^{\opn{crit}}$, completing the proof of Theorem \ref{ThmCCaseBIntro} in the introduction.

\section{Temperedness of the distribution}

In this section we prove Theorem \ref{ThmBTemperedIntro}; namely, if we assume we're in Case A (\S \ref{CaseASubSubsection}) and fix a small slope good $p$-stabilised automorphic representation $\tilde{\pi} = (\pi, \Theta_{\pi, p})$ (resp. good $p$-stabilised automorphic representation $\tilde{\sigma}$) of $G(\mbb{A})$ (resp. $\opn{GL}_2(\mbb{A})$) of weight $\nu = (r_1, r_2; -(r_1+r_2))$, level $(N_0, M_0)$ and character $\chi_0$ (resp. of weight $1+t_2$, level $N_2$, and character $\chi_2$) such that $t_2 \leq r_1-r_2$, then we prove the following theorem:

\begin{theorem} \label{ThmBInMainBody}
    Let $L_p(\pi \times \sigma, -) \in \mathscr{D}^{\opn{la}}(\mbb{Z}_p^{\times}, L)$ denote the locally analytic distribution constructed in Theorem \ref{FourVarCaseATheorem}, and let $h = v_p\left(\Theta_{\pi, p}(U_{p, \opn{Kl}}^{\circ})\right)$. Then $L_p(\pi \times \sigma, -)$ is tempered of growth $\leq h+2$ in the sense of \cite[Definition 3.10]{barrera2021p}. Moreover, if $h+2 < r_1-r_2-t_2+1$, then $L_p(\pi \times \sigma, -)$ is uniquely determined (as a tempered distribution of growth $\leq h+2$) by its values $L_p(\pi \times \sigma, j+\chi)$ for $0 \leq j \leq r_1-r_2-t_2$ and $\chi$ finite-order.
\end{theorem}

\begin{remark}
    This is not the optimal growth bound. We expect that $L_p(\pi \times \sigma, -)$ is tempered of growth $h$, however we were unable to show this.
\end{remark}

We note that the last part of this theorem follows from \cite[Proposition I.4.5]{colmez1998theorie}, so we will focus on proving that the $p$-adic $L$-function is tempered of growth $\leq h+2$ for the rest of this section. Concretely this amounts to checking the following property. For $n \geq 1$, let $\mathcal{W}_n \subset \mathcal{W}(\mbb{Z}_p^{\times})_L$ denote the quasi-Stein open subspace of ``$(n-1)$-accessible weights'' as in \cite[Definition 4.1.1]{lz-coleman}. By \cite[Lemma 4.1.5]{lz-coleman}, for any open affinoid $U \subset \mathcal{W}_n$, the universal character $\kappa_U \colon \mbb{Z}_p^{\times} \to \mathcal{O}(U)^{\times}$ is $n$-analytic (i.e., given by a power series on cosets of the form $a + p^n \mbb{Z}_p \subset \mbb{Z}_p^{\times}$). If $U_n \subset \mathcal{W}_n$ is open affinoid, then the sections $\mathcal{O}(U_n)$ naturally form a $L$-Banach algebra with norm $|\!| - |\!|_n$ such that the unit ball is $\mathcal{O}^+(U_n)$. Let $\mu_n \in \mathcal{O}(U_n)$ denote the restriction of $L_p(\pi \times \sigma, -) \in \mathcal{O}(\mathcal{W}(\mbb{Z}_p^{\times})_L)$. Then we need to show that there exists a constant $C \in \mbb{R}_{>0}$ such that $|\!| \mu_n |\!|_n \leq p^{(h+2)n}C$ for all $n \geq 1$ and $U_n \subset \mathcal{W}_n$.

\subsubsection{} 

We first need to understand integrality properties of the family of nearly overconvergent Eisenstein series. Recall that the spaces $\mathscr{N}^{\dagger, n\opn{-an}}_{\opn{GL}_2,\zeta_{H_i}} \defeq \opn{H}^0\left( \widehat{K}^{H_i}_{\beta}, \zeta_{H_i} \right)^{(\ddagger, n\opn{-an}, +)}$ are independent of $\beta$, so we may take $\beta = 1$ from now on (and often drop it from the notation). We define an integral version of this space as
\[
\mathscr{N}^{\dagger, n\opn{-an}, \circ}_{\opn{GL}_2,\zeta_{H_i}} \defeq \varinjlim_V \opn{H}^0\left( V, [\mbb{I}_{\overline{\mathcal{P}}_{\opn{GL}_2, n}^{\square}}(\zeta_{H_i})^{\circ}] \right)
\]
which is again independent of $\beta$. Here the colimit is over all strict neighbourhoods of $\mathcal{X}_{H_i, \opn{id}, \beta}$, one defines the $\overline{\mathcal{P}}_{\opn{GL}_2, n}^{\square}$-representation $\mbb{I}_{\overline{\mathcal{P}}_{\opn{GL}_2, n}^{\square}}(\zeta_{H_i})^{\circ}$ in exactly the same way as for $\mbb{I}_{\overline{\mathcal{P}}_{\opn{GL}_2, n}^{\square}}(\zeta_{H_i})$ but replacing $\mbb{A}^{1, \opn{an}}$ with $\mbb{G}_a^+ = \opn{Spa}(\mbb{Q}_p\langle t \rangle, \mbb{Z}_p\langle t \rangle)$, and the sheaf $[\mbb{I}_{\overline{\mathcal{P}}_{\opn{GL}_2, n}^{\square}}(\zeta_{H_i})^{\circ}]$ is given by 
\[
\left(\pi_*\mathcal{O}^+_{\mathcal{P}_{H_i, \opn{dR}, n}} \hatot \; \mbb{I}_{\overline{\mathcal{P}}_{\opn{GL}_2, n}^{\square}}(\zeta_{H_i})^{\circ}\right)^{\overline{\mathcal{P}}_{\opn{GL}_2, n}^{\square}}.
\]
We will need a strengthening of Lemma \ref{TheEisSeriesIsNearlyOCLemma}.

\begin{lemma} \label{M1M2EisLemma}
    Let $\kappa_n \colon \mbb{Z}_p^{\times} \to \mathcal{O}(U_n)^{\times}$ denote the universal character, $\xi = 1-r_2$, and $\zeta_{H_1} = (-1-t_1; \xi)$. Let $\zeta_{H_1, n}$ denote $\zeta_{H_1}$ viewed as a character valued in $\mathcal{O}(U_n)^{\times}$. Then there exist integers $M_1, M_2 \geq 1$ (independent of $n$ and $U_n$) such that
    \[
    \mathcal{E}^{\Phi^{(p)}}_{\xi}(t_1 - \kappa_n, \kappa_n) \in p^{-2n-M_2} \mathscr{N}^{\dagger, (n + M_1)\opn{-an}, \circ}_{\opn{GL}_2,\zeta_{H_1, n}}.
    \]
\end{lemma}
\begin{proof}
    Set $(A, A^+) = (\mathcal{O}(U_n), \mathcal{O}^+(U_n))$ for ease of notation. Set $\lambda_{H_1, n} = \zeta_{H_1,n}+(t_1-\kappa_n)(2;-1)$ and consider the family of Eisenstein series $\mathcal{E}^{\Phi^{(p)}}_{\xi+\kappa_n-t_1}(0, 2\kappa_n-t_1) \in \mathscr{N}^{\dagger}_{\opn{GL}_2, A} \defeq \mathscr{N}^{\dagger}_{\opn{GL}_2} \hatot_L A$. We claim that there exists an integer $M_3 \geq 0$ (independent of $n$ and $U_n$) such that 
    \[
    \mathcal{E}^{\Phi^{(p)}}_{\xi+\kappa_n-t_1}(0, 2\kappa_n-t_1) \in p^{-M_3}\mathscr{N}^{\dagger, n\opn{-an}, \circ}_{\opn{GL}_2,\lambda_{H_1,n}}.
    \]
    Indeed, we note that this family of Eisenstein series is \emph{overconvergent}, hence the analyticity is determined by the analyticity of the weight. Furthermore, integrality (for overconvergent forms) can be checked on the level of $q$-expansions, and it is easy to see that the integrality of the $q$-expansion of this family of Eisenstein series only depends on $\Phi^{(p)}$ (see \cite[Theorem 7.6]{LPSZ}).

    Recall from Lemma \ref{TheEisSeriesIsNearlyOCLemma} that $\mathscr{N}^{\dagger}_{\opn{GL}_2, A}$ comes equipped with an action $\star$ of $C^{\opn{la}}(\mbb{Z}_p, A)$. The image of $\mathcal{E}^{\Phi^{(p)}}_{\xi}(t_1 - \kappa_n, \kappa_n)$ in $\mathscr{N}^{\dagger}_{\opn{GL}_2, A}$ can be described as
    \[
    \phi \star \mathcal{E}^{\Phi^{(p)}}_{\xi+\kappa_n-t_1}(0, 2\kappa_n-t_1)
    \]
    where $\phi \colon \mbb{Z}_p \to A$ is the extension-by-$0$ of the character $t_1 - \kappa_n$. It therefore suffices to show that there exist integers $M_1, M_2 \geq 1$ (independent of $n$ and $U_n$) such that $\phi \star -$ maps $p^{-M_3}\mathscr{N}^{\dagger, (n+M_1)\opn{-an}, \circ}_{\opn{GL}_2,\lambda_{H_1,n}}$ into $p^{-2n-M_2}\mathscr{N}^{\dagger, (n+M_1)\opn{-an}, \circ}_{\opn{GL}_2,\zeta_{H_1,n}}$. But this follows from Proposition \ref{CepsilonIntegralityAppendixProp} by taking $\varepsilon = (p^n(p-1))^{-1}$.
\end{proof}

\subsubsection{}

We now study the integrality of the trilinear period in \S \ref{padicWtNOCsssec}. We continue to work at level $\beta = 1$ and omit $\beta$ from the notation. For brevity, we set $\mathcal{Q}_n \defeq \mathcal{P}^{\ddagger, n\opn{-an}}_{\lambda}$ with $\lambda = r_1+1-t_1$. We let $\mbb{D}_{\overline{\mathcal{P}}^{\square}_{G,n+1}}(\kappa_G^*)^{\circ}$ denote the integral version of $\mbb{D}_{\overline{\mathcal{P}}^{\square}_{G,n+1}}(\kappa_G^*)$ defined in exactly the same way, but replacing $\mbb{A}^{1, \opn{an}}$ with $\mbb{G}_a^+$ (and taking continuous linear functionals valued in $\mbb{Z}_p$). We use similar notation for the other $p$-adic nearly representations. Note that the branching law $\opn{br}^{n\opn{-an}}$ preserves integrality. We let $[\mbb{D}_{\overline{\mathcal{P}}^{\square}_{G,n+1}}(\kappa_G^*)^{\circ}]$ denote the sheaf associated with the representation $\mbb{D}_{\overline{\mathcal{P}}^{\square}_{G,n+1}}(\kappa_G^*)^{\circ}$ using $+$-sheaves. We use similar notation for the other $p$-adic nearly representations. We can also twist these sheaves by $\mathcal{O}^+(-D_G)$ and $\mathcal{O}^+(-D_H)$.

We define $\mathscr{N}^{\dagger, n\opn{-an}, \circ}_{G, \kappa_G^*}$ to be the image of the map
\[
\varinjlim_V \opn{H}^2_{\mathcal{Z}_{G, <1} \cap V}\left( V, [\mbb{D}_{\overline{\mathcal{P}}^{\square}_{G,n+1}}(\kappa_G^*)^{\circ}](-D_G) \right) \to \opn{H}^2\left( K^G_{1}, \kappa_G^*; \opn{cusp} \right)^{(\ddagger, n\opn{-an}, -)}
\]
where the colimit is over all strict neighbourhoods of $\mathcal{X}_{G, w_1}$.

\begin{lemma} \label{QnIntegralityLemma}
    With notation as in Lemma \ref{M1M2EisLemma}, there exists an integer $M_3 \geq 0$ (independent $n$ and $U_n$) such that 
    \[
    \mathcal{Q}_m\left( \mathscr{N}^{\dagger, m\opn{-an}, \circ}_{G, \kappa_G^*}, \mathscr{N}^{\dagger, m\opn{-an}, \circ}_{\opn{GL}_2, \zeta_{H_1, n}}, \mathscr{N}^{\dagger, m\opn{-an}, \circ}_{\opn{GL}_2, \zeta_{H_2}} \right) \subset p^{-M_3}\mathcal{O}^+(U_n)
    \]
    for any $m \geq 1$, where $\zeta_{H_2} = (-1-t_2; 0)$.
\end{lemma}
\begin{proof}
    Set $\widehat{K}_H = (\widehat{K}^p \cap H(\mbb{A}_f^p)) K^H_{\diamondsuit}(p)$. Due to the functoriality of the torsors $\mathcal{P}_{?, \opn{dR}, m}$, one can show (by the same construction for $\mathcal{Q}_m$) that for any $x \in \mathscr{N}^{\dagger, m\opn{-an}, \circ}_{G, \kappa_G^*}$, $y \in \mathscr{N}^{\dagger, m\opn{-an}, \circ}_{\opn{GL}_2, \zeta_{H_1, n}}$, and $z \in \mathscr{N}^{\dagger, m\opn{-an}, \circ}_{\opn{GL}_2, \zeta_{H_2}}$, the element $\hat{\iota}^*(x) \smile (y \sqcup z)$ is in the image of the map
    \begin{equation} \label{H2cdaggerOmegaEqn}
    \opn{H}^2_{\mathcal{Z}_{H, \opn{id}}}\left( \mathcal{X}_H(\widehat{K}_H)_A, [V_{M_H}(-2\rho_{H, \opn{nc}})^{\circ}] \right) \to \opn{H}^2_{\mathcal{Z}_{H, \opn{id}}}\left( \mathcal{X}_H(\widehat{K}_H)_A, [V_{M_H}(-2\rho_{H, \opn{nc}})] \right) .
    \end{equation}
    Here $A = \mathcal{O}(U_n)$, we have used excision to write the cohomology groups in this particular way, and $\rho_{H, \opn{nc}}$ denotes the half-sum of the positive roots of $H$ not lying in $M_H$. Note that $[V_{M_H}(-2\rho_{H, \opn{nc}})] \cong \Omega^2_{\mathcal{X}_H(\widehat{K}_H)}$ and $[V_{M_H}(-2\rho_{H, \opn{nc}})]^{\circ} \subset \Omega^2_{\mathcal{X}_H(\widehat{K}_H)}$ is an $\mathcal{O}^+_{\mathcal{X}_H(\widehat{K}_H)}$-lattice. Hence, up to multiplying by a power of $p$ which is independent of $n$ and $U_n$, we may assume without loss of generality that $[V_{M_H}(-2\rho_{H, \opn{nc}})]^{\circ} \cong \Omega^{2, +}_{\mathcal{X}_H(\widehat{K}_H)}$, where $\Omega^{2, +}_{\mathcal{X}_H(\widehat{K}_H)}$ denotes the wedge square of the sheaf of integral differentials. Let $\mathcal{X}^{\dagger}_{H, \opn{id}}$ denote the dagger space associated with $\mathcal{X}_{H, \opn{id}}$. Then the cohomology groups in (\ref{H2cdaggerOmegaEqn}) can be identified with the compactly supported cohomology of $\mathcal{X}^{\dagger}_{H, \opn{id}}$. But $\mathcal{X}^{\dagger}_{H, \opn{id}}$ admits a smooth weak formal model $\mathfrak{X}^{\dagger}_{H, \opn{id}}$, hence the map in (\ref{H2cdaggerOmegaEqn}) is identified with
    \[
    \opn{H}^2_c\left(\mathfrak{X}^{\dagger}_{H, \opn{id}, A^+}, \Omega^2_{\mathfrak{X}^{\dagger}_{H, \opn{id},A^+}} \right) \to \opn{H}^2_c\left(\mathcal{X}^{\dagger}_{H, \opn{id}, A}, \Omega^2_{\mathcal{X}^{\dagger}_{H, \opn{id},A}} \right) .
    \]
    In particular, there is a trace map $\opn{tr} \colon \opn{H}^2_c\left(\mathfrak{X}^{\dagger}_{H, \opn{id}, A^+}, \Omega^2_{\mathfrak{X}^{\dagger}_{H, \opn{id},A^+}} \right) \to A^+$ compatible with the trace map on the generic fibre. This implies the required integrality.
\end{proof}

We almost have all the ingredients needed to prove Theorem \ref{ThmBInMainBody}. More precisely, the distribution $\mu_n$ can be described as 
\[
\mu_n = \mathcal{Q}_{(n+M_1)}\left( \eta^{(n+M_1)}, \mathcal{E}^{\Phi^{(p)}}_{\xi}( t_1 - \kappa_n, \kappa_n ), \omega_{\sigma} \right)
\]
where $\eta^{(n+M_1)}$ denotes the $(n+M_1)$-analytic class $\gamma_{0, S} \cdot \tilde{\eta}_{\pi}$. Note that there exists an integer $M_4 \geq 0$ (independent of $n$ and $U_n$) such that $p^{M_4}\omega_{\sigma}$ is integral. Let $\alpha_{\opn{Kl}} \defeq \Theta_{\pi, p}(U_{p, \opn{Kl}}^{\circ})$. Suppose that we can find an integer $M_5 \geq 0$ (independent of $n$ and $U_n$) such that $p^{M_5}\alpha_{\opn{Kl}}^m \eta^{(m)}$ is integral for all $m \geq 1$. Then Lemma \ref{M1M2EisLemma} and Lemma \ref{QnIntegralityLemma} imply that
\[
\mu_n \in p^{-M_5} \alpha^{-(n+M_1)}_{\opn{Kl}} p^{-2n-M_2} p^{-M_4-M_3} \mathcal{O}^+(U_n)
\]
hence we have 
\[
|\!| \mu_n |\!|_n \leq p^{n(h+2)} C
\]
where $C = p^{M_5+M_1h+M_2+M_4+M_3}$ is independent of $n$ and $U_n$. Here we have used that, by definition, $|\alpha_{\opn{Kl}}| = p^{-h}$. It therefore remains to understand the potential denominators of the classes $\{\eta^{(m)}\}_{m \geq 1}$.

\subsubsection{}

We keep the notation of the previous two sections. We will study the correspondence associated with $s_{\opn{Kl}} = w_G^{\opn{max}} t_{\opn{Kl}} w_G^{\opn{max}}$. 

\begin{lemma} \label{FactorisationOfKlingenUpLemma}
    For any $m \geq 1$, one has a factorisation
    \[
\begin{tikzcd}
{\opn{H}^2\left(K^G_1, \kappa_G^*; \opn{cusp}\right)^{(\ddagger, (m+1)\opn{-an}, -)}} \arrow[d] \arrow[r, "s_{\opn{Kl}}"]                         & {\opn{H}^2\left(K^G_1, \kappa_G^*; \opn{cusp}\right)^{(\ddagger, (m+1)\opn{-an}, -)}} \arrow[d] \\
{\opn{H}^2\left(K^G_1, \kappa_G^*; \opn{cusp}\right)^{(\ddagger, m\opn{-an}, -)}} \arrow[r, "s_{\opn{Kl}}"] \arrow[ru, "{\rho_{m, m+1}}", dashed] & {\opn{H}^2\left(K^G_1, \kappa_G^*; \opn{cusp}\right)^{(\ddagger, m\opn{-an}, -)}}              
\end{tikzcd}
    \]
    and the morphism $\rho_{m, m+1}$ maps $\mathscr{N}^{\dagger, m\opn{-an}, \circ}_{G, \kappa_G^*}$ into $\mathscr{N}^{\dagger, (m+1)\opn{-an}, \circ}_{G, \kappa_G^*}$.
\end{lemma}
\begin{proof}
    The factorisation follows from the fact that action of $w_1 t_{\opn{Kl}} w_1^{-1} = t_{\opn{Kl}}$ maps $\mbb{I}_{\overline{\mathcal{P}}^{\square, \circ}_{G, m}}(\kappa_G)$ into $\mbb{I}_{\overline{\mathcal{P}}^{\square, \circ}_{G, m+1}}(\kappa_G)$ (because conjugation by $t_{\opn{Kl}}^{-1}$ contracts the roots $\left( \begin{smallmatrix} 1 & & & \\ * & 1 & & \\ * & & 1 & \\ * & * & * & 1 \end{smallmatrix} \right)$). The integrality of the map $\rho_{m, m+1}$ follows from the fact that the Hecke operator corresponding to $s_{\opn{Kl}}$ is optimally normalised. Indeed, the cohomological correspondence $(w_1^{-1} \star \mbf{1})(s_{\opn{Kl}})\psi_{s_{\opn{Kl}}} = \psi_{s_{\opn{Kl}}}$ in Definition \ref{PsitCohCorrespDef} is integral because $t_{\opn{Kl}}$ maps $\mbb{I}_{\overline{\mathcal{P}}^{\square, \circ}_{G, m}}(\kappa_G)^{\circ}$ into $\mbb{I}_{\overline{\mathcal{P}}^{\square, \circ}_{G, m+1}}(\kappa_G)^{\circ}$, and the trace map associated with the morphism $p_1$ in the Hecke correspondence is integral by the comparison between higher Coleman theory and higher Hida theory in \cite[\S 6]{HHTBoxerPilloni}. 
\end{proof}

We now let $M_5 \geq 0$ be an integer such that $p^{M_5}\eta^{(1)}$ is integral. From the factorisation in Lemma \ref{FactorisationOfKlingenUpLemma}, we see that
\[
p^{M_5}\alpha_{\opn{Kl}}^{m}\eta^{(m)} = \alpha_{\opn{Kl}} s_{\opn{Kl}}^{m-1} \cdot (p^{M_5} \eta^{(m)}) = \alpha_{\opn{Kl}} \rho_{m-1, m} \rho_{m-2, m-1} \cdots \rho_{1, 2} (p^{M_5} \eta^{(1)}) \; \in \; \mathscr{N}^{\dagger, m\opn{-an}, \circ}_{G, \kappa_G^*}
\]
for any $m \geq 1$, where we have used the fact that $\eta^{(m)}$ is an eigenvector for the action of $s_{\opn{Kl}}$ with eigenvalue $\alpha_{\opn{Kl}}$. Combining this with the discussion in the previous section completes the proof of Theorem \ref{ThmBInMainBody}.

\appendix

\section{Local zeta integral computations} \label{LocalZetaSSecAppendix}

Let $p$ be a prime. In this section, we explain how the results in \cite[\S 8]{LPSZ} can be applied to compute certain local zeta integrals at Iwahori level (rather than at Klingen level). To be consistent, we will (mostly) use the same notation as in \emph{loc.cit.} in this section; the notation will therefore differ slightly from the main body of this article.  

Let $\psi \colon \mbb{Q}_p \to \mbb{C}^{\times}$ denote the standard additive character which is trivial on $\mbb{Z}_p$ (but not $p^{-1}\mbb{Z}_p$). Let $\pi$ be an unramified smooth irreducible representation of $\opn{GSp}_4(\mbb{Q}_{p})$ which is unitary. Let $\sigma$ be a unitary unramified principal series representation of $\opn{GL}_2(\mbb{Q}_p)$, and let $\chi \colon \mbb{Q}_p^{\times} \to \mbb{C}^{\times}$ be a finite-order character of conductor $p^r$ ($r \geq 0$). We assume that both $\pi$ and $\sigma$ are generic, and we fix a Whittaker model $W_{\pi}(-)$ (resp. $W_{\sigma}(-)$) with respect to $\psi$ (resp. $\psi^{-1}$) of $\pi$ (resp. $\sigma$). Following \cite[\S 8.4]{LPSZ}, we write $\pi = I_{P_{\opn{Kl}}}^G(\tau \boxtimes \theta)$ as the normalised induction from the upper-triangular Klingen parabolic subgroup to $G(\mbb{Q}_p)$ of a smooth irreducible unramified representation $\tau \boxtimes \theta$ of $\opn{GL}_2 \times \opn{GL}_1$ (which we identify with the Levi of $P_{\opn{Kl}}$ in the same way as in \emph{loc.cit.}). We identify the generic representation $\tau$ with the normalised induction $I_{B_{\opn{GL}_2}}^{\opn{GL}_2}(\alpha)$, where $\alpha \colon T_{\opn{GL}_2}(\mbb{Q}_p) \to \mbb{C}^{\times}$ is an unramified character. Let $\alpha_1 = \alpha(p, 1)$ and $\alpha_2 = \alpha(1, p)$. We will consider the following twisted local zeta integral:
\[
Z(\varphi, \lambda, \Phi; \chi, s) \defeq \int_{Z_G(\mbb{Q}_p)N_{H}(\mbb{Q}_p) \backslash H(\mbb{Q}_p)} W_{\varphi}(h) f^{\Phi}(h_1; \chi^{2} \chi_{\pi}\chi_{\sigma}, s) W_{\lambda}(h_2) \chi(\opn{det}h_2) dh
\]
where $\varphi \in \pi$, $\lambda \in \sigma$, $\Phi \in C_c^{\infty}(\mbb{Q}_p^{\oplus 2})$ and $\chi_{\pi}$ and $\chi_{\sigma}$ denote the central characters of $\pi$ and $\sigma$.

Recall that $K_{\opn{B}} \subset \opn{GSp}_4(\mbb{Q}_p)$ denotes the standard upper-triangular Iwahori subgroup, and $K_{\opn{B}, \beta} \defeq K^G_{\opn{Iw}}(p^{\beta})$ is the depth $p^{\beta}$ upper-triangular Iwahori subgroup. Fix an element $\varphi \in \pi^{K_{\opn{B}}}$ which is an eigenvector for the actions of $U_{p, \opn{Si}}$ and $U_{p, \opn{Kl}}$ with eigenvalues $\alpha_1$ and $p^2 \alpha_1 \alpha_2$ respectively (note that $\{ \alpha_1, \alpha_2, \theta(p) \alpha_1, \theta(p) \alpha_2 \}$ are the Satake parameters of $\pi$). 

\begin{definition}
    For $\beta \geq 1$, we define $\varphi'_{\beta} \in \pi^{K_{\opn{B}, \beta}}$ to be the vector
    \[
     \varphi'_{\beta} \defeq (p^{7/2} \alpha_1^2 \alpha_2)^{-\beta} [K_{\opn{B}, \beta} s_{p, \opn{B}}^{\beta} w_G^{\opn{max}} K_{\opn{B}}] \cdot \varphi, \quad t_{p, \opn{B}} = \left( \begin{smallmatrix} p^3 & & & \\ & p^2 & & \\ & & p & \\ & & & 1 \end{smallmatrix} \right), \; s_{p, \opn{B}} = w_{G}^{\opn{max}} t_{p, \opn{B}} (w_{G}^{\opn{max}})^{-1}
    \]
    which is an eigenvector for $U_{p, \opn{Si}}'$ and $U_{p, \opn{Kl}}'$ with eigenvalues $p^{3/2}\alpha_1$ and $p^2 \alpha_1 \alpha_2$ respectively. Moreover, one has $\opn{Tr}_{\beta}(\varphi'_{\beta+1}) = \varphi_{\beta}'$ where $\opn{Tr}_{\beta} \colon \pi^{K_{\opn{B}, \beta+1}} \to \pi^{K_{\opn{B}, \beta}}$ denotes the (unnormalised) trace map. Under the identifications above, we can view $\varphi_{\beta}' \colon G(\mbb{Q}_p) \to \left( \tau \boxtimes \theta \right) \otimes \delta_{P_{\opn{Kl}}}^{1/2}$ as the (unique up to scalar) function which is invariant under $K_{\opn{B}, \beta}$, supported on $P_{\opn{Kl}}(\mbb{Q}_p) K_{\opn{B}, \beta}$, and satisfies $\varphi_{\beta}'(1) = \delta_{P_{\opn{Kl}}}^{1/2} \cdot ( \xi_{\beta} \otimes \theta)$, where $\xi_{\beta} \in \tau^{K^{\opn{GL}_2}_{\opn{Iw}}(p^{\beta})}$ is supported on $B_{\opn{GL}_2}(\mbb{Q}_p) K^{\opn{GL}_2}_{\opn{Iw}}(p^{\beta})$. The functions $\xi_{\beta}$ are compatible under the unnormalised trace maps from level $K^{\opn{GL}_2}_{\opn{Iw}}(p^{\beta+1})$ to $K^{\opn{GL}_2}_{\opn{Iw}}(p^{\beta})$.
\end{definition}

We have the following result: 

\begin{proposition} \label{PrelimFormOfZetapProp}
    Let $\lambda \in \sigma$ be any element which is fixed by $B_{\opn{GL}_2}(\mbb{Z}_p)$ and suppose that $\Phi = \Phi_{p, \chi^{-1}, \chi}$ with notation as in \cite[Definition 7.5]{LPSZ}. For $\beta \gg 2$ sufficiently large (depending on $\lambda$ and $\chi$): 
    \begin{enumerate}
        \item $Z(\gamma \cdot \varphi_{\beta}', \lambda, \Phi; \chi, s) = 0$;
        \item One has
        \[
        Z(\hat{\gamma} \cdot \varphi_{\beta}', \lambda, \Phi; \chi, s) = p^{-4\beta} \cdot C \cdot \frac{L(\tau \times \sigma, \chi, s)}{L(\tau^{\vee} \times \sigma^{\vee}, \chi^{-1}, 1-s) \varepsilon(\tau \times \sigma, \chi, s)} W_{\lambda}(1)
    \]
    where $C \in \mbb{C}^{\times}$ is a non-zero scalar independent of $\chi$, $\beta$, $s$, and $\lambda$. Here $L(\tau \times \sigma, \chi, s)$ (resp. $\varepsilon(\tau \times \sigma, \chi, s)$) denotes the Rankin--Selberg local $L$-factor (resp. Langlands--Deligne local factor) associated with the representation $\tau \boxtimes (\sigma \otimes \chi)$ of $\opn{GL}_2(\mbb{Q}_p) \times \opn{GL}_2(\mbb{Q}_p)$.
    \end{enumerate}
\end{proposition}
\begin{proof}
    Note that the support of $\Phi$ is contained in $\mbb{Z}_p^{\times} \oplus \mbb{Q}_p$. Following exactly the same proof as in \cite[Proposition 8.15]{LPSZ}, with the only modification being that one can substitute the conclusion of \cite[Lemma 8.18]{LPSZ} with ``$k_a(u, v, w) \in K_{\opn{B}, \beta}$'' (since $k_a(u, v, w)$ lies in the upper-triangular depth $p^{\beta}$ Klingen parahoric subgroup if and only if it lies in $K_{\opn{B}, \beta}$ because the $(3, 2)$-th entry of $k_a(u, v, w)$ is zero), one has the formula
    \[
    Z(\gamma \cdot \varphi_{\beta}', \lambda, \Phi; \chi, s) = (\star) \cdot \int_{x \in \mbb{Q}_p^{\times}} W_{\xi_{\beta}}\left( \begin{smallmatrix} x & \\ & 1 \end{smallmatrix} \right) W^{\Phi}\left( \begin{smallmatrix} x & \\ & 1 \end{smallmatrix} \right) W_{\lambda}\left( \begin{smallmatrix} x & \\ & 1 \end{smallmatrix} \right) \chi(x) \tfrac{\theta(x)}{|x|} d^{\times} x
    \]
    where $(\star)$ is a non-zero scalar multiplied by the inverse of the $\gamma$-factor associated with $\tau \boxtimes (\sigma \otimes \chi)$. Note that from \cite[Remark 8.2]{LPSZ} and the definition of $\Phi$, one can easily compute that:
    \[
    W^{\Phi}\left( \begin{smallmatrix} x & \\ & 1 \end{smallmatrix} \right) = |x|^s \int_{a \in \mbb{Q}_p^{\times}} \Phi'(xa, a^{-1})\chi^2\chi_{\pi}\chi_{\sigma}(a) |a|^{2s-1} d^{\times}a = \left\{ \begin{array}{cc} \chi^{-1}(x) & \text{ if } x \in \mbb{Z}_p^{\times} \\ 0 & \text{ otherwise } \end{array} \right.  .
    \]
    We therefore see that $Z(\gamma \cdot \varphi_{\beta}', \lambda, \Phi; \chi, s) = (\star) \cdot W_{\xi_{\beta}}(1) W_{\lambda}(1)$. To conclude the proof of part (1), it suffices to show that $W_{\xi_{\beta}}(1) = 0$. But this can be explicitly computed:
    \begin{align} \label{WxiBZeroEqn}
        W_{\xi_{\beta}}(1) = \int_{u \in \mbb{Q}_p} \xi_{\beta}\left(\eta \left( \begin{smallmatrix} 1 & u \\ & 1 \end{smallmatrix} \right) \right) \psi^{-1}(u) du = \sum_{h \geq \beta} \xi_{\beta}\left( \begin{smallmatrix} -p^{h} & 1 \\ & -p^{-h} \end{smallmatrix} \right) \int_{p^{-h}\mbb{Z}_p - p^{-h+1}\mbb{Z}_p} \psi^{-1}(u) du = 0
    \end{align}
    using the equality $\left( \begin{smallmatrix} -1/m & 1 \\ & -m \end{smallmatrix} \right) \left( \begin{smallmatrix} 1 &  \\ 1/m & 1 \end{smallmatrix} \right) = \left( \begin{smallmatrix}  & 1 \\ -1 &  \end{smallmatrix} \right) \left( \begin{smallmatrix} 1 & m \\ & 1 \end{smallmatrix} \right)$ for $m \in \mbb{Q}_p^{\times}$ (recall that $\beta \geq 2$ so the integral over $p^{-h}\mbb{Z}_p - p^{-h+1}\mbb{Z}_p$ in (\ref{WxiBZeroEqn}) vanishes).

    For the proof of part (2), let $K_{\opn{Kl}, \beta} \subset G(\mbb{Z}_p)$ denote the subgroup of elements which land in $P_{\opn{Kl}}$ modulo $p^{\beta}$. Then one can easily compute that 
    \[
    \left( K_{\opn{Kl}, \beta} \cap \gamma^{-1}H(\mbb{Q}_p)\gamma \right) \backslash K_{\opn{Kl}, \beta} / K_{\opn{B}, \beta} = \{ 1, w_1 \}
    \]
    i.e., this double coset space has two representatives given by $1$ and $w_1$. One can check that $K_{\opn{B}, \beta} \cap \gamma^{-1}H(\mbb{Q}_p)\gamma = K_{\opn{Kl}, \beta} \cap \gamma^{-1}H(\mbb{Q}_p)\gamma$, and 
    \[
    [K_{\opn{Kl}, \beta} \cap \gamma^{-1}H(\mbb{Q}_p) \gamma : w_1 K_{\opn{B}, \beta} w_1^{-1} \cap K_{\opn{Kl}, \beta} \cap \gamma^{-1}H(\mbb{Q}_p) \gamma] = p^{\beta}
    \]
    with a set of representatives of $\left( \gamma K_{\opn{Kl}, \beta} \gamma^{-1} \cap H(\mbb{Q}_p) \right) / \left( \hat{\gamma} K_{\opn{B}, \beta} \hat{\gamma}^{-1} \cap H(\mbb{Q}_p) \right)$ given by
    \[
    \left\{ \delta_a \defeq \left( \begin{smallmatrix} 1 & & & \\ & 1 & a & \\ & & 1 & \\ & & & 1 \end{smallmatrix} \right) : \text{ for } a \in \mbb{Z}/p^{\beta}\mbb{Z}  \right\} .
    \]
    Let $T \colon \pi^{K_{\opn{B}, \beta}} \to \pi^{K_{\opn{Kl}, \beta}}$ denote the unnormalised trace map. Then the above discussion implies that 
    \[
    \gamma \cdot T(\varphi_{\beta}') = \gamma \cdot \varphi_{\beta}' + \sum_{a \in \mbb{Z}/p^{\beta}\mbb{Z}} \delta_a \cdot \hat{\gamma} \cdot \varphi_{\beta}' 
    \]
    hence we see that 
    \[
    Z(\gamma \cdot T(\varphi_{\beta}'), \lambda, \Phi; \chi, s) = Z(\gamma \cdot \varphi_{\beta}', \lambda, \Phi; \chi, s) + \sum_{a \in \mbb{Z}/p^{\beta}\mbb{Z}} Z(\delta_a\hat{\gamma} \cdot \varphi_{\beta}', \lambda, \Phi; \chi, s) = p^{\beta} Z(\hat{\gamma} \cdot \varphi_{\beta}', \lambda, \Phi; \chi, s) .
    \]
    Noting that $T(\varphi_{\beta}')$ is supported on $P_{\opn{Kl}}(\mbb{Q}_p)K_{\opn{Kl}, \beta}$ and its value at the identity is equal to $\varphi_{\beta}'$ evaluated at the identity (because $B(\mbb{Q}_p)K_{\opn{B}, \beta} \cap K_{\opn{Kl}, \beta} = K_{\opn{B}, \beta}$), we can apply \cite[Proposition 8.15]{LPSZ} to find that
    \[
    Z(\gamma \cdot T(\varphi_{\beta}'), \lambda, \Phi; \chi, s) = \frac{p^{3-3\beta}\xi_{\beta}(1)}{(p+1)^2(p-1) \gamma(\tau \times \sigma_{\chi}, s)} W_{\xi}(1) W_{\lambda}(1)
    \]
    where $\gamma(\tau \times \sigma_{\chi}, s)$ is the $\gamma$-factor associated with the representation $\tau \boxtimes \sigma_{\chi} = \tau \boxtimes (\sigma \otimes \chi)$, and $\xi \in \tau$ is the spherical vector normalised so that $\xi(1) = 1$. We now take $C = \tfrac{p^{3}\xi_{\beta}(1)}{(p+1)^2(p-1)} W_{\xi}(1)$ which is independent of $\beta$, $\chi$, $s$, and $\lambda$.
\end{proof}

Let $\{\mu_1, \mu_2\}$ denote the Satake parameters of $\sigma$. 

\begin{corollary} \label{FinalZetapFormulaCorollary}
    Suppose that $\lambda$ is fixed by $K^{\opn{GL}_2}_{\opn{Iw}}(p)$ and that $W_{\lambda}(1) \neq 0$ (e.g., one can take $\lambda$ to be an eigenvector under the action of $[K^{\opn{GL}_2}_{\opn{Iw}}(p) \left( \begin{smallmatrix} p & \\ & 1 \end{smallmatrix}\right) K^{\opn{GL}_2}_{\opn{Iw}}(p)]$). Suppose that $\Phi = \Phi_{p, \chi^{-1}, \chi}$, and set $m \defeq \opn{max}(1, r)$. Then for any $\beta \geq \opn{max}(2m, m+1)$, one has
    \[
    Z(\hat{\gamma} \cdot \varphi_{\beta}', \lambda, \Phi; \chi, s) \doteq p^{-4\beta} \cdot \left\{ \begin{array}{cc} \prod_{i,j=1}^2 \frac{(1-p^{s-1}\alpha_i^{-1}\mu_j^{-1})}{(1-p^{-s}\alpha_i\mu_j)} & \text{ if } r=0 \\ G(\chi^{-1})^{-4} p^{4rs} (\alpha_1 \alpha_2 \mu_1 \mu_2 )^{-2r} & \text{ if } r \geq 1 \end{array} \right.
    \]
    where $\doteq$ means up to a non-zero scalar which is independent of $\beta$, $\chi$, and $s$.
\end{corollary}
\begin{proof}
    By analysing the proof of \cite[Proposition 8.15]{LPSZ} and using the fact that $\Phi$ has support in $\mbb{Z}_p^{\times} \oplus p^{-m}\mbb{Z}_p$, we see that the conclusion of Proposition \ref{PrelimFormOfZetapProp}(2) holds for $\beta \geq \opn{max}(2m, m+1)$. The result now follows from the definition of local Rankin--Selberg $L$-factors and $\varepsilon$-factors for principal series representations.
\end{proof}


\section{Nearly overconvergent modular forms}

In this appendix, we study some integrality properties of nearly overconvergent modular forms which arise from analysing the various bounds in \cite[\S 4]{DiffOps}. We begin with some notation. Let $\mathscr{N}^{\dagger}$ denote the space of nearly overconvergent modular forms as in \emph{op.cit.}; explicitly it can be described as $\mathscr{N}^{\dagger} = \varinjlim_U\opn{H}^0(U, \mathcal{O}_U)$ as $U$ varies over all open subsets of $P^{\opn{an}}_{\opn{GL}_2, \opn{dR}}$ (viewed as a torsor over $\mathcal{X}_{\opn{GL}_2}(K^p K^{\opn{GL}_2}_{\opn{Iw}}(p^{\beta}))$) containing the closure of $\mathcal{IG}^{\opn{tor}}_{\opn{GL}_2, \opn{id}, \beta}$. Note that both the Igusa tower and $\mathscr{N}^{\dagger}$ are independent of $\beta \geq 1$ (see \cite[Lemma 2.4.1]{DiffOps}). 

For any $n \geq 1$, one has a reduction of structure $\pi_n \colon \mathcal{P}_{\opn{GL}_2, \opn{dR}, n} \to U_n$ of $P^{\opn{an}}_{\opn{GL}_2, \opn{dR}}$ to a $\overline{\mathcal{P}}^{\square}_{\opn{GL}_2, n}$-torsor, where $U_n$ is a sufficiently small strict neighbourhood of $\mathcal{X}_{\opn{GL}_2, \opn{id}, \beta}$ in $\mathcal{X}_{\opn{GL}_2}(K^p K^{\opn{GL}_2}_{\opn{Iw}}(p^{\beta}))$. We set:
\begin{align*}
    \mathscr{N}^{\dagger, n\opn{-an}} &\defeq \varinjlim_{V \subset U_n} \opn{H}^0\left(V, (\pi_n)_* \mathcal{O}_{\mathcal{P}_{\opn{GL}_2, \opn{dR}, n}} \right) \\
    \mathscr{N}^{\dagger, n\opn{-an}, \circ} &\defeq \varinjlim_{V \subset U_n} \opn{H}^0\left(V, (\pi_n)_* \mathcal{O}^+_{\mathcal{P}_{\opn{GL}_2, \opn{dR}, n}} \right)
\end{align*}
where both colimits are over all strict neighbourhoods of $\mathcal{X}_{\opn{GL}_2, \opn{id}} = \mathcal{X}_{\opn{GL}_2, \opn{id}, \beta}$ contained in $U_n$. Both $\mathscr{N}^{\dagger, n\opn{-an}}$ and $\mathscr{N}^{\dagger, n\opn{-an}, \circ}$ are independent of $\beta$ (and the choice of reduction of structure). Furthermore, we have $\mathscr{N}^{\dagger, n\opn{-an}} = \mathscr{N}^{\dagger, n\opn{-an}, \circ}[1/p]$. Finally, we set $\mathscr{N}^{n\opn{-an}}_{\opn{HT}} \defeq \opn{H}^0\left( \mathcal{X}_{\opn{GL}_2, \opn{id}}, (\pi_n)_* \mathcal{O}_{\mathcal{P}_{\opn{GL}_2, \opn{dR}, n}} \right)$ and similarly for $\mathscr{N}^{n\opn{-an}, \circ}_{\opn{HT}}$.

\begin{proposition} \label{CepsilonIntegralityAppendixProp}
    Let $0 < \varepsilon \leq 1/(p-1)$ and let $C_{\varepsilon}(\mbb{Z}_p, \mbb{Q}_p)$ denote the Banach space of $\varepsilon$-analytic functions $\mbb{Z}_p \to \mbb{Q}_p$ as in \cite[\S 4.1]{DiffOps}. Set $\upsilon(\varepsilon) \defeq -\opn{log}_p(\varepsilon)$ and let $[\upsilon(\varepsilon)] \in \mbb{Z}$ denote its integer part. Let $\star$ denote the action of $C^{\opn{la}}(\mbb{Z}_p, \mbb{Q}_p)$ on $\mathscr{N}^{\dagger}$ as in Theorem A of \emph{op.cit.}. Then there exist integers $M_1, M_2 \geq 1$ (which are independent of $\varepsilon$) such that:
    \[
    C_{\varepsilon}(\mbb{Z}_p, \mbb{Q}_p) \star \left( \mathscr{N}^{\dagger, ([\upsilon(\varepsilon)]+M_1)\opn{-an}, \circ} \right) \subset p^{-2[\upsilon(\varepsilon)]-M_2} \mathscr{N}^{\dagger, ([\upsilon(\varepsilon)]+M_1)\opn{-an}, \circ}
    \]
    for all $\varepsilon$.
\end{proposition}
\begin{proof}
    We first claim a similar statement holds over the ordinary locus, namely there exist integers $M_1', M_2' \geq 1$ (independent of $\varepsilon$) such that 
    \[
    C_{\varepsilon}(\mbb{Z}_p, \mbb{Q}_p) \star \left( \mathscr{N}_{\opn{HT}}^{([\upsilon(\varepsilon)]+M_1')\opn{-an}, \circ} \right) \subset p^{-2[\upsilon(\varepsilon)]-M_2'} \mathscr{N}_{\opn{HT}}^{([\upsilon(\varepsilon)]+M_1')\opn{-an}, \circ} .
    \]
    For this, let $\nabla$ denote the operator corresponding to the identity function $\mbb{Z}_p \hookrightarrow \mbb{Q}_p$. One has an isometry
    \[
    \mathscr{N}^{m\opn{-an}}_{\opn{HT}} \hookrightarrow \mathscr{N}^{m\opn{-an}}_{\opn{HT}, \infty} \defeq \opn{H}^0\left( \mathcal{P}_{\opn{GL}_2, \opn{dR}, m}^{\infty\opn{-ord}}, \mathcal{O}_{\mathcal{P}_{\opn{GL}_2, \opn{dR}, m}^{\infty\opn{-ord}}} \right)
    \]
    where $\mathcal{P}_{\opn{GL}_2, \opn{dR}, m}^{\infty\opn{-ord}} \defeq \mathcal{P}_{\opn{GL}_2, \opn{dR}, m} \times_{U_m} \mathcal{IG}_{\opn{GL}_2, \opn{id}}^{\opn{tor}}$. Let $|\!|-|\!|_{m}$ denote the Banach norm on $\mathscr{N}^{m\opn{-an}}_{\opn{HT}, \infty}$ with unit ball given by $\mathscr{N}^{m\opn{-an}, \circ}_{\opn{HT}, \infty} \defeq \opn{H}^0\left( \mathcal{P}_{\opn{GL}_2, \opn{dR}, m}^{\infty\opn{-ord}}, \mathcal{O}^+_{\mathcal{P}_{\opn{GL}_2, \opn{dR}, m}^{\infty\opn{-ord}}} \right)$. It suffices to prove a similar statement for these spaces.

    Note that, on $\mathscr{N}^{m\opn{-an}, \circ}_{\opn{HT}, \infty}$ (or more precisely on the versions defined locally with respect to opens in $\mathfrak{IG}_{\opn{GL}_2, \opn{id}}$), the operator $\nabla$ (which is integral) is congruent modulo $p^m$ to a ``nilpotent operator'' $\nabla'$ (as in \cite[\S 4.3]{DiffOps}) extending the Atkin--Serre operator. We now analyse the construction of the $\varepsilon$-analytic action as in Proposition 5.2.3 of \emph{op.cit.}. The construction of this action relies on two key results (Proposition 4.3.4 and Proposition 4.2.4 in \emph{op.cit.}).

    Let $0 < \varepsilon \leq 1/(p-1)$, and let $C_{\varepsilon/2} \in \mbb{R}_{>0}$ be such that 
    \[
    p^{-k\varepsilon/2} \left| \! \left| \binom{\nabla'}{k} \right| \! \right|_m \leq C_{\varepsilon/2} .
    \]
    for all $k \geq 0$. We know that this exists by \cite[Proposition 4.3.4]{DiffOps} -- in fact analysing the proof further (and using the fact that 
    \[
    \binom{k-r}{k_1, \dots, k_\ell} \binom{k}{r} = \binom{k}{k_1, \dots, k_{\ell}, r}
    \]
    so its $p$-adic valuation is bounded by $\opn{log}_p(k)$), we can take $C_{\varepsilon/2}$ to be such that $p^{-k\varepsilon/2+\opn{log}_p(k)} = kp^{-k\varepsilon/2} \leq C_{\varepsilon/2}$ for all $k \geq 1$. It is easy to show that $C_{\varepsilon/2} = \varepsilon^{-1} (2 \opn{ln}(p)^{-1})$ satisfies this property, where $\opn{ln}(-)$ denotes the natural logarithm.

    Take $M_1' \geq 1/(p-1)+\opn{log}_p(2 \opn{ln}(p)^{-1})+1$, then we note that $N_{\varepsilon} \defeq [\upsilon(\varepsilon)] + M_1'$ satisfies $p^{-N_{\varepsilon} + 1/(p-1)} \leq C_{\varepsilon/2}^{-1}$. By analysing the proof of \cite[Proposition 4.2.4]{DiffOps} (and again using the fact that 
    \[
    \binom{k}{a} \binom{a}{k_1, \dots, k_r} = \binom{k}{k-a, k_1, \dots, k_{r}}
    \]
    so its $p$-adic valuation is bounded by $\opn{log}_p(k)$), we see that
    \[
    p^{-k \varepsilon} \left|\! \left| \binom{\nabla}{k} \right|\! \right|_{N_{\varepsilon}} \leq p^{-k\varepsilon/2+\opn{log}_p(k)} C_{\varepsilon/2} \leq C_{\varepsilon/2}^2
    \]
    for all $k \geq 1$. Take $M_2' \geq 2 + 2\opn{log}_p(2\opn{ln}(p)^{-1})$. One can then calculate that
    \[
    C_{\varepsilon/2}^2 \leq p^{2[\upsilon(\varepsilon)] + M_2'} .
    \]
    Putting this all together, we see that for any $f \in C_{\varepsilon}(\mbb{Z}_p, \mbb{Q}_p)$, one has
    \[
    |\!| f \star - |\!|_{[\upsilon(\varepsilon)] + M_1'} \leq p^{2[\upsilon(\varepsilon)] + M_2'}
    \]
    which proves the claim over the ordinary locus.

    Let $M_1 \geq M_1'+1$ and $M_2 \geq M_2'+1$. We can write $\mathscr{N}^{\dagger, ([\upsilon(\varepsilon)]+M_1)\opn{-an}}$ as an injective direct limit of Banach spaces $V_0 \to V_1 \to \cdots \to V_{\infty} = \mathscr{N}^{([\upsilon(\varepsilon)]+M_1)\opn{-an}}_{\opn{HT}}$ satisfying the property in \cite[Proposition 4.4.1]{DiffOps}. Let $\gamma < \varepsilon$ be such that $\upsilon(\gamma) + M_1' \leq \upsilon(\varepsilon) + M_1$ and $[\upsilon(\gamma)] = [\upsilon(\varepsilon)]$. Then we know that we have an action of $C_{\gamma}(\mbb{Z}_p, \mbb{Q}_p)$ on $\mathscr{N}^{([\upsilon(\varepsilon)] + M_1)\opn{-an}}_{\opn{HT}} = V_{\infty}$, hence Proposition 4.4.1 in \emph{op.cit.} implies that $\mathscr{N}^{\dagger, ([\upsilon(\varepsilon)] + M_1)\opn{-an}}$ carries an action of $C_{\varepsilon}(\mbb{Z}_p, \mbb{Q}_p)$. Furthermore, given $x \in \mathscr{N}^{\dagger, ([\upsilon(\varepsilon)] + M_1)\opn{-an}, \circ}$ and $f \in C_{\varepsilon}(\mbb{Z}_p, \mbb{Q}_p)$, we have that $f \star x$ lies in $p^{-(2[\upsilon(\gamma)] + M_2')}\mathscr{N}^{([\upsilon(\varepsilon)]+M_1)\opn{-an}, \circ}_{\opn{HT}} = p^{-(2[\upsilon(\varepsilon)] + M_2')}\mathscr{N}^{([\upsilon(\varepsilon)]+M_1)\opn{-an}, \circ}_{\opn{HT}}$, hence we must have $f \star x \in p^{-(2[\upsilon(\varepsilon)] + M_2)}\mathscr{N}^{\dagger, ([\upsilon(\varepsilon)] + M_1)\opn{-an}, \circ}$. This concludes the proof.
\end{proof}


\newcommand{\etalchar}[1]{$^{#1}$}
\renewcommand{\MR}[1]{}
\providecommand{\bysame}{\leavevmode\hbox to3em{\hrulefill}\thinspace}
\providecommand{\MR}{\relax\ifhmode\unskip\space\fi MR }
\providecommand{\MRhref}[2]{%
  \href{http://www.ams.org/mathscinet-getitem?mr=#1}{#2}
}
\providecommand{\href}[2]{#2}

\Addresses

\end{document}